\documentclass[reqno]{amsart}

\usepackage{amsmath,amssymb,latexsym,amsthm}
\usepackage{graphicx}
\usepackage{color}
\usepackage{subfigure}
\usepackage{multirow}
\usepackage{mathptmx}
\usepackage{url,bm}
\usepackage{enumerate}

\newtheorem{theorem}{Theorem}[section]
\newtheorem{corollary}{Corollary}[section]
\newtheorem{proposition}{Proposition}[section]
\newtheorem{lemma}{Lemma}[section]
\newtheorem{remark}{Remark}[section]
\newtheorem{Assumption}{Assumption}[section]
\newtheorem{Condition}{Condition}[section]

\newcommand{\Second}{\textup{I}\!\textup{I}}
\numberwithin{equation}{section}

\title[Asymptotic Analysis of LLE]{Think globally, fit locally under the Manifold Setup: Asymptotic Analysis of Locally Linear Embedding}
\author{Hau-Tieng~Wu}
\address{Hau-Tieng Wu\\
Department of Mathematics and Department of Statistical Science\\
Duke University}
\email{hauwu@math.duke.edu}

\author{Nan~WU}
\address{Nan Wu\\
Department of Mathematics\\
University of Toronto}
\email{n.wu@mail.utoronto.ca}

\begin{document}

%
%
%
%
%
%

\maketitle

\begin{abstract}
Since its introduction in 2000, locally linear embedding (LLE) has been widely applied in data science. We provide an asymptotical analysis of the LLE under the manifold setup. We show that for a general manifold, asymptotically we may not obtain the Laplace-Beltrami operator, and the result may depend on the non-uniform sampling, unless a correct regularization is chosen. We also derive the corresponding kernel function, which indicates that the LLE is not a Markov process. A comparison with the other commonly applied nonlinear algorithms, particularly the diffusion map, is provided, and its relationship with the locally linear regression is also discussed.
\end{abstract}   

%
%

\section{Introduction}\label{Section:Introduction}

Dimension reduction is a fundamental step in data analysis. In past decades, due to the demanding need for analyzing 
the large scale, massive and complicated datasets accompanying technological advances, there have been many efforts to solve this problem from different angles. The resulting algorithms could be roughly classified into two types, linear and nonlinear. Linear methods include principal component analysis (PCA), multidimensional scaling, and others. Nonlinear methods include ISOMAP \cite{Tenenbaum_deSilva_Langford:2000}, locally linear embedding (LLE) \cite{Roweis_Saul:2000} and its variations like Hessian LLE \cite{Donoho_Grimes:2003} and modified LLE \cite{Zang_Wang:2006}, eigenmap \cite{Belkin_Niyogi:2003}, diffusion map (DM) \cite{Coifman_Lafon:2006}, local tangent space alignment \cite{Zhang_Zha:2004}, vector diffusion map \cite{Singer_Wu:2012,Singer_Wu:2016}, horizontal diffusion map \cite{Gao:2016}, maximal variance unfolding \cite{Weinberger_Saul:2006}, and t-distributed stochastic neighbor embedding \cite{VanderMaaten_Hinton:2008}, to name a few. 

The subject of this paper, the LLE, was published in Science in 2000 \cite{Roweis_Saul:2000}. It has been cited almost 10,000 times, according to the Google Scholar as of mid-January, 2017. The algorithm is designed to be intuitive and simple. It has also been found to be efficient and practical. It contains two main parts. First, for each data point, determine its nearest neighbors, and catch the local geometric structure of the dataset through finding the barycenter coordinate for those neighboring points by a regularization. {This is the ``fit locally'' part of the LLE. Second,} by viewing the barycenter coordinates as the ``weights'' for the neighboring points, the eigenvectors and eigenvalues of the associated ``affinity matrix'' are evaluated to organize the data points. {This is the ``think globally'' part of the LLE.}
However, unlike the fruitful theoretical results from discussing the diffusion-based approach like DM \cite{Belkin_Niyogi:2005,Hein_Audibert_Luxburg:2005,Singer:2006,Gine_Koltchinskii:2006,Coifman_Lafon:2006,Smolyanov:2007,Belkin_Niyogi:2007,VonLuxburg_Belkin_Bousquet:2008,Wang:2015,GarciaTrillos_Slepcev:2015,Singer_Wu:2016,ElKaroui_Wu:2016b}, to the best of our knowledge, a systematic analysis of the LLE algorithm has not been undertaken, except an ad hoc argument shown in \cite{Belkin_Niyogi:2003} based on some conditions. %

{The main contribution of this paper is analyzing the ``fit locally'' part of the LLE. Based on a careful analysis of the barycentric coordinate by the covariance matrix analysis, we provide an asymptotic {\em pointwise convergence} analysis of the LLE under the manifold setup.\footnote{While it is not explored in this paper, we mention that based on the established pointwise convergence, we could further understand the ``think globally'' part of the LLE algorithm from the spectral geometry viewpoint \cite{Berard:1986,Berard_Besson_Gallot:1994}.} }
Although it is widely believed that under the manifold setup, asymptotically the LLE should lead to the Laplace-Beltrami operator, in this paper, we show that it might not always be the case. It fundamentally depends on the geometric structure of the data set. Specifically, under the assumption that the point cloud is (non-) uniformly sampled from a low dimensional manifold isometrically embedded in the Euclidean space, we show that the asymptotical behavior of the LLE depends on the regularization. If the regularization is chosen properly, we obtain the Laplace-Baltrami operator, even if the sampling is non-uniform. If the regularization is not chosen properly, the acquired information will be contaminated by the extrinsic information (the second fundamental form), and we even obtain the fourth order differential operator in some extreme cases. To catch the dependence on the extrinsic information, we carefully analyze the ``local covariance structure'' of the dataset up to the higher order.
{One key step toward the analysis is establishing the kernel function associated with the LLE that comes from the barycentric coordinate estimation. Via the established kernel function, we have a direct comparison of the LLE and other relevant nonlinear machine learning algorithms, for example, the eigenmap and DM.
Unlike the eigenmap or DM, the LLE in general is not a diffusion process on the dataset, and the convergence rate might be different, depending on the regularization. }
In the end, we link the LLE back to the widely applied kernel regression technique, the locally linear regression (LLR), and the error in variable problem.

The paper is organized as follows. In Section \ref{Section:Review}, we review the LLE algorithm. In Section \ref{Section:Asymptotics}, we provide the asymptotical analysis of the LLE under the manifold setup. 
In Section \ref{Section:Numerics}, we provide numerical simulations to support our theoretical findings. 
{The relationship between two common nearest neighbor search schemes is discussion in Section \ref{Section:KNNvsEpsilon}.}
The relationship between the LLE, the LLR, and the shrinkage scheme for the high dimensional covariance matrix are discussed in \ref{Section:LLR}. 
The discussions is shown in Section \ref{Section:Conclusion}. The technical proofs of the theorems are included in the Appendix. The perturbation argument of the eigenvalues and eigenvectors of a symmetric matrix is summarized in Section \ref{Section:Perturbation}. The statement of technical lemmas for the proof is given in Section \ref{Section:StatementTechnicalLemmas}. The covariance structure analysis is provided in Section \ref{Section:Proof:PropositionForCovariance}. The proofs of the main Theorems are given in Appendices \ref{Section:Theoremt1} and  \ref{Section:Theoremt0}. The technical lemmas for the theorems are given in Section \ref{Section:LemmasProof}.

Here we fix the notations used in this paper. For $d\in \mathbb{N}$, $I_{d\times d}$ means the identity matrix of size $d\times d$. For $n\in\mathbb{N}$, denote $\boldsymbol{1}_{n}$ to be the $n$-dim vector with all entries $1$. For $\epsilon\geq 0$, denote $B^{\mathbb{R}^p}_{\epsilon}(x):=\{y\in \mathbb{R}^p|\,\|x-y\|_{\mathbb{R}^p}\leq \epsilon\}$. Denote $e_i=[0,\cdots ,1, \cdots 0]^\top \in\mathbb{R}^p$ to be the unit $p$-dim vector with $1$ in the i-th entry. 
For $p,r\in \mathbb{N}$ so that $r\leq p$, denote $J_{p,r}\in \mathbb{R}^{p\times r}$ so that the $(i,i)$ entry is $1$ for $i=1,\ldots,r$, and zeros elsewhere and { denote $\bar J_{p,r}\in \mathbb{R}^{p\times r}$ so that the $(p-r+i,i)$ entry is $1$ for $i=1,\ldots,r$, and zeros elsewhere.} $I_{p,r}:=J_{p,r}J_{p,r}^\top$ is a $p\times p$ matrix so that the $(i,i)$-th entry is $1$ for $i=1,\ldots,r$ and $0$ elsewhere; and $\bar I_{p,r}:=\bar J_{p,r}\bar J_{p,r}^\top$ is a $p\times p$ matrix so that the $(i,i)$-th entry is $1$ for $i=p-r+1,\ldots,p$ and $0$ elsewhere.
Denote $S(p)$ to be the set of real symmetric matrix of size $p\times p$, $O(p)$ to be the orthogonal group in dimension $p$, and $\mathfrak{o}(p)$ to be the set of anti-symmetric matrix of size $p\times p$. For $M\in \mathbb{R}^{p\times p}$, denote $M^\top$ to be the transpose of $M$ and $M^\dagger$ to be the Moore-Penrose pseudo-inverse of $M$. 
For $a,b\in\mathbb{R}$, we use $a\wedge b:=\min\{a,b\}$ and $a\vee b:=\max\{a,b\}$ to simplify the notation.
We summarize the commonly used notations for the asymptotical analysis in Table \ref{Table:Notations} for the convenience of the readers.

\begin{table}\label{Table:Notations}
\caption{Commonly used notations in this paper.}
\begin{tabular}{|l|l|} 
\hline $Symbol$ & $Meaning$\\ 
\hline\hline 
$p$ & Dimension of the ambient space \\ 
$d$ & Dimension of the low-dimensional Riemannian manifold \\ 
$(M,g)$ &  $d$-dimensional smooth Riemannian manifold \\ 
$dV$ & Riemannian volume form of $(M,g)$ \\ 
$\exp_{x}$ & Exponential map at $x$\\ 
$T_{x}M$ & Tangent space of $M$ at $x$ \\ 
$\texttt{Ric}_{x}$ & Ricci curvature tensor of $(M,g)$ at $x$ \\ 
$\iota$, $\iota_*$ & Isometric embedding of $M$ into $\mathbb{R}^p$ and its differential \\ 
$\Second_{x}$ & Second fundamental form of the embedding $\iota$ at $x$ \\
$P$ & Probability density function on $\iota(M)$\\
$n\in\mathbb{N}$ & Number of data points sampled from $M$ \\ 
$\mathcal{X}=\{z_i\}_{i=1}^n$ & Point cloud sampled from $\iota(M)\subset \mathbb{R}^p$ \\ 
$w_{z_k} \in \mathbb{R}^N$ & Barycentric coordinates of $z_k$ with respect to data points in \\
                    & the $\epsilon$-neighborhood \\
\hline 
\end{tabular}
\end{table}

\section{Review of the Locally Linear Embedding}\label{Section:Review}

We start by summarizing the LLE algorithm. Suppose $\mathcal{X}=\{z_i\}_{i=1}^n \subset \mathbb{R}^p$ is the provided dataset, or the point cloud.

\begin{enumerate}
\item Fix $\epsilon>0$. For each $z_k\in \mathcal{X}$, denote $\mathcal{N}_{z_k}:=B^{\mathbb{R}^p}_{\epsilon}(z_k) \cap (\mathcal{X} \setminus\{z_k\}) =\{ z_{k,j}\}_{j=1}^{n_k}$, where $n_k\in\mathbb{N}$ is the number of points in $\mathcal{N}_{z_k}$. $\mathcal{N}_{z_k}$ is called the $\epsilon$-radius neighborhood of $z_k$. Alternatively, we can also fix a number $K$, and choose the K nearest points of $z_k$. {This is called the $K$ nearest neighbors (KNN) scheme. While the $\epsilon$-radius neighborhood scheme and the KNN scheme are closely related, they are different. In this paper, we study the LLE with the $\epsilon$-radius neighborhood scheme, and postpone the discussion of the relationship between these two schemes to Section \ref{Section:KNNvsEpsilon}.}

\item For each $z_k\in \mathcal{X}$, find its barycentric coordinate associated with $\mathcal{N}_{z_k}$ by
\begin{equation}
w_{z_k}= \arg\min_{w\in \mathbb{R}^{n_k},w^\top  \boldsymbol{1}_{n_k}=1} \|z_k-\sum_{j=1}^{n_k}w(j)z_{k,j} \|^2 \in \mathbb{R}^{n_k}. \label{Definition:FindBarycentric}
\end{equation}
Notice that $w_{z_k}$ satisfies $w_{z_k}^\top  \boldsymbol{1}_{n_k}=\sum_{j=1}^{n_k} w_{z_k}(j)=1$.

\item Define a $n \times n$ matrix $W${, called the {\em LLE matrix},} by
\begin{equation}\label{Definition:Wmatrix:LLE}
W_{k,l} = \left\{ 
\begin{array}{ll} 
w_{z_k}(j) & \mbox{if $z_l=z_{k,j} \in \mathcal{N}_{z_k}$};\\ 
0 & \mbox{otherwise}. 
\end{array} \right.
\end{equation}

\item To reduce the dimension of $\mathcal{X}$, it is suggested in \cite{Roweis_Saul:2000} to embed $\mathcal{X}$ into a low dimension Euclidean space
\begin{equation}
z_k\mapsto Y_k=[v_1(k) ,\cdots, v_{\ell}(k)]^\top  \in \mathbb{R}^{\ell},
\end{equation}
for each $z_k\in \mathcal{X}$, where $\ell$ is the dimension of the embedded points chosen by the {user}, and $v_1, \cdots ,v_{\ell}\in \mathbb{R}^n$ are eigenvectors of $(I-W)^\top (I-W)$ corresponding to the $\ell$ smallest eigenvalues. 
Note that this is equivalent to {minimizing} the cost function 
$\sum_{k=1}^n \|Y_k-\sum_{l=1}^{n}W_{k,l}Y_l\|^2$.
\end{enumerate}

Although the algorithm looks relatively simple, there are actually several details that should be discussed prior to the asymptotical analysis. To simplify the discussion, we focus on one point $z_k\in \mathcal{X}$ and assume that there are $N$ data points in $\mathcal{N}_{z_k}=\{z_{k,1}, \cdots, z_{k,N}\}$. 
To find the barycentric coordinate of $z_k$, we define the \textit{local data matrix} associated with $\mathcal{N}_{z_k}$:
\begin{equation}
G_n:=\begin{bmatrix}
| & & |\\
z_{k,1}-z_k & \ldots & z_{k,N}-z_k\\
| & & | 
\end{bmatrix}\in \mathbb{R}^{p \times N}.
\end{equation}
It is important to note that $G_n$ depends not only on $n$, but also $\epsilon$ and $z_k$. However,  we only keep $n$ to make the notation easier. The other notations in this section are simplified in the same way. 
Minimizing (\ref{Definition:FindBarycentric}) is equivalent to minimizing the functional $w^\top G_n^\top G_nw$ over $w\in\mathbb{R}^N$ under the constraint $w^\top \boldsymbol{1}_N=1$. 
Here, $G_n^\top G_n$ is the Gramian matrix associated with the dataset $\{z_{k,1}-z_k , \cdots, z_{k,N}-z_k \}$. 
In general, $G_n^\top G_n$ might be singular, and it is suggested in \cite{Roweis_Saul:2000} to stabilize the algorithm by regularizing the equation by
\begin{equation}
(G_n^\top G_n+c I_{N\times N})y=\boldsymbol{1}_N\,, \label{Section2:GGyN}
\end{equation}
where $c>0$ is the regularizer chosen by the user. For example, in \cite{Roweis_Saul:2000}, $c$ is suggested to be $\frac{\delta}{N}$, where $0<\delta < \|G_n\|^2_F$ is chosen by the user and $\|G_n\|_F$ is the Frobenius norm of $G_n$. It has been observed that the LLE is sensitive to the choice of the regularizer (see, for example, \cite{Zang_Wang:2006}). We will later quantify this dependence under the manifold setup. Using the Lagrange multiplier method, the minimizer is
\begin{equation}
w_{n}=\frac{y_{n}}{y_{n}^\top \boldsymbol{1}_N},
\end{equation}
where $y_{n}$ is the solution of (\ref{Section2:GGyN}). We will consider the regularized equation (\ref{Section2:GGyN}) in the following discussion.

{ Next, we explicitly express $w_{n}$,} which is the essential step toward the asymptotical analysis. Suppose $\texttt{rank}(G_n^\top G_n)=r_n$. Note that $r_n =\texttt{rank}(G_nG_n^\top ) =\texttt{rank}(G_n)\leq p$, so $G_n^\top G_n$ is singular when $p < N$. Moreover, $G_n^\top G_n$ is positive semidefinite. 
Denote the eigen-decomposition of $G_n^\top G_n$ as $V_n\Lambda_n V_n^\top $, where 
\begin{equation}
\Lambda_n=\texttt{diag}(\lambda_{n,1},\lambda_{n,2},\ldots,\lambda_{n,N}),
\end{equation} 
{$\lambda_{n,1} \geq \lambda_{n,2} \geq \cdots \geq \lambda_{n,r_n}> \lambda_{n,r_n+1}= \cdots =\lambda_{n,N}=0$,} and
\begin{equation}
V_n=\begin{bmatrix}| && |\\
v_{n,1}&\ldots&v_{n,N}\\ |&&|
\end{bmatrix}\in O(N).
\end{equation}
Clearly, { $\{v_{n,i}\}_{i=r_n+1}^N$} form an orthonormal basis of the null space of $\texttt{Null}(G_n^\top G_n)$, which is equivalent to $\texttt{Null}(G_n)$. 
Then (\ref{Section2:GGyN}) is equivalent to solving
\begin{equation}\label{regularized equation}
V_n (\Lambda_n+c I_{N\times N})V_n^\top y=\boldsymbol{1}_N\,,
\end{equation}
and the solution is
\begin{align}
y_{n}&\,=V_n (\Lambda_n+c I_{N\times N} )^{-1}V_n^\top  \boldsymbol{1}_N\nonumber\\
&\,=c^{-1}\boldsymbol{1}_N+V_n\big[(\Lambda_n+c I_{N\times N})^{-1}-c^{-1}I_{N\times N}\big]V_n^\top  \boldsymbol{1}_N\label{Algorithm:solutiony}\,.
\end{align}
Therefore, 
{
\begin{equation}
w^\top_n=\frac{\boldsymbol{1}^\top_N+\boldsymbol{1}^\top_N V_n\big[c (\Lambda_n+c I_{N\times N})^{-1}-I_{N\times N}\big]V_n^\top  }
{N+\boldsymbol{1}^\top_N V_n\big[c (\Lambda_n+c I_{N\times N})^{-1}-I_{N\times N}\big]V_n^\top \boldsymbol{1}_N}.\label{Definition:Wf:FirstVersion}
\end{equation}
}

Without recasting (\ref{Definition:Wf:FirstVersion}) into a proper form, it is not clear how to capture the geometric information contained in (\ref{Definition:Wf:FirstVersion}). Observe that while $G_n^\top G_n$ is the Gramian matrix, $G_nG_n^\top $ is related to the sample covariance matrix associated with $\mathcal{N}_{z_k}$. We call $\frac{1}{n}G_nG_n^\top $ the \textit{local sample covariance matrix}.\footnote{The usual sample covariance matrix associated with $\mathcal{N}_{z_k}$ is defined as $\frac{1}{n-1}\sum_{j=1}^N(z_{k,j}-\mu_k)(z_{k,j}-\mu_k)^\top $, where $\mu_k=\frac{1}{n}\sum_{j=1}^Nz_{k,j}$.} Clearly, $r_n \leq p$ and $G_nG_n^\top $ and $G_n^\top G_n$ share the same positive eigenvalues, $\lambda_{n,1} \cdots  \lambda_{n,r_n}$. Denote the eigen-decomposition of 
$G_nG_n^\top$ as $U_n\bar\Lambda_n U_n^\top$, where $U_n\in O(p)$ and $\bar\Lambda_n$ is a $p\times p$ diagonal matrix. By a direct calculation, the first $r_n$ columns of $V_n$ are related to $U_n$ by
\begin{equation}
V_{n}J_{N,r_n}=G_n^\top U_{n}(\bar\Lambda^\dagger_n)^{1/2}J_{p,r_n} \label{Equation:vni_uni_relationship}\,,
\end{equation} 
where $V_n=[V_{n}J_{N,r_n} | V_n\bar J_{N,N-r_n}]$. Since $\big(\Lambda_n+c I_{N\times N}\big)^{-1}-c^{-1}I_{N\times N}$ has only $r_n$ non-zero diagonal entries, { based on (\ref{Algorithm:solutiony})}, we have 
\begin{align*}
y_{n}^\top =&\,c^{-1}\boldsymbol{1}_N^\top+\boldsymbol{1}_N^\top V_n\big[(\Lambda_n+c I_{N\times N})^{-1}-c^{-1}I_{N\times N}\big]V_n^\top \\
=&\,c^{-1}\boldsymbol{1}_N^\top+\boldsymbol{1}_N^\top G_n^\top U_{n}(\bar\Lambda^\dagger_n)^{1/2}J_{p,r_n}J_{p,r_n}^\top  \big[(\bar\Lambda_n+c I_{p\times p})^{-1}-c^{-1}I_{p\times p}\big] J_{p,r_n}  J_{p,r_n}^\top   (\bar\Lambda^\dagger_n)^{1/2}U_{n}^\top   G_n \,.\nonumber
\end{align*}
{
Note that we have
\begin{align}
&U_{n}(\bar\Lambda^\dagger_n)^{1/2}J_{p,r_n}J_{p,r_n}^\top  \big[(\bar\Lambda_n+c I_{p\times p})^{-1}-c^{-1}I_{p\times p}\big] J_{p,r_n}  J_{p,r_n}^\top   (\bar\Lambda^\dagger_n)^{1/2}U_{n}^\top\nonumber\\
=\,&-c^{-1}U_{n}J_{p,r_n}J_{p,r_n}^\top  (\bar\Lambda_n+c I_{p\times p})^{-1} J_{p,r_n}  J_{p,r_n}^\top U_{n}^\top\,,
\end{align}
}
which could be understood as a ``regularized pseudo-inverse''. Specifically, when $c$ is small, we have 
\begin{equation}
U_nJ_{p,r_n}J_{p,r_n}^\top  (\bar\Lambda_n+c I_{p\times p})^{-1} J_{p,r_n}  J_{p,r_n}^\top U_n^\top \approx (G_nG_n^\top)^\dagger.
\end{equation}
Denote
\begin{equation}\label{Definition:Irho:Soft}
\mathcal{I}_c(G_nG_n^\top):=U_nJ_{p,r_n}J_{p,r_n}^\top  (\bar\Lambda_n+c I_{p\times p})^{-1} J_{p,r_n}  J_{p,r_n}^\top U_n^\top.
\end{equation}
{
Hence, we can recast (\ref{Algorithm:solutiony}) and (\ref{Definition:Wf:FirstVersion}) into
\begin{align}
y_{n}^\top =&\,c^{-1}\boldsymbol{1}_N^\top- c^{-1} \boldsymbol{1}_N^\top G_n^\top \mathcal{I}_c(G_nG_n^\top)G_n
\end{align}
and
\begin{align}
w^\top _n&\,=\frac{\boldsymbol{1}_N^\top-\boldsymbol{1}_N^\top G_n^\top \mathcal{I}_c(G_nG_n^\top)   G_n}
{N-\boldsymbol{1}_N^\top G_n^\top \mathcal{I}_c(G_nG_n^\top)  G_n\boldsymbol{1}_N}=\frac{\boldsymbol{1}_N^\top-\mathbf{T}^\top_{n,z_k}  G_n}
{N-\mathbf{T}^\top_{n,z_k}  G_n\boldsymbol{1}_N}\,,\label{Expansion:LLEweightedKernel}
\end{align}
}
where
\begin{align}
\mathbf{T}_{n,z_k}:= \mathcal{I}_c(G_nG_n^\top)  G_n\boldsymbol{1}_N\label{Definition:Tn}
\end{align}
is chosen in order to have a better geometric insight into the LLE algorithm. 
We now summarize the expansion of the barycentric coordinate.

\begin{proposition} \label{Proposition:p1}
Take a data set $\mathcal{X}=\{z_i\}_{i=1}^n \subset \mathbb{R}^p$. Suppose there are $N$ data points in the $\epsilon$ neighborhood of $z_k$, namely $\{z_{k,1}, \cdots, z_{k,N}\} \subset B^{\mathbb{R}^p}_{\epsilon}(z_k) \cap (\mathcal{X} \setminus\{z_k\})$. Assume $p < N$. Let $G_n^\top G_n$ be the Gramian matrix associated with $\{z_{k,1}-z_k, \cdots, z_{k,N}-z_k\}$ and let $\{{\lambda}_{n,i}\}_{i=1}^r$ and $\{u_{n,i}\}_{i=1}^r$, where $r\leq p$ is the rank of $G_n^\top G_n$, be the nonzero eigenvalues and the corresponding orthonormal eigenvectors of $G_nG_n^\top $ satisfying (\ref{Equation:vni_uni_relationship}). With $\mathbf{T}_{n,z_k}$ defined in (\ref{Definition:Tn}), the barycentric coordinates of $z_k$ coming from the regularized equation (\ref{Section2:GGyN}) is 
\begin{equation}
w^\top _n
=\frac{\boldsymbol{1}_N^\top - \mathbf{T}^\top _{n,z_k}G_n}
{N -  \mathbf{T}^\top _{n,z_k}G_n\boldsymbol{1}_N}.\label{Definition:Wf:SecondVersion}
\end{equation}
\end{proposition}

{
\begin{remark}
The denominator $N - \mathbf{T}^\top _{n,z_k}G_n\boldsymbol{1}_N$ is the sum of all entries of the numerator $\boldsymbol{1}_N^\top - \mathbf{T}^\top _{n,z_k}G_n$. We could thus view the LLE matrix defined in (\ref{Definition:Wmatrix:LLE}) as a ``normalized kernel'' defined on the point cloud. However, we mention that while all entries of $w_n$ are summed to $1$, the vector $\boldsymbol{1}_N^\top - \mathbf{T}^\top _{n,z_k}G_n$ might have negative entries, depending on the vector $\mathbf{T}^\top _{n,z_k}$. Therefore, in general, $W$ cannot be understood as a transition matrix. 
\end{remark}
}

How the LLE achieves the nonlinear dimension reduction and captures the geometric structure of the point cloud could thus be understood by understanding $\mathbf{T}_{n,z_k}$. 
In the next section, we will show that under the manifold assumption, $\mathbf{T}_{n,z_k}$ is intimately related to the ``normal bundle'' associated with the manifold, and see how the selection of { $c$ }influences the convergence behavior.

\section{Asymptotic behavior of LLE}\label{Section:Asymptotics}

In this section, we focus on the asymptotic analysis of LLE under the manifold setup. We start by introducing the manifold setup and assumptions for the analysis.

\subsection{Manifold Setup}\label{Section:ManifoldSetup}
Let $X$ be a $p$-dimensional random vector. 
Assume that the range of $X$ is supported on a $d$-dimensional compact, smooth Riemannian manifold $(M,g)$ isometrically embedded in $\mathbb{R}^p$ via $\iota:M\hookrightarrow \mathbb{R}^p$, where we assume that $M$ is boundary-free to simplify the discussion. 
Denote $d(\cdot,\cdot)$ to be the geodesic distance associated with $g$. 
For the tangent space $T_yM$ on $y\in M$, denote $\iota_*T_{y}M$ to be the embedded tangent space in $\mathbb{R}^p$. Denote $\exp_{y}:T_yM\to M$ to be the exponential map at $y$. {Denote $\text{Ric}$ to be the Ricci curvature, $\nabla$ to be the covariant derivative and $\Delta$ to be the Laplace-Beltrami operator.} Unless otherwise stated, in this paper we will carry out the calculation with the normal coordinate. 

{ Let $z=\iota(y)$. Denote $\Second_{y}$ to be the second fundamental form of $\iota$ at $y$. Denote the normal space at $z$ as $(\iota_*T_{y}M)^\bot$, which could be viewed as $\mathbb{R}^{p-d}$. Recall that the second fundamental form at $y$ is a symmetric bilinear map from $T_{y}M \times T_{y}M$ to $(\iota_*T_{y}M)^\bot$. If $S^{d-1}$ is the $(d-1)$-dim unit sphere in $T_{y}M$ and $\theta= (\theta^1, \cdots, \theta^d)\in S^{d-1}$, then for a fixed $e_k \in (\iota_*T_{y}M)^\bot$, we can expand $e_k^\top\Second_y(\theta,\theta)$ as $\sum_{i,j=1}^d p^k_{ij}\theta^i\theta^j$, where $p^k_{ij}\in \mathbb{R}$. The eigenvalues of the matrix $A^{(k)}\in \mathbb{R}^{d\times d}$, where $A^{(k)}_{ij}=p^k_{ij}$ for $i,j=1,\ldots,d$, are the {\em principal curvatures} at $z$ in the direction $e_k$. }

We now quickly summarize how the probability density function (p.d.f.) associated with $X$ is defined \cite{Cheng_Wu:2013}. The random vector $X:\Omega \rightarrow \mathbb{R}^p$ is a measurable function with respect to the probability space $(\Omega,\mathcal{F},\mathcal{P})$, where $\mathcal{P}$ is the probability measure defined on the sigma algebra $\mathcal{F}$ in $\Omega$. 
By assumption, the range of $X$ is supported on $\iota(M)$. Let $\tilde{\mathcal{B}}$ be the Borel sigma algebra of $\iota(M)$, and denote by $\tilde{\mathcal{P}}_X$ the probability measure defined on $\tilde{\mathcal{B}}$ that is induced from $P$. 
{If $\tilde{\mathcal{P}}_{X}$ is absolutely continuous with respect to the volume density on $\iota(M)$, by the Radon-Nikodym theorem, $ d \tilde{\mathcal{P}}_{X}(z)=P(z)\iota_* d V(z)$, where $ d V$ is the volume form associated with the metric $g$, $\iota_* d V(z)$ is the induced measure on $\iota(M)$ via $\iota$, and $P$ is a non-negative measurable function defined on $\iota(M)$. 
We call $P$ {\it the p.d.f. of $X$ on $M$}. When $P$ is constant, we call $X$ a \textit{uniform} random sampling scheme; otherwise it is \textit{nonuniform}.}

To facilitate the discussion and the upcoming analysis, we make the following assumption about the random vector $X$ and the regularity of the associated p.d.f..

\begin{Assumption}\label{AssumptionPDF}
Assume $\tilde{\mathcal{P}}_X$ is absolutely continuous with respect to the volume density on $\iota(M)$ so that $ d \tilde{\mathcal{P}}_X=P\iota_* d V$, where $P$ is a measurable function. We further assume that $P \in C^5(\iota(M))$ and there exist $P_m>0$ and $P_M\geq P_m$ so that $P_m \leq P(x) \leq P_M < \infty$ for all $x\in \iota(M)$. 
\end{Assumption}

Let $\mathcal{X}=\{\iota(x_i)\}_{i=1}^n\subset \iota(M)\subset \mathbb{R}^p$ denote a set of identical and independent (i.i.d.) random samples from $X$, where $x_i\in M$. We could then run the LLE on $\mathcal{X}$. For $\iota(x_k) \in \mathcal{X}$ and $\epsilon>0$, we have $\mathcal{N}_{\iota(x_k)}:=\{\iota(x_{k,1}), \cdots, \iota(x_{k,N})\} \subset B^{\mathbb{R}^p}_{\epsilon}(\iota(x_k)) \cap (\mathcal{X} \setminus \{\iota(x_k)\})$. Take $G_n\in \mathbb{R}^{p\times N}$ to be the local data matrix associated with $\mathcal{N}_{\iota(x_k)}$ and evaluate the barycentric coordinate $w_n=[w_{n,1}, \cdots ,w_{n,N}] ^\top \in \mathbb{R}^N$. Again, although $G_n$ and $w_n$ depend on $\epsilon$, $n$, and $x_k$, to ease the notation, we only keep $n$ to indicate that we have finite sampling points.

\subsection{Local covariance structure and local PCA}\label{Section:LocalPCA}

We call 
\begin{equation}
C_x:=\mathbb{E}[(X-\iota(x))(X-\iota(x))^{\top}\chi_{B_{\epsilon}^{\mathbb{R}^p}(\iota(x))}(X)]\in\mathbb{R}^{p\times p}
\end{equation}
the \textit{local covariance matrix} at $\iota(x)\in \iota(M)$, which is the covariance matrix associated with the local PCA \cite{Singer_Wu:2012,Cheng_Wu:2013}.
In the proof of the LLE under the manifold setup, the eigen-structure of $C_x$ plays an essential role due to its relationship with the barycentric coordinate. 
{Geometrically, for a $d$-dim manifold, the first $d$ eigenvectors of $C_x$ corresponding to the largest $d$ eigenvalues provide an estimated basis for the embedded tangent space $\iota_*T_xM$, and the remaining eigenvectors form an estimated basis for the normal space at $\iota(x)$. }
To be more precise, a smooth manifold can be well-approximated locally by an affine subspace. However, this approximation cannot be perfect, if the curvature exists. It is well-known that the contribution of curvature is of high order. For the purpose of fitting the manifold, we can ignore its contribution. For example, in \cite{Singer_Wu:2012,Cheng_Wu:2013} the local PCA is applied to estimate the tangent space. However, in the LLE, the curvature plays an essential role and a careful analysis is needed to understand its role. In Lemma \ref{Lemma:5}, we show a generalization of the result shown in \cite{Singer_Wu:2012,Cheng_Wu:2013} by expanding the $C_x$ up to the third order for the sake of capturing the LLE behavior.
The third order term is needed for analyzing the regularization step shown in (\ref{Section2:GGyN}).

\begin{Assumption}\label{AssumptionTangent}
Since the barycentric coordinate is rotational and translational invariant, without loss of generality, we assume that the manifold is translated and rotated properly, so that $\iota_*T_xM$ is spanned by $e_1,\ldots,e_d$. 
\end{Assumption}

\begin{proposition} \label{Proposition:1}
Fix $x\in M$ and suppose Assumption \ref{AssumptionTangent} holds. When $\epsilon$ is sufficiently small, we have
\begin{align*}
C_x =\frac{|S^{d-1}|  {P}(x)}{d(d+2)} \epsilon^{d+2}
\Big(\begin{bmatrix}
I_{d \times d} & 0 \\
0& 0  \\
\end{bmatrix}+
\begin{bmatrix}
M^{(2)}_{11} & M^{(2)}_{12}  \\
M^{(2)}_{21} & M^{(2)}_{22} 
\end{bmatrix}\epsilon^{2}+
\begin{bmatrix}
M^{(4)}_{11} & M^{(4)}_{12}  \\
M^{(4)}_{21} & M^{(4)}_{22}  \\
\end{bmatrix}\epsilon^{4}+O(\epsilon^{6})\Big) \,,
\end{align*}
where 
$M^{(2)}_{11},\,M^{(4)}_{11}\in S(d)$, $M^{(2)}_{22},\,M^{(4)}_{22}\in S(p-d)$, $M^{(2)}_{12},\,M^{(4)}_{12}\in \mathbb{R}^{d\times (p-d)}$, $M^{(2)}_{12}={M^{(2)\top}_{21}}$, and $M^{(4)}_{12}={M^{(4)\top}_{21}}$.
These matrices are defined in (\ref{Proof:LemmaD5:Definition:tildeM12}), (\ref{Proof:LemmaD5:Definition:barM22}), (\ref{Proof:LemmaD5:Definition:tildeM11}), and  (\ref{Proof:LemmaD5:Definition:tildeM22}), {and $S(d)$ and $S(p-d)$ are defined in the end of Section \ref{Section:Introduction}}.
$M^{(2)}_{22}$ depends on $\Second_x$ but does not depend on the p.d.f. $P$, and $M^{(4)}_{22}$ depends on the $\Second_x$ and its derivatives, the Ricci curvature, and $P$.
\end{proposition}

The proof of Proposition \ref{Proposition:1} is postponed to Section \ref{Section:Proof:PropositionForCovariance}. 
{ Since $P$ is bounded by $P_m$ from below, when $\epsilon$ is sufficiently small, 
the $\epsilon^{d+2}$ term is dominant and the largest $d$ eigenvalues of $C_x$ are of order $\epsilon^{d+2}$. The other eigenvalues of $C_x$ are of higher order and depend on 
the $\epsilon^{d+4}$ term or even the
$\epsilon^{d+6}$ term.  
The behavior of eigenvectors is more complicated, due to the possible multiplicity of the corresponding eigenvalues. 

To precisely calculate the eigenvalues and the corresponding eigenvectors of $C_x$, we apply the perturbation technique.
We summarize the key steps here. 
Proposition \ref{Proposition:1} provides a Taylor expansion of $C_x$ in terms of $\epsilon$ up to the third order, and we could view $C_x$ as a function depending on $\epsilon$ around $0$. Consider the eigen-decomposition of $C_x$ as  
\begin{equation} \label{eigenvalue equation of C_x}
C_x U_x=U_x \Lambda_x \,,
\end{equation}
where $\Lambda_x$ is diagonal and $U_x \in O(p)$.
$\Lambda_x$ and $U_x$ satisfy 
$\Lambda_x=\Lambda_x(0) \epsilon^{d+2}+\Lambda_x^{'}(0) \epsilon^{d+4}+O(\epsilon^{d+6})$ and $
U_x=U_x(0) \epsilon^{d+2}+U_x^{'}(0) \epsilon^{d+4}+O(\epsilon^{d+6})$.
Therefore, we obtain $U_x$ and $\Lambda_x$ if we find $\Lambda_x(0)$, $\Lambda_x^{'}(0)$, $U_x(0)$ and $U_x^{'}(0)$. To achieve this goal, we differentiate (\ref{eigenvalue equation of C_x}), and compare terms with the same order of $\epsilon$. This technique fails to uniquely determine $U_x$ when the eigenvalue repeats,
and we need higher order terms in $C_x$ to determine the eigenvectors. The details could be found in Appendix \ref{Section:Perturbation}.

To simplify the statement of the eigen-structure, following Assumption \ref{AssumptionTangent}, we make one more assumption.}

\begin{Assumption}\label{AssumptionNormal}
Following Assumption \ref{AssumptionTangent}, without loss of generality, we assume that the manifold is translated and rotated properly, so that $e_{d+1}, \cdots ,e_p$ ``diagonalize'' the second fundamental form; that is, $M^{(2)}_{22}$ in Proposition \ref{Proposition:1} is diagonalized to $\Lambda_2^{(2)}=\text{diag}(\lambda^{(2)}_{d+1},\ldots,\lambda^{(2)}_{p})$. 
\end{Assumption}

The eigen-structure of the local covariance matrix is summarized in the following Proposition. The detailed proof of the Proposition is postponed to Section \ref{Section:Proof:PropositionForCovariance}.

\begin{proposition} \label{Proposition:2}
Fix $x\in M$. Suppose $\epsilon$ is sufficiently small and Assumptions \ref{AssumptionTangent} and \ref{AssumptionNormal} hold. The eigen-decomposition of $C_x=U_x\Lambda_x U_x^{\top}$, where $U_x\in O(p)$ and $\Lambda_x\in\mathbb{R}^{p\times p}$ is a diagonal matrix, is summarized below.

\textbf{Case 1:} When all diagonal entries of $\Lambda_2^{(2)}$ are nonzero, we have:
\begin{align*}
\Lambda_x&=\frac{|S^{d-1}|  {P}(x)\epsilon^{d+2}}{d(d+2)}\begin{bmatrix}
I_{d\times d} +\epsilon^2 \Lambda^{(2)}_1+\epsilon^4\Lambda^{(4)}_1 & 0 \\
0 & \epsilon^2\Lambda^{(2)}_2+\epsilon^4\Lambda^{(4)}_2 \\
\end{bmatrix}+O(\epsilon^6), \\
U_x&=U_x(0)(I_{p\times p}+\epsilon^2 \mathsf{S})+O(\epsilon^4)\in O(p),
\end{align*}
where $\Lambda^{(2)}_1,\Lambda^{(4)}_1\in \mathbb{R}^{d\times d}$ and $\Lambda^{(4)}_2\in \mathbb{R}^{(p-d)\times (p-d)}$ are diagonal matrices with diagonal entries of order $1$, $U_x(0)=\begin{bmatrix}X_1 & 0 \\ 0 & X_2\end{bmatrix}\in O(p)$, $X_1\in O(d)$, $X_2\in O(p-d)$, and $\mathsf{S}\in \mathfrak{o}(p)$. The explicit expression of these matrices are listed in (\ref{Proof:LemmaD6:Condition1:S1})-(\ref{description of X_2}).

\textbf{Case 2:} When $l$ diagonal entries for $\Lambda_2^{(2)}$ are $0$, where $1 \leq l \leq p-d$, we have the following eigen-decomposition under some conditions. 
Divide $C_x$ into blocks corresponding to the multiplicity $l$ as 
\begin{align} 
C_x =&\,\frac{|S^{d-1}|  {P}(x)}{d(d+2)} \epsilon^{d+2}
\Big(\begin{bmatrix}
I_{d \times d} & 0 &0 \\
0& 0 & 0 \nonumber \\
0& 0 & 0 \nonumber \\
\end{bmatrix}+
\begin{bmatrix}
M^{(2)}_{11} & M^{(2)}_{12,1} &  M^{(2)}_{12,2}  \\
M^{(2)}_{21,1} & \Lambda^{(2)}_{2,1} &  0 \\
M^{(2)}_{21,2} & 0 &  0 \\
\end{bmatrix}\epsilon^{2}\label{Expansion:LocalCovarianceMatrix:Case2}\\
&\quad+\begin{bmatrix}
M^{(4)}_{11} & M^{(4)}_{12,1} &  M^{(4)}_{12,2}  \\
M^{(4)}_{21,1} & M^{(4)}_{22,11} &  M^{(4)}_{22,12} \\
M^{(4)}_{21,2} & M^{(4)}_{22,21} &  M^{(4)}_{22,22} \\
\end{bmatrix}\epsilon^{4}+O(\epsilon^{6})\Big) \,, 
\end{align}
where $M^{(2)}_{12,1},M^{(4)}_{12,1}\in \mathbb{R}^{d\times (p-d-l)}$, $M^{(2)}_{12,2},M^{(4)}_{12,2}\in \mathbb{R}^{d\times l}$,
$M^{(2)}_{12,1}={M^{(2)\top}_{21,1}}$, $M^{(4)}_{12,1}={M^{(4)\top}_{21,1}}$, $M^{(2)}_{12,2}={M^{(4)\top}_{21,2}}$, $M^{(2)}_{12,2}={M^{(4)\top}_{21,2}}$,
$M^{(4)}_{22,11}\in S(p-d-l)$,  $M^{(4)}_{22,22}\in S(l)$, $M^{(4)}_{22,12}\in \mathbb{R}^{(p-d-l)\times l}$, and $M^{(4)}_{22,21}={M^{(4)\top}_{22,12}}$.

Denote the eigen-decomposition of the matrix $M^{(4)}_{22,22}-2M^{(2)}_{21,2}M^{(2)}_{12,2}$ as
\begin{equation}
M^{(4)}_{22,22}-2M^{(2)}_{21,2}M^{(2)}_{12,2}=U_{2,2}\Lambda^{(4)}_{2,2}U_{2,2}^\top,
\end{equation}
where $U_{2,2}\in O(l)$ and $\Lambda^{(4)}_{2,2}=\texttt{diag}[\lambda^{(4)}_{p-l+1},\ldots,\lambda^{(4)}_{p}]$ is a diagonal matrix. If we further assume that all diagonal entries of $\Lambda^{(4)}_{2,2}$ are nonzero,
we have
\begin{align*}
\Lambda_x&=\frac{|S^{d-1}|  {P}(x)\epsilon^{d+2}}{d(d+2)}\begin{bmatrix}
I_{d\times d} +\epsilon^2 \Lambda^{(2)}_1+\epsilon^4\Lambda^{(4)}_1 & 0 & 0\\
0 & \epsilon^2\Lambda^{(2)}_{2,1}+\epsilon^4\Lambda^{(4)}_{2,1} & 0\\
0 & 0 & \epsilon^4\Lambda^{(4)}_{2,2} \\
\end{bmatrix}+O(\epsilon^6), \\
U_x&=U_x(0)(I_{p\times p}+\epsilon^2 \mathsf{S})+O(\epsilon^4)\in O(p),
\end{align*}
where $\Lambda^{(4)}_1$ and $\Lambda^{(4)}_{2,1}$ are diagonal matrices, $U_x(0)=\begin{bmatrix}X_1 &0 &0\\ 0 & X_{2,1} & 0 \\ 0 & 0 & X_{2,2}\end{bmatrix}\in O(p)$, $X_1\in O(d)$, $X_{2,1}\in O(p-d-l)$, $X_{2,2}\in O(l)$, and $\mathsf{S}\in \mathfrak{o}(p)$. The explicit formula for these matrices are listed in (\ref{Proof:LemmaD6:Condition2:first})-(\ref{Proof:LemmaD6:Condition2:last}).

\end {proposition}

In general, the eigen-structure of $C_x$ may be more complicated than the two cases considered in Proposition \ref{Proposition:2}. In this general case, we could apply the same perturbation theory to evaluate the eigenvalues. Since the proof is similar but there is extensive notational loading, and it does not bring further insight to the LLE, we skip details of these more general situations.

\subsection{Variance analysis of the LLE} 
We now study the asymptotic behavior of the LLE. Under the manifold setup, from now on, we fix 
\begin{equation}
c=n\epsilon^{d+\rho}, 
\end{equation}
and we call $\rho$ the {\em regularization order}.
By (\ref{Definition:Wf:SecondVersion}), for $\boldsymbol{v}\in\mathbb{R}^N$, we have 
\begin{align}
\sum_{j=1}^N w_k(j)\boldsymbol{v}(j)=\frac{\boldsymbol{1}_N^\top \boldsymbol{v}-\boldsymbol{1}_N^\top G_n^\top \mathcal{I}_{n\epsilon^{d+\rho}}(G_nG_n^\top)  G_n\boldsymbol{v}}
{N -  \boldsymbol{1}_N^\top G_n^\top \mathcal{I}_{n\epsilon^{d+\rho}}(G_nG_n^\top) G_n\boldsymbol{1}_N }.
\end{align}
Before proceeding, we provide {\em a} geometric interpretation of this formula. By the eigen-decomposition $G_nG_n^\top=U_n\bar \Lambda_n U_n^\top$ and the fact that $\mathcal{I}_{n\epsilon^{d+\rho}}(G_nG_n^\top)=U_nJ_{p,r_n}J_{p,r_n}^\top(\bar\Lambda_n+{n\epsilon^{d+\rho}} I_{p\times p})^{-1}J_{p,r_n}J_{p,r_n}^\top U_n^\top=U_n\mathcal{I}_{n\epsilon^{d+\rho}}(\bar\Lambda_n)U_n^\top $ by the definition of $\mathcal{I}_\rho$ in (\ref{Definition:Irho:Soft}), we have $\boldsymbol{1}_N^\top G_n^\top \mathcal{I}_{n\epsilon^{d+\rho}}(G_nG_n^\top)  G_n\boldsymbol{v}=\boldsymbol{1}_N^\top G_n^\top U_n \mathcal{I}_{n\epsilon^{d+\rho}}(\bar\Lambda_n)U_n^\top  G_n\boldsymbol{v}$ and $\boldsymbol{1}_N^\top G_n^\top \mathcal{I}_{n\epsilon^{d+\rho}}(G_nG_n^\top)  G_n\boldsymbol{1}=\boldsymbol{1}_N^\top G_n^\top U_n \mathcal{I}_{n\epsilon^{d+\rho}}(\bar\Lambda_n)U_n^\top  G_n\boldsymbol{1}_N$. By the discussion of the local PCA in Section \ref{Section:LocalPCA}, $U_n^\top  G_n$ means evaluating the coordinates of all neighboring points of $\iota(x_k)$ {with the basis composed of the column vectors of $U_n$}, $U_n^\top  G_n \boldsymbol{1}$ means the mean coordinate of all neighboring points, $\mathcal{I}_{n\epsilon^{d+\rho}}(\bar \Lambda_n)$ means a regularized weighting of the coordinates that helps to enhance the nonlinear geometry of the point cloud, and $G_n^\top U_n \mathcal{I}_{n\epsilon^{d+\rho}}(\bar\Lambda_n)U_n^\top  G_n$ is a quadratic form of the averaged coordinates of all neighboring points. We could thus view the ``kernel'' part, $\boldsymbol{1}_N^\top G_n^\top U_n \mathcal{I}_{n\epsilon^{d+\rho}}(\bar\Lambda_n)U_n^\top  G_n$, as preserving the geometry of the point cloud, by evaluating how strongly the weighted coordinates of neighboring points are related to the mean coordinate of all neighboring points by the inner product.

Asymptotically, by the law of large numbers, when conditional on $\iota(x_k)$, 
\begin{equation*}
\frac{1}{n}G_n\boldsymbol{1}_N=\frac{1}{n} \sum_{j=1}^N(\iota(x_{k,j})-\iota(x_k))\xrightarrow[]{n\to \infty} \mathbb{E}[(X-\iota(x_k))\chi_{B_{\epsilon}^{\mathbb{R}^p}(\iota(x_k))}(X)]
\end{equation*}
and we ``expect'' the following holds
\begin{equation*}
n \mathcal{I}_{n\epsilon^{d+\rho}}(G_nG_n^\top) = \mathcal{I}_{\epsilon^{d+\rho}}(\frac{1}{n}G_nG_n^\top) \xrightarrow[]{n\to \infty}\mathcal{I}_{\epsilon^{d+\rho}}(C_{x_k}).
\end{equation*}
Also, we would ``expect'' to have
\begin{align*}
n \mathcal{I}_{n\epsilon^{d+\rho}}(G_nG_n^\top)\frac{1}{n}G_n\boldsymbol{1}_N \xrightarrow[]{n\to \infty} \mathcal{I}_{\epsilon^{d+\rho}}(C_{x_k})\big[\mathbb{E}(X-\iota(x_k))\chi_{B_{\epsilon}^{\mathbb{R}^p}(x_k)}\big]=:\mathbf{T}_{\iota(x_k)}\,.  
\end{align*}
Hence, for $f\in C(\iota(M))$, for $\iota(x_k)$ and its corresponding $\mathcal{N}_{\iota(x_k)}$, we would ``expect'' to have 
\begin{align}
\sum_{j=1}^N w_n(j)f(x_{k,j})
&\,\xrightarrow[]{n\to \infty}\frac{\mathbb{E}[\chi_{B_{\epsilon}^{\mathbb{R}^p}(x_k)}(X)f(X)]- \mathbf{T}_{\iota(x_k)}^\top \mathbb{E}[(X-\iota(x_k))\chi_{B_{\epsilon}^{\mathbb{R}^p}(x_k)}(X)f(X)]}{\mathbb{E}[\chi_{B_{\epsilon}^{\mathbb{R}^p}(x_k)}(X)]- \mathbf{T}_{\iota(x_k)}^\top \mathbb{E}[(X-\iota(x_k))\chi_{B_{\epsilon}^{\mathbb{R}^p}(x_k)}(X)]}\nonumber\\
&\,={
\frac{\mathbb{E}[f(X)(1-\mathbf{T}_{\iota(x)}^\top (X-\iota(x)))\chi_{B_{\epsilon}^{\mathbb{R}^p}(x)}(X)]}{\mathbb{E}[(1-\mathbf{T}_{\iota(x)}^\top (X-\iota(x)))\chi_{B_{\epsilon}^{\mathbb{R}^p}(x)}(X)]}}.
\end{align}
However, it is not possible to directly see how the { convergence} happens, due to the dependence among different terms and how the regularized pseudo-inverse converges. The dependence on the regularization order is also not clear. A careful theoretical analysis is needed. 
 
To proceed with the proof, we need to discuss a critical observation. Note that the term $C_x$ might be ill-conditioned for the pseudo-inverse procedure, and the regularized pseudo inverse depends on how the regularization penalty $\rho$ is chosen. As we will see later, the choice of $\rho$ is critical for the outcome. The ill-conditionedness depends on the manifold geometry, and can be complicated. In this paper we focus on the following three cases. 

\begin{Condition}\label{Condition:1}
Follow the notations used in Proposition \ref{Proposition:2}. 
For the local covariance matrix $C_x$ with the rank $r$, without loss of generality, we consider the following three cases:
\begin{itemize}
\item Case 0: $r=d$;
\item Case 1: $r=p>d$, and $\lambda^{(2)}_{d+1},\ldots,\lambda^{(2)}_p$ are nonzero;
\item Case 2: $r=p>d$, $\lambda^{(2)}_{d+1},\ldots,\lambda^{(2)}_{p-l}$, are nonzero, where $1 \leq l \leq p-d$, $\lambda^{(2)}_{p-l+1}=\ldots=\lambda^{(2)}_{p}=0$, and $\lambda^{(4)}_{p-l+1},\ldots,\lambda^{(4)}_{p}$ are nonzero. 
\end{itemize}
\end{Condition}

At first glance, it is limited to assume that when $r>d$, we have $r=p$ in Cases 1 and 2. However, it is general enough in the following sense. In Cases 1 and 2, if $C_x$ is degenerate, that is, $d<r<p$, it means that locally the manifold only occupies a lower dimensional affine subspace. Therefore, the sampled data are constrained to this affined subspace, and hence the rank of the local sample covariance matrix satisfies $r_n\leq r$. As a result, the analysis can be carried out only on this affine subspace without changing the outcome.
More general situations could be studied by the same analysis techniques shown below, but they will not provide more insights about our understanding of the algorithm and will introduce additional notational burdens. 
For $f \in C(\iota(M))$, define
\begin{equation}
Qf(x):=\frac{\mathbb{E}[f(X)(1-\mathbf{T}_{\iota(x)}^\top (X-\iota(x)))\chi_{B_{\epsilon}^{\mathbb{R}^p}(x)}(X)]}{\mathbb{E}[(1-\mathbf{T}_{\iota(x)}^\top (X-\iota(x)))\chi_{B_{\epsilon}^{\mathbb{R}^p}(x)}(X)]},\label{Definition:Qf:KernelExpansion}
\end{equation}
The following theorem summarizes the relationship between the LLE and $Qf$ under these three cases.

\begin{theorem} \label{Theorem:t0}
Fix $f \in C(\iota(M))$. Suppose the regularization order is $\rho\in \mathbb{R}$, $\epsilon=\epsilon(n)$ so that $\frac{\sqrt{\log(n)}}{n^{1/2}\epsilon^{d/2+1}}\to 0$ and $\epsilon\to 0$ as $n\to \infty$. 
With probability greater than $1-n^{-2}$, for all $x_k\in\mathcal{X}$,
under different conditions listed in Condition \ref{Condition:1}, we have:
\begin{align}
&\sum_{j=1}^N w_k(j)f(x_{k,j})-f(x_k)\\
=\,&
\left\{
\begin{array}{ll}
Qf(x_k)-f(x_k)+O\Big(\frac{\sqrt{\log (n)}}{n^{1/2}\epsilon^{d/2-1}}\Big)&\mbox{in Case 0}\\
Qf(x_k)-f(x_k)+O\Big(\frac{\sqrt{\log (n)}}{n^{1/2}\epsilon^{d/2+[(-1)\vee (0\wedge (\rho-4)]}}\Big)&\mbox{in Cases 1,2}
\end{array}
\right.\nonumber
\end{align}
Particularly, when $\rho\leq 3$, with probability greater than $1-n^{-2}$, for all $x_k\in\mathcal{X}$, for all Cases listed in Condition \ref{Condition:1}, we have:
\begin{align}
&\sum_{j=1}^N w_k(j)f(x_{k,j})-f(x_k)=Qf(x_k)-f(x_k)+O\Big(\frac{\sqrt{\log (n)}}{n^{1/2}\epsilon^{d/2-1}}\Big).
\end{align}
\end{theorem}

The proof of Theorem \ref{Theorem:t0} is postponed to Appendix \ref{Section:Theoremt0}. 
Note that the convergence rate of Case 0 is fast, no matter what regularization order $\rho$ is chosen, while the convergence rate of Case 1 and Case 2 depends on $\rho$. This theorem echoes several practical findings of the LLE that the choice of regularization is critical in the performance, and it suggests that we should choose $\rho=3$.

\begin{remark}
We should compare the convergence rate of the LLE with that of the DM. The convergence rate of Case 0 is the same as that of the eigenmap or the DM without any normalization \cite{Singer_Wu:2016}, while the convergence rate of Case 1 and Case 2 is the same as that of the $\alpha$-normalized DM \cite{Coifman_Lafon:2006} when $\rho\geq 4$ \cite{Singer_Wu:2016}. Note that the main convergence rate bottleneck for the $\alpha$-normalized DM comes from the probability density function estimation, while the convergence bottleneck for the LLE is the regularized pseudo-inverse.  
\end{remark}

\subsection{The kernel function corresponding to the LLE}

Theorem \ref{Theorem:t0} describes how the LLE could be viewed as a ``diffusion process'' on the dataset. 
Note that
\begin{align}
&{
\mathbb{E}[f(X)(1-\mathbf{T}_{\iota(x)}^\top (X-\iota(x)))\chi_{B_{\epsilon}^{\mathbb{R}^p}(x)}(X)]}\\
=&\,\int_{M} (1-\mathbf{T}_{\iota(x_k)}^\top (\iota(y)-\iota(x_k)))\chi_{B_{\epsilon}^{\mathbb{R}^p}(x_k)}(\iota(y))f(\iota(y)) P(y)dV(y)\nonumber
\end{align}
Therefore, we can view $w_n$ as a ``zero-one'' kernel supported on $B_{\epsilon}^{\mathbb{R}^p}(x_k)\cap \iota(M)$ with the correction depending on $\mathbf{T}_{\iota(x_k)}$. Note that after the correction, the whole operator may no longer be a diffusion.

\begin{corollary}\label{kernel}
The integral kernel associated with the LLE when the regularization order is $\rho\in\mathbb{R}$ is 
\begin{equation}
K_{\texttt{LLE}}(x,y)=[1- \mathbf{T}_{\iota(x)}^\top(\iota(y)-\iota(x))]\chi_{B_{\epsilon}^{\mathbb{R}^p}(\iota(x)) \cap \iota(M)}(\iota(y)),
\end{equation}
where $x,y\in M$ and 
\begin{equation}
\mathbf{T}_{\iota(x)}:= \mathcal{I}_{\epsilon^{d+\rho}}(C_{x})\big[\mathbb{E}(X-\iota(x))\chi_{B_{\epsilon}^{\mathbb{R}^p}(x)}\big]    \in \mathbb{R}^p.\label{Definition:Tx:ContinuousCase}
\end{equation}
\end{corollary}

Note that $K_{\texttt{LLE}}$ depends on $\epsilon$, the geometry of the manifold near $x$, and $\rho$ via $\mathbf{T}_{\iota(x)}$.
We provide some properties of the kernel function $K_{\texttt{LLE}}$. By a direct expansion, we have
$\mathbf{T}_{\iota(x)}^\top=\sum_{i=1}^r\frac{u_i^\top\mathbb{E}[(X-x_k)\chi_{B_{\epsilon}^{\mathbb{R}^p}(x_k)}(X)]}{\lambda_i+\epsilon^{d+\rho}} u_i^\top$, where $u_i$ and $\lambda_i$ are the $i$-th eigen-pair of $C_x$.
Since $| \mathbb{E}(X-\iota(x_k))\chi_{B_{\epsilon}^{\mathbb{R}^p}(x_k)}(X) |$ is bounded above by $\text{vol}(M) \epsilon$ , $\lambda_i+\epsilon^{d+\rho}$ is bounded below by $\epsilon^{d+\rho}$ and each $u_i$ is a unit vector, $|\mathbf{T}_{x_k}|$ is bounded above by $\sum_{i=1}^{r}\frac{\epsilon \text{vol}(M)}{\lambda_i+\epsilon^{d+\rho}}$. Consequently, we have the following proposition. 
\begin{proposition}
The kernel $K_{\texttt{LLE}}$ is compactly supported and is in $L^2(M\times M)$. Thus, the linear operator $A:L^2(M,PdV)\to L^2(M, PdV)$ defined by
\begin{equation}
{
Af(x):=\mathbb{E}[f(X)(1-\mathbf{T}_{\iota(x)}^\top (X-\iota(x)))\chi_{B_{\epsilon}^{\mathbb{R}^p}(x)}(X)]}\label{Definition:Aoperator}
\end{equation}
is Hilbert-Schmidt. 
\end{proposition}

Note that the kernel function $K_{\texttt{LLE}}(x,\cdot)$ depends on $x$ and hence the manifold, and the kernel is dominated by normal bundle information, due to the regularized pseudo-inverse procedure. 
For example, if $M$ is an affine subspace of $\mathbb{R}^p$ and the data is uniformly sampled, then $\mathbb{E}[(X-x)\chi_{B_{\epsilon}^{\mathbb{R}^p}(x)}(X)]=0$. Consequently, $\mathbf{T}_x=0$ and $K(x,y)=1$. 
If $M$ is $S^{p-1}$, a unit sphere centered at origin embedded in $\mathbb{R}^p$ and the data is uniformly sampled, the first dominant $p-1$ eigenvectors are perpendicular to $x$ and the last eigenvector is parallel to $x$. By a direct calculation, $\mathbb{E}[(X-x)\chi_{B_{\epsilon}^{\mathbb{R}^p}(x)}(X)]$ is parallel to $x$ and hence $K(x,y)$ behaves like a quadratic function $1-cu_p^\top (y-x)=1-cx^\top (y-x)$, where $c$ is the constant depending on the eigenvalues.

\subsection{Bias analysis}

For $f \in C(\iota(M))$, by the definition of $A$, we have 
\begin{equation}
Qf(x)=\frac{(Af)(x)}{(A1)(x)},\label{Definition:Qf:KernelExpansion}
\end{equation}
where $1$ means the constant function.
We now provide an {\em approximation of identity} expansion of the $Q$ operator. By a direct expansion, we have
\begin{equation}
Af(x)
=\int_MK_{\texttt{LLE}}(x,y)f(\iota(y))P(y)d V(y) .
\end{equation}
While the formula of the $Q$ operator looks like the diffusion process commonly encountered in the graph Laplacian based approach, like the DM \cite{Coifman_Lafon:2006}, the proof and the result are essentially different. 
To ease the notation, define
\begin{align}
\mathfrak{N}_0(x):= \frac{1}{|S^{d-1}|}\int_{S^{d-1}}\Second_{x}(\theta,\theta)d\theta,\label{Definition:N0x:forsimplification}
\end{align}
\vspace{-15pt}
\begin{align}
\mathfrak{M}_2(x):=\frac{1}{|S^{d-1}|}\int_{S^{d-1}} \Second_x(\theta,\theta) \theta\theta^{\top} d\theta,\quad\mathfrak{H}_f(x):=\texttt{tr}(\mathfrak{M}_2(x) \nabla^2 f(x)),\nonumber
\end{align}
where $f \in C^3(\iota(M))$.

\begin{theorem} \label{Theorem:t1}
Suppose $f \in C^3(\iota(M))$ and $P\in C^5(\iota(M))$ and fix $x \in M$. 
Assume that Assumptions \ref{AssumptionTangent} and \ref{AssumptionNormal} hold and the regularization order is $\rho\in\mathbb{R}$.
Following the same notations used in Proposition \ref{Proposition:2}, we have the following result 
\begin{equation}
Qf(x)-f(x)=(\mathfrak{C_1}(x)+\mathfrak{C_2}(x))\epsilon^2+O(\epsilon^3),
\end{equation}
where $\mathfrak{C_1}(x)$ and  $\mathfrak{C_2}(x)$ depend on different cases stated in Condition \ref{Condition:1}.

$\bullet$ \textbf{Case 0.} In this case,
\begin{align}
& \mathfrak{C_1}(x)=\frac{1}{d+2}\big[\frac{1}{2}\Delta f(x) +\frac{\nabla f(x) \cdot \nabla P(x)}{P(x)}-\frac{\nabla {f}(x)\cdot \nabla {P}(x)}{{P}(x)+\frac{d(d+2)}{|S^{d-1}|}\epsilon^{\rho-2}} \big]\,,\\
& \mathfrak{C_2}(x)=0\,.
\end{align}

$\bullet$ \textbf{Case 1.}
In this case, 
\begin{align}
\mathfrak{C_1}(x)&=\frac{\frac{1}{d+2}\big[\frac{1}{2}\Delta {f}(x)+ \frac{\nabla f(x) \cdot \nabla P(x)}{P(x)}-\frac{\nabla {f}(x)\cdot \nabla {P}(x)}{{P}(x)+\frac{d(d+2)}{|S^{d-1}|}\epsilon^{\rho-2}}\big] }
{1-\frac{d}{2(d+2) } \sum_{i=d+1}^p\frac{(\mathfrak{N}^{\top}_0(x)e_i)^2}{\frac{2}{d}\lambda^{(2)}_i+\frac{2(d+2)}{P(x)|S^{d-1}|}\epsilon^{\rho-4}}}\,,
\end{align}
\begin{align}
\mathfrak{C_2}(x)&=-\frac{\frac{1}{4(d+4) }\sum_{i=d+1}^p\frac{(\mathfrak{N}^{\top}_0(x)e_i)(\mathfrak{H}^{\top}_f(x)e_{i})}{\frac{2}{d}\lambda^{(2)}_i+\frac{2(d+2)}{P(x)|S^{d-1}|}\epsilon^{\rho-4}}}{\frac{1}{d}-\frac{1}{2(d+2) } \sum_{i=d+1}^p\frac{(\mathfrak{N}^{\top}_0(x)e_i)^2}{\frac{2}{d}\lambda^{(2)}_i+\frac{2(d+2)}{P(x)|S^{d-1}|}\epsilon^{\rho-4}}}.
\end{align}

$\bullet$ \textbf{Case 2.}
In this case, 
\begin{align}
\mathfrak{C_1}(x)&=\frac{\frac{1}{d+2}\big[\frac{1}{2}\Delta {f}(x)+ \frac{\nabla f(x) \cdot \nabla P(x)}{P(x)}-\frac{\nabla {f}(x)\cdot \nabla {P}(x)}{{P}(x)+\frac{d(d+2)}{|S^{d-1}|}\epsilon^{\rho-2}}\big] }
{1-\frac{d}{2(d+2) } \sum_{i=d+1}^{p-l}\frac{(\mathfrak{N}^{\top}_0(x)e_{i})^2}{\frac{2}{d}\lambda^{(2)}_i+\frac{2(d+2)}{P(x)|S^{d-1}|}\epsilon^{\rho-4}}}\,,
\end{align}
\begin{align}
\mathfrak{C_2}(x)&=-\frac{\frac{1}{4(d+4) } \sum_{i=d+1}^{p-l}\frac{(\mathfrak{N}^{\top}_0(x)e_i)(\mathfrak{H}^{\top}_f(x)e_{i})}{\frac{2}{d}\lambda^{(2)}_i+\frac{2(d+2)}{P(x)|S^{d-1}|}\epsilon^{\rho-4}}}{\frac{1}{d}-\frac{1}{2(d+2) } \sum_{i=d+1}^{p-l}\frac{(\mathfrak{N}^{\top}_0(x)e_{i})^2}{\frac{2}{d}\lambda^{(2)}_i+\frac{2(d+2)}{P(x)|S^{d-1}|}\epsilon^{\rho-4}}}.
\end{align}

\end{theorem}

The proof of this long theorem is postponed to Appendix \ref{Section:Theoremt1}. 
{
Intuitively, based on the approximation of the identity, the kernel representation of the $Q$ operator suggests that asymptotically we get the function value back, with the second order derivative popping out in the second order error term. In the GL setup, it has been well known that the second order derivative term is the Laplace-Beltrami operator when the p.d.f. is constant \cite{Coifman_Lafon:2006}. However, due to the interaction between the geometric structure and the barycentric coordinate, the LLE usually does not lead to the Laplace-Beltrami operator, unless under special situations. 
Note that while we could still see the Laplace-Beltrami operator in $\mathfrak{C}_1$, it is contaminated by other quantities, including $\mathfrak{N}_0(x)$, $\mathfrak{H}_f(x)$ and $\lambda^{(2)}_i$. These terms all depend on the second fundamental form. When $\rho>4$, the curvature term appears in the $\epsilon^2$ order term.}

This theorem states that the asymptotic behavior of LLE is sensitive to the choice of $\rho$. We discuss each case based on different choices of $\rho$.  
If $\rho <2$, for all cases,  
\begin{align}
\mathfrak{C_1}(x)=\frac{1}{(d+2)}\big[\frac{1}{2}\Delta f(x) +\frac{\nabla f(x) \cdot \nabla P(x)}{P(x)} \big] \quad\mbox{and}\quad \mathfrak{C_2}(x)=0,\label{Theorem:t1:SmallRho}
\end{align}
which comes from the fact that when $\epsilon^{\rho}$ is large, $\mathbf{T}_{\iota(x)}$ is small, and hence $K_{\texttt{LLE}}$ is dominated by $1$. Note that not only the Laplacian-Beltrami operator but also the p.d.f are involved, if the sampling is non-uniform. Therefore, when $\rho$ is chosen too small, the resulting asymptotic operator is the Laplace-Beltrami operator, only when the sampling is uniform.
If $\rho=3$, for all cases we have
\begin{align}
\mathfrak{C_1}(x)=\frac{1}{2(d+2)}\Delta f(x) \quad\mbox{and}\quad \mathfrak{C_2}(x)=0.\label{Theorem:t1:BestRho}
\end{align}
In this case, we recover the Laplacian-Beltrami operator, and the asymptotic result of the LLE is independent of the non-uniform p.d.f.. 
This theoretical finding partially explains why such regularization could lead to a good result.
If $\rho>4$, since $\epsilon^{d+\rho}$ is smaller than all eigenvalues of the local covariance matrix, asymptotically $\epsilon^{d+\rho}$ is negligible and the result depends on different cases considered in Condition \ref{Condition:1}:
for Case 0, we have
\begin{align}
\mathfrak{C_1}(x)=\frac{1}{2(d+2)}\Delta f(x)\quad\mbox{and}\quad  \mathfrak{C_2}(x)=0\,,\nonumber
\end{align}
for Case 1, we have
\begin{align}
\mathfrak{C_1}(x)&=\frac{\frac{1}{2(d+2)}\Delta {f}(x)}
{1-\frac{d^2}{4(d+2) } \sum_{i=d+1}^p\frac{(\mathfrak{N}^{\top}_0(x)e_i)^2}{\lambda^{(2)}_i}},\quad\mathfrak{C_2}(x)&=-\frac{\frac{d}{8(d+4) }\sum_{i=d+1}^p\frac{(\mathfrak{N}^{\top}_0(x)e_i)(\mathfrak{H}^{\top}_f(x)e_{i})}{\lambda^{(2)}_i}}{\frac{1}{d}-\frac{d}{4(d+2) } \sum_{i=d+1}^p\frac{(\mathfrak{N}^{\top}_0(x)e_i)^2}{\lambda^{(2)}_i}}\,,\nonumber
\end{align}
and for Case 2, we have
\begin{align}
\mathfrak{C_1}(x)&=\frac{\frac{1}{2(d+2)}\Delta {f}(x)}
{1-\frac{d^2}{4(d+2) } \sum_{i=d+1}^{p-l}\frac{(\mathfrak{N}^{\top}_0(x)e_i)^2}{\lambda^{(2)}_i}},\quad
\mathfrak{C_2}(x)&=-\frac{\frac{d}{8(d+4) }\sum_{i=d+1}^{p-l}\frac{(\mathfrak{N}^{\top}_0(x)e_i)(\mathfrak{H}^{\top}_f(x)e_{i})}{\lambda^{(2)}_i}}{\frac{1}{d}-\frac{d}{4(d+2) } \sum_{i=d+1}^{p-l}\frac{(\mathfrak{N}^{\top}_0(x)e_i)^2}{\lambda^{(2)}_i}}.\nonumber
\end{align}
Note that when $\rho>4$, we do not get the Laplace-Beltrami operator asymptotically in Cases 1 and 2. Furthermore, the behavior of LLE is dominated by the curvature and is independent of the p.d.f..  

{ It is worth mentioning a specific situation when $\rho>4$. Suppose the principal curvatures are equal to $\mathfrak{p}\in \mathbb{R}$ in the direction $e_i$, where $i=d+1,\ldots,p$, and vanish in the other directions. Then, there is a choice of basis $e_1, \ldots, e_d$ so that $\Second_x(\theta,\theta) \cdot e_i=\sum_{j=1}^d \mathfrak{p} \theta_j^2= \mathfrak{p}$, where $\theta=(\theta_1,\ldots,\theta_d)\in S^{d-1}$. Under this specific situation, by a direct expansion, we have a simplification that 
\begin{equation*}
\frac{d}{8(d+4)}(\mathfrak{N}^{\top}_0(x)e_i)(\mathfrak{H}^{\top}_f(x)e_{i})=\frac{1}{2(d+2)}\Delta {f}(x)\,,
\end{equation*}
which leads to $\mathfrak{C_1}(x)+\mathfrak{C_2}(x)=0$. Therefore, asymptotically we obtain a fourth order term.

The relationship between $\epsilon$ and the intrinsic geometry of the manifold requires further discussion, in order to better understand how the curvature plays a role in the whole analysis. We mention that the statement ``suppose $\epsilon$ is sufficiently small'' in Proposition \ref{Proposition:1}, Proposition \ref{Proposition:2} and Theorem \ref{Theorem:t1} is a technical condition needed in the proof of Lemma \ref{Lemma:3}, which describes how well we could estimate the local geodesic distance by the ambient space metric. This technical condition depends on the fact that the exponential map is a diffeomorphism only if it is restricted to a subset of $\iota_*T_xM$ that is bounded by the injectivity radius of the manifold. 
That is, $\epsilon$ needs to be less than the injectivity radius. For any closed (compact without boundary) and smooth manifold, it is clear that different kinds of curvatures are bounded and the injectivity radius is strictly positive, so there exists $\epsilon_0>0$ less than the injectivity radius, so that for all $\epsilon\leq \epsilon_0$, the statement ``suppose $\epsilon$ is sufficiently small'' is satisfied. The relationship between the curvature and the $\epsilon_0$ could be further elaborated by quoting the well known result in \cite{Cheeger:1982}: for a closed Riemannian manifolds of dimension $d$ with the sectional curvature bounded by $K$, where $K\geq 0$, and with the volume lower bound $v$, where $v>0$, the injectivity radius is bounded below by $i(d,K,v)>0$, where $i(d,K,v)$ can be expressed explicitly in terms of $d$, $K$ and $v$. Hence, $\epsilon_0$ needs to satisfy $\epsilon_0<i(d,K,v)$.

}

{
\subsection{Convergence of the LLE}

By combining the variation analysis and the bias analysis shown above, we conclude the following {\em pointwise} convergence theorem for the LLE, when we have a proper choice of $\rho$.

\begin{theorem} \label{Corollary:LLE:Asymptotitcs}
Take $f \in C(\iota(M))$, $\rho=3$, and $\epsilon=\epsilon(n)$ so that $\frac{\sqrt{\log(n)}}{n^{1/2}\epsilon^{d/2+1}}\to 0$ and $\epsilon\to 0$ as $n\to \infty$. With probability greater than $1-n^{-2}$, for all $x_k\in\mathcal{X}$, 
\begin{align*}
\frac{1}{\epsilon^2}\Big[\sum_{j=1}^N w_k(j)f(x_{k,j})-f(x_k)\Big]=\frac{1}{2(d+2)}\Delta f(x)+O(\epsilon)+O\Big(\frac{\sqrt{\log (n)}}{n^{1/2}\epsilon^{d/2+1}}\Big).
\end{align*}
\end{theorem}
Based on the Borel-Cantelli Lemma, it is clear that asymptotically the LLE converges almost surely. 
For practical purposes, we need to discuss the bandwidth choice when $\rho= 3$. Based on the assumption about the relationship between $n$ and $\epsilon$, we have $\frac{\sqrt{\log(n)}}{n^{1/2}\epsilon^{d/2+1}}\to 0$ as $n\to \infty$, but the convergence rate of $\frac{\sqrt{\log(n)}}{n^{1/2}\epsilon^{d/2+1}}$ might be slower than $\epsilon\to 0$. Suppose we call a bandwidth ``optimal'', if it balances the standard deviation and the bias for all cases in Condition \ref{Condition:1}; that is, $\frac{\sqrt{\log(n)}}{n^{1/2} \epsilon^{d/2+1}}\asymp\epsilon$. We then have $\frac{n}{\log(n)}\asymp\frac{1}{\epsilon^{d+4}}$, and we can estimate the optimal bandwidth from $n$.

}

\section{Numerical Examples}\label{Section:Numerics}

{We adapt the LLE code provided in \url{https://www.cs.nyu.edu/~roweis/lle/code.html} to implement the LLE with the $\epsilon$-radius neighborhood.
The Matlab code for the figures can be found in \url{https://sites.google.com/site/hautiengwu/home/download}. }

\subsection{Sphere}\label{Section:Example:Sphere}
{
Suppose that $S^{p-1} \in \mathbb{R}^p$ is the unit sphere in $\mathbb{R}^p$. Denote $H_{k}$ to be the space of homogeneous polynomials in $\mathbb{R}^p$ restricted on $S^{p-1}$. We have that the space $H_{k}$ is the eigenspace of the Laplace-Beltrami operator on $S^{p-1}$ corresponding to eigenvalue $-k(k+p-2)$, and the dimension of $H_k$ is $\begin{pmatrix}p+k-1\\p-1\end{pmatrix}-\begin{pmatrix}p+k-3\\p-1\end{pmatrix}$ \cite{stein2016}.
In this example, we show that if we choose a $\epsilon^{d+\rho}$ that is too small, then we are not going to get the Laplace-Beltrami operator. }
When $\rho=8$, which is much greater than $3$, by Theorem \ref{Theorem:t1}, 
we have
\begin{align}
 Qf(x_k) -f(x_k)  
=&\,\bigg(\frac{-(p-1)}{8(p+3)(p+5)}\sum_{i=1}^{p-1} \partial^4_i  f(x_k)-\frac{(p-1)}{24(p+3)(p+5)}\sum_{i \not =j}  \partial^2_i \partial^2_j f(x_k)\nonumber\\
&\quad-\frac{p+1}{24(p+3)(p+5)}\sum_{i=1}^{p-1} \partial^2_i  f(x_k)  \bigg) \epsilon^4+O(\epsilon^6). \label{Example:Sp:Expansion}
\end{align}
A detailed calculation is shown in Section \ref{Section:Calculation} (a calculation for the torus case is also provided). It is obvious that asymptotically, we get the fourth order differential operator, instead of the Laplace-Beltrami operator. 
Specifically, when $p=2$, or $S^1$, 
\begin{align}
& Qf(x_k)-f(x_k)=-\frac{1}{280}\big(f''''(x_k)+f''(x_k) \big) \epsilon^4+O(\epsilon^6). \label{Example:AsympotitcS1}
\end{align}
We mention that if the data set $\{x_i\}_{i=1}^n$ is non-uniformly sampled based on the p.d.f. $P$ from $S^1$, then for any $x_k$ we have $Qf(x_k)-f(x_k) =\,C \epsilon^4+O(\epsilon^6)$,
where $C$ depends on the first four order differentiation of $f$ at $x_k$ and the first three order differentiations of $P$ at $x_k$.

{We now} numerically show the relationship between the non-uniform sampling scheme and the regularization term.  
Fix $n=30,000$. Take non-uniform sampling points {$\theta_i:=2\pi U_i+0.3\sin(2\pi i/n)$} on $(0,2\pi]$, where $i=1,\ldots,n$ {and $U_i$ is the uniform distribution on $[0,1]$}, and construct $\mathcal{X}_2=\{(\cos(\theta_i),\sin(\theta_i))^\top\}_{i=1}^n\subset\mathbb{R}^2$. Run the LLE with $\epsilon=0.0002$ and different $\rho$'s, and evaluate the first $400$ eigenvalues. Based on the theory, we know that when $\rho<3$, the asymptotic depends on the non-uniform density function; when $\rho=3$, we recover the Laplace-Beltrami operator in the $\epsilon^2$ order; when $\rho>3$, we get the fourth order differential operator in the $\epsilon^4$, which depends on the non-uniform density function.
See Figure \ref{Figure:S1nonuniform} for a comparison of the estimated eigenvalues and the predicted eigenvalues under different setups. We clearly see that the eigenvalues are well predicted under different $\rho$. When $\rho=8$, {we} get the fourth order term that depends on the non-uniform density function; when $\rho=3$, {the LLE is independent of the non-uniform density function} and we recover the spectrum of the Laplace-Beltrami operator in the second order term, as is predicted by the developed theory; when $\rho=-5$, the non-uniform density function comes into play, and the eigenvalues are slightly shifted. {To enhance the visualization, the difference between the estimated eigenvalues of $S^1$ and the theoretical values are shown on the middle subplot.} The eigenfunctions provide more information. When $\rho=-5$ and $\rho=8$, the dependence of the eigenfunctions on the non-uniform density function could be clearly seen. 

\begin{figure}[h!]
\centering
\subfigure[$S^1$ Eigenvalues]{
\includegraphics[width=0.42\columnwidth]{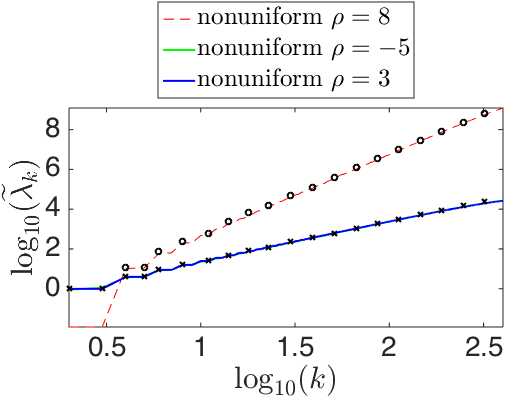}\label{Fig1:S1ev}
}
\subfigure[$S^1$ Eigenvalues error]{
\includegraphics[width=0.42\columnwidth]{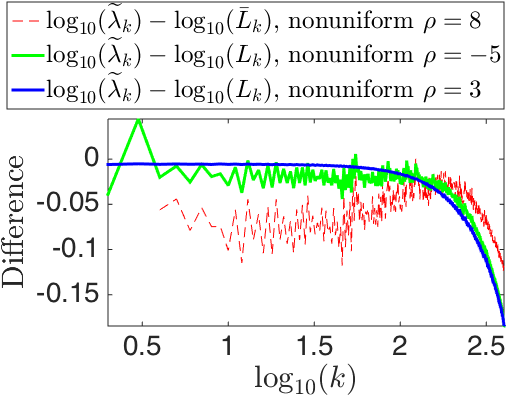}\label{Fig1:S1evErr}
}
\subfigure[$S^1$ Eigenfunctions]{
\includegraphics[width=0.42\columnwidth]{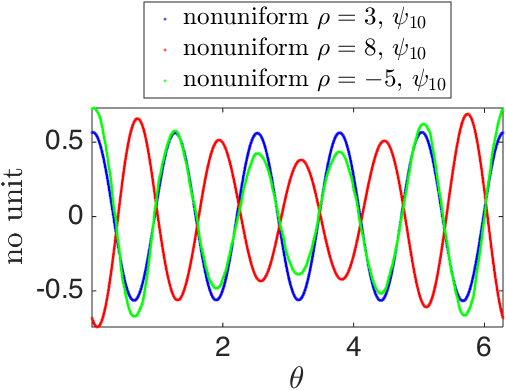}\label{Fig1:S1ef}
}
\caption{The first $400$ eigenvalues of the LLE on $30,000$ points sampled from $S^1$ under a non-uniform sampling scheme with $\rho=-5,3,8$. $\lambda_k$ and $\psi_k$ are the $k$-th largest eigenvalue and the corresponding eigenfunction of the LLE under different situations. $\widetilde{\lambda}_k$ denotes the estimated $k$-th smallest eigenvalue of the Laplace-Beltrami operator or the fourth order differential operator under different situations. The theoretical value, $L_k:=\lceil \frac{k-1}{2}\rceil^2$ for the Laplace-Beltrami operator and $\bar{L}_k:=\lceil \frac{k-1}{2}\rceil^4-\lceil \frac{k-1}{2}\rceil^2$ for the fourth order differential operator $f''''+f''$, where $\lceil x\rceil$ means the the least integer greater than or equal to $x$, are provided for a comparison. 
The eigenvalues and the theoretical values under different setups are shown in \ref{Fig1:S1ev}, with $L_k$ shown as the black crosses and $\bar{L}_k$ as the black circles. To enhance the visualization, the deviation of the evaluated eigenvalues from the theoretical values under different setups are shown in \ref{Fig1:S1evErr}. The tenth eigenfunctions associated with the tenth  largest eigenvalues of the LLE under different setups are shown in \ref{Fig1:S1ef}. Note that when $\rho=3$, we recover the spectrum of the Laplace-Beltrami operator when the sampling is non-uniform; when $\rho=-5$, the non-uniform density function comes into play, and the eigenvalues are shifted from the theoretical value. 
For the non-uniform sampling scheme and $\rho=8$, theoretically the first three eigenvalues come from the six order term and depend on the non-uniform density function. Therefore, numerically the first three eigenvalues are non-zero. 
When $\rho=-5$ and $\rho=8$, the eigenfunctions are the same (up to the global rotation), and depend on the non-uniform density function. 
}\label{Figure:S1nonuniform}
\end{figure}

{Next, we show the results on $S^2$ with different radii under the non-uniform sampling scheme with $\rho=3$ and different $\epsilon$'s.  
Fix $n=30,000$. Take uniform sampling points $x_i=(x_{i1},x_{i2},x_{i3})^\top\in S^2\subset \mathbb{R}^3$, where $i=1,\ldots,n$, randomly choose $n/10$ points, randomly perturb those $n/10$ points by setting $\bar{x}_{i3}:=x_{i3}+1-\cos(2\pi U_i)$, where $U_i$ is the uniform distribution on $[0,1]$, and $y_i:=\frac{(x_{i1},x_{i2},\bar{x}_{i3})^\top}{\|(x_{i1},x_{i2},\bar{x}_{i3})^\top\|}$. As a result, $\mathcal{Y}:=\{y_i\}_{i=1}^n\subset S^2$ is nonuniformly distributed on $S^2$. Denote $r\mathcal{Y}$ to be the scaled sampling points on the sphere with radius $r>0$. 
Run the LLE on $r\mathcal{Y}$ with different $\epsilon$'s, and evaluate the first $400$ eigenvalues. We consider $r=0.5,1,2$. For $r=1$, consider $\epsilon=0.02$; for $r=0.5$, consider $\epsilon=0.02/4$ and $0.02/6$; for $r=2$, consider $\epsilon=0.02\times 4$ and $0.02\times 3$. Based on the theory, when $\rho=3$, the LLE is independent of the non-uniform density function and we obtain the eigenvalues of the Laplace-Beltrami operator in all cases.
See Figure \ref{Figure:S2Eigenvalues} for the results under different setups. Theoretically, the eigenvalues of $S^2$ without counting multiplicities are $\nu_i=-i(i+1)$, where $i=0,1,\ldots$. The multiplicity of $\nu_i$ is $2i+1$. When the radius is $r>0$, the eigenvalues are scaled by $r^{-2}$. The eigenvalues, as is shown in Figure \ref{Figure:S2Eigenvalues}, can be well estimated by the LLE, and the gap between the eigenvalues of spheres with different radii is predicted. The sawtooth behavior of the error comes from the spectral convergence behavior of eigenvalues with multiplicities. Note that there are $19$ eigenvalues with multiplicity greater than $1$ in the first $400$ eigenvalues, which match the $19$ oscillations found in Figure\ref{Fig2:S2evErr}. 
The eigenfunctions are shown in Figure \ref{Fig2:S2ef}. As is predicted, the first eigenfunction is constant, as is shown in $\psi_1$. The eigenspace of $\nu_1$ is spanned by three linear functions $x$, $y$, and $z$, restricted on $S^2$. Theresore, $\psi_4$ is a linear. The eigenspace of $\nu_\ell$ is spanned by spherical harmonics of order $\ell$, and its oscillation is illustrated in $\psi_9$ associated with $\nu_2$ and $\psi_{16}$ associated with $\nu_3$.}

\begin{figure}[h!]
\centering
\subfigure[$S^2$ Eigenvalues]{
\includegraphics[width=0.42\columnwidth]{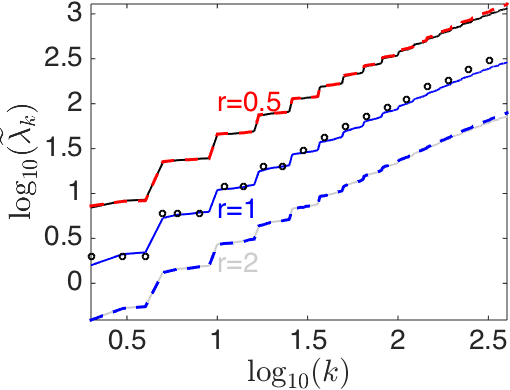}\label{Fig2:S2ev}
}
\subfigure[$S^2$ Eigenvalues error]{
\includegraphics[width=0.42\columnwidth]{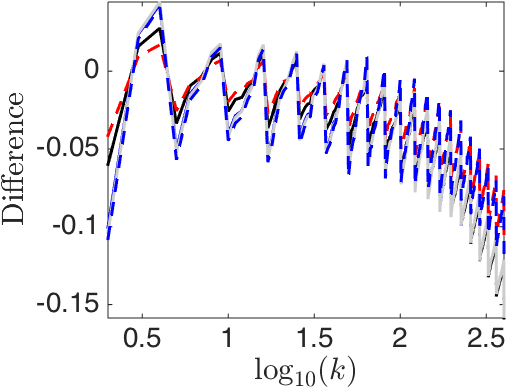}\label{Fig2:S2evErr}
}
\subfigure[$S^2$ Eigenfunctions]{
\includegraphics[width=0.42\columnwidth]{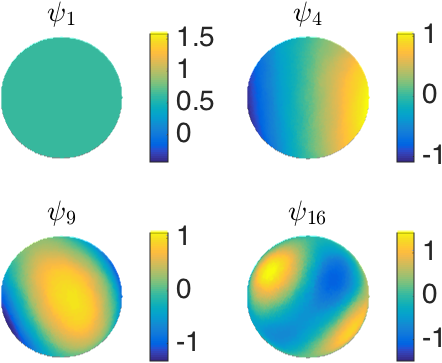}\label{Fig2:S2ef}
}
\caption{\ref{Fig2:S2ev}: the first $400$ eigenvalues of the LLE with $\rho=3$ but different $\epsilon$, over a $n=30,000$ non-uniform sampling points on $S^2$ with different radii $r>0$. $\tilde{\lambda}_k$ is the $k$-th smallest eigenvalue of the Laplace-Beltrami operator estimated by the LLE under different situations.
When $r=0.5$ (respectively $r=1$ and $r=2$), $\tilde{\lambda}_k$ are shown in the black (respectively blue and gray) curve. The results with different $\epsilon$ are shown as the red dash (respectively blue dash) when $r=0.5$ (respectively $r=2$).
The theoretical eigenvalues for the canonical $S^2$ (with the radius $1$), denoted as $L_k$, $k=1,\ldots$, are provided for a comparison (superimposed as black circles).
\ref{Fig2:S2evErr}: to enhance the visualization, the difference between the theoretical values and numerical values, $\log_{10}(\tilde{\lambda}_k)-\log_{10}(L_k)$, are shown with the same color and line properties as those shown on \ref{Fig2:S2ev}. Some eigenfunctions evaluated when $r=0.5$ are shown on \ref{Fig2:S2ef}.}\label{Figure:S2Eigenvalues}
\end{figure}

\subsection{Examine the kernel}

We now show the numerical simulations of the corresponding kernel on the unit circle $S^1$ embedded in $\mathbb{R}^2$. We take a uniform grid $\theta_i:=2\pi i/n$ on $(0,2\pi]$, where $n\in \mathbb{N}$ and $i=1,\ldots,n$, and construct $\mathcal{X}=\{x_i:=(\cos(\theta_i),\sin(\theta_i))^\top\}_{i=1}^n\subset\mathbb{R}^2$, which could be viewed as a uniform sampled set from the unit circle. We fix $n=10,000$. We then run the LLE with {$\epsilon=[(\cos(\theta_{K/2})-1)^2+\sin(\theta_{K/2})^2]^{1/2}$, where $K\in \mathbb{N}$}. See Figure \ref{Figure:S1Kernel} for an example of the corresponding kernels when $K=80$, and $K=320$. Note that the constructed normalized kernel, $\frac{K_{\texttt{LLE}}(x_{1000},y)}{\int K_{\texttt{LLE}}(x_{1000},y) dV(y)}$, is non-positive.

Next, we show the numerical simulations of the corresponding kernel on the $1$-dim flat torus $\mathbb{T}^1\sim \mathbb{R}/\mathbb{Z}$ with the induced metric from the canonical metric on $\mathbb{R}^1$. We take a uniform grid on $\mathbb{T}^1$ as $\{\theta_i=2\pi i/n\}_{i=1}^n$, and take $\mathcal{X}=\{x_i:=(\cos(\theta_i),\sin(\theta_i))^\top\}_{i=1}^n\subset\mathbb{R}^2$ to illustrate the flat torus. Fix $n=10,000$ and run the LLE with {$\epsilon=|\theta_{K/2}|$, where $K\in \mathbb{N}$}. See Figure \ref{Figure:S1Kernel} for an example of the corresponding kernels when $K=80$ and $K=320$. The constructed normalized kernel, as the theory predicts, is constant. Note that in this case, we can view the flat $1$-dim flat torus as the unit circle, when we have the access to the geodesic distance information on the manifold.

Finally, we take a look at the unit sphere $S^2$ embedded in $\mathbb{R}^3$ with the center at $(0,0,1)$, and its corresponding kernel. We uniformly sample $n$ points, $\mathcal{X}=\{x_i\}_{i=1}^n\subset\mathbb{R}^3$, from ${S}^2$. Fix $n=10,000$ and run the LLE with $400$ nearest neighbors. See Figure \ref{Figure:S1Kernel} for the corresponding kernel. Note that the normalized kernel is not positive. These examples show that even with the simple manifolds, the corresponding kernels might be complicated.

\begin{figure}[h!]
\centering
\subfigure[$S^1$ kernel]{
\includegraphics[width=0.42\columnwidth]{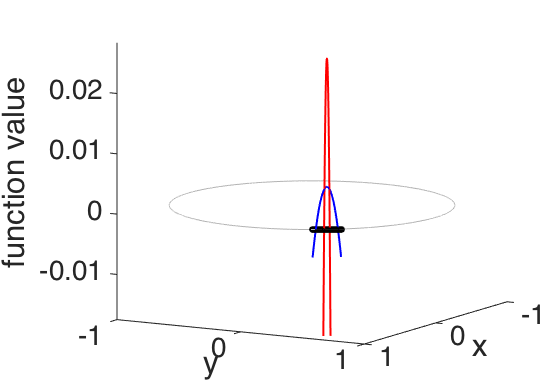}\label{Fig3:S1}
}
\subfigure[$\mathbb{T}^1$ kernel]{
\includegraphics[width=0.42\columnwidth]{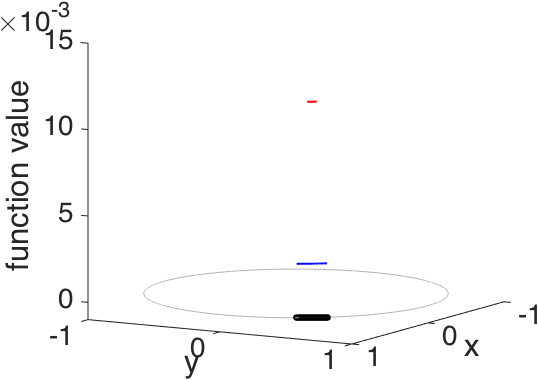}\label{Fig3:T1}
}
\subfigure[$S^2$ kernel]{
\includegraphics[width=0.42\columnwidth]{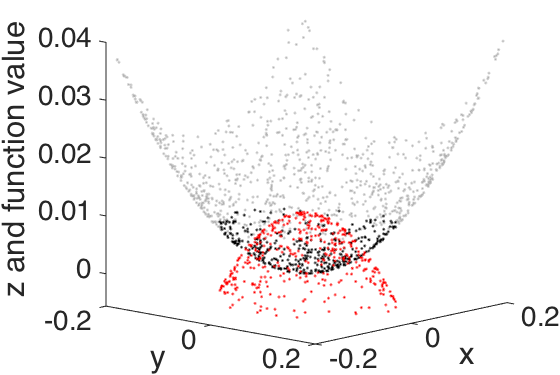}\label{Fig3:S2}
}
\caption{\ref{Fig3:S1}: the sampled $S^1$ is illustrated as the gray circle embedded in the $(x,y)$-plane. The black thick line indicates the first $320$ neighbors of the central point $x_{1000}$. The red line is the corresponding normalized kernel, $\frac{K_{\texttt{LLE}}(x_{1000},y)}{\int K_{\texttt{LLE}}(x_{1000},y) dV(y)}$, when $K=80$, and the blue line is the corresponding normalized kernel when $K=320$. It is clear that the kernel changes sign. 
\ref{Fig3:T1}: a surrogate of the sampled flat $1$-dim torus $\mathbb{T}^1$ is illustrated as the gray circle embedded in the $(x,y)$-plane. The black thick line indicates the first $320$ neighbors of the central point $x_{1000}$. The red line is the corresponding normalized kernel when $K=80$, and the blue line is the corresponding normalized kernel when $K=320$. In this flat manifold case, the kernel is constant.
\ref{Fig3:S2}: a surrogate of the uniformly sampled $\mathbb{S}^2$. Only the first $10,000$ nearest points of the chosen $x=(0,0,0)$ are plotted as the gray points. Note that the scale of the $x$ and $y$ axes and the $z$ axis are different. The black points indicate the first $400$ neighbors of $x$. The red points are the corresponding normalized kernel values when $K=400$. It is clear that the kernel is non-positive.
}\label{Figure:S1Kernel}
\end{figure}

\subsection{Two-dimensional random tomography example}

To further examine the capability of the LLE from the viewpoint of nonlinear dimension reduction, we consider the two-dimensional random tomography problem \cite{Singer_Wu:2013a}. It is chosen because its geometrical structure is well known and complicated. 

We briefly describe the dataset and refer the reader with interest to \cite{Singer_Wu:2013a}. The classical two-dimensional transmission computerized tomography problem is to recover the function $f:\mathbb{R}^2\rightarrow\mathbb{R}$ from its Radon transform. In the parallel beam model, the Radon transform of $f$ is given by the line integral
$R_\theta f (s)=\int_{x\cdot \theta=s}f(x)dx$, 
where $\theta\in S^1$ is perpendicular to the beaming direction $\theta^\perp \in S^1$, where $S^1$ is the unit circle, and $s\in \mathbb{R}$. We call $\theta$ the {\em projection direction} and $R_\theta f$ the {\em projected image}.
There are cases, however, in which we only have the projected images and the projection directions are unknown. 
In such cases, the problem at hand is to estimate $f$ from these projected images without knowing their corresponding projection directions. 
To better study this random projection problem, we need the following facts and assumptions. First, we know that for $f\in L^2(\mathbb{R}^2)$ with a compact support within $B_1(0)$, the map $R_{\cdot} f:\theta\in S^1\mapsto  L^2([-1,1])$ is continuous \cite{Singer_Wu:2013a}.
To simplify the discussion, we assume that
there is no symmetry in $f$; that is, $R_{\theta_1} f$ and $R_{\theta_2} f$ are different for all pairs of $\theta_1\neq \theta_2$.
Next, take $S:=\{s_i\}_{i=1}^p$ to be the chosen set of sampling points on $[-1,1]$, where $p\in \mathbb{N}$. In this example, we assume that $S$ is a uniform grid on $[-1,1]$; that is, $s_i=-1+2(i-1)/(p-1)$. For $\theta\in S^1$, denote the discretization of the projection image $R_\theta f$ as $D_S:L^2([-1,1])\to \mathbb{R}^p$, which is defined by
$D_S:R_{\theta}f \mapsto \left(R_{\theta}f\star h_\epsilon(s_1),R_{\theta}f\star h_\epsilon(s_2),\ldots,R_{\theta}f\star h_\epsilon(s_p)\right)^\top\in \mathbb{R}^p$, where $h_\epsilon(x):=\frac{1}{\epsilon}h(\frac{x}{\epsilon})$, $h$ is a Schwartz function, $h_\epsilon$ converges weakly to the Dirac delta measure at $0$ as $\epsilon\to 0$. Note that, in general, $R_{\theta}f$ is a $L^2$ function when $f$ is a $L^2$ function. Therefore, we need a convolution to model the sampling step. We assume that the discretization $D_S$ is dense enough, so that $M^1:=\{D_p\circ R_{\theta} f\}_{\theta\in S^1}$ is also simple.
In other words, we assume that $p$ is large enough so that $M^1$ is a one-dimensional closed simple curved embedded in $\mathbb{R}^p$ and $M^1$ is diffeomorphic to $S^1$.
Finally, we sample finite points from $S^1$ uniformly and obtain the simulation.

With the above facts and assumptions, we sample the Radon transform $\mathcal{X}:=\{x_i:=D_S\circ R_{\theta_i} f\}_{i=1}^{n}\subset \mathbb{R}^{p}$ with finite projection directions $\left\{\theta_i\right\}_{i=1}^n$, where $\left\{\theta_i\right\}_{i=1}^n$ is a finite uniform grid on $S^1$; that is, $\mathcal{X}$ is sampled from the one-dimensional manifold $M^1$. 
For the simulations with the Shepp-Logan phantom, we take $n=4096$, and the number of discretization points was $p=128$. It has been shown in \cite{Singer_Wu:2013a}, that the DM could recover the $M^1$ up to diffeomorphism, that is, we could achieve the nonlinear dimensional reduction. In order to avoid distractions, we do not consider any noise as is considered in \cite{Singer_Wu:2013a}, and focus our analysis on the clean dataset.  
The Shepp-Logan image, some examples of the projections and {the results of PCA, DM and LLE}, are shown in Figure \ref{Figure:SheppLogan}. As is shown in \cite{Singer_Wu:2013a}, the PCA fails to embed $\mathcal{X}$ with only the first three principal components, while the DM succeeded. There can be additional discussion for the DM, particularly its robustness to the noise and metric design. They have been extensively discussed in \cite{Singer_Wu:2013a}, so they are not discussed here. {For the LLE, we take $\epsilon=0.004$. The embedding results of the LLE with different regularization orders, $\rho=8,3,-5$, are shown.}
Due to the complicated geometrical structure, we encounter difficulty even to recover the topology of $M^1$ by the LLE, if the regularization order is not chosen properly.
%

To examine whether the sign of the kernel corresponding to the LLE is indeterminate in this database, we fixed $x_{3555}\in \mathcal{X}$, and apply the PCA to visualize its $K=150$ neighbors. The kernel function is shown in Figure \ref{Figure:SheppLogan} as the color encoded on the embedded points. The sign of the kernel is indeterminate, as is predicted by the above theory due to the existence of curvature.
%
%
In summary, we should be careful when we apply the LLE to a complicated real database.

\begin{figure}[h!]
\centering
\includegraphics[width=0.325\columnwidth]{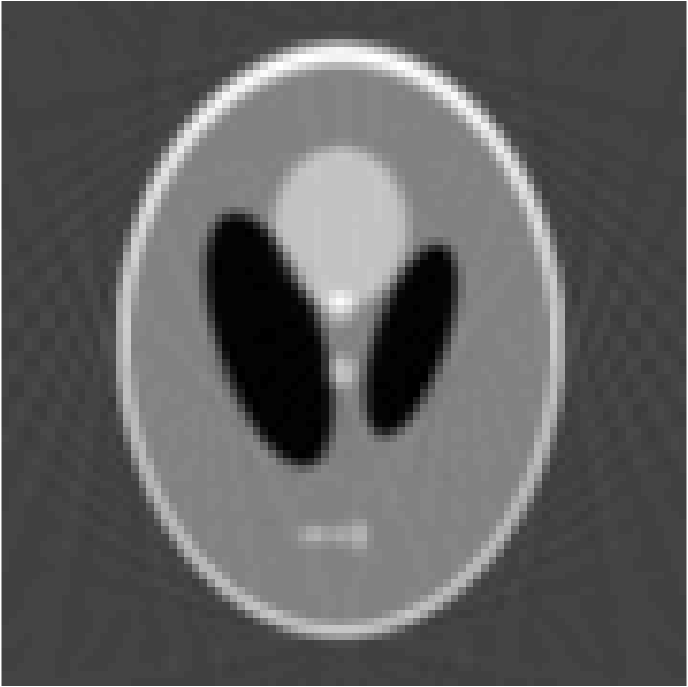}
\includegraphics[width=0.325\columnwidth]{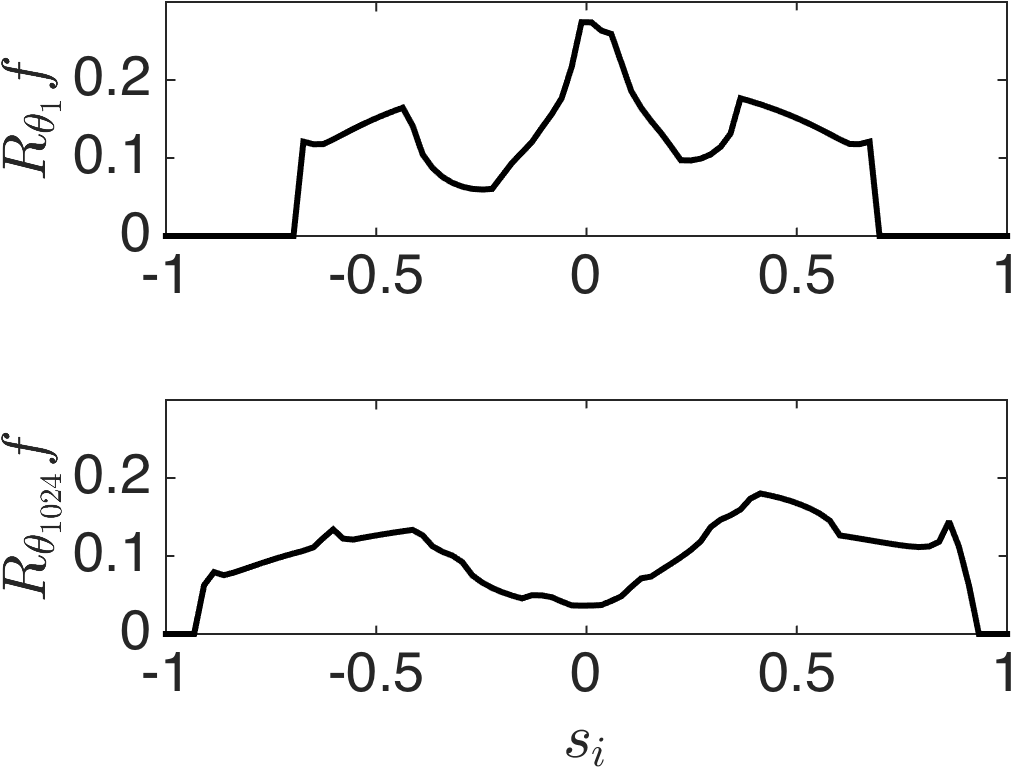}
\includegraphics[width=0.325\columnwidth]{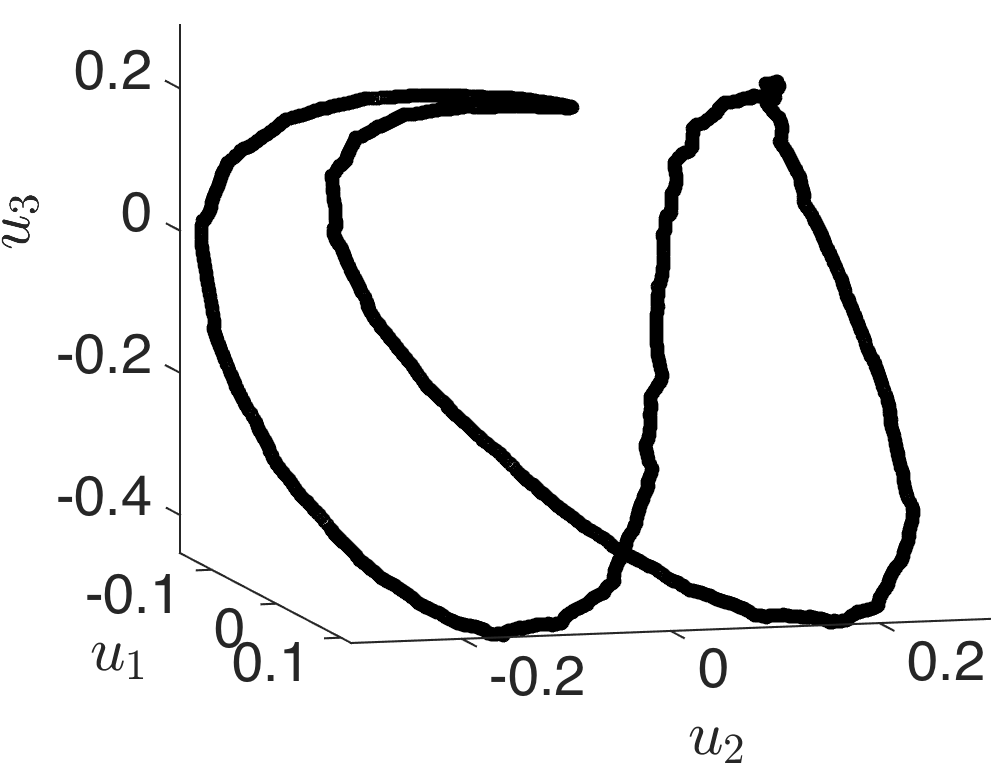}\\
\includegraphics[width=0.325\columnwidth]{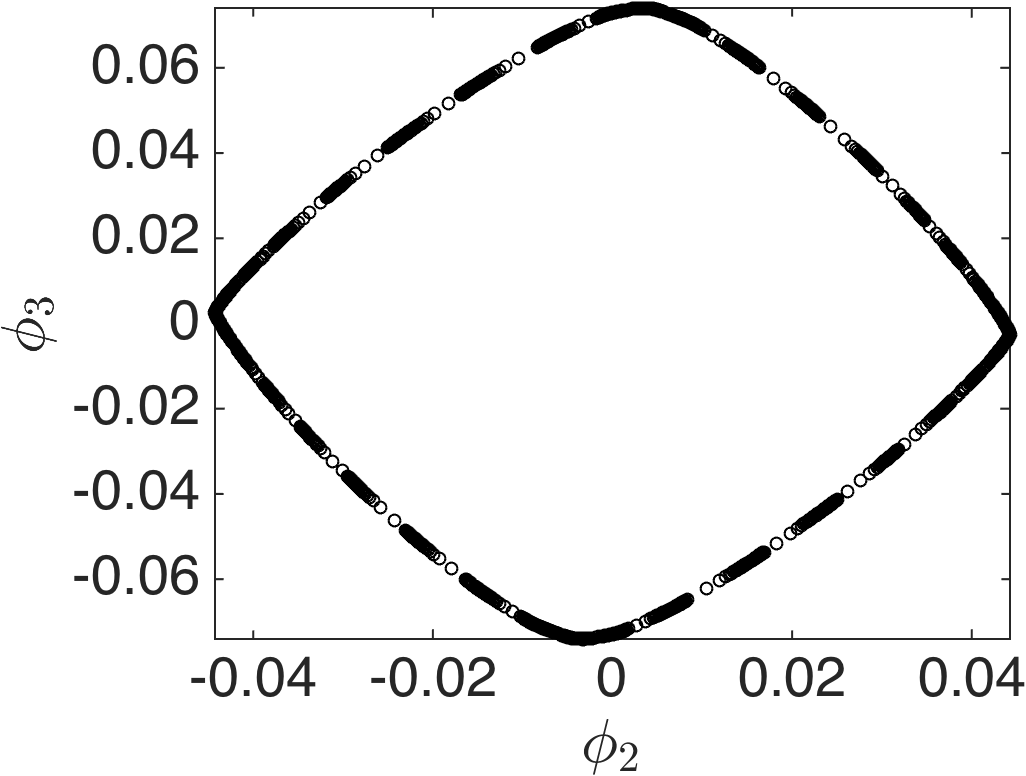}
\includegraphics[width=0.325\columnwidth]{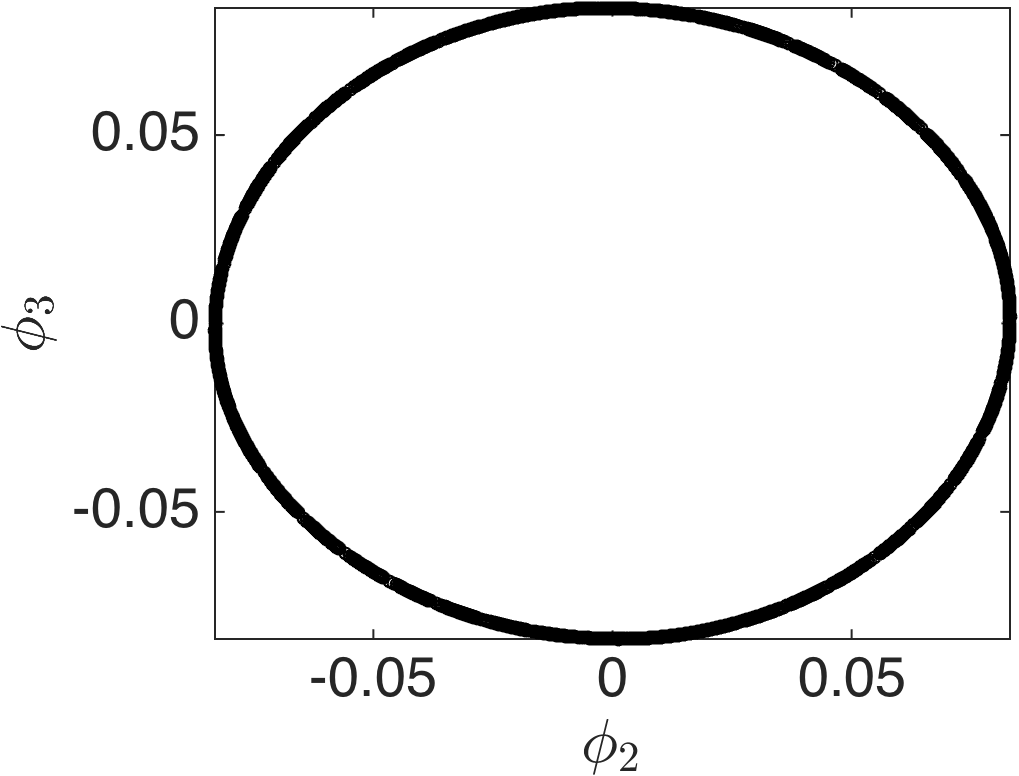}
\includegraphics[width=0.325\columnwidth]{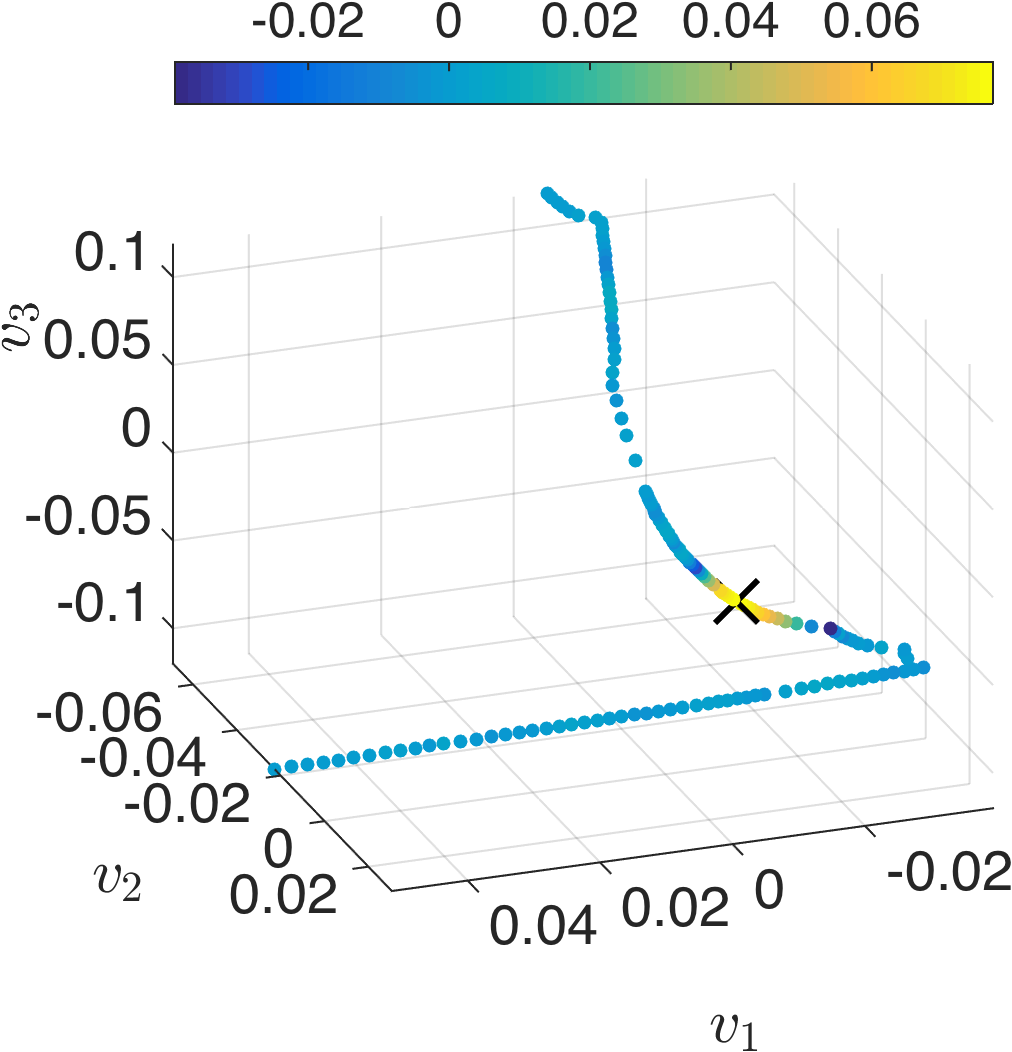}\\
\includegraphics[width=0.325\columnwidth]{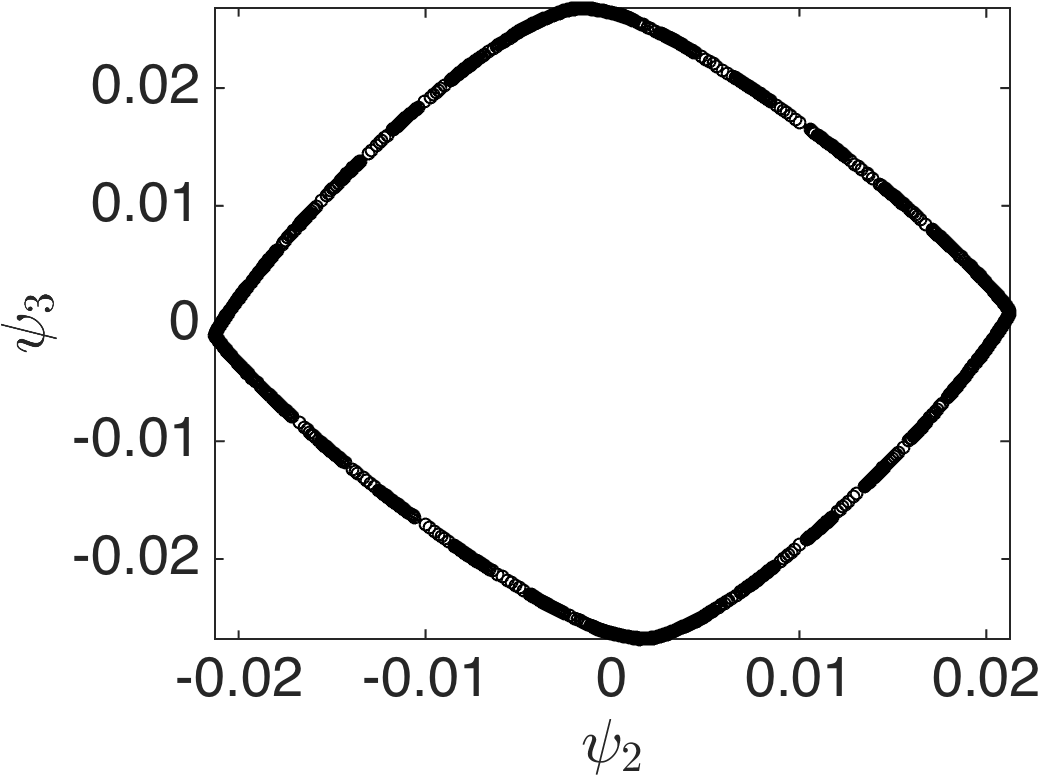}
\includegraphics[width=0.325\columnwidth]{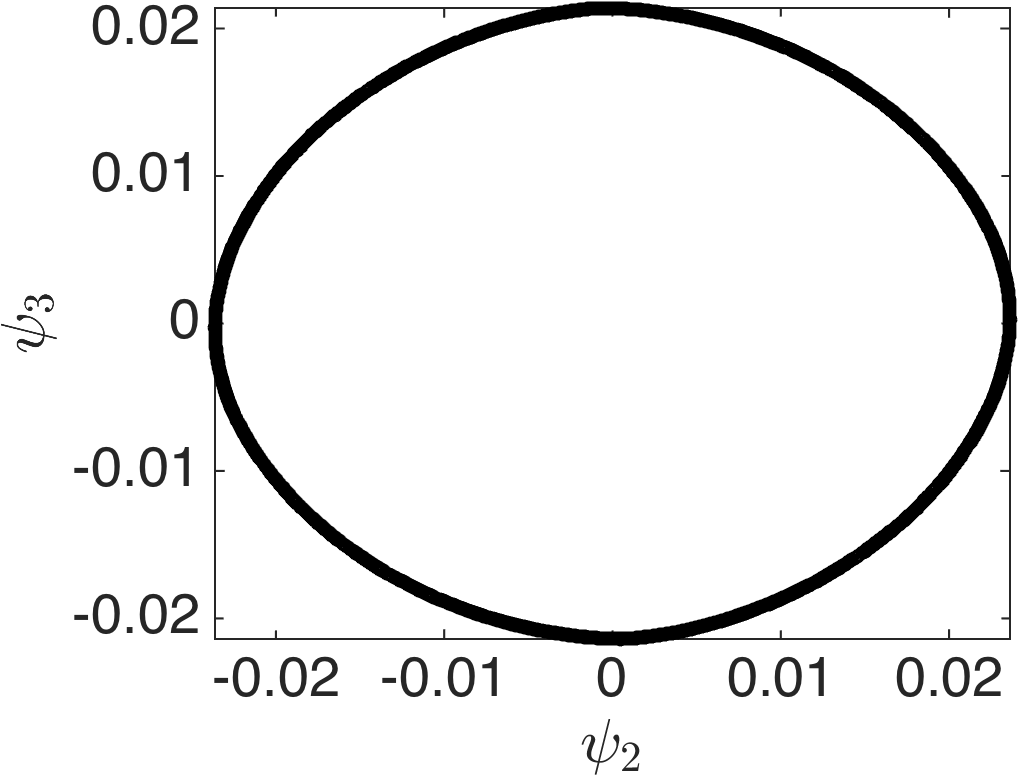}
\includegraphics[width=0.325\columnwidth]{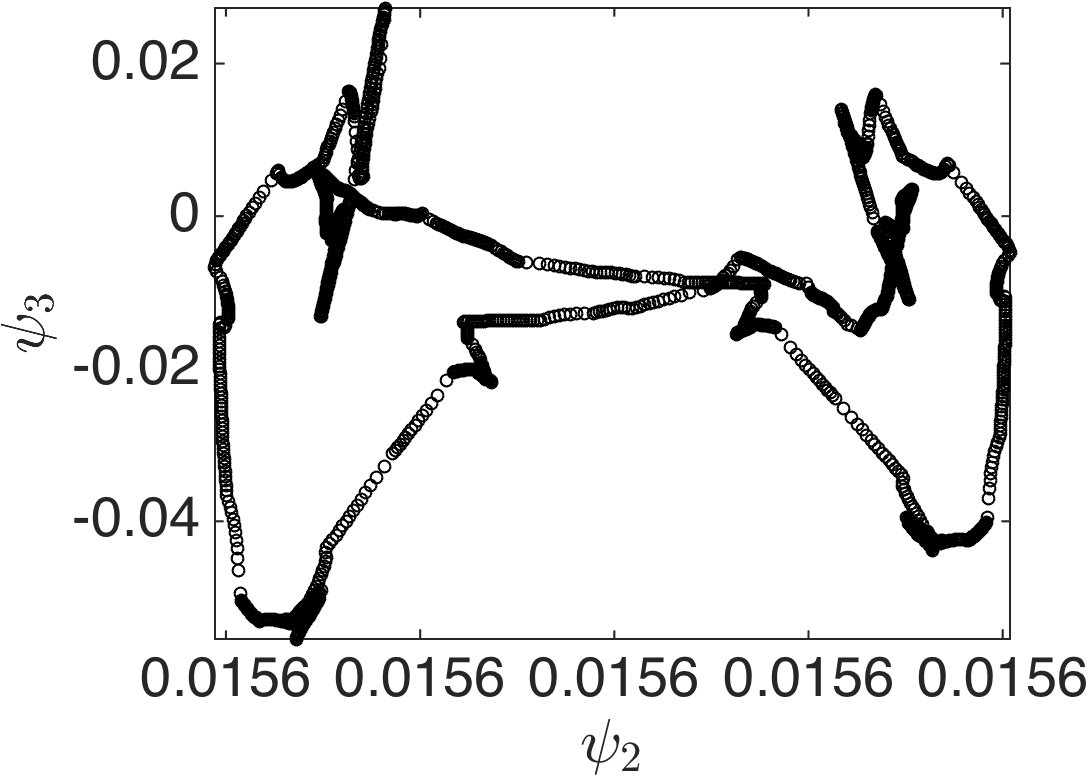}
\caption{Top row: the left panel is the Shepp-Logan phantom, the middle panel shows two projection images from two different projection directions, and the right panel shows the linear dimension reduction of the dataset by the first three principal components, $u_1,u_2$ and $u_3$. Middle row: the left panel shows the diffusion map (DM) of the dataset, where the embedding is done by choosing the first two non-trivial eigenvectors of the graph Laplacian, $\phi_2$ and $\phi_3$, and we simply take the Gaussian kernel to design the affinity without applying the $\alpha$-normalization technique \cite{Coifman_Lafon:2006}, the middle panel shows the DM of the dataset, where we apply the $\alpha$-normalization technique when $\alpha=1$, and the right panel shows that the sign of the kernel corresponding to the locally linear embedding (LLE) is indeterminate, where the black cross indicates $x_{3555}$, and the kernel value on its neighbors are encoded by color (the neighbors are visualized by the top three principal components, $v_1$, $v_2$, and $v_3$).
Bottom row: the embedding using the second and third eigenvectors of the LLE, $\psi_2$ and $\psi_3$, under different setups are shown. The left panel shows the result with $\rho=-5$, the middle panel shows the result with $\rho=3$, and the right panel shows the result with $\rho=8$. The results shows the importance of choosing the regularization and are explained by the theory.}\label{Figure:SheppLogan}
\end{figure}

\section{$\epsilon$-radius neighborhood v.s. $K$ nearest neighborhood}\label{Section:KNNvsEpsilon}

{
In the original article \cite{Roweis_Saul:2000}, the KNN scheme was proposed for the LLE algorithm. However, the analysis in this paper has been based on the $\epsilon$-radius neighborhood scheme. These two schemes are closely related asymptotically from the viewpoint of density function estimation \cite{Moore_Yackel:1977}. The following argument shows that the developed theorems are actually transferrable to the KNN scheme under the manifold setup.

We follow the notations in Section \ref{Section:ManifoldSetup}. 
For $\iota(x_k)\in \mathcal{X}$, take $K$ nearest neighbors of $\iota(x_k)$, namely $\iota(x_{k,1}), \ldots, \iota (x_{k,K})$, with respect to the Euclidean distance.  
Intuitively, $K$ is closely related to the volume of the minimal ball centered at $x_k$ with the radius $\epsilon(x_k)$ containing the $K$ nearest neighbors of $x_k$, where ${\epsilon(x_k)}$ depends on $K$ and the p.d.f.; that is, we expect to have 
\begin{equation}\label{Expected:K:epsilon:relationship}
nP(x_k)\text{vol}(D_{x_k})\approx K\,,
\end{equation}
where $D_x:=B^{\mathbb{R}^p}_{{\epsilon(x)}}(\iota(x)) \cap \iota(M)$ is the minimal ball centered at $x\in M$ with the radius $\epsilon(x)>0$ so that $D_x$ contains the $K$ nearest neighbors of $x$.
Under the smoothness assumption of the p.d.f. and the manifold setup, we claim that asymptotically when $n\to \infty$, this relationship holds uniformly over the manifold a.s., if $K=K(n)$, $K/\log(n)\to \infty$ and $K/n\to 0$ as $n\to \infty$. This claim could be achieved by slightly modifying the argument for the Theorem in \cite{Devroye_Wagner:1977} to obtain the large deviation bound for (\ref{Expected:K:epsilon:relationship}) when $n$ is finite. To bound $\text{Pr}\{\sup_{x \in M}|\frac{K}{n\text{vol}(D_{x})}-P(x)|>\alpha\}$, where $\alpha>0$, it is sufficient to bound the two terms on the right hand side of \cite[equation (10)]{Devroye_Wagner:1977}.
By a straightforward calculation of the equations on page 539 in  \cite{Devroye_Wagner:1977}, we achieve the bound $\text{Pr}\{\sup_{x \in M}|\frac{K}{n\text{vol}(D_{x})}-P(x)|>\alpha\}\leq \text{poly}(n)e^{-cK\alpha^3}$, where $c$ is a constant depending on  $d$ and the upper bounds of $P(x)$  on $M$, and $\text{poly}(n)=3(1+2^{p+3}n^{p+3})$.\footnote{{This can be observed by combining \cite[equations (6) (7) (9) and (10)]{Devroye_Wagner:1977}. The second term on the right hand side of \cite[equation (10)]{Devroye_Wagner:1977} is dominated by the first term. To bound the first term, we can substitute $\delta=\frac{K \beta}{4n(P_M+\beta)}$ and $M=\frac{4k P_M}{n \beta}$ into the fourth unlabeled equation on page 539  in  \cite{Devroye_Wagner:1977}, where $P_M$ is the upper bound of p.d.f. In the fourth unlabeled equation, $\alpha$ is the upper bound of the volume ratio of $B^{\mathbb{R}^p}_{2\epsilon(x)}(\iota(x)) \cap \iota(M)$ and  $B^{\mathbb{R}^p}_{\epsilon(x)}(\iota(x)) \cap \iota(M)$, which can be chosen as $3^d$ when $\epsilon(x)$ is sufficiently small. Finally, we use the fact that when $\beta$ is small, the equation follows.}}    
Therefore, if we choose $\alpha=(\frac{2p+10}{c})^{1/3}(\frac{\log n}{K})^{1/3}$, with probability greater than $1-n^{-2}$, we have uniformly $\frac{K}{n\text{vol}(D_{x})}=P(x)+O(\alpha)$. Note that by the assumption, $\alpha\to 0$ as $n\to \infty$. We conclude that with probability greater than $1-n^{-2}$, 
\begin{equation}\label{epsilon approximated by K/n}
\epsilon(x)=\Big(\frac{d}{|S^{d-1}|}\Big)^{1/d} \Big(\frac{K}{nP(x)}\Big)^{1/d}\Big(1+O\Big(\Big(\frac{\log n}{K}\Big)^{1/3}\Big)\Big)\,,
\end{equation}
where we use the fact that
$\text{vol}(D_{x})=\frac{|S^{d-1}|}{d} \epsilon(x)^{d}+O(\epsilon(x)^{d+1})$
when $\epsilon(x)$ is sufficiently small. It is transparent that $\epsilon(x)$ depends on $n$ and $\epsilon(x)\to 0$ a.s. as $n\to \infty$ since $K(n)/n\to 0$ by assumption. In other words, $\epsilon$ is not a constant value. It is a function depending on the p.d.f.. If we requre $K=K(n)$ to additionally satisfy $\frac{K(n)}{n}\frac{K(n)^{d/2}}{\log(n)^{d/2}}\to \infty$, then $\epsilon(x_k)$ satisfies $\frac{\sqrt{n}}{n^{1/2}\epsilon(x)^{d/2+1}}\to 0$ a.s..
On the other hand, notice that the statement of Theorem \ref{Corollary:LLE:Asymptotitcs} is pointwise. Therefore, its proof could be directly employed to the case when $\epsilon$ is chosen pointwisely, and hence the KNN scheme. As a result, 
if we take $\rho=3$ and is $K/n \to 0$, $K/\log(n)\to \infty$, and $(K/n)(K/\log(n))^{d/2}\to \infty$ when $n\to \infty$, by plugging (\ref{epsilon approximated by K/n}) into Theorem \ref{Corollary:LLE:Asymptotitcs}, when $n$ is sufficiently large, the following convergence holds for all $x_k$ with probability greater than $1-2n^{-2}$:
\begin{align}
\sum_{j=1}^K w_k(j)f(x_{k,j})&-f(x_k)=\frac{(\frac{d}{|S^{d-1}|})^{1/d}}{2(d+2)}\frac{\Delta f(x_k)}{P(x_k)^{2/d}}\Big(\frac{K}{n}\Big)^{2/d}\nonumber\\
&+O\Big(\Big(\frac{\log(n)}{K}\Big)^{1/3}\Big(\frac{K}{n}\Big)^{2/d}\Big)+O\Big(\Big(\frac{\log(n)}{K}\Big)^{1/2}\Big(\frac{K}{n}\Big)^{1/d}\Big)\label{RelationshipKNNepsilonAsymptoticalExpansion}\,.
\end{align}
In summary, unless the sampling is uniform, we do not obtain the Laplace-Beltrami operator with the KNN scheme. Based on the expansion (\ref{RelationshipKNNepsilonAsymptoticalExpansion}), to obtain the Laplace-Beltrami operator with the KNN scheme, we could numerically consider a ``normalized LLE matrix''; that is, find the eigen-structure of $\tilde{L}:=\mathcal{E}^{-1}(W-I)$, where $W$ is the ordinary LLE matrix, and $\mathcal{E}\in \mathbb{R}^{n\times n}$ is a diagonal matrix so that $\mathcal{E}_{ii}=\epsilon(x_i)^2$. Since the analysis of the pointwise convergence of $\tilde{L}$ is similar to that of Theorem \ref{Corollary:LLE:Asymptotitcs}, we skip the details here.

%
}

\section{Relationship with two statistical topics}\label{Section:LLR}

\subsection{Locally linear regression}
Based on the above theoretical study under the manifold setup, we could link the LLE to the locally linear regression (LLR) \cite{fan_gijbels:1996,Cheng_Wu:2013}. Recall that in the LLR, we locally fit a linear function to the response, and the associated kernel depends on the inverse of a variation of the covariance matrix. We summarize how the LLR is operated.
Consider the following regression model 
\begin{equation}\label{model1}
Y=m(X)+\sigma(X)\,\xi,
\end{equation}
where $\xi$ is a random error independent of $X$ with $\mathbb{E} (\xi)=0$ and $\operatorname{Var}(\xi)=1$, and both the regression function $m$ and the conditional variance function  $\sigma^2$ are defined on $\mathbb{R}^d$. 
Let  $\{(X_l,Y_l)\}_{l=1}^n$ denote a random sample observed from model (\ref{model1}) 
with $\mathcal{X}:=\{X_l\}_{l=1}^n$ being sampled from $X$. 
Given $\{(X_l,Y_l)\}_{l=1}^n$ and $x\in \mathbb{R}^d$, the problem is then to estimate $m({x})$ assuming enough smoothness of $m$.
Choose a smooth kernel function with fast decay $K:[0,\infty]\to \mathbb{R}$ and a bandwidth $\epsilon>0$. 
The LLR estimator for $m(x)$ is defined as $e_1^\top \hat{\boldsymbol{\beta}}_x$, where 
\begin{align}
\hat{\boldsymbol{\beta}}_x&=\arg\min_{\boldsymbol{\beta}\in\mathbb{R}^{d+1}} (\mathbf{Y}-\mathbf{X}_x\boldsymbol\beta)^\top\mathbf{W}_{x}(\mathbf{Y}-\mathbf{X}_x\boldsymbol\beta)\,,\label{functionalmfd}\\
\mathbf{Y}&=\left(Y_{1},\ldots,Y_{n}\right)^\top,\quad \mathbf{X}_x=\bigg[
\begin{array}{ccc}
1 & \dots & 1 \\
X_{1} & \dots & X_{n} \\
\end{array}
\bigg]^\top\in \mathbb{R}^{n\times (d+1)},\nonumber\\
\mathbf{W}_{x}&=\texttt{diag}\left(K_\epsilon(X_{1},x),\ldots,K_\epsilon(X_{n},x)\right)\in\mathbb{R}^{n\times n},\nonumber
\end{align}
and $K_\epsilon(X_{l},x):=\epsilon^{-d}K\big(\|X_{l}-x\|_{\mathbb{R}^d}\big/\epsilon\big)$. 
By a direct expansion, (\ref{functionalmfd}) becomes
\begin{equation}\label{functional2}
\hat{\boldsymbol{\beta}}_x=(\mathbf{X}_x^\top\mathbf{W}_{x}\mathbf{X}_x)^{-1}\mathbf{X}^\top_x\mathbf{W}_{x}\mathbf{Y}
\end{equation}
if $(\mathbf{X}^\top_x\mathbf{W}_{x}\mathbf{X}_x)^{-1}$ exists. We have $\mathbf{X}_x=\begin{bmatrix}\boldsymbol{1}_n^\top\\ \mathbf{G}_x\end{bmatrix}$, where $\mathbf{G}_x$ is the data matrix associated with $\{X_i\}_{i=1}^n$ centered at $x$. By yet another direct expansion by the block inversion,
\begin{equation}
e_1^\top \hat{\boldsymbol{\beta}}_x={w^{(\texttt{LLR})}_x}^\top\mathbf{Y}\,,
\end{equation}
where $w^{(\texttt{LLR})}_x$ is called the ``smoothing kernel'' and satisfies
\begin{equation}
w^{(\texttt{LLR})}_x:=\frac{\boldsymbol{1}_n^\top\mathbf{W}_x-\boldsymbol{1}_n^\top\mathbf{W}_x\mathbf{G}^\top_x(\mathbf{G}_x\mathbf{W}_x\mathbf{G}_x^\top)^{-1}\mathbf{G}_x\mathbf{W}_x}
{\boldsymbol{1}_n^\top\mathbf{W}_x\boldsymbol{1}_n-\boldsymbol{1}_n^\top\mathbf{W}_x\mathbf{G}^\top_x(\mathbf{G}_x\mathbf{W}_x\mathbf{G}_x^\top)^{-1}\mathbf{G}_x\mathbf{W}_x\boldsymbol{1}_n}.
\end{equation}
Through a direct comparison, we see that the vector $w^{(\texttt{LLR})}_x$ is almost the same as the weight matrix in the LLE algorithm shown in (\ref{Expansion:LLEweightedKernel}), except the weighting by the chosen kernel -- in the LLE, the kernel function {and its support are both determined by the data}, while in the LLR the kernel is selected in the beginning and the data points are weighted by the chosen kernel like $\mathbf{G}_x\mathbf{W}_x$. If we choose the kernel to be a zero-one kernel with the support on the ball centered at $x$ with the radius $\epsilon$, then we ``recover'' (\ref{Expansion:LLEweightedKernel}). 

{Under the low dimensional manifold setup, $\mathbf{G}_x\mathbf{W}_x\mathbf{G}_x^\top$ might not be of full rank. Note that the term $\mathbf{G}_x\mathbf{W}_x \mathbf{G}_x^\top$ is the weighted local covariance matrix, which is considered in \cite{Singer_Wu:2012} to estimate the tangent space. 
Unlike the regularized pseudo-inverse (\ref{Definition:Irho:Soft}) in the LLE, to handle this degeneracy issue, in LLR the data matrix $\mathbf{G}_x$ is constructed by projecting the point cloud to the estimated tangent plane. This projection step could be understood as taking the Moore-Penrose pseudo-inverse approach to handle the degeneracy.}
We mention that in \cite[Section 6]{Cheng_Wu:2013}, the relationship between the LLR and the manifold learning under the manifold setup is established. {It is shown that asymptotically, the smooth matrix from the kernel $w^{(\texttt{LLR})}_x$ leads to the Laplace-Beltrami operator. The result is parallel to the reported result in this paper. }

These relationships between the LLE and the LLR suggest the possibility of fitting the data locally by taking the {locally polynomial regression into account, and generalizing the barycentric coordinate by fitting a polynomial function locally. This might lead to a variation of the LLE that catches more delicate structure of the manifold, in a different adaptive way}. Since this direction is outside the scope of this paper, the study of this possibility is left to future studies.

\subsection{Error in variable}

In this work, we analyze the LLE under the assumption that the dataset is randomly sampled directly from a manifold, without any influence of the noise. However, the noise is inevitable and a further study is needed. {By the analysis, we observe that the LLE takes care of the error in variable challenge ``in some sense''.

Suppose the dataset is $\{y_i\}_{i=1}^n \subset \mathbb{R}^p$,  where $y_i=z_i+\xi_i$, $z_i$ is supported on a manifold and $\xi_i$ is an i.i.d. noise with good properties. The question is to ask how much information the LLE could recover from $\{z_i\}_{i=1}^n$. A parallel problem for the GL, or the more general graph connection Laplacian (GCL), has been studied in \cite{ElKaroui:2010a,ElKaroui_Wu:2016b}. It shows that the spectral properties of the GL and GCL are robust to noise. For the LLE, while a similar analysis could be applied, if we view the LLE as a kernel method and show a similar result, we mention that we might benefit by taking the special algorithmic structure of the LLE into account. 

When the dimension of the dataset is high, the noise might have a nontrivial behavior. For example, when the dimension of the database $p=p(n)$ satisfies $p(n)/n\to \gamma>0$ when $n\to \infty$ (known as the {\em large $p$ and large $n$ setup}), it is problematic to even estimate the covariance matrix. Note that the covariance matrix is directly related to the LLE algorithm
%
since the covariance matrix appears in the regularized pseudo inverse, $\mathcal{I}_{n\epsilon^{d+\rho}}(\bar G_n\bar G_n^\top)$, where $\bar G_n$ is the local data matrix associated with $y_k$ determined from the noisy database $\{y_i\}_{i=1}^n$, and $\bar G_n\bar G_n^\top$ is the covariance matrix.  
Under the large $p$ and large $n$ setup,  
the eigenvalues and eigenvectors of the covariance matrix will both be biased, depending on the ``signal-to-noise ratio'' and $\gamma$ \cite{Johnstone:2006}.
A careful manipulation of the noise, or a modification of the covariance matrix estimator, is needed in order to address these introduced biases. For example, the ``shrinkage technique'' was introduced to correct the eigenvalue bias with a theoretical guarantee \cite{Singer_Wu:2013a,Donoho_Gavish_Johnstone:2013}. The covariance matrix estimator based on the shrinkage technique is $\tilde{C}_n := \sum_{l=1}^p f(\lambda_l) u_lu_l^\top$, where $u_l$ and $\lambda_l$ form the $l$-th eigenpair of $\bar G_n\bar G_n^\top$ and $f$ is the designed shrinkage function.

A direct comparison shows that the regularized pseudo inverse in the LLE behaves like a shrinkage technique. Recall that $\mathcal{I}_{n\epsilon^{d+\rho}}(\bar G_n\bar G_n^\top)=\sum_{l=1}^{r_n} \frac{1}{\lambda_l+n\epsilon^{d+\rho}} u_lu_l^\top$ (\ref{Definition:Irho:Soft}), where $r_n$ is the rank of $\bar G_n\bar G_n^\top$, the shrinkage function is $f(x)=\frac{1}{x+n\epsilon^{d+\rho}}\chi_{(0,\infty)}(x)$, and $\chi$ is the indicator function. 
Although how $f$ corrects the noise impact is outside the scope of this paper,
it would be potential to carefully improve the regularized pseudo inverse by taking the shrinkage technique into account. In other words, by modifying the barycentric coordinate evaluation and applying the technique discussed in \cite{ElKaroui:2010a,ElKaroui_Wu:2016b}, it is possible to improve the LLE algorithm. An extensive study of the topic will be reported in the upcoming research.}

\section{Conclusions and Discussion}\label{Section:Conclusion}

We provide an asymptotical analysis of the LLE under the manifold setup. The theoretical results indicate that asymptotically, the LLE generally may not give the expected Laplace-Beltrami operator, unless the regularization is chosen properly. From the integral operator viewpoint, the corresponding kernel of the LLE in general is not positive. Therefore, the LLE in general is not a diffusion operator. Some direct calculations of the LLE operator over simple manifolds, like the sphere, indicate that asymptotically the fourth order differential operator might pop out as the dominant term, if the regularization is chosen to be too small. The numerical results support the theoretical findings. {In addition, we also discuss the relationship between the LLE and two statistical problems, the LLR and the error in variable problem, and point out the potential future work.}

There are more important topics we do not explore in this paper. 
{First, note that the pointwise convergence result established in this paper comes from a careful analysis of the ``fit locally'' part of the LLE algorithm. However, it is not sufficient to fully understand the ``think globally'' part of the LLE algorithm. Recall that we evaluate the eigen-decomposition of the LLE matrix for the embedding in the last step of the LLE algorithm. The theoretical and numerical results suggest that the eigen-structure of the LLE matrix provides an approximation of the eigen-structure of the Laplace-Beltrami operator. The embedding in the last step could therefore be understood from the point of view of the spectral embedding theory \cite{Berard:1986,Berard_Besson_Gallot:1994}. The eigen-structure of the LLE matrix integrates the local information. As a result, we catch the ``think globally'' part. However, the pointwise convergence is not strong enough to guarantee the spectral convergence. In other words, we need to show that asymptotically, the eigen-decomposition provides a proper approximation of the eigen-structure of the Laplace-Beltrami operator. While a similar proof of that in \cite{Singer_Wu:2016} could be slightly modified to achieve the spectral convergence of the LLE, however, more may be needed, such as the spectral convergence rate, from the statistical viewpoint.
Recently, there have been some relevant works for the GL under the manifold model in this direction \cite{GarciaTrillos_Slepcev:2015,Wang:2015}. Based on the special structure of the LLE, like the regularization, the optimal convergence rate of the LLE could be different and additional exploration is needed. The} result will be reported in the future work. 

Another important topic is the appearance of the fourth order differential operator in the LLE, when the manifold has a special structure and the regularization is improperly chosen. Although it would be a by-product, it would be interesting to ask if it is possible to take the fourth order differential operator into account in the data analysis and which kind of information could be extracted from the dataset. {It would also be interesting to ask if it is possible to directly obtain the fourth order differential operator for more general manifolds with a slight modification of the LLE algorithm. A direct benefit of this possibility is linked back to the regression problem, such as the LLR. If we could directly eliminate the second order term, the regression result could be more accurate.} We leave this study direction to the future work.

\section*{Acknowledgement}

Hau-tieng Wu's research is partially supported by Sloan Research Fellow FR-2015-65363. He acknowledges the continuous support from the Department of Mathematics, University of Toronto.

\bibliographystyle{plain}
\bibliography{bib}

\clearpage
\appendix

\setcounter{page}{1}
\renewcommand{\thepage}{SI.\arabic{page}}

{\large\center Online Supplementary Information for \\
\Large \center \textbf{Think globally, fit locally under the Manifold Setup}\\ \Large \center \textbf{Asymptotic Analysis of Locally Linear Embedding}\\
\large\center by Hau-Tieng Wu and Nan Wu\\}

\section{Perturbation analysis of eigenvalue and eigenvectors}\label{Section:Perturbation}

Suppose $A: \mathbb{R} \rightarrow S(p)$, where $S(p)$ is the set of real symmetric $p \times p$ matrices, is an analytic function around $0$. In this appendix, we are going to introduce an algorithm to calculate the eigenvalues and orthonormal eigenvectors of $A(\epsilon)$ when {$\epsilon$ is small enough.} The method introduced in this appendix follows the standard approach, like \cite{Andrew1998,VanDerAn_TerMorsche:2007}. For discussion of more general matrices, interested readers are referred to \cite{VanDerAn_TerMorsche:2007}. 

Suppose 
\begin{equation}
A(0)=\begin{bmatrix}
\lambda I_{d\times d} & 0 \\
0 & 0 \\
\end{bmatrix},\nonumber
\end{equation}
where $0<d< p$ and $\lambda\neq 0$. Decompose $A(0)$ by
\begin{equation} \label{A2}
A(0)X(0)=X(0)\Lambda(0) ,
\end{equation}
where $\Lambda(0)=A(0)$ is a diagonal matrix consisting of eigenvalues of $A(0)$, and 
\begin{equation}
X(0)=
\begin{bmatrix}
X_1 & 0 \\
0 & X_2 \\
\end{bmatrix}\in O(p),\nonumber
\end{equation}
where $X_1\in O(d)$ and $X_2\in O(p-d)$. 
Note that due to the possible nontrivial multiplicity of eigenvalues, $X(0)$ may not be uniquely determined.
Take the Taylor expansion of $A$ around $0$ as
\begin{equation}
A(\epsilon)=A(0)+A'(0)\epsilon+\frac{1}{2}A''(0)\epsilon^2+O(\epsilon^3)\,,\nonumber
\end{equation} 
where $\epsilon>0$ is sufficiently small, $A'(0)$ and $A''(0)$ are divided into blocks of the same size as those of $A(0)$ by
\begin{equation}
A'(0)=\begin{bmatrix}
A'_{11} & A'_{12} \\
A'_{21} & A'_{22} \\
\end{bmatrix}, \quad 
A''(0)=\begin{bmatrix}
A''_{11} & A''_{12} \\
A''_{21} & A''_{22} \\
\end{bmatrix},\nonumber
\end{equation}
where $A'_{11}\in S(d)$, $ A'_{22}\in S(p-d)$, $A''_{11}\in S(d)$ and $A''_{22}\in S(p-d)$. 
Let $\Lambda(\epsilon)\in \mathbb{R}^{p\times p}$ be the diagonal matrix consisting of eigenvalues  of $A(\epsilon)$ and $X(\epsilon)\in \mathbb{R}^{p\times p}$ be the matrix formed by the corresponding eigenvectors, i.e.
\begin{equation} \label{A1}
A(\epsilon)X(\epsilon)=X(\epsilon)\Lambda(\epsilon)\,.
\end{equation}
Since $A$ is symmetric, $X$ and $\Lambda$ are both analytic around $0$ based on \cite[Section 3.6.2, Theorem 1]{baumgArtel1985analytic}. We thus have the following Taylor expansion when $\epsilon$ is sufficiently small:
\begin{equation}
\Lambda(\epsilon)=\Lambda(0)+\epsilon\Lambda'(0)+\frac{1}{2}\epsilon^2\Lambda''(0)+O(\epsilon^3), \nonumber
\end{equation}
\begin{equation}
X(\epsilon)=X(0)+X'(0)\epsilon+O(\epsilon^2).\nonumber
\end{equation}
Here $\Lambda(0)$, $\Lambda'(0)$ and $\Lambda''(0)$ are all diagonal matrices and columns of $X(\epsilon)$ form an {\em orthogonal} set. 
Note that if we normalize $X(\epsilon)$ to be in $O(p)$, then by the fact that the Lie algebra of $O(p)$ is the set of anti-symmetric matrices, we know that $X(0)^{-1}X'(0)$ is an anti-symmetric matrix. 
We discuss the eigendecomposition of $A(\epsilon)$ under two different setups, depending on the multiplicity of eigenvalues.

\medskip

\textbf{When there is no repeated eigenvalue in both $A'_{11}$ and $A'_{22}$.}  In the first case, we assume that the eigenvalues of $A'_{11}$ are distinct and the eigenvalues of $A'_{22}$ are distinct (but the eigenvalues of $A'_{11}$ and the eigenvalues of $A'_{22}$ could overlap). 
To get $\Lambda(\epsilon)$ up to the first order, we need to solve $\Lambda'(0)$. 
To determined $\Lambda'(0)$, 
we check the first order derivative of $A(\epsilon)$ at $\epsilon=0$. 
Differentiate (\ref{A1}) and we get
\begin{equation} \label{A3}
A'(0)X(0)+A(0)X'(0)=X'(0)\Lambda(0)+X(0)\Lambda'(0)\,.\nonumber
\end{equation}
Denote
\begin{equation}
\Lambda'(0)=
\begin{bmatrix}
\Lambda'_1 & 0 \\
0 & \Lambda'_2 \\
\end{bmatrix}\nonumber
\end{equation}
and set 
\begin{equation}
X'(0)=X(0)C, \nonumber
\end{equation}
where
\begin{equation}
C=
\begin{bmatrix}
C_{11} & C_{12} \\
C_{21} & C_{22} \\
\end{bmatrix}\in\mathbb{R}^{p\times p}.\nonumber
\end{equation}
If we substitute $X(0)$, $X'(0)$, and $\Lambda'(0)$ into (\ref{A3}), we have the following linear equations by comparing blocks:
\begin{align}
A'_{11}X_1&=X_1 \Lambda'_1,\label{AppendixB:Norepeat:Eq1} \\
A'_{22}X_2&= X_2\Lambda'_2,\label{AppendixB:Norepeat:Eq4}\\
A'_{12}X_2&=-\lambda X_1C_{12},\label{AppendixB:Norepeat:Eq2} \\
A'_{21}X_1&= \lambda X_2C_{21} , \label{AppendixB:Norepeat:Eq3}
\end{align}
By (\ref{AppendixB:Norepeat:Eq1}) and (\ref{AppendixB:Norepeat:Eq4}), $\Lambda'_1$ and $\Lambda'_2$ are eigenvalue matrices of $A'_{11}$ and $A'_{22}$, and $X_1$ and $X_2$ are the corresponding eigenvector matrices, and we obtain the first order approximation of the eigenvalues.
Note that above equations hold without assuming that the eigenvalues of $A'_{11}$ are distinct and the eigenvalues of $A'_{22}$ are distinct. Also note that although we could obtain the first order relationship between the eigenvectors of $A(\epsilon)$ and $A(0)$, without assuming distinct eigenvalues, the eigenvectors may not be unique. 

If we want to further get $\Lambda(\epsilon)$ up to the second order and solve $X(\epsilon)$ uniquely up to the first order, we need to solve $\Lambda'(0)$, $\Lambda''(0)$, $X(0)$, and $X'(0)$. To solve $\Lambda'(0)$, $\Lambda''(0)$, $X(0)$, and $X'(0)$, we need the assumption that $A'_{11}$ and $A'_{22}$ have no repeated eigenvalues, while we allow eigenvalues of $A'_{11}$ to be same as those of $A'_{22}$. 
By (\ref{AppendixB:Norepeat:Eq2}) and (\ref{AppendixB:Norepeat:Eq3}) we have 
\begin{align}
C_{12}&=-\lambda^{-1}X_1^\top A'_{12}X_2 , \label{AppendixB:Norepeat:Eq2.1}\\
C_{21}&=\lambda^{-1}X_2^\top A'_{21}X_1 \label{AppendixB:Norepeat:Eq3.1}.
\end{align}
Clearly, since $A'_{11}$ and $A'_{22}$ do not have repeated eigenvalues, $X_1$ and $X_2$ are uniquely defined and $C_{12}$ and $C_{21}$ can be uniquely determined. 
Since the information of $C_{11}$ and $C_{22}$ are not available from (\ref{A3}), we need the higher order derivative of $A(\epsilon)$. Differentiate $A(\epsilon)$ twice, we get
\begin{equation} \label{A4}
A''(0)X(0)+2A'(0)X'(0)+A(0)X''(0)=X''(0)\Lambda(0)+2X'(0)\Lambda'(0)+X(0)\Lambda''(0) .
\end{equation}
Now, we further substitute $X(0)$ and $X'(0)=X(0)C$ into (\ref{A4}), and get
\begin{align}
A''_{11}X_1-X_1\Lambda''_1&=2X_1(C_{11}\Lambda'_1-\Lambda'_1C_{11})-2A'_{12}X_2C_{21},\label{AppendixB:Norepeat:Eq5}\\
A''_{22}X_2-X_2\Lambda''_2&=2X_2(C_{22}\Lambda'_2-\Lambda'_2C_{22})-2A'_{21}X_1C_{12}.\label{AppendixB:Norepeat:Eq6}
\end{align}
Since $X_1\in O(d)$ and $X_2\in O(p-d)$, we have
\begin{align} 
X^\top _{1}(A''_{11}X_1+2A'_{12}X_2C_{21})&=2(C_{11}\Lambda'_1-\Lambda'_1C_{11})+\Lambda''_1\label{A5} \\
X^\top _{2}(A''_{22}X_2+2A'_{21}X_1C_{12})&=2(C_{22}\Lambda'_2-\Lambda'_2C_{22})+
\Lambda''_2.\label{A6}
\end{align}
Since that diagonal entries of $C_{11}\Lambda'_1-\Lambda'_1C_{11}$ and $C_{22}\Lambda'_2-\Lambda'_2C_{22}$ are zero, and the off-diagonal entries of $\Lambda''_1$ and $\Lambda''_2$ are zero, the off diagonal entries of $C_{11}$ and $C_{22}$, as well as $\Lambda''_{11}$ and $\Lambda''_{22}$, can be found from (\ref{A5}). Specially, since $A'_{11}$ and $A'_{22}$ do not have repeated eigenvalues, we have
\begin{align}
(C_{11})_{m,n}&=\frac{-1}{(\Lambda'_1)_{m,m}-(\Lambda'_1)_{n,n}}e_m^\top\big(X_1^\top A_{11}''X_1+\frac{2}{\lambda}X_1^\top A_{12}'A_{21}'X_1\big)e_n,\nonumber\\
(\Lambda_{11}'')_{m,m}&=e_m^\top\big(X_1^\top A_{11}''X_1+\frac{2}{\lambda}X_1^\top A_{12}'A_{21}'X_1\big)e_m,\nonumber
\end{align}
where $1\leq m\neq n\leq d$ and
\begin{align}
(C_{22})_{m,n}&=\frac{-1}{(\Lambda'_2)_{m,m}-(\Lambda'_2)_{n,n}}e_m^\top \big(X_2^\top A_{22}''X_2-\frac{2}{\lambda}X_2^\top A_{21}'A_{12}'X_2\big)e_n,\nonumber\\
(\Lambda_{22}'')_{m,m}&=e_m^\top \big(X_2^\top A_{22}''X_2-\frac{2}{\lambda}X_2^\top A_{21}'A_{12}'X_2\big)e_m,\nonumber
\end{align}
where $1\leq m\neq n\leq p-d$. By the above evaluation, we know $C_{i,j}=-C_{j,i}$ for $1\leq i\neq j\leq p$, and what is left unknown is the diagonal entries of $C$.
To determine the diagonal entries of $C$, we normalize $X(\epsilon)=X(0)+X(0)C\epsilon+O(\epsilon^2)$ so that $X(\epsilon)\in O(p)$. We thus have
\begin{align} \label{A7}
I_{p\times p}&=(X(0)+X'(0)\epsilon+O(\epsilon^2))^\top (X(0)+X'(0)\epsilon+O(\epsilon^2))\nonumber\\
&=X(0)^\top X(0)+(C^\top X(0)^\top X(0)+X(0)^\top X(0)C)\epsilon+O(\epsilon^2)\nonumber\\
&=I_{p\times p}+2\epsilon \texttt{diag}(C)+O(\epsilon^2),
\end{align}
where the last equality holds since $C_{i,j}=-C_{j,i}$ when $i\neq j$, and $\texttt{diag}(C)$ is a diagonal matrix so that $\texttt{diag}(C)_{i,i}=C_{i,i}$ for $i=1,\ldots,p$. As a result, we know that the diagonal entries of $C$ are of order $\epsilon$. As a result, we have the following solution to the eigenvalues and eigenvectors of $A(\epsilon)$:
\begin{align}
\Lambda(\epsilon)&=\begin{bmatrix}
\lambda I_{d\times d} +\epsilon\Lambda'_1+\frac{1}{2}\epsilon^2\Lambda''_1 & 0 \\
0 & \epsilon\Lambda'_2+\frac{1}{2}\epsilon^2\Lambda''_2 \\
\end{bmatrix}+O(\epsilon^3),\label{AppendixB:Norepeat:ResultLambda} \\
X(\epsilon)&=X(0)(I_{p\times p}+\epsilon \mathsf{S})+O(\epsilon^2)\in O(p),\label{AppendixB:Norepeat:ResultEigenVector}
\end{align}
where $\mathsf{S}:=C-\texttt{diag}(C)$
and the last equality holds since the entries of $\texttt{diag}(C)$ are of order $\epsilon$. Note that $\mathsf{S}$ is an anti-symmetric matrix. This result could be understood from the fact that the Lie algebra of $O(p)$ is the set of anti-symmetric matrices, and the tangent vector at $X(0)$ leading to $X(\epsilon)$ is $X(0)\mathsf{S}$.

\medskip
\textbf{When there exists a repeated eigenvalue in $A'_{22}$.}  In this case, we assume that $A'_{22}$ may have repeated eigenvalues, and to simplify the discussion, we assume that $A'_{11}$ does not have a repeated eigenvalue. Recall (\ref{AppendixB:Norepeat:Eq1}) and (\ref{AppendixB:Norepeat:Eq4}).
Write
\begin{equation}
\Lambda'_2=
\begin{bmatrix}
\Lambda'_{2,1} &0\\
0 & \Lambda'_{2,2} \\
\end{bmatrix},\nonumber
\end{equation}
where $\Lambda'_{2,2}\in \mathbb{R}^{l\times l}$, $1\leq l\leq p-d$, is a diagonal matrix with the same diagonal entries, denoted as $\gamma\in\mathbb{R}$. To simplify the discussion, we assume that the diagonal entries of $\Lambda'_{2,1}\in \mathbb{R}^{(p-d-l)\times (p-d-l)}$ are all distinct and are different from $\gamma$. Hence, we have
\begin{equation}
\Lambda'(0)=
\begin{bmatrix}
\Lambda'_1 & 0 &0\\
0 & \Lambda'_{2,1} &0\\
0 & 0 & \Lambda'_{2,2} \\
\end{bmatrix}.\nonumber
\end{equation}
Let $\Gamma_1\in O(d)$ be the orthonormal eigenvector matrix of $A'_{11}$ and $\Gamma_2\in O(p-d)$ be any
orthonormal eigenvector matrix of $A'_{22}$. Define 
\begin{equation}
\Gamma=
\begin{bmatrix}
\Gamma_1 & 0 \\
0 & \Gamma_2 \\
\end{bmatrix}.\nonumber
\end{equation}
Consider 
\begin{equation}
\tilde{A}(\epsilon) =\Gamma^{-1}A(\epsilon)\Gamma\,.\label{AppendixA:Repeated:AtildeDefinition}
\end{equation}
Note that $A(\epsilon)$ has the same eigenvalue matrix as $\tilde{A}(\epsilon)$.  By a direct expansion,
\begin{align}
\tilde{A}(\epsilon)=&\,\Gamma^{-1}A(\epsilon)\Gamma\nonumber \\
=&\,\Gamma^{-1}A(0)\Gamma+\Gamma^{-1}A'(0)\Gamma\epsilon+\frac{1}{2}\Gamma^{-1}A''(0)\Gamma\epsilon^2+O(\epsilon^3)  \nonumber \\
=&\,\tilde{A}(0)+\tilde{A}'(0)\epsilon+\frac{1}{2}\tilde{A}''(0)\epsilon^2+O(\epsilon^3)  \nonumber \,,
\end{align}
where $\tilde{A}(0):=\Gamma^{-1}A(0)\Gamma$,  $\tilde{A}'(0):=\Gamma^{-1}A'(0)\Gamma$, and $\tilde{A}''(0):=\Gamma^{-1}A''(0)\Gamma$.
By the assumption of $A(0)$, we have 
\begin{equation}
\tilde{A}(0)=\Gamma^{-1}A(0)\Gamma=A(0)\nonumber.
\end{equation}
 Furthermore, we have 
\begin{align}
\tilde{A}'(0)=\Gamma^{-1}A'(0)\Gamma=
\begin{bmatrix}
\Gamma_1^{-1}A'_{11}\Gamma_1 & \Gamma_1^{-1}A'_{12}\Gamma_2\\
\Gamma_2^{-1}A'_{21}\tilde{X}_1 & \Gamma_2^{-1}A'_{22}\Gamma_2  \\
\end{bmatrix}
=\begin{bmatrix}
\Lambda'_1 & \Gamma_1^{-1}A'_{12}\Gamma_2\\
\Gamma_2^{-1}A'_{21}\Gamma_1 & \Lambda'_2  \\
\end{bmatrix}\,,\nonumber
\end{align}
where the last equality holds since $\Gamma_1$ and $\Gamma_2$ are eigenvector matrices of $A'_{11}$ and $A'_{22}$. 
Then, we divide $\tilde{A}(0)$, $\tilde{A}'(0)$ and $\tilde{A}''(0)$ and $\Lambda''(0)$ into blocks in the same way as that of $\Lambda'(0)$:
\begin{align*}
\tilde{A}(0)&=\begin{bmatrix}
\lambda I_{d\times d} & 0 & 0\\
0 & 0 & 0\\
0 & 0 & 0\\
\end{bmatrix} \quad 
\tilde{A}'(0)=\begin{bmatrix}
\Lambda'_1 & \tilde{A}'_{12,1} & \tilde{A}'_{12,2} \\
\tilde{A}'_{21,1} & \Lambda'_{2,1} & 0 \\
\tilde{A}'_{21,2} & 0 & \Lambda'_{2,2} \\
\end{bmatrix} \\ 
\tilde{A}''(0)&=\begin{bmatrix}
\tilde{A}''_{11} & \tilde{A}''_{12,1} & \tilde{A}''_{12,2} \\
\tilde{A}''_{21,1} & \tilde{A}''_{22,11} & \tilde{A}''_{22,12} \\
\tilde{A}''_{21,2} & \tilde{A}''_{22,21} & \tilde{A}''_{22,22} \\
\end{bmatrix}\quad
\Lambda''(0)=
\begin{bmatrix}
\Lambda''_1 & 0 &0\\
0 & \Lambda''_{2,1} &0\\
0 & 0 & \Lambda''_{2,2} \\
\end{bmatrix} \,,\nonumber
\end{align*}
where we use the following notations for the blocks of $\tilde{A}'(0)$:
\begin{align}
\Gamma_2^{-1}A'_{21}\Gamma_1=\begin{bmatrix}
\tilde{A}'_{21,1} \\
\tilde{A}'_{21,2} \\
\end{bmatrix} \quad 
\Gamma_1^{-1}A'_{12}\Gamma_2=\begin{bmatrix}
\tilde{A}'_{12,1} & \tilde{A}'_{12,2} \\
\end{bmatrix}\,.\nonumber
\end{align}
%
If $\tilde{X}(\epsilon)$ is an orthonormal eigenvector matrix of $\tilde{A}(\epsilon)$, by (\ref{AppendixA:Repeated:AtildeDefinition}), we have 
\begin{equation}
X(\epsilon)=\Gamma \tilde{X}(\epsilon)\,.\nonumber
\end{equation}
By the expansion $\tilde{X}(\epsilon)=\tilde{X}(0)+\epsilon\tilde{X}'(0)+O(\epsilon^2)$, we have $X(0)=\Gamma \tilde{X}(0)$ and $X'(0)=\Gamma \tilde{X}'(0)$. Therefore, it is sufficient to find $\tilde{X}(0)$ and $\tilde{X}'(0)$.  
Since
\begin{align}
\tilde{A}'_{22}=\Gamma_2^{-1}A'_{22}\Gamma_2 =
\begin{bmatrix}
\Lambda'_{2,1} &0\\
0 & \Lambda'_{2,2}\\
\end{bmatrix}\nonumber
\end{align} 
is a diagonal matrix after the conjugation with $\Gamma$, by the  assumption about the eigenvalues and (\ref{AppendixB:Norepeat:Eq1}) and (\ref{AppendixB:Norepeat:Eq4}), we have
\begin{align}
\tilde{X}(0)=
\begin{bmatrix}
\tilde{X}_1 & 0 &0\\
0 & \tilde{X}_{2,1} &0\\
0 & 0 & \tilde{X}_{2,2} \\
\end{bmatrix} .\nonumber
\end{align}
Similarly, define $\tilde{X}'(0) =\tilde{X}(0)C$, where we divide $C$ into blocks in the same way as that of $\Lambda'(0)$:
\begin{align}
C=\begin{bmatrix}
C_{11} & C_{12,1} & C_{12,2} \\
C_{21,1} & C_{22,11} & C_{22,12} \\
C_{21,2} & C_{22,21} & C_{22,22} \\
\end{bmatrix}\,.\nonumber
\end{align}
Under such a block decomposition, we apply (\ref{A3}) to $\tilde{A}'(0)$, and we have
\begin{align}
\Lambda'_1\tilde{X}_1&=\tilde{X}_1 \Lambda'_1,\label{AppendixB:Repeat:Eq1} \\
\Lambda'_{2,1}\tilde{X}_{2,1}&= \tilde{X}_{2,1}\Lambda'_{2,1}.\label{AppendixB:Repeat:Eq41}\\
\Lambda'_{2,2}\tilde{X}_{2,2}&= \tilde{X}_{2,2}\Lambda'_{2,2}.\label{AppendixB:Repeat:Eq42}\\
\tilde{A}'_{12,1}\tilde{X}_{2,1}&=-\lambda \tilde{X}_1C_{12,1},\label{AppendixB:Repeat:Eq21} \\
\tilde{A}'_{12,2}\tilde{X}_{2,2}&=-\lambda \tilde{X}_1C_{12,2},\label{AppendixB:Repeat:Eq22} \\
\tilde{A}'_{21,1}\tilde{X}_1&= \lambda \tilde{X}_{2,1}C_{21,1} , \label{AppendixB:Repeat:Eq31}\\
\tilde{A}'_{21,2}\tilde{X}_1&= \lambda \tilde{X}_{2,2}C_{21,2}\,. \label{AppendixB:Repeat:Eq32}
\end{align}
Then, we apply (\ref{A4}) to $\tilde{A}''(0)$, we have
\begin{align}
\tilde{A}''_{11}\tilde{X}_1-\tilde{X}_1\Lambda''_1&=2\tilde{X}_1(C_{11}\Lambda'_1-\Lambda'_1C_{11})-2\tilde{A}'_{12,1}\tilde{X}_{2,1}C_{21,1}-2\tilde{A}'_{12,2}\tilde{X}_{2,2}C_{21,2},\label{AppendixB:Repeat:Eq5}\\
\tilde{A}''_{22,11}\tilde{X}_{2,1}-\tilde{X}_{2,1}\Lambda''_{2,1}&=2\tilde{X}_{2,1}(C_{22,11}\Lambda'_{2,1}-\Lambda'_{2,1}C_{22,11})-2\tilde{A}'_{21,1}\tilde{X}_1C_{12,1},\label{AppendixB:Repeat:Eq61}\\
\tilde{A}''_{22,22}\tilde{X}_{2,2}-\tilde{X}_{2,2}\Lambda''_{2,2}&=2\tilde{X}_{2,2}(C_{22,22}\Lambda'_{2,2}-\Lambda'_{2,2}C_{22,22})-2\tilde{A}'_{21,2}\tilde{X}_1C_{12,2},\label{AppendixB:Repeat:Eq62}\\
\tilde{A}_{22,12}''\tilde{X}_{2,2}&=2\tilde{X}_{2,1}(C_{22,12}\Lambda'_{2,2}-\Lambda'_{2,1}C_{22,12})-2\tilde{A}'_{21,1}\tilde{X}_1C_{12,2},\label{AppendixB:Repeat:Eq63}\\
\tilde{A}_{22,21}''\tilde{X}_{2,1}&=2\tilde{X}_{2,2}(C_{22,21}\Lambda'_{2,1}-\Lambda'_{2,2}C_{22,21})-2\tilde{A}'_{21,2}\tilde{X}_1C_{12,1}.\label{AppendixB:Repeat:Eq64}
\end{align}
Since $\Lambda'_1$ and $\Lambda'_{2,1}$ both have distinct diagonal entries, by (\ref{AppendixB:Repeat:Eq1}) and (\ref{AppendixB:Repeat:Eq41}), we have 
\begin{equation}
\tilde{X}_1= I_{d \times d}\nonumber
\end{equation}
and 
\begin{equation}
\tilde{X}_{2,1}= I_{(p-d-l) \times (p-d-l)}.\nonumber
\end{equation}
 In this case, $C_{12,1}$ and $C_{21,1}$ can be uniquely determined by (\ref{AppendixB:Repeat:Eq21}) and (\ref{AppendixB:Repeat:Eq31}), and we have
\begin{align}
C_{12,1}&= \frac{-1}{\lambda}\tilde{A}'_{12,1},\label{AppendixB:Repeat:Eq21S} \\
C_{21,1} &= \frac{1}{\lambda}\tilde{A}'_{21,1}. \label{AppendixB:Repeat:Eq31S}
\end{align} 
Similarly, by (\ref{AppendixB:Repeat:Eq22}) and (\ref{AppendixB:Repeat:Eq32}), we have
\begin{align}
C_{12,2}&= \frac{-1}{\lambda}\tilde{A}'_{12,2},\label{AppendixB:Repeat:Eq22S} \\
C_{21,2}&= \frac{1}{\lambda}\tilde{A}'_{21,2},\label{AppendixB:Repeat:Eq32S}
\end{align}
By plugging (\ref{AppendixB:Repeat:Eq22S}) into (\ref{AppendixB:Repeat:Eq62}), and use the assumption that $\Lambda'_{2,2}=\gamma I_{l\times l}$, we can solve
$\Lambda''_{2,2}$. Indeed, since $\Lambda'_{2,2}$ is a scalar multiple of the identity matrix, $C_{22,22}\Lambda'_{2,2}-\Lambda'_{2,2}C_{22,22}=0$ in (\ref{AppendixB:Repeat:Eq62}). Thus, (\ref{AppendixB:Repeat:Eq62}) becomes
\begin{equation} \label{A8}
(\tilde{A}''_{22,22}-2\lambda^{-1}\tilde{A}'_{21,2}\tilde{A}'_{12,2})\tilde{X}_{2,2}=\tilde{X}_{2,2}\Lambda''_{2,2},
\end{equation}
and $\Lambda''_{2,2}$ and $\tilde{X}_{2,2}$ are eigenvalue and orthonormal eigenvector matrices of $\tilde{A}''_{22,22}-2\lambda^{-1}\tilde{A}'_{21,2}\tilde{A}'_{12,2}$. Thus, we have obtained the eigenvalue information.
However, note that in general $\tilde{X}_{2,2}$ cannot be uniquely determined. 

Suppose we want to uniquely determine the eigenvectors, $\tilde{X}(\epsilon)$, we have to further assume that $\Lambda''_{2,2}$ does not have repeated diagonal entries; {that is, eigenvalues of $\tilde{A}''_{22,22}-2\lambda^{-1}\tilde{A}'_{21,2}\tilde{A}'_{12,2}$ do not repeat.} Under this assumption, $\tilde{X}_{2,2}$ is uniquely determined, and we can proceed to solve $C$.
With $\tilde{X}_{2,2}$, from (\ref{AppendixB:Repeat:Eq63}) and (\ref{AppendixB:Repeat:Eq64}) we can solve $C_{22,12}$ and $C_{22,21}$ since $\Lambda_{2,2}'$ is a scalar multiple of the identity matrix and the diagonal entries of $\Lambda_{2,1}'$ are assumed to be different from $\Lambda_{2,2}$. In fact, we have
 \begin{align}
C_{22,12}&=(\gamma I_{(p-d-l)\times (p-d-l)}-\Lambda_{2,1}')^{-1}(\frac{1}{2}\tilde{A}_{22,12}''\tilde{X}_{2,2}+\tilde{A}_{21,1}'C_{12,2}),\label{AppendixB:Repeat:Eq63S} \\
C_{22,21}&=\tilde{X}_{2,2}^\top (\frac{1}{2}\tilde{A}_{22,21}''+\tilde{A}_{21,2}'C_{12,1})(\Lambda_{2,1}'-\gamma I_{(p-d-l)\times(p-d-l)})^{-1}\,.\label{AppendixB:Repeat:Eq64S}
\end{align}
Next, $\Lambda''_1$, $\Lambda''_{2,1}$ and the off-diagonal entries of $C_{11}$ and $C_{22,11}$ are solved by rewriting (\ref{AppendixB:Repeat:Eq5}) and (\ref{AppendixB:Repeat:Eq61}) as
\begin{align} 
&2(C_{11}\Lambda'_1-\Lambda'_1C_{11})+\Lambda''_1=(\tilde{A}''_{11}+2\tilde{A}'_{12,1}C_{21,1}+2\tilde{A}'_{12,2}\tilde{X}_{2,2}C_{21,2})\label{AppendixB:Repeat:Eq5S} \\
&2(C_{22,11}\Lambda'_{2,1}-\Lambda'_{2,1}C_{22,11})+\Lambda''_{2,1}=(\tilde{A}''_{22,11}+2\tilde{A}'_{21,1}C_{12,1})\,. \label{AppendixB:Repeat:Eq61S}
\end{align}
Therefore, with the assumption that $\Lambda''_{2,2}$ does not have repeated diagonal entries, we have
\begin{align}
(C_{11})_{m,n}&=\frac{-1}{(\Lambda'_1)_{m,m}-(\Lambda'_1)_{n,n}}e_m^{\top}\big((\frac{1}{2}\tilde{A}''_{11}+A'_{12,1}C_{21,1}+\tilde{A}'_{12,2}\tilde{X}_{2,2}C_{21,2})\big)e_n,\nonumber\\
(C_{22,11})_{m,n}&=\frac{-1}{(\Lambda'_{2,1})_{m,m}-(\Lambda'_{2,1})_{n,n}}e_m^{\top}\big(  (\frac{1}{2}\tilde{A}''_{22,11}+\tilde{A}'_{21,1}C_{12,1}) \big)e_n\nonumber,
\end{align}
where $1 \leq m \not= n \leq d$ and $d+1 \leq i \not= j \leq p-l$.
However, the problem cannot be solved and more information is needed. Indeed, note that (\ref{AppendixB:Repeat:Eq62}) can be rewritten as
\begin{align}
\Lambda''_{2,2}&=\tilde{X}_{2,2}^\top (\tilde{A}''_{22,22}\tilde{X}_{2,2}+2\tilde{A}'_{21,2}C_{12,2})=\tilde{X}_{2,2}^\top (\tilde{A}''_{22,22}-\frac{2}{\lambda}\tilde{A}'_{21,2}\tilde{A}'_{12,2})\tilde{X}_{2,2}\label{AppendixB:Repeat:Eq62S}
\end{align}
since $C_{22,22}\Lambda'_{2,2}-\Lambda'_{2,2}C_{22,22}=0$, which is the same as (\ref{A8}).
Thus, it is not informative and we need higher order derivatives of $A(\epsilon)$ at $0$ to solve $C_{22,22}$. 

Suppose we know $A'''(0)$, and denote $\tilde{A}'''(0)=\Gamma^{-1}A'''(0)\Gamma$, which is divided correspondingly as
\begin{align}
\tilde{A}'''(0)&=\begin{bmatrix}
\tilde{A}'''_{11} & \tilde{A}'''_{12,1} & \tilde{A}'''_{12,2} \\
\tilde{A}'''_{21,1} & \tilde{A}'''_{22,11} & \tilde{A}'''_{22,12} \\
\tilde{A}'''_{21,2} & \tilde{A}'''_{22,21} & \tilde{A}'''_{22,22} \\
\end{bmatrix}\,.
\end{align}
Then, if we differentiate (\ref{A1}) three times and use the similar method as before, we get
\begin{align}
(C_{22,22})_{m,n}=&\,\frac{-1}{(\Lambda''_{2,2})_{m,m}-(\Lambda''_{2,2})_{n,n}}e_m^{\top}\Big[\tilde{X}_{2,2}^{\top}\tilde{A}_{22,22}'''\tilde{X}_{2,2}-\frac{2}{\lambda^2}\tilde{X}_{2,2}^{\top}\tilde{A}_{21,2}'(\tilde{A}'_{11}-\gamma I_{d\times d})\tilde{A}_{12,2}'\tilde{X}_{2,2}  \label{Proof:ApendixA:C22_22} \\ 
 &+\frac{1}{\lambda}\tilde{X}_{2,2}^{\top} \tilde{A}_{21,2}'\tilde{A}_{12,2}''\tilde{X}_{2,2}+\frac{1}{\lambda}\tilde{X}_{2,2}^{\top}\tilde{A}_{21,2}''\tilde{A}_{12,2}'\tilde{X}_{2,2} \nonumber \\
&-\frac{4}{\lambda^2}\tilde{X}_{2,2}^\top (\tilde{A}_{21,2}'\tilde{A}_{12,1}')(\gamma I_{(p-d-l)\times(p-d-l)}-\Lambda'_{2,1})^{-1} (\tilde{A}'_{21,1}\tilde{A}'_{12,2})\tilde{X}_{2,2}\Big]e_n\,.\nonumber
\end{align}

By normalizing $\tilde{X}(\epsilon)$, we can get the diagonal terms of $C_{11}$, $C_{22,11}$ and $C_{22,22}$, which are of order $\epsilon$. As a result, we have
\begin{align}
\Lambda(\epsilon)&=\begin{bmatrix}
\lambda I_{d\times d} +\epsilon\Lambda'_1+\epsilon^2\Lambda''_1 & 0 & 0\\
0 & \epsilon\Lambda'_{2,1}+\epsilon^2\Lambda''_{2,1}& 0 \\
0 & 0 & \epsilon\Lambda'_{2,2}+\epsilon^2\Lambda''_{2,2}
\end{bmatrix}+O(\epsilon^3),\label{AppendixB:Repeat:ResultLambda} \\
\tilde{X}(\epsilon)&=\tilde{X}(0)(I_{p\times p}+\epsilon (C-\texttt{diag}(C)))+O(\epsilon^2)\in O(p),\label{AppendixB:Repeat:ResultEigenVector}
\end{align}
where the last equality holds since the entries of $\texttt{diag}(C)$ are of order $\epsilon$. Finally, we can find
$X(0)$  and $X'(0)$ by using
\begin{align}
X(0)=\Gamma \tilde{X}(0), \quad X'(0)=\Gamma \tilde{X}'(0).\nonumber
\end{align}

\textbf{General cases.} 
In general, if $\Lambda'_1$ or $\Lambda'_{2,1}$ has repeated diagonal entries, we divide them into more blocks, and the block with the same diagonal entries can be treated in the same way as we treated $\Lambda'_{2,2}$ above. We skip details here.

\section{Technical lemmas for the proof}\label{Section:StatementTechnicalLemmas}

In this section we prepare several technical lemmas. 
For $v\in\mathbb{R}^p$, we use the following notation to simplify the proof:
\begin{align} \label{vectornotation}
v=[\![v_1,\,v_2]\!]\in \mathbb{R}^p\,,
\end{align}
where $v_1\in \mathbb{R}^{d}$ forms the first $d$ coordinates of $v$ and $v_2\in\mathbb{R}^{p-d}$ forms the last $p-d$ coordinates of $v$. 
Thus, under Assumptions \ref{AssumptionTangent} and \ref{AssumptionNormal}, for $v=[\![v_1,\,v_2]\!]\in T_{\iota(x)}\mathbb{R}^p$, $v_1=J_{p,d}^\top v$ is tangential to $\iota_*T_xM$ and $v_2=\bar{J}_{p,p-d}^\top v$ is the coordinate of the normal component of $v$ associated with a chosen basis of the normal bundle. 
The first three lemmas are basic facts about the exponential map, the normal coordinate, and the volume form. The proofs of Lemmas \ref{Lemma:1} and \ref{Lemma:2} are standard and we skip the proof. Interested readers are referred to \cite{Singer_Wu:2012}.

\begin{lemma} \label{Lemma:1}
Fix $x\in M$. If we use the polar coordinate $(t,\theta)\in [0,\infty)\times S^{d-1}$ to parametrize $T_{x}M$, the volume form has the following expansion:
\begin{align}
dV=\Big(&\, t^{d-1}-\frac{1}{6}\texttt{Ric}_{x}(\theta,\theta)t^{d+1}-\frac{1}{12}\nabla_\theta \texttt{Ric}_{x}(\theta,\theta) t^{d+2}\nonumber \\
&\,-\big(\frac{1}{40}\nabla^2_{\theta\theta}\texttt{Ric}_{x}(\theta,\theta)+\frac{1}{180}\sum_{a,b=1}^d\texttt{R}_x(\theta,a,\theta,b)\texttt{R}_x(\theta,a,\theta,b)- \frac{1}{72}\texttt{Ric}_{x}(\theta,\theta)^2\big) t^{d+3}\nonumber\\
&\,+O(t^{d+4})\Big) dt d\theta,\nonumber
\end{align}
where $\texttt{R}_x$ is the Riemannian curvature of $(M,g)$ at $x$.
If we use the Cartesian coordinate to parametrize $T_{x}M$, the volume form has the following expansion
\begin{align}
dV=\Big(&1-\sum_{i,j=1}^d\frac{1}{6}\texttt{Ric}_{x}(\partial_i,\partial_j) u^iu^j-\sum_{i,j,k=1}^d\frac{1}{12}\nabla_k\texttt{Ric}_{x}(\partial_i,\partial_j) u^iu^ju^k \nonumber\\
&-\sum_{i,j,k,l=1}^d\Big[\frac{1}{40}\nabla^2_{kl}\texttt{Ric}_{x}(\partial_i,\partial_j)+\frac{1}{180}\sum_{a,b=1}^d\texttt{R}_x(\partial_i,\partial_a,\partial_j,\partial_b)\texttt{R}_x(\partial_k,\partial_a,\partial_l,\partial_b)\nonumber\\
&- \frac{1}{72}\texttt{Ric}_{x}(\partial_i,\partial_j)\texttt{Ric}_{x}(\partial_k,\partial_l)\Big] u^iu^ju^ku^l+O(\|u\|^5)\Big) du,\nonumber
\end{align}
where $u=u^i\partial_i\in T_{x}M$. 
\end{lemma}

\begin{lemma} \label{Lemma:2}
Fix $x\in M$. For $u\in T_xM$ with $\|u\|$ sufficiently small, we have the following Taylor expansion:
\begin{align}
\iota \circ \exp_x(u)-\iota(x)
=&\,\iota_{*}u+\frac{1}{2}\Second_x(u,u)+\frac{1}{6}\nabla_{u}\Second_x(u,u)\nonumber\\
&+\frac{1}{24}\nabla^2_{uu}\Second_x(u,u)+\frac{1}{120}\nabla^3_{uuu}\Second_x(u,u)+O(\|u\|^6).\nonumber
\end{align}
\end{lemma}

\begin{lemma} \label{Lemma:3}
Fix $x\in M$. If we use the polar coordinate $(t,\theta)\in [0,\infty)\times S^{d-1}$ to parametrize $T_{x}M$, when $\tilde{t}=\|\iota \circ \exp_x(\theta t)-\iota(x)\|_{\mathbb{R}^p}$ is sufficiently small, we have
\begin{align*}
\tilde{t} = &\,  t-\frac{1}{24} \|\Second_x(\theta,\theta)\|^2 t^3 -\frac{1}{24} \nabla_{\theta}\Second_x(\theta,\theta)  \cdot  \Second_x(\theta,\theta) t^4 - \big(\frac{1}{80}\nabla^2_{\theta\theta}\Second_x(\theta,\theta) \cdot \Second_x(\theta,\theta)\nonumber \\
&\,+ \frac{1}{90}\nabla_{\theta}\Second_x(\theta,\theta) \cdot \nabla_{\theta}\Second_x(\theta,\theta) +\frac{1}{1152} \|\Second_x(\theta,\theta)\|^4\big)t^5+O(t^6)\,,  \\
t = &\,  \tilde{t} + \frac{1}{24} \|\Second_x(\theta,\theta)\|^2 \tilde {t}^3 +\frac{1}{24} \nabla_{\theta}\Second_x(\theta,\theta)  \cdot  \Second_x(\theta,\theta)  \tilde{t}^4 +\big(\frac{1}{80}\nabla^2_{\theta\theta}\Second_x(\theta,\theta) \cdot \Second_x(\theta,\theta)\nonumber \\
&\, +\frac{1}{90}\nabla_{\theta}\Second_x(\theta,\theta) \cdot \nabla_{\theta}\Second_x(\theta,\theta)  +\frac{7}{1152} \|\Second_x(\theta,\theta)\|^4\big) \tilde{t}^5+O(\tilde{t}^6). 
\end{align*}
Hence, $(\iota \circ \exp_x)^{-1}(B^{\mathbb{R}^p}_{\tilde{t}}(\iota(x))\cap \iota(M)) \subset T_xM^d$ is star shaped.
\end{lemma}

The proof of Lemma \ref{Lemma:3} could be found in Appendix \ref{Section:LemmasProof}. {The essence of Lemma \ref{Lemma:3} is describing how well we could estimate the local geodesic distance by the ambient space metric. When the manifold setup is considered in an algorithm, this Lemma could be helpful in the analysis since most of time we only have an access to the ambient space metric, but not the intrinsic Riemannian metric. }

\begin{remark}
{This lemma could be applied to analyze other nonlinear dimension reduction algorithms under the manifold model, for example, the ISOMAP \cite{Tenenbaum_deSilva_Langford:2000}. Recall that the ISOMAP algorithm is composed of two steps. First, build up an undirected affinity graph from the point cloud with a chosen nearest neighbor scheme, and find the shortest distance for each pair of data points on the affinity graph. Second, run the multidimensional scaling algorithm with the obtained pairwise distances and hence the dimension reduction. Although it is out of the scope of the current paper, we mention that Lemma \ref{Lemma:3} could help us better understand the global geodesic distance approximation error in the first step of ISOMAP.}  
\end{remark}

To alleviate the notational load, we denote 
\begin{equation}
\tilde{B}_{\epsilon}(x):=\iota^{-1}( B_{\epsilon}^{\mathbb{R}^p}(\iota(x))\cap \iota(M))\subset M\label{Definition:Notation:SmallBallonTangentSpace},
\end{equation}
and for a sufficiently small $\epsilon$, by Lemma \ref{Lemma:3}, denote
\begin{align}
\tilde{\epsilon} = &\,  \epsilon-\frac{1}{24} \|\Second_x(\theta,\theta)\|^2 \epsilon^3 -\frac{1}{24} \nabla_{\theta}\Second_x(\theta,\theta)  \cdot  \Second_x(\theta,\theta) \epsilon^4 - \big(\frac{1}{80}\nabla^2_{\theta\theta}\Second_x(\theta,\theta) \cdot \Second_x(\theta,\theta)\nonumber \\
&\,+ \frac{1}{90}\nabla_{\theta}\Second_x(\theta,\theta) \cdot \nabla_{\theta}\Second_x(\theta,\theta) +\frac{1}{1152} \|\Second_x(\theta,\theta)\|^4\big)\epsilon^5+O(\epsilon^6).\nonumber
\end{align}

To have a more succinct proof, we prepare the following integration, which comes from a direct expansion and the proof is skipped.
\begin{lemma}\label{Lemma:4}
For $d\in\mathbb{N}$, $\gamma> -d$ and $h_1,h_2,h_3\in \mathbb{R}$, we have the following asymptotical expansion when $\epsilon$ is sufficiently small:
\begin{align*}
&\int_{0}^{\epsilon+h_1 \epsilon^3+ h_2 \epsilon^4+h_3\epsilon^5+O(\epsilon^6)} t^{d-1+\gamma}dt\\
=&\frac{\epsilon^{d+\gamma}}{d+\gamma}\Big(1+(d+\gamma)h_1\epsilon^2+(d+\gamma)h_2\epsilon^3+\big[(d+\gamma)h_3+\frac{(d+\gamma)(d+\gamma-1)}{2}h_1^2\big]\epsilon^4\Big)+O(\epsilon^{d+\gamma+5})\nonumber.
\end{align*}
\end{lemma}

In the next lemma, we calculate the asymptotical expansion of few quantities that we are going to use in proving the main theorem. Note that in order to capture the extra terms introduced by the barycentric coordinate, we calculate the normal term 2 orders higher than those for the tangential direction. To handle the normal component is the main reason we need the $C^5$ regularity for $P$.
        
\begin{lemma} \label{Lemma:5}
Fix $x\in M$ and assume Assumptions \ref{AssumptionTangent} and \ref{AssumptionNormal} hold. When $\epsilon$ is sufficiently small, we have the following expansion for $\mathbb{E}[f(X)\chi_{B_{\epsilon}^{\mathbb{R}^p}(\iota(x))}(X)]$:
\begin{align}
\mathbb{E}[f(X)&\chi_{B_{\epsilon}^{\mathbb{R}^p}(\iota(x))}(X)]
=\,\frac{|S^{d-1}|}{d}f(x)P(x)\epsilon^d\label{Lemma:4:Statement1:withf}\\
&+\frac{|S^{d-1}|}{d(d+2)}\Big[\frac{1}{2}P(x)\Delta f(x)+\frac{1}{2}f(x)\Delta P(x)+\nabla f(x)\cdot \nabla P(x)\nonumber\\
&\qquad+\frac{s(x)f(x)P(x)}{6}+\frac{d(d+2)\omega(x)f(x)P(x)}{24}\Big]\epsilon^{d+2}+O(\epsilon^{d+3})\nonumber
\end{align}
and the following expansion for $\mathbb{E}[(X-\iota(x))f(X)\chi_{B_{\epsilon}^{\mathbb{R}^p}(\iota(x))}(X)]\in\mathbb{R}^p$:
\begin{align}
\mathbb{E}[(X-\iota(x))f(X)\chi_{B_{\epsilon}^{\mathbb{R}^p}(x)}(X)]=\, [\![v_1,\,v_2]\!]+O(\epsilon^{d+5}),\label{Lemma:4:Statement2:withf}
\end{align}
where $v_1\in\mathbb{R}^d$ and $v_2\in\mathbb{R}^{p-d}$ are defined in (\ref{Proof:LemmaD4:v1}) and (\ref{Proof:LemmaD4:v2}) respectively, which contain terms of order $\epsilon^{d+2}$ and $\epsilon^{d+4}$. 
\end{lemma}

The proof of Lemma \ref{Lemma:5} is postponed to Section \ref{Section:LemmasProof}.
We comment that if $\mathbf{T}_{\iota(x)}=0$, then (\ref{Definition:Qf:KernelExpansion}) becomes $\frac{\mathbb{E}[f(X)\chi_{B_{\epsilon}^{\mathbb{R}^p}(\iota(x))}(X)]}{\mathbb{E}[\chi_{B_{\epsilon}^{\mathbb{R}^p}(\iota(x))}(X)]}$, which is nothing but the diffusion process corresponding to the zero-one kernel with a compact support. Thus, the analysis is the same as those for the diffusion map shown in \cite{Coifman_Lafon:2006}, except that the kernel is discontinuous. When  $\mathbf{T}_{\iota(x)}\neq 0$, $\mathbb{E}[(X-\iota(x))f(X)\chi_{B_{\epsilon}^{\mathbb{R}^p}(x)}(X)]$ is the new component specific to the LLE, and $\mathbf{T}_{\iota(x)}$ contributes to the ``correction of the kernel''.

Recall the definition of $\mathbf{T}_{\iota(x)}= \mathcal{I}_{\epsilon^{d+\rho}}(C_{x})\big[\mathbb{E}(X-\iota(x))\chi_{B_{\epsilon}^{\mathbb{R}^p}(x)}\big]$ in (\ref{Definition:Tx:ContinuousCase}), which could be expanded as $\sum_{i=1}^r\frac{\mathbb{E}[(X-\iota(x))\chi_{B_{\epsilon}^{\mathbb{R}^p}(\iota(x))}(X)] \cdot u_i}{\lambda_i+\epsilon^{d+\rho}}u_i \in \mathbb{R}^p$, where $r$ is the rank of $C_x$ and $u_i$ and $\lambda_i$ form the $i$-th eigen-pair of $C_x$. Clearly, $\mathbf{T}_{\iota(x)}$ is dominated by those ``small'' eigenvalues of $C_x$. 

Define the notation to simplify the statement of the next lemma:
\begin{equation}
\tilde{J}:=\bar J_{p,p-d}J_{p-d,p-d-l}\in \mathbb{R}^{p\times(p-d-l)}.
\end{equation}

\begin{lemma} \label{Lemma:6}
Fix $x\in M$ and assume Assumptions \ref{AssumptionTangent} and \ref{AssumptionNormal} hold. Suppose $\epsilon$ is sufficiently small. Following the same notations used in Proposition \ref{Proposition:2}, under three conditions shown in Condition \ref{Condition:1}, $\mathbf{T}_{\iota(x)}$ satisfies:

\textbf{Case 0.} $\mathbf{T}_{\iota(x)}=[\![v_1,\,v_2]\!]+[\![O(\epsilon^2),\,0]\!]$, where
\begin{align}
v_1=\,\frac{J_{p,d}^\top \iota_*\nabla {P}(x)}{{P}(x)+\frac{d(d+2)}{|S^{d-1}|}\epsilon^{\rho-2}}+O(\epsilon^2),\qquad v_2=\,0.\nonumber
\end{align}

\textbf{Case 1.} $\mathbf{T}_{\iota(x)}=[\![v_1,\,v_2]\!]+[\![O(\epsilon^2),\,O(1)]\!]$, where
\begin{align}
v_1=\,&\frac{J_{p,d}^\top \iota_*\nabla P(x)}{{P}(x)+\frac{d(d+2)}{|S^{d-1}|}\epsilon^{\rho-2}}+\sum_{i=d+1}^p\frac{\mathfrak{N}^\top _0(x)\bar{J}_{p,p-d}X_2\bar{J}_{p,p-d}^\top e_i}{\frac{2}{d}\lambda^{(2)}_i+\frac{2(d+2)}{P(x)|S^{d-1}|}\epsilon^{\rho-4}} X_1\mathsf{S}_{12}\bar{J}_{p,p-d}^\top e_i\,,\nonumber\\
v_2=\,&\frac{1}{ \epsilon^2}\sum_{i=d+1}^p\frac{\mathfrak{N}^\top _0(x)\bar{J}_{p,p-d}X_2\bar{J}_{p,p-d}^\top e_{i}}{\frac{2}{d}\lambda^{(2)}_i+\frac{2(d+2)}{P(x)|S^{d-1}|}\epsilon^{\rho-4}}X_2\bar{J}_{p,p-d}^\top e_{i}\,,\nonumber
\end{align}
and $\mathfrak{N}_0(x)$ is defined in (\ref{Definition:N0x:forsimplification}).

\textbf{Case 2.} $\mathbf{T}_{\iota(x)}=[\![v_1,\,v_2]\!]+[\![O(\epsilon^2),\,O(1)]\!]$, where 
\begin{align}
v_1=\,&\frac{J_{p,d}^\top \iota_*\nabla P(x)}{{P}(x)+\frac{d(d+2)}{|S^{d-1}|}\epsilon^{\rho-2}}+\sum_{i=d+1}^{p-l}\frac{\mathfrak{N}^\top _0(x)\tilde{J}X_{2,1}\tilde{J}^\top e_i}{\frac{2}{d}\lambda^{(2)}_i+\frac{2(d+2)}{P(x)|S^{d-1}|}\epsilon^{\rho-4}}X_{2,1}\mathsf{S}_{12,1}\tilde{J}^\top e_i\nonumber\\
&\qquad+\sum_{i=p-l+1}^p\frac{\alpha_i }{\frac{|S^{d-1}|  {P}(x)}{d(d+2)}\lambda^{(4)}_i+\epsilon^{\rho-6}} X_1\mathsf{S}_{12,2}\bar J_{p,l}^\top e_{i}\,,\label{Expansion:Tx:Case2:v1}\\
v_2=\,&\frac{1}{\epsilon^2}\sum_{i=d+1}^{p-l}\frac{\mathfrak{N}^\top _0(x)\tilde{J}X_{2,1}\tilde{J}^\top e_{i}}{\frac{2}{d}\lambda^{(2)}_i+\frac{2(d+2)}{P(x)|S^{d-1}|}\epsilon^{\rho-4}}\begin{bmatrix}X_{2,1} & 0 \\ 0 & X_{2,2} \end{bmatrix}\bar{J}_{p,p-d}^\top e_i\nonumber\\
&\quad+\frac{1}{\epsilon^2}\sum_{i=p-l+1}^p\frac{\alpha_i}{\frac{|S^{d-1}|  {P}(x)}{d(d+2)}\lambda^{(4)}_i+\epsilon^{\rho-6}}\begin{bmatrix}X_{2,1} & 0 \\ 0 & X_{2,2} \end{bmatrix}\bar{J}_{p,p-d}^\top e_i\,,\label{Expansion:Tx:Case2:v2}
\end{align}
where $\alpha_i\in \mathbb{R}$ is defined in (\ref{Definition:Alphai:ProofC6}).
\end{lemma}

The proof of Lemma \ref{Lemma:6} is postponed to Appendix \ref{Section:LemmasProof}.

\section{Proofs of Propositions 3.1 and 3.2}\label{Section:Proof:PropositionForCovariance}

The idea of the proof is same as that in \cite{Singer_Wu:2012,Cheng_Wu:2013}, except that we are going to calculate it more explicitly. 

\subsection{Proof of Proposition \ref{Proposition:1}}

We use the notation $\langle \cdot, \cdot \rangle$ to mean the inner product and use the notation (\ref{Definition:Notation:SmallBallonTangentSpace}). 
The $(m,n)$-th entry of $C_x=\mathbb{E}[(X-\iota(x))(X-\iota(x))^{\top}\chi_{B_{\epsilon}^{\mathbb{R}^p}(\iota(x))}(X)]$ is
\begin{align}
& e_m^{\top}C_xe_n= \int_{\tilde{B}_{\epsilon}(x)} \langle \iota(y)-\iota(x),e_m \rangle \langle \iota(y)-\iota(x),e_n\rangle P(y)dV(y).\label{Proof:Lemma5:Expression1}
\end{align}
The quantities $\iota \circ \exp_{x}(\theta t)$, $\tilde{\epsilon}$ and $dV$ need to be expanded up to higher order terms. By the change of variable $y=\exp_x(t\theta)$, where $(t,\theta)\in [0,\infty)\times S^{d-1}$ constitutes the polar coordinate, we have the following expressions:
\begin{align*}
\iota \circ \exp_{x}(\theta t)-\iota(x)&\,=K_1(\theta)t+K_2(\theta)t^2+K_3(\theta)t^3+K_4(\theta)t^4+K_5(\theta)t^5+O(t^6)\\
\tilde{\epsilon}&\,= \epsilon+H_1(\theta) \epsilon^3+ H_2(\theta) \epsilon^4+H_3(\theta)\epsilon^5+O(\epsilon^6)\,\\
dV(\exp_x(t\theta))&\,=t^{d-1}+R_1(\theta)t^{d+1}+R_2(\theta)t^{d+2}+R_3(\theta)t^{d+3}+O(t^{d+4})\,\\
P(\exp_x(t\theta))&\,=P_0+P_1(\theta)+P_2(\theta)t^2+P_3(\theta)t^3+P_4(\theta)t^4+O(t^5)\,,
\end{align*}
where 
\begin{align*}
\left\{
\begin{array}{l}
\displaystyle K_1(\theta)=\iota_{*}\theta, \quad K_2(\theta)=\frac{1}{2}\Second_{x}(\theta,\theta), \quad K_3(\theta)=\frac{1}{6}\nabla_{\theta}\Second_{x}(\theta,\theta), \\
\displaystyle  K_4(\theta)=\frac{1}{24}\nabla^2_{\theta\theta}\Second_{x}(\theta,\theta), \quad K_5(\theta)=\frac{1}{120}\nabla^3_{\theta\theta\theta}\Second_{x}(\theta,\theta), 
\end{array}
\right.
\end{align*}
by Lemma (\ref{Lemma:1}),
\begin{align*}
\left\{
\begin{array}{l}
\displaystyle  H_1(\theta)=\,\frac{1}{24} \|\Second_x(\theta,\theta)\|^2,\quad H_2(\theta)=\frac{1}{24} \nabla_{\theta}\Second_x(\theta,\theta)  \cdot  \Second_x(\theta,\theta), \\ 
\displaystyle  H_3(\theta)=\,\frac{1}{80}\nabla^2_{\theta\theta}\Second_x(\theta,\theta) \cdot \Second_x(\theta,\theta)\\
\displaystyle \qquad\qquad +\frac{1}{90}\nabla_{\theta}\Second_x(\theta,\theta) \cdot \nabla_{\theta}\Second_x(\theta,\theta) +\frac{7}{1152} \|\Second_x(\theta,\theta)\|^4 , 
\end{array}
\right.
\end{align*}
by Lemma (\ref{Lemma:2}),
\begin{align*}
\left\{
\begin{array}{l}
\displaystyle R_1(\theta)=\,-\frac{1}{6}\texttt{Ric}_{x}(\theta,\theta) ,\quad R_2(\theta)=-\frac{1}{12}\nabla_\theta \texttt{Ric}_{x}(\theta,\theta) , \\
\displaystyle R_3(\theta)=\,-\frac{1}{40}\nabla^2_{\theta\theta}\texttt{Ric}_{x}(\theta,\theta)\\
\displaystyle\qquad\qquad-\frac{1}{180}\sum_{a,b=1}^d\texttt{R}_x(\theta,a,\theta,b)\texttt{R}_x(\theta,a,\theta,b)+ \frac{1}{72}\texttt{Ric}_{x}(\theta,\theta)^2  , 
\end{array}
\right.
\end{align*}
by Lemma (\ref{Lemma:3}), and
\begin{align*}
\left\{
\begin{array}{l}
\displaystyle P_0:=P(x),\quad P_1(\theta):=\nabla_\theta P(x),\quad P_2(\theta):=\frac{1}{2}\nabla^2_{\theta,\theta} P(x)\\ 
\displaystyle P_3(\theta):=\frac{1}{6}\nabla^3_{\theta,\theta,\theta} P(x),\quad P_4(\theta):=\frac{1}{24}\nabla^4_{\theta,\theta,\theta,\theta} P(x).
\end{array}
\right.
\end{align*}
Note that $H_1$, $H_3$, $R_1$, $R_3$, $P_0$, $P_2$, and $P_4$ are even functions on $S^{d-1}$ and $H_2$, $R_2$, $P_1$ and $P_3$ are odd on $S^{d-1}$. 
Similarly, for $m,n=1,\ldots,p$, we have
\begin{align}
& \langle \iota \circ \exp_{x}(\theta t)-\iota(x),e_m \rangle \langle\iota \circ \exp_{x}(\theta t)-\iota(x),e_n\rangle \nonumber \\ 
=&\,  A_{m,n}(\theta)t^2+B_{m,n}(\theta)t^3 + C_{m,n}(\theta)t^4+D_{m,n}(\theta)t^5+E_{m,n}(\theta)t^6+O(t^7) \nonumber\,,
\end{align} 
where
\begin{align*}
A_{m,n}(\theta) =\,& \langle K_1(\theta),e_m \rangle  \langle K_1(\theta),e_n \rangle   \\
B_{m,n}(\theta) =\,& \langle K_2(\theta),e_m \rangle  \langle K_1(\theta),e_n \rangle +  \langle K_1(\theta),e_m \rangle \langle K_2(\theta),e_n \rangle   \\
 C_{m,n}(\theta)  =\,&  \langle K_2(\theta),e_m \rangle  \langle K_2(\theta),e_n \rangle + \langle K_1(\theta),e_m \rangle  \langle K_3(\theta),e_n \rangle  \\
& +  \langle K_3(\theta),e_m \rangle \langle K_1(\theta),e_n \rangle    \nonumber \\
D_{m,n}(\theta) =\,& \langle K_1(\theta),e_m \rangle  \langle K_4(\theta),e_n \rangle + \langle K_2(\theta),e_m \rangle  \langle K_3(\theta),e_n \rangle  \\
                      & +  \langle K_3(\theta),e_m \rangle \langle K_2(\theta),e_n \rangle + \langle K_4(\theta),e_m \rangle  \langle K_1(\theta),e_n \rangle \nonumber \\
E_{m,n}(\theta) =\,& \langle K_1(\theta),e_m \rangle  \langle K_5(\theta),e_n \rangle + \langle K_2(\theta),e_m \rangle  \langle K_4(\theta),e_n \rangle  \\
                      & + \langle K_3(\theta),e_m \rangle \langle K_3(\theta),e_n \rangle + \langle K_4(\theta),e_m \rangle  \langle K_2(\theta),e_n \rangle + \langle K_5(\theta),e_m \rangle  \langle K_1(\theta),e_n \rangle\,.  \nonumber 
\end{align*} 
Observe that $A_{m,n}$, $C_{m,n}$ and $E_{m,n}$, for $m,n=1,\ldots,p$, are even functions on $S^{d-1}$, while $B_{m,n}$ and $D_{m,n}$ are odd functions on $S^{d-1}$. But plugging these expressions into (\ref{Proof:Lemma5:Expression1}), we have
\begin{align}
  e_m^{\top}C_xe_n= &\, \int_{S^{d-1}}\int_{0}^{\tilde\epsilon}\big(A_{m,n}(\theta)t^2+B_{m,n}(\theta)t^3 + C_{m,n}(\theta)t^4 \nonumber\\
 &\qquad\qquad\qquad\qquad\qquad\qquad +D_{m,n}(\theta)t^5 +  E_{m,n}(\theta)t^6+O(t^7)\big)\nonumber\\
 &\qquad\times  \big(P_0+P_1(\theta)t+ P_2(\theta)t^2+P_3(\theta)t^3+P_4(\theta)t^4+O(t^5)\big) \nonumber\\
 &\qquad\times \big(t^{d-1}+R_1(\theta)t^{d+1}+R_2(\theta)t^{d+2}+R_3(\theta)t^{d+3}+O(t^{d+4})\big) dt d\theta \nonumber.
\end{align}
We now collect terms of the same order to simplify the calculation. We focus on those terms with the order less than or equal to $\epsilon^{d+6}$. 
\begin{align*}
   e_m^{\top}C_xe_n
= &\, \int_{S^{d-1}}\int_{0}^{\tilde\epsilon} 
P_0 A_{m,n}(\theta)t^{d+1}+\big(P_0B_{m,n}(\theta) +  P_1(\theta)A_{m,n}(\theta)\big)t^{d+2}\\
&+\big(P_0 A_{m,n}(\theta)R_1(\theta) +P_0 C_{m,n}(\theta)+P_1(\theta)B_{m,n}(\theta)+P_2(\theta)A_{m,n}(\theta)\big)t^{d+3} \nonumber \\
& + \big(P_0A_{m,n}(\theta)R_2(\theta) + P_0 B_{m,n}(\theta)R_1(\theta)+ P_1(\theta) A_{m,n}(\theta)R_1(\theta)+P_0D_{m,n}(\theta) \nonumber \\
 &\qquad +  P_1(\theta)C_{m,n}(\theta)+P_2(\theta)B_{m,n}(\theta)+P_3(\theta)A_{m,n}(\theta)\big)t^{d+4}\nonumber\\
&+\big(P_0A_{m,n}(\theta)R_3(\theta) + P_0B_{m,n}(\theta)R_2(\theta)+P_1(\theta)A_{m,n}(\theta)R_2(\theta)+P_0C_{m,n}(\theta)R_1(\theta)\nonumber\\
&\qquad  +P_1(\theta)B_{m,n}(\theta)R_1(\theta)  + P_2(\theta)A_{m,n}(\theta)R_1(\theta) +P_0E_{m,n}(\theta)+P_1(\theta)D_{m,n}(\theta)\nonumber\\
&\qquad + P_2(\theta)C_{m,n}(\theta)  + P_3(\theta)B_{m,n}(\theta)+P_4(\theta)A_{m,n}(\theta)\big)t^{d+5} dtd\theta +O(\epsilon^{d+7})  \nonumber
\end{align*} 
By further expanding the integration of $t$ over $[0,\tilde\epsilon]=[0,\epsilon+H_1(\theta) \epsilon^3+ H_2(\theta) \epsilon^4+H_3(\theta)\epsilon^5+O(\epsilon^6)]$ by Lemma \ref{Lemma:4}, 
we have
\begin{align*} 
e_m^{\top}C_xe_n = \epsilon^{d+2} Q_{m,n}^{(0)}(x)+\epsilon^{d+4} Q_{m,n}^{(2)}(x)+ \epsilon^{d+6}Q_{m,n}^{(4)}(x) + O(\epsilon^{d+7}),
\end{align*} 
where 
\begin{align*}
Q_{m,n}^{(0)}(x)&=P_0\int_{S^{d-1}} \frac{ A_{m,n}(\theta)}{d+2} d\theta,\\
Q_{m,n}^{(2)}(x)&= \int_{S^{d-1}} P_0 A_{m,n}(\theta)H_1(\theta) +\frac{1}{d+4}\big[ P_0A_{m,n}(\theta)R_1(\theta)+ P_0C_{m,n}(\theta)\nonumber \\
&\qquad\qquad+ P_1(\theta)B_{m,n}(\theta) + P_2(\theta)A_{m,n}(\theta)\big]d\theta\nonumber\,,
\end{align*}
and
\begin{align*}
Q_{m,n}^{(4)}(x)&=\int_{S^{d-1}}  \bigg( P_0 A_{m,n}(\theta)H_3(\theta)+\frac{d+1}{2}P_0    A_{m,n}(\theta)H_1^2(\theta)  \\
  &+\big[ P_0 B_{m,n}(\theta)+P_1(\theta)A_{m,n}(\theta)\big]H_2(\theta)\nonumber\\
  &+\big[ P_0 A_{m,n}(\theta)R_1(\theta)+ P_0 C_{m,n}(\theta) +P_1(\theta)B_{m,n}(\theta) + P_2(\theta)A_{m,n}(\theta)\big]H_1(\theta)\nonumber\\
  & +\frac{1}{d+6} \big[  P_0A_{m,n}(\theta)R_3(\theta)+P_0B_{m,n}(\theta)R_2(\theta)+P_1(\theta)A_{m,n}(\theta)R_2(\theta)+ P_0C_{m,n}(\theta)R_1(\theta) \nonumber \\
&\qquad\qquad + P_1(\theta)B_{m,n}(\theta)R_1(\theta) + P_2(\theta)A_{m,n}(\theta)R_1(\theta) + P_0E_{m,n}(\theta)+P_1(\theta)D_{m,n}(\theta) \nonumber \\
&\qquad\qquad + P_2(\theta) C_{m,n}(\theta) + P_3(\theta)B_{m,n}(\theta)+P_4(\theta)A_{m,n}(\theta)\big] \bigg) d\theta\nonumber.
\end{align*}

To finish the proof,
we evaluate $Q^{(0)}_{m,n}$, $Q^{(2)}_{m,n}$, and $Q^{(4)}_{m,n}$, for $1\leq m,n\leq p$. 
Due to Assumptions \ref{AssumptionTangent} and \ref{AssumptionNormal}, $\{e_1, \cdots , e_d\}$ is an orthonormal basis of  $\iota_*T_x M$ and $\{e_{d+1}, \cdots , e_p\}$ is an orthonormal basis of  $(\iota_*T_x M)^\bot $. There, we have $\langle K_1(\theta),e_i \rangle = \langle \iota_*\theta,e_i \rangle =0$ for $i=d+1, \cdots, p$. Using Lemma \ref{Lemma:2} and symmetry of sphere, we can evaluate the term of order $\epsilon^{d+2}$ in $C_x$.
For $1 \leq m=n \leq d$, we have
\begin{equation}
Q^{(0)}_{m,n}=\frac{P_0}{d+2}\int_{S^{d-1}}A_{m,n}(\theta)d\theta=\frac{P(x)}{d+2} \int_{S^{d-1}} | \langle \iota_*\theta, e_1 \rangle |^2d\theta=\frac{|S^{d-1}| P(x)}{d(d+2)};
\end{equation}
for other $m$ and $n$, $\int_{S^{d-1}}A_{m,n}(\theta)d\theta=0$. 
Thus, the coefficient of the $\epsilon^{d+2}$ term is $\frac{|S^{d-1}|{P}(x)}{d(d+2)} 
\begin{bmatrix}
I_{d \times d} & 0 \\
0& 0 \nonumber \\
\end{bmatrix}$. Denote $M^{(0)}_{11}=I_{d\times d}\in\mathbb{R}^{d\times d}$ and $M^{(0)}_{12}=0\in\mathbb{R}^{d\times (p-d)}$, $M^{(0)}_{21}={M^{(0)}_{12}}^\top$ and $M^{(0)}_{22}=0\in\mathbb{R}^{(p-d)\times (p-d)}$.

Next, we evaluate the term of order $\epsilon^{d+4}$ in $C_x$. Note that $\langle\Second_{x}(\theta,\theta),e_m \rangle =0$, for $m=1, \cdots, d$, so $B_{m,n}(\theta)=0$. 
Thus, for $1\leq m,n\leq d$, by a direct calculation,
\begin{align} 
Q^{(2)}_{m,n} 
=\,& \frac{P(x)}{24} \int_{S^{d-1}} \langle \iota_*\theta,e_m \rangle  \langle \iota_*\theta,e_n \rangle \|\Second_{x}(\theta,\theta)\|^2d\theta \label{Proof:LemmaD5:Definition:tildeM11} \\
&- \frac{P(x)}{6(d+4)}\int_{S^{d-1}} \langle \iota_*\theta,e_m \rangle  \langle \iota_*\theta,e_n \rangle \texttt{Ric}_{x}(\theta,\theta)d\theta \nonumber \\
& -   \frac{P(x)}{6(d+4)} \int_{S^{d-1}} \langle \iota_*\theta,e_m \rangle  \langle \Second_{x}(e_n,\theta),\Second_{x}(\theta,\theta) \rangle+ \langle \iota_*\theta,e_n \rangle \langle \Second_{x}(e_m,\theta) , \Second_{x}(\theta,\theta) \rangle d\theta   \nonumber \\
& +  \frac{1}{2(d+4)}  \int_{S^{d-1}} \nabla^2_{\theta,\theta}P(x)
\langle \iota_*\theta,e_m \rangle  \langle \iota_*\theta,e_n \rangle d\theta \nonumber,
\end{align} 
where we use the fact that $ \langle \nabla_{\theta}\Second_{x}(\theta,\theta), e_m \rangle =-\langle \Second_{x}(e_m,\theta), \Second_{x}(\theta,\theta) \rangle$ when $m=1,\ldots,d$.
By defintion, it is clear that $A_{m,n}(\theta)=0$ when $1\leq m\leq d$ and $d+1\leq n\leq p$. Thus, for $1\leq m\leq d$ and $d+1\leq n\leq p$, we have
\begin{align}
 Q^{(2)}_{m,n} =&\, \frac{P(x)}{6(d+4)}\int_{S^{d-1}} \langle \iota_*\theta,e_m \rangle  \langle \nabla_{\theta}\Second_{x}(\theta,\theta),e_n \rangle d\theta \label{Proof:LemmaD5:Definition:tildeM12} \\
&+ \frac{1}{2(d+4)}\int_{S^{d-1}} \nabla_\theta P(x) \langle \iota_*\theta,e_m \rangle  \langle \Second_{x}(\theta,\theta),e_n \rangle d\theta \nonumber.
\end{align}
By definition, for $d+1\leq m,n \leq p$, $A_{m,n}(\theta)=B_{m,n}(\theta)=0$, and hence
\begin{equation}
Q^{(2)}_{m,n}=\frac{P(x)}{4(d+4)}\int_{S^{d-1}} \langle \Second_{x}(\theta,\theta),e_m  \rangle  \langle  \Second_{x}(\theta,\theta) ,e_n \rangle d\theta\,.\label{Proof:LemmaD5:Definition:tildeQ22}
\end{equation}

Finally, we evaluate the $\epsilon^{d+6}$ term. Again, recall the fact that when $d+1\leq m,n \leq p$, $A_{m,n}(\theta)=0$ and $B_{m,n}(\theta)=0$. Therefore, $Q^{(4)}_{m,n}$, where $d+1\leq m,n\leq p$, consists of only  
\begin{align}
Q^{(4)}_{m,n}=\,&\int_{S^{d-1}} P_0C_{m,n}(\theta)H_1(\theta)\nonumber\\
&+\frac{1}{d+6}\bigg(P_0C_{m,n}(\theta)R_1(\theta)+P_0E_{m,n}(\theta)+P_1(\theta)D_{m,n}(\theta)+P_2(\theta)C_{m,n}(\theta)\bigg) d\theta.\nonumber
\end{align}
Based on Lemmas \ref{Lemma:1}, \ref{Lemma:2} and \ref{Lemma:3}, for $d+1\leq m,n\leq p$, we have
\begin{align} 
Q^{(4)}_{m,n}=&\, \frac{P(x)}{96}\int_{S^{d-1}}  \langle \Second_{x}(\theta,\theta),e_m  \rangle  \langle  \Second_{x}(\theta,\theta) ,e_n \rangle \|\Second_{x}(\theta,\theta)\|^2 d\theta \label{Proof:LemmaD5:Definition:barM22} \\
                                & -  \frac{P(x)}{24(d+6)} \int_{S^{d-1}} \langle \Second_{x}(\theta,\theta),e_m  \rangle  \langle  \Second_{x}(\theta,\theta) ,e_n \rangle \texttt{Ric}_{x}(\theta,\theta) d\theta\nonumber \\
                             & + \frac{P(x)}{48(d+6)}\int_{S^{d-1}} \langle \Second_{x}(\theta,\theta),e_m\rangle \langle \nabla^2_{\theta\theta}\Second_{x}(\theta,\theta),e_n \rangle \nonumber \\
 &\qquad\qquad\qquad + \langle \nabla^2_{\theta\theta}\Second_{x}(\theta,\theta),e_m \rangle \langle \Second_{x}(\theta,\theta),e_n\rangle d\theta\nonumber \\
                             & + \frac{P(x)}{36(d+6)} \int_{S^{d-1}} \langle \nabla_{\theta}\Second_{x}(\theta,\theta),e_m \rangle \langle \nabla_{\theta}\Second_{x}(\theta,\theta),e_n \rangle d\theta\nonumber \\
                             & +  \frac{1}{12(d+6)} \int_{S^{d-1}} \nabla_\theta P(x)\Big(\langle \Second_{x}(\theta,\theta),e_m \rangle \langle \nabla_{\theta}\Second_{x}(\theta,\theta),e_n \rangle \nonumber \\
&\qquad\qquad\qquad+\langle \nabla_{\theta}\Second_{x}(\theta,\theta),e_m \rangle \langle \Second_{x}(\theta,\theta),e_n \rangle \Big) d\theta \nonumber \\
                         & +   \frac{1}{4(d+6)}  \int_{S^{d-1}} \nabla^2_{\theta,\theta}P(x) \langle \Second_{x}(\theta,\theta),e_m  \rangle  \langle  \Second_{x}(\theta,\theta) ,e_n \rangle d\theta \nonumber.
\end{align}
Since we only need to evaluate $Q^{(4)}_{mn}$, where $d+1\leq m,n\leq p$, for the LLE analysis, we omit the calculation of the other pairs of $m,n$.
We thus conclude that 
\begin{align}
C_x =\,\epsilon^{d+2}\frac{|S^{d-1}|  {P}(x)}{d(d+2)} 
\Big(&\begin{bmatrix}
I_{d \times d} & 0 \\
0& 0  \\
\end{bmatrix}+
\begin{bmatrix}
M^{(2)}_{11} & M^{(2)}_{12}  \\
M^{(2)}_{21} & M^{(2)}_{22} 
\end{bmatrix}\epsilon^{2}\nonumber\\
&+
\begin{bmatrix}
M^{(4)}_{11} & M^{(4)}_{12}  \\
M^{(4)}_{21} & M^{(4)}_{22}  \\
\end{bmatrix}\epsilon^{4}+O(\epsilon^{6})\Big) \,,\label{Proof:LemmaD5:FinalResult}
\end{align}
where $M^{(j)}_{11}\in \mathbb{R}^{d\times d}$ is defined as 
\begin{equation}
e_m^\top M^{(j)}_{11}e_n=\frac{d(d+2)}{|S^{d-1}|  {P}(x)} Q^{(j)}_{m,n},\label{Proof:LemmaD5:Definition:tildeM11}
\end{equation} 
for $m,n=1,\ldots,d$ and $j=2,4$, 
$M^{(j)}_{22}\in \mathbb{R}^{(p-d)\times (p-d)}$ is defined as 
\begin{equation}
e_m^\top M^{(j)}_{22}e_n=\frac{d(d+2)}{|S^{d-1}|  {P}(x)} Q^{(j)}_{m+d,n+d}, \label{Proof:LemmaD5:Definition:tildeM22}
\end{equation} 
for $m,n=1,\ldots,p-d$ and $j=2,4$,
$M^{(2)}_{12}\in \mathbb{R}^{d\times (p-d)}$ is defined as 
\begin{align}
e_m^\top M^{(2)}_{12}e_n=&\,\frac{d(d+2)}{|S^{d-1}|  {P}(x)}Q^{(2)}_{m,n+d} \nonumber\\
=&\, \frac{d(d+2)}{6|S^{d-1}|(d+4)}\int_{S^{d-1}} \langle \iota_*\theta,e_m \rangle  \langle \nabla_{\theta}\Second_{x}(\theta,\theta),e_n \rangle d\theta \label{Proposition32:Proof:M2:12:Expansion}  \\
&+ \frac{d(d+2)}{2(d+4)|S^{d-1}|P(x)}\int_{S^{d-1}} \nabla_\theta P(x) \langle \iota_*\theta,e_m \rangle  \langle \Second_{x}(\theta,\theta),e_n \rangle d\theta \nonumber
\end{align} 
for $m=1,\ldots,d$ and $n=1,\ldots,p-d$, and $M^{(2)}_{21}={M^{(2)}_{12}}^\top$.

\subsection{Proof of Proposition \ref{Proposition:2}}

We now evaluate the eigenvalue and eigenvectors of $C_x$ shown in (\ref{Proof:LemmaD5:FinalResult}) based on the technique introduced in Appendix \ref{Section:Perturbation}. 

For Case 0 in Condition \ref{Condition:1}, when $\epsilon$ is sufficiently small, we have
\begin{align}
C_x:=\frac{|S^{d-1}|  {P}(x)\epsilon^{d+2}}{d(d+2)} 
\Big(\begin{bmatrix}
I_{d \times d} & 0 \\
0& 0  \\
\end{bmatrix}+
\begin{bmatrix}
O(\epsilon^2)&0  \\
0 & 0 
\end{bmatrix}\Big)\nonumber
\end{align}
and hence the $d$ non-zero eigenvalues satisfies 
\begin{align}
\Lambda_x&=\frac{|S^{d-1}|  {P}(x)\epsilon^{d+2}}{d(d+2)}\begin{bmatrix}
I_{d\times d} +O(\epsilon^2) & 0 \\
0 & 0 \\
\end{bmatrix}+O(\epsilon^4), \nonumber\\
U_x(\epsilon)&=U_x(0)(I_{p\times p}+\epsilon^2 \mathsf{S})+O(\epsilon^4)\in O(p),\nonumber
\end{align}
where 
$U_x(0)=\begin{bmatrix}
X_1 & 0 \\
0 & X_2\\
\end{bmatrix}$, $X_1 \in O(d)$, $X_2 \in O(p-d)$, and $\mathsf{S}\in\mathfrak{o}(p)$. Note that in this case $\mathsf{S}$, $X_1$ and $X_2$ cannot be uniquely determined by the order $\epsilon^{d+2}$ part of $C_x$. 

For Case 1 in Condition \ref{Condition:1}, when $\epsilon$ is sufficiently small, we have
\begin{align*}
&\Lambda_x=\frac{|S^{d-1}|  {P}(x)\epsilon^{d+2}}{d(d+2)}\begin{bmatrix}
I_{d\times d} +\epsilon^2\Lambda^{(2)}_1+\epsilon^4\Lambda^{(4)}_1 & 0 \\
0 & \epsilon^2\Lambda^{(2)}_2+\epsilon^4\Lambda^{(4)}_2 \\
\end{bmatrix}+O(\epsilon^3), \\
&U_x(\epsilon)=U_x(0)(I_{p\times p}+\epsilon^2 \mathsf{S})+O(\epsilon^4)\in O(p),
\end{align*}
where 
$U_x(0)=\begin{bmatrix}
X_1 & 0 \\
0 & X_2\\
\end{bmatrix}$, $X_1 \in O(d)$, and $X_2 \in O(p-d)$, and 
\begin{equation}
\mathsf{S}:=\begin{bmatrix}\mathsf{S}_{11} & \mathsf{S}_{12}\\ \mathsf{S}_{21} & \mathsf{S}_{22} \end{bmatrix}\in\mathfrak{o}(p)\label{Proof:LemmaD6:Condition1:S1}.
\end{equation} 
Since  $M^{(2)}_{22}$ is a diagonal matrix under Assumptions \ref{AssumptionTangent} and \ref{AssumptionNormal}, (\ref{AppendixB:Norepeat:Eq4}) implies that it is $M^{(2)}_{22}=\Lambda^{(2)}_2$. From (\ref{AppendixB:Norepeat:Eq2.1}) and (\ref{AppendixB:Norepeat:Eq3.1}), we have
\begin{align}
\mathsf{S}_{12}&=-X_1^\top M^{(2)}_{12}X_2 ,\label{Proof:LemmaD6:Condition1:S4} \\
\mathsf{S}_{21}&=X_2^\top M^{(2)}_{21}X_1.\label{Proof:LemmaD6:Condition1:S5}
\end{align}
If all eigenvalues of $M^{(2)}_{11}$ are distinct, then $X_1$ could be uniquely determined; if all eigenvalues of $M^{(2)}_{22}$ are distinct, since it is a diagonal matrix, $X_2$ is identity matrix.  Moreover, $\Lambda_{1}''$, $\Lambda_{2}''$ and $\mathsf{S}$ can be uniquely determined: 
\begin{align}
\Lambda_{1}''&=\texttt{diag}\big(X_1^{\top}M^{(4)}_{11}X_1+2X_1^{\top}M^{(2)}_{12}M^{(2)}_{21}X_1\big),\label{Proof:LemmaD6:Condition1:S2}\\
\Lambda_{2}''&=\texttt{diag}\big(M^{(4)}_{22}-2 M^{(2)}_{21}M^{(2)}_{12}\big),\label{Proof:LemmaD6:Condition1:S3}\\
(\mathsf{S}_{11})_{m,n}&=\frac{-1}{(\Lambda^{(2)}_1)_{m,m}-(\Lambda^{(2)}_1)_{n,n}}e_m^{\top}\big(X_1^{\top}M^{(4)}_{11}X_1+2X_1^{\top}M^{(2)}_{12}M^{(2)}_{21}X_1\big)e_n,\label{Proof:LemmaD6:Condition1:S6}\\
(\mathsf{S}_{22})_{i,j}&=\frac{-1}{(\Lambda^{(2)}_2)_{i,i}-(\Lambda^{(2)}_2)_{j,j}}e_i^{\top}\big(M^{(4)}_{22}-2M^{(2)}_{21}M^{(2)}_{12}\big)e_i,\label{Proof:LemmaD6:Condition1:S7}
\end{align}
where $1\leq m\neq n\leq d$ and $1\leq i\neq j\leq p-d$.
On the other hand, if $M^{(2)}_{22}$ has $q+t$ distinct eigenvalues, where $q,t\geq 0$, and $q$ eigenvalues are simple, then based on Appendix B, we have
\begin{align} \label{description of X_2}
X_2=\begin{bmatrix}
I_{q \times q} & 0 & \cdots &0 \\
0 & X_2^{1} & \cdots & 0\\
0 & 0 &  \ddots &0\\
0 &0  & \cdots & X_2^{t}\\
\end{bmatrix}
\end{align}
since $M^{(2)}_{22}$ is diagonal under Assumption \ref{AssumptionNormal}.
Each of $X_2^1 \cdots X_2^{t}$ corresponds to a repeated eigenvalue, and each of them is an orthogonal matrix whose dimension depends on the multiplicity of the repeated eigenvalue. We mention that they may be uniquely determined by higher order terms in $C_x$ as described in Appendix \ref{Section:Perturbation}.

For Case 2 in Condition \ref{Condition:1},  when $\epsilon$ is sufficiently small, by dividing all matrices into blocks of the same size, we have 
\begin{align*}
\Lambda_x&=\frac{|S^{d-1}|  {P}(x)\epsilon^{\epsilon^{d+2}}}{d(d+2)}\begin{bmatrix}
I_{d\times d} +\epsilon^2 \Lambda^{(2)}_1+\epsilon^4\Lambda^{(4)}_1 & 0 & 0\\
0 & \epsilon^2\Lambda^{(2)}_{2,1}+\epsilon^4\Lambda^{(4)}_{2,1} & 0\\
0 & 0 & \epsilon^4\Lambda^{(4)}_{2,2} \\
\end{bmatrix}+O(\epsilon^6),
\end{align*}
\begin{align}
U_x(\epsilon)&=U_x(0)(I_{p\times p}+\epsilon^2 \mathsf{S})+O(\epsilon^4)\in O(p),   \label{Proof:LemmaD6:Condition2:first}
\end{align}
\begin{align*}
U_x(0)&=\begin{bmatrix}
X_1 & 0 & 0\\
0 & X_{2,1} & 0\\
0 & 0 & X_{2,2} \\
\end{bmatrix}\in O(p) ,\quad
\mathsf{S}=
\begin{bmatrix}
\mathsf{S}_{11} & \mathsf{S}_{12,1} & \mathsf{S}_{12,2} \\ 
\mathsf{S}_{21,1} & \mathsf{S}_{22,11}  & \mathsf{S}_{22,12} \\
\mathsf{S}_{21,2} & \mathsf{S}_{22,21} & \mathsf{S}_{22,22}
\end{bmatrix}\in \mathfrak{o}(p)\,,
\end{align*}
by Assumptions \ref{AssumptionTangent} and \ref{AssumptionNormal}, where $\Lambda^{(2)}_1$ is the eigenvalue matrix of $M^{(2)}_{11}$, diagonal entries of $\Lambda^{(2)}_{2,1}$ are nonzero, $X_1 \in O(d)$,  $X_{2,1} \in O(p-d-l)$ and $X_{2,2} \in O(l)$. By (\ref{AppendixB:Norepeat:Eq2.1}) and (\ref{AppendixB:Norepeat:Eq3.1}), we have
\begin{align}
\mathsf{S}_{12,1}&= -X_1^\top M^{(2)}_{12,1}X_{2,1}, \nonumber\\
\mathsf{S}_{21,1} &= X_{2,1}^\top M^{(2)}_{21,1}X_1, \nonumber\\
\mathsf{S}_{12,2}&= -X_1^\top M^{(2)}_{12,2}X_{2,2},\label{Expansion:S122:ForthOrder}\\
\mathsf{S}_{21,2}&=X_{2,2}^\top M^{(2)}_{21,2}X_1\,. \nonumber
\end{align}
If the eigenvalues of $M^{(2)}_{11}$ are distinct, then $X_1$ is the corresponding orthonormal eigenvector matrix. $\Lambda^{(2)}_{2,2}=0$ by the assumption of Case 2 in Condition \ref{Condition:1}. Recall that $\Lambda^{(4)}_{2,2}$ is the eigenvalue matrix of $M^{(4)}_{22,22}-2\frac{d(d+2)} {|S^{d-1}|  {P}(x)}M^{(2)}_{21,2}M^{(2)}_{12,2}$. If  $\Lambda^{(4)}_{2,2}$  has different diagonal entries  then $X_{2,2}$ is the corresponding orthonormal eigenvector matrix. Recall that if $\Lambda^{(2)}_1$ , $\Lambda^{(2)}_{2,1}$ and $\Lambda^{(4)}_{2,2}$, each has distinct diagonal entries, then $X(0)$ and $\mathsf{S}$ can be determined uniquely, and we have
\begin{align*} 
\mathsf{S}_{22,12}&=(-\Lambda_{2,1}')^{-1} (\frac{1}{2}M^{(4)}_{22,12}X_{2,2}+M^{(2)}_{21,1}X_1 \mathsf{S}_{12,2}), \\
\mathsf{S}_{22,21}&=X_{2,2}^\top (\frac{1}{2} M^{(4)}_{22,21}+M^{(2)}_{21,2}X_1 \mathsf{S}_{12,1})(\Lambda_{2,1}^{(2)})^{-1}, \\
\Lambda^{(4)}_1& =\texttt{diag}\big[X^\top _{1}(M^{(4)}_{11}X_1+2M^{(2)}_{12,1}\mathsf{S}_{21,1}+2M^{(2)}_{12,2}X_{2,2}\mathsf{S}_{21,2})\big] \\
\Lambda^{(4)}_{2,1}&=\texttt{diag}\big[ (M^{(4)}_{22,11}+2M^{(2)}_{21,1}X_1\mathsf{S}_{12,1}) \big], \\
(\mathsf{S}_{11})_{m,n}&=\frac{-1}{(\Lambda^{(2)}_1)_{m,m}-(\Lambda^{(2)}_1)_{n,n}}e_m^{\top}\big[X^\top _{1}(\frac{1}{2}M^{(4)}X_1+M^{(2)}_{12,1}\mathsf{S}_{21,1}+M^{(2)}_{12,2}X_{2,2}\mathsf{S}_{21,2})\big]e_n,
\end{align*}
where $1 \leq m \not= n \leq d$, and
\begin{align}
(\mathsf{S}_{22,11})_{m,n}&=\frac{-1}{(\Lambda^{(2)}_{2,1})_{m,m}-(\Lambda^{(2)}_{2,1})_{n,n}}e_m^{\top}\big[  (\frac{1}{2}M^{(4)}_{22,11}+M^{(2)}_{21,1}X_1\mathsf{S}_{12,1}) \big]e_n. \label{Proof:LemmaD6:Condition2:last}
\end{align}
where $d+1 \leq m \not= n \leq p-l$. However, we need higher order derivative of $C_x$ to solve $\mathsf{S}_{22,22}$ following the same step as evaluating (\ref{Proof:ApendixA:C22_22}). We skip the details here.
Finally, if diagonal entries of $\Lambda^{(2)}_{2,1}$ are distinct, then $X_{2,1}$ is the identity matrix. If $\Lambda^{(2)}_{2,1}$ or $\Lambda^{(4)}_{2,2}$ contains repeated eigenvalues, then it can be described as (\ref{description of X_2}). We also skip the details here.

\section{Proof of Theorem 3.2}\label{Section:Theoremt1}

{%

We need the following Proposition for the proof.

\begin{proposition} \label{Lemma:5.1}
Suppose $l=\texttt{nullity}(M_{22}^{(2)})>0$ and Assumptions \ref{AssumptionTangent} and \ref{AssumptionNormal} hold. Then 
$\langle \Second_{x}(\theta,\theta),e_i  \rangle =0$ for $p-l+1 \leq i \leq p$. Moreover, for $m,n=p-l+1,\ldots,p$, we have
\begin{align*}
 & \big[M^{(4)}_{22,22}-2M^{(2)}_{21,2}M^{(2)}_{12,2}\big]_{m-p+l,n-p+l}
=\,\frac{d(d+2)}{36(d+6)|S^{d-1}|} \int_{S^{d-1}} \langle \nabla_{\theta}\Second_{x}(\theta,\theta),e_m \rangle \langle \nabla_{\theta}\Second_{x}(\theta,\theta),e_n \rangle d\theta  \\
  &- \frac{d^2(d+2)^2}{18|S^{d-1}|^2(d+4)^2} \sum_{k=1}^d  \int_{S^{d-1}} \langle \nabla_{\theta}\Second_{x}(\theta,\theta),e_m \rangle \langle \iota_*\theta,e_k \rangle  d\theta \int_{S^{d-1}} \langle \iota_*\theta,e_k \rangle  \langle \nabla_{\theta}\Second_{x}(\theta,\theta),e_n \rangle d\theta .\nonumber 
\end{align*}

\end{proposition}
}

{This Proposition essentially says that if $\texttt{nullity}(M_{22}^{(2)})=l>0$ and $M_{22}^{(2)}$ is diagonalized as in (\ref{Expansion:LocalCovarianceMatrix:Case2}), then geometrically $e_{p-l+1}, \ldots ,e_p$ are perpendicular to the second fundamental form $\Second_{x}(\theta,\theta)$. Furthermore, the eigenvalues of order $\epsilon^{d+6}$ in Case 2 of Proposition \ref{Proposition:2} depend only on the third order derivative of the embedding, $\nabla_{\theta}\Second_{x}(\theta,\theta)$, in those directions. }

\begin{proof}
{
Suppose $l=\texttt{nullity}(M_{22}^{(2)})>0$. By Assumption \ref{AssumptionNormal}, $M_{22}^{(2)}$ is diagonalized as in (\ref{Expansion:LocalCovarianceMatrix:Case2}). Therefore, based on (\ref{Proof:LemmaD5:Definition:tildeQ22}) and (\ref{Proof:LemmaD5:Definition:tildeM22}), we have

\begin{equation}
\int_{S^{d-1}} \langle \Second_{x}(\theta,\theta), e_m  \rangle  \langle  \Second_{x}(\theta,\theta) , e_m \rangle d\theta\,=0, \label{Proof:Lemma6:Condition2ForTheM222}
\end{equation}
where $m=p-l+1,\ldots,p$. }

If we denote $\theta=\theta^i\partial_i \in S^{d-1}\subset T_xM$, the following expression for the second fundamental form holds: 
\begin{equation}
\langle \Second_{x}(\theta,\theta),e_m  \rangle = \sum_{i=1}^d p^m_{ii}\,{\theta^i}^2 +2\sum_{i<j} p^m_{ij}\,\theta^i \theta^j\,,\label{Proof:Lemma6:ExpressionPim}
\end{equation}
where $p^m_{ij}=\langle \Second_{x}(\partial_i,\partial_j),e_m  \rangle\in \mathbb{R}$, $i,j=1,\ldots,d$, are the corresponding coefficients.
Note that $\iota_*\partial_i=e_i$ for $i=1,\ldots,d$.
By plugging (\ref{Proof:Lemma6:ExpressionPim}) into (\ref{Proof:Lemma6:Condition2ForTheM222}), we have
\begin{align}
0=&\,\int_{S^{d-1}}\Big[ \sum_{i=1}^d (p^m_{ii}{\theta^i}^2)^2+4\sum_{k=1}^d p^m_{kk}{\theta^k}^2\sum_{i <j} p^m_{ij}\theta^i \theta^j+4\big(\sum_{i <j} p^m_{ij}\theta^i \theta^j\big)^2 \Big]d\theta\nonumber\\
=&\, \frac{1}{d(d+2)}|S^{d-1}|\Big(3\sum_{i=1}^d (p^m_{ii})^2 +2 \sum_{i < j} p^m_{ii}p^m_{jj}+4\sum_{i < j} (p^m_{ij})^2\Big)\nonumber\\
=&\,2\sum_{i=1}^d (p^m_{ii})^2 + \big(\sum_{i = 1}^d p^m_{ii}\big)^2+4\sum_{i < j} (p^m_{ij})^2,\nonumber
\end{align}
which leads to the conclusion that $p^m_{ij}=0$ for all $i$ and $j$.
To get the expansion of $\big[M^{(4)}_{22,22}-2M^{(2)}_{21,2}M^{(2)}_{12,2}\big]_{m-p+l,n-p+l}$, we directly plug the above formula to (\ref{Proof:LemmaD5:Definition:tildeM12}) and (\ref{Proof:LemmaD5:Definition:barM22}) and get the claim.
\end{proof}


We introduce the following notations to simplify the proof:
\begin{align*}
\omega(x)&:=\frac{1}{|S^{d-1}|}\int_{S^{d-1}}\|\Second_x(\theta,\theta)\|^2d\theta\\
\mathfrak{N}_1(x)&:= \frac{1}{|S^{d-1}|}\int_{S^{d-1}}\|\Second_x(\theta,\theta)\|^2\Second_{x}(\theta,\theta)d\theta\\ 
\mathfrak{N}_2(x)&:= \frac{1}{|S^{d-1}|}\int_{S^{d-1}} \Second_x(\theta,\theta)\texttt{Ric}_{x}(\theta,\theta) d \theta\\ 
\mathfrak{M}_1(x)&:=\frac{1}{|S^{d-1}|}\int_{S^{d-1}}\|\Second_x(\theta,\theta)\|^2\theta \theta^{\top} d\theta\\
\mathfrak{M}_2(x)&:=\frac{1}{|S^{d-1}|}\int_{S^{d-1}} \Second_x(\theta,\theta) \theta\theta^{\top} d\theta \\
\mathfrak{R}_0(x)&:=\frac{1}{|S^{d-1}|}\int_{S^{d-1}} \theta \nabla_{\theta}\Second_x(\theta,\theta)  \cdot  \Second_x(\theta,\theta) d\theta \\
\mathfrak{R}_1(x)&:=\frac{1}{|S^{d-1}|}\int_{S^{d-1}}\nabla_\theta\Second_x(\theta,\theta)\theta^{\top} d\theta \\
\mathfrak{R}_2(x)&:=\frac{1}{|S^{d-1}|}\int_{S^{d-1}}\nabla_{\theta\theta}\Second_x(\theta,\theta) d\theta .
\end{align*}
For $f \in C^3(\iota(M))$ and $P\in C^5(M)$, define
\begin{align}
\Omega_f &:=\frac{1}{2}\nabla f(x)^{\top}\mathfrak{M}_2(x) \nabla P(x) +\frac{1}{4}P(x)\texttt{tr}(\mathfrak{M}_2(x) \nabla^2 f(x))+\frac{1}{6}P(x)\mathfrak{R}_1(x)\nabla f(x)\label{Definition:Omegaf:Proof}\\
\mathfrak{J}_f(x)&:=\frac{1}{|S^{d-1}|}\iota_*\int_{S^{d-1}}
\theta \big(\frac{1}{6}\nabla^3_{\theta,\theta,\theta}f(x)P(x)+\frac{1}{6}\nabla^3_{\theta,\theta,\theta}P(x)f(x)+\frac{1}{2}\nabla^2_{\theta,\theta}f(x)\nabla_\theta P(x)\nonumber\\
&\quad+\frac{1}{2}\nabla^2_{\theta,\theta}P(x)\nabla_\theta f(x)-\frac{1}{6}\texttt{Ric}_{x}(\theta,\theta)[f(x)\nabla_{\theta}P(x)+\nabla_{\theta}f(x)P(x)]\big)d \theta\,.\nonumber
\end{align}
We prepare some calculations. By Lemma \ref{Lemma:5},
we have
\begin{align}
\mathbb{E}[\chi_{B_{\epsilon}^{\mathbb{R}^p}(\iota(x))}(X)]\,=&\frac{|S^{d-1}|}{d}P(x)\epsilon^d+\frac{|S^{d-1}|}{d(d+2)}\Big[\frac{1}{2}\Delta P(x)\label{Expansion:ProofTheorem2:1}\\
&+\frac{s(x)P(x)}{6}+\frac{d(d+2)\omega(x)P(x)}{24}\Big]\epsilon^{d+2}+O(\epsilon^{d+3}),\nonumber
\end{align}
and hence
\begin{align}
&\mathbb{E}[(f(X)-f(x))\chi_{B_{\epsilon}^{\mathbb{R}^p}(\iota(x))}(X)]\label{Expansion:ProofTheorem2:f}\\
=\,&\mathbb{E}[f(X)\chi_{B_{\epsilon}^{\mathbb{R}^p}(\iota(x))}(X)]-f(x)\mathbb{E}[\chi_{B_{\epsilon}^{\mathbb{R}^p}(\iota(x))}(X)]\nonumber\\
= &\,\frac{|S^{d-1}|}{d(d+2)}\Big[\frac{1}{2} {P}(x)\Delta {f}(x)+ \nabla f(x) \cdot \nabla P(x)\Big] \epsilon^{d+2}+O( \epsilon^{d+3})\nonumber.
\end{align}
Similarly, by Lemma \ref{Lemma:5}, we have
\begin{align}
\mathbb{E}[(X-\iota(x))\chi_{B_{\epsilon}^{\mathbb{R}^p}(\iota(x))}(X)]=[\![{v}_1,{v}_2]\!],\label{Expansion:ProofTheorem2:X}
\end{align}
where
\begin{align}
 v_1=\, &\frac{|S^{d-1}|}{d+2} \frac{J_{p,d}^{\top}\iota_*\nabla {P}(x)}{d}\epsilon^{d+2}+\frac{|S^{d-1}|}{24} J_{p,d}^{\top}\iota_{*} \big(\mathfrak{M}_1(x)\nabla P(x)+P(x) \mathfrak{R}_0(x)\big)\epsilon^{d+4}\nonumber\\
&+ \,\frac{|S^{d-1}|}{d+4}J_{p,d}^{\top}\big[\mathfrak{J}_1(x) +\frac{1}{6}\mathfrak{R}_1(x) \nabla P(x)+\frac{1}{24}P(x)\mathfrak{R}_2(x) \big]\epsilon^{d+4}+O(\epsilon^{d+5})\nonumber,
\end{align}
and
\begin{align}
 v_2=\,& \frac{|S^{d-1}|}{d+2} \frac{{P}(x)\bar J_{p,p-d}^{\top}\mathfrak{N}_0(x)}{2}\epsilon^{d+2}+\frac{|S^{d-1}|}{24}\frac{P(x)\bar J_{p,p-d}^{\top}\mathfrak{N}_1(x)}{2}\epsilon^{d+4}\nonumber\\
&+\,\frac{|S^{d-1}|}{(d+4)}\bar J_{p,p-d}^{\top}\big[\frac{1}{4}\texttt{tr}(\mathfrak{M}_2(x) \nabla^2 P(x))-\frac{1}{12}f(x)P(x)\mathfrak{N}_2(x)\big]\epsilon^{d+4}\nonumber\\
&+\, \frac{|S^{d-1}|}{6(d+4)}\bar J_{p,p-d}^{\top}\big[\mathfrak{R}_1(x) \nabla P(x)+\frac{1}{4}P(x)\mathfrak{R}_2(x) \big]\epsilon^{d+4}  +O(\epsilon^{d+5})\,.\nonumber
\end{align}
Again, by Lemma \ref{Lemma:5}, we have 
\begin{align}
&\mathbb{E}[(X-\iota(x))(f(X)-f(x))\chi_{B_{\epsilon}^{\mathbb{R}^p}(\iota(x))}(X)]\label{Expansion:ProofTheorem2:Xf}\\
=\,&\mathbb{E}[(X-\iota(x))f(X)\chi_{B_{\epsilon}^{\mathbb{R}^p}(\iota(x))}(X)]-f(x)\mathbb{E}[(X-\iota(x))\chi_{B_{\epsilon}^{\mathbb{R}^p}(\iota(x))}(X)]\nonumber\\
=\,&[\![v_1,v_2]\!],\nonumber
\end{align}
where
\begin{align}
v_1=\, &\frac{|S^{d-1}|}{d(d+2)} P(x)J_{p,d}^{\top}\iota_*\nabla {f}(x)\epsilon^{d+2}+\,\frac{|S^{d-1}|}{24}P(x) J_{p,d}^{\top}\iota_{*} \mathfrak{M}_1(x)\nabla {f}(x)\epsilon^{d+4}\nonumber\\
&+ \,\frac{|S^{d-1}|}{d+4}J_{p,d}^{\top}\big[\mathfrak{J}_f(x)-f(x)\mathfrak{J}_1(x) +\frac{1}{6}P(x)\mathfrak{R}_1(x)\nabla f(x) \big]\epsilon^{d+4}+O(\epsilon^{d+5})\nonumber
\end{align}
and
\begin{align}
v_2  =\,& \frac{|S^{d-1}|}{d+4}\bar{J}_{p,p-d}^{\top}\Big(\frac{1}{2}\nabla f(x)^{\top}\mathfrak{M}_2(x) \nabla P(x) +\frac{1}{4}P(x)\texttt{tr}(\mathfrak{M}_2(x) \nabla^2 f(x))\nonumber\\
&+\frac{1}{6}P(x)\mathfrak{R}_1(x)\nabla f(x)\Big)\epsilon^{d+4}+O(\epsilon^{d+5}) \nonumber \\
=\, & \frac{|S^{d-1}|}{d+4}\bar{J}_{p,p-d}^{\top} \Omega_f \epsilon^{d+4}+O(\epsilon^{d+5}). \nonumber
\end{align}
With the above preparation, we are ready to prove Theorem \ref{Theorem:t1}.

\begin{proof}[Proof of Theorem \ref{Theorem:t1}]
The proof is straightforward, and we show it case by case. 

\underline{Case 0 in Condition \ref{Condition:1}}. In this case, by Lemma \ref{Lemma:6}, (\ref{Expansion:ProofTheorem2:Xf}), and (\ref{Expansion:ProofTheorem2:X}),
\begin{align}
&\mathbf{T}_{\iota(x)}^{\top}\mathbb{E}[X(f(X)-f(x))\chi_{B_{\epsilon}^{\mathbb{R}^p}(\iota(x))}(X)]
=\frac{|S^{d-1}|}{d(d+2)}\frac{P(x)\nabla {f}(x)\cdot \nabla {P}(x)}{{P}(x)+\frac{d(d+2)}{|S^{d-1}|}\epsilon^{\rho-2}}\epsilon^{d+2}+O(\epsilon^{d+4})\nonumber
\end{align}
and
\begin{align}
&\mathbf{T}_{\iota(x)}^{\top}\mathbb{E}[(X-\iota(x))\chi_{B_{\epsilon}^{\mathbb{R}^p}(\iota(x))}(X)]
=\frac{|S^{d-1}|}{d(d+2)}\frac{\nabla P(x)\cdot \nabla {P}(x)}{P(x)+\frac{d(d+2)}{|S^{d-1}|}\epsilon^{\rho-2}}\epsilon^{d+2}+O(\epsilon^{d+4}),\nonumber
\end{align}
and hence
\begin{align}
&\mathbb{E}[(f(X)-f(x))\chi_{B_{\epsilon}^{\mathbb{R}^p}(\iota(x))}(X)]-\mathbf{T}_{\iota(x)}^{\top}\mathbb{E}[(X-\iota(x))(f(X)-f(x))\chi_{B_{\epsilon}^{\mathbb{R}^p}(\iota(x))}(X)] \nonumber\\
=\,&\frac{|S^{d-1}|}{d(d+2)}\Big[\frac{1}{2} {P}(x)\Delta {f}(x)+ \nabla f(x) \cdot \nabla P(x)-\frac{P(x)\nabla {f}(x)\cdot \nabla {P}(x)}{{P}(x)+\frac{d(d+2)}{|S^{d-1}|}\epsilon^{\rho-2}}\Big] \epsilon^{d+2}+O(\epsilon^{d+4})\,.  \nonumber
\end{align}
Note that $\mathbf{T}_{\iota(x)}^{\top}\mathbb{E}[(X-\iota(x))\chi_{B_{\epsilon}^{\mathbb{R}^p}(\iota(x))}(X)]$ is of order $O(\epsilon^{d+2})$ for any $\rho$, therefore
\begin{align}
&\mathbb{E}[\chi_{B_{\epsilon}^{\mathbb{R}^p}(\iota(x))}(X)]-\mathbf{T}_{\iota(x)}^{\top}\mathbb{E}[(X-\iota(x))\chi_{B_{\epsilon}^{\mathbb{R}^p}(\iota(x))}(X)]=\frac{|S^{d-1}|}{d}P(x)\epsilon^d +O(\epsilon^{d+2}).\nonumber
\end{align}
As a result, we conclude that
\begin{align}
& Qf(x)-f(x)=\frac{1}{(d+2)}\big[\frac{1}{2}\Delta f(x) +\frac{\nabla f(x) \cdot \nabla P(x)}{P(x)}-\frac{\nabla {f}(x)\cdot \nabla {P}(x)}{{P}(x)+\frac{d(d+2)}{|S^{d-1}|}\epsilon^{\rho-2}} \big]\epsilon^2+O(\epsilon^4) .\nonumber
\end{align}

\underline{Case 1 in Condition \ref{Condition:1}}. 
Observe that by Lemma \ref{Lemma:6}, the tangential component of $\mathbf{T}_{\iota(x)}$ is of order $O(1)$ and the normal component of $\mathbf{T}_{\iota(x)}$ is of order $O(\frac{1}{\epsilon^2})$. Hence, by (\ref{Expansion:ProofTheorem2:X})
\begin{align}
&\mathbf{T}_{\iota(x)}^{\top}\mathbb{E}[(X-\iota(x))\chi_{B_{\epsilon}^{\mathbb{R}^p}(\iota(x))}(X)]\nonumber\\
=&\,\frac{|S^{d-1}|{P}(x)}{2(d+2) } \sum_{i=d+1}^p\frac{(\mathfrak{N}^{\top}_0(x)\bar{J}_{p,p-d}X_2\bar{J}_{p,p-d}^{\top}e_{i})^2}{\frac{2}{d}\lambda^{(2)}_i+\frac{2(d+2)}{P(x)|S^{d-1}|}\epsilon^{\rho-4}}\epsilon^{d}+O(\epsilon^{d+2})\,, \nonumber
\end{align}
and hence
\begin{align}
&\mathbb{E}[\chi_{B_{\epsilon}^{\mathbb{R}^p}(\iota(x))}(X)]-\mathbf{T}_{\iota(x)}^{\top}\mathbb{E}[(X-\iota(x))\chi_{B_{\epsilon}^{\mathbb{R}^p}(\iota(x))}(X)] \nonumber\\
=&\bigg[\frac{|S^{d-1}|}{d}P(x)-\frac{|S^{d-1}|{P}(x)}{2(d+2) } \sum_{i=d+1}^p\frac{(\mathfrak{N}^{\top}_0(x)\bar{J}_{p,p-d}X_2\bar{J}_{p,p-d}^{\top}e_{i})^2}{\frac{2}{d}\lambda^{(2)}_i+\frac{2(d+2)}{P(x)|S^{d-1}|}\epsilon^{\rho-4}}\bigg]\epsilon^d +O(\epsilon^{d+2}).\nonumber
\end{align}
Similarly, by (\ref{Expansion:ProofTheorem2:Xf}) 
\begin{align}
&\mathbf{T}_{\iota(x)}^{\top}\mathbb{E}[(X-\iota(x))(f(X)-f(x))\chi_{B_{\epsilon}^{\mathbb{R}^p}(\iota(x))}(X)]\nonumber \\
=\, &\bigg[\frac{|S^{d-1}|P(x)}{d(d+2)}\Big(\frac{\nabla {f}(x)\cdot \nabla {P}(x)}{{P}(x)+\frac{d(d+2)}{|S^{d-1}|}\epsilon^{\rho-2}}+\sum_{i=d+1}^p\frac{\mathfrak{N}^{\top}_0(x)\bar{J}_{p,p-d}X_2\bar{J}_{p,p-d}^{\top}e_i}{\frac{2}{d}\lambda^{(2)}_i+\frac{2(d+2)}{P(x)|S^{d-1}|}\epsilon^{\rho-4}}\nabla {f}(x)^{\top}J_{p,d}X_1\mathsf{S}_{12}\bar{J}_{p,p-d}^{\top}e_i\Big) \nonumber \\
& +\frac{|S^{d-1}|}{(d+4) } \sum_{i=d+1}^p\frac{\mathfrak{N}^{\top}_0(x)\bar{J}_{p,p-d}X_2\bar{J}_{p,p-d}^{\top}e_{i}}{\frac{2}{d}\lambda^{(2)}_i+\frac{2(d+2)}{P(x)|S^{d-1}|}\epsilon^{\rho-4}}\Omega_f^{\top}\bar{J}_{p,p-d}X_2\bar{J}_{p,p-d}^{\top}e_{i}\bigg ]\epsilon^{d+2}  +O(\epsilon^{d+3} )\,,\nonumber
\end{align}
which could be significantly simplified.
Since $M^{(2)}_{12}$ satisfies 
(\ref{Proposition32:Proof:M2:12:Expansion}), by a direct expansion we have that
\begin{equation} 
\nabla {f}(x)^{\top}J_{p,d}M_{12}^{(2)}=\frac{d(d+2)}{P(x)(d+4)}\Big(\frac{1}{2}\nabla f(x)^{\top}\mathfrak{M}_2(x) \nabla P(x) +\frac{1}{6}P(x)\mathfrak{R}_1(x)\nabla f(x)\Big)^{\top} \bar{J}_{p,p-d}\nonumber\,.
\end{equation}
By (\ref{Proof:LemmaD6:Condition1:S4}), we have $X_1\mathsf{S}_{12}=-M^{(2)}_{12} X_2$,
and hence 
\begin{align} 
&\frac{|S^{d-1}|{P}(x)}{d(d+2)}\frac{\mathfrak{N}^{\top}_0(x)\bar{J}_{p,p-d}X_2\bar{J}_{p,p-d}^{\top}e_i}{\frac{2}{d}\lambda^{(2)}_i+\frac{2(d+2)}{P(x)|S^{d-1}|}\epsilon^{\rho-d-4}}\nabla {f}(x)^{\top}J_{p,d}X_1\mathsf{S}_{12}\bar{J}_{p,p-d}^{\top}e_i \nonumber\\
=\,&-\frac{|S^{d-1}|}{d+4}\frac{\mathfrak{N}^{\top}_0(x)\bar{J}_{p,p-d}X_2\bar{J}_{p,p-d}^{\top}e_i}{\frac{2}{d}\lambda^{(2)}_i+\frac{2(d+2)}{P(x)|S^{d-1}|}\epsilon^{\rho-d-4}}\Big(\frac{1}{2}\nabla f(x)^{\top}\mathfrak{M}_2(x) \nabla P(x) 
+\frac{1}{6}P(x)\mathfrak{R}_1(x)\nabla f(x)\Big)^{\top} \bar{J}_{p,p-d} X_2\bar{J}_{p,p-d}^{\top}e_i \,.\nonumber
\end{align}
Combining this with $\Omega_f$ defined in (\ref{Definition:Omegaf:Proof}),
the second and third terms in $\mathbf{T}_{\iota(x)}^{\top}\mathbb{E}[(X-\iota(x))(f(X)-f(x))\chi_{B_{\epsilon}^{\mathbb{R}^p}(\iota(x))}(X)]$ are simplified. As a result, we have
\begin{align}
&\mathbb{E}[(f(X)-f(x))\chi_{B_{\epsilon}^{\mathbb{R}^p}(\iota(x))}(X)]-\mathbf{T}_{\iota(x)}^{\top}\mathbb{E}[(X-\iota(x))(f(X)-f(x))\chi_{B_{\epsilon}^{\mathbb{R}^p}(\iota(x))}(X)] \nonumber\\
=&\,\bigg[\frac{|S^{d-1}|}{d(d+2)}\Big(\frac{1}{2} {P}(x)\Delta {f}(x)+ \nabla f(x) \cdot \nabla P(x)-\frac{P(x)\nabla {f}(x)\cdot \nabla {P}(x)}{{P}(x)+\frac{d(d+2)}{|S^{d-1}|}\epsilon^{\rho-2}}\Big)\nonumber \\
& -\frac{|S^{d-1}|P(x)}{4(d+4) } \sum_{i=d+1}^p\frac{\mathfrak{N}^{\top}_0(x)\bar{J}_{p,p-d}X_2\bar{J}_{p,p-d}^{\top}e_i}{\frac{2}{d}\lambda^{(2)}_i+\frac{2(d+2)}{P(x)|S^{d-1}|}\epsilon^{\rho-4}}\mathfrak{H}^{\top}_f(x)\bar{J}_{p,p-d}X_2\bar{J}_{p,p-d}^{\top}e_{i}\bigg ]\epsilon^{d+2}+O(\epsilon^{d+3}) \nonumber\,.
\end{align}
To finish the proof for Case 1, we claim that 
\begin{align}\label{projection onto the eigenspace}
\sum_{i=d+1}^p\frac{\mathfrak{N}^{\top}_0(x)\bar{J}_{p,p-d}X_2\bar{J}_{p,p-d}^{\top}e_i}{\frac{2}{d}\lambda^{(2)}_i+\frac{2(d+2)}{P(x)|S^{d-1}|}\epsilon^{\rho-4}}\mathfrak{H}^{\top}_f(x)\bar{J}_{p,p-d}X_2\bar{J}_{p,p-d}^{\top}e_{i}=\sum_{i=d+1}^p\frac{\mathfrak{N}^{\top}_0(x)e_i\mathfrak{H}^{\top}_f(x)e_{i}}{\frac{2}{d}\lambda^{(2)}_i+\frac{2(d+2)}{P(x)|S^{d-1}|}\epsilon^{\rho-4}}.\nonumber
\end{align}
Recall (\ref{description of X_2}). Suppose there are $q+t$ eigenvalues of $M^{(2)}_{22}$, where $q,t\in \mathbb{N}\cup\{0\}$, so that $q$ eigenvalues are simple. We have  
\begin{align} 
X_2=\begin{bmatrix}
I_{q \times q} & 0 & \cdots &0 \\
0 & X_2^{1} & \cdots & 0\\
0 & 0 &  \ddots &0\\
0 &0  & \cdots & X_2^t\\
\end{bmatrix}\,,\nonumber
\end{align}
where $X_2^j$, where $j=1,\ldots,t$ are orthogonal matrices whose size is the multiplicity of the associated eigenvalue.
Suppose $X_2^1\in O(\alpha)$, where $\alpha>1$.  Then 
\begin{align*}
\sum_{i=d+q}^{d+q+\alpha}(\mathfrak{N}^{\top}_0(x)\bar{J}_{p,p-d}X_2\bar{J}_{p,p-d}^{\top}e_i)(\mathfrak{H}^{\top}_f(x)\bar{J}_{p,p-d}X_2\bar{J}_{p,p-d}^{\top}e_{i})=\sum_{i=d+q}^{d+q+\alpha}(\mathfrak{N}^{\top}_0(x)e_i)(\mathfrak{H}^{\top}_f(x)e_{i})
\end{align*}
since the left hand side is the inner product between the projections of $\mathfrak{N}_0(x)$ and $\mathfrak{H}_f(x)$ onto the eigenspace.
By a similar argument for the other blocks, we conclude the claim.
By exactly the same argument we have
\begin{align}
\sum_{i=d+1}^p\frac{(\mathfrak{N}^{\top}_0(x)\bar{J}_{p,p-d}X_2\bar{J}_{p,p-d}^{\top}e_i)^2}{\frac{2}{d}\lambda^{(2)}_i+\frac{2(d+2)}{P(x)|S^{d-1}|}\epsilon^{\rho-4}}=\sum_{i=d+1}^p\frac{(\mathfrak{N}^{\top}_0(x)e_i)^2}{\frac{2}{d}\lambda^{(2)}_i+\frac{2(d+2)}{P(x)|S^{d-1}|}\epsilon^{\rho-4}}.\nonumber
\end{align}
In conclusion, we have
\begin{equation}
Qf(x)-f(x)=(\mathfrak{C_1}(x)+\mathfrak{C_2}(x))\epsilon^2+O(\epsilon^3),\nonumber
\end{equation}
where
\begin{align}
\mathfrak{C_1}(x)&=\frac{\frac{1}{d(d+2)}\Big[\frac{1}{2}\Delta {f}(x)+ \frac{\nabla f(x) \cdot \nabla P(x)}{P(x)}-\frac{\nabla {f}(x)\cdot \nabla {P}(x)}{{P}(x)+\frac{d(d+2)}{|S^{d-1}|}\epsilon^{\rho-2}}\Big] }{\frac{1}{d}-\frac{1}{2(d+2) } \sum_{i=d+1}^p\frac{(\mathfrak{N}^{\top}_0(x)e_i)^2}{\frac{2}{d}\lambda^{(2)}_i+\frac{2(d+2)}{P(x)|S^{d-1}|}\epsilon^{\rho-4}}}\nonumber
\end{align}
and
\begin{align}
\mathfrak{C_2}(x)&=-\frac{\frac{1}{4(d+4) }\sum_{i=d+1}^p\frac{(\mathfrak{N}^{\top}_0(x)e_i)(\mathfrak{H}^{\top}_f(x)e_{i})}{\frac{2}{d}\lambda^{(2)}_i+\frac{2(d+2)}{P(x)|S^{d-1}|}\epsilon^{\rho-4}}}{\frac{1}{d}-\frac{1}{2(d+2) } \sum_{i=d+1}^p\frac{(\mathfrak{N}^{\top}_0(x)e_i)^2}{\frac{2}{d}\lambda^{(2)}_i+\frac{2(d+2)}{P(x)|S^{d-1}|}\epsilon^{\rho-4}}}.\nonumber
\end{align}

\underline{Case 2 in Condition \ref{Condition:1}}. 
In this case, by (\ref{Expansion:Tx:Case2:v1}) and (\ref{Expansion:Tx:Case2:v2}), we rewrite $\mathbf{T}_{\iota(x)}$ as
\begin{align}
\mathbf{T}_{\iota(x)}
=&\,[\![v_1,\,v_2]\!]+[\![O(\epsilon^2),\,O(1)]\!],\nonumber
\end{align}
where 
\begin{align}
v_1=\,&\frac{J_{p,d}^\top \iota_*\nabla P(x)}{{P}(x)+\frac{d(d+2)}{|S^{d-1}|}\epsilon^{\rho-2}}+\sum_{i=d+1}^{p-l}\frac{\mathfrak{N}^\top _0(x)\tilde{J}X_{2,1}\tilde{J}^\top e_i}{\frac{2}{d}\lambda^{(2)}_i+\frac{2(d+2)}{P(x)|S^{d-1}|}\epsilon^{\rho-4}}X_{2,1}\mathsf{S}_{12,1}\tilde{J}^\top e_i\nonumber\\
&\qquad+\sum_{i=p-l+1}^p\frac{\alpha_i }{\frac{|S^{d-1}|  {P}(x)}{d(d+2)}\lambda^{(4)}_i+\epsilon^{\rho-6}} X_1\mathsf{S}_{12,2}(\bar J_{p,l})^\top e_{i}\nonumber
\end{align}
and
\begin{align}
v_2=\,&\frac{1}{\epsilon^2}\sum_{i=d+1}^{p-l}\frac{\mathfrak{N}^\top _0(x)\tilde{J}X_{2,1}\tilde{J}^\top e_{i}}{\frac{2}{d}\lambda^{(2)}_i+\frac{2(d+2)}{P(x)|S^{d-1}|}\epsilon^{\rho-4}} X_{2,2} \bar{J}_{p,l}^\top e_i\nonumber\\
&\quad+\frac{1}{\epsilon^2}\sum_{i=p-l+1}^p\frac{\alpha_i}{\frac{|S^{d-1}|  {P}(x)}{d(d+2)}\lambda^{(4)}_i+\epsilon^{\rho-6}}X_{2,2} \bar{J}_{p,l}^\top e_i\,,\nonumber
\end{align}
Note that $\frac{\alpha_i }{\frac{|S^{d-1}|  {P}(x)}{d(d+2)}\lambda^{(4)}_i+\epsilon^{\rho-6}}$ is of order $1$ or smaller, no matter what regularization order $\rho$ is chosen.
Rewrite (\ref{Expansion:ProofTheorem2:X}) up to $O(\epsilon^{d+4})$ as 
\begin{align}
\mathbb{E}[(X-\iota(x))\chi_{B_{\epsilon}^{\mathbb{R}^p}(\iota(x))}(X)]=&\Big[\!\!\!\Big[\frac{|S^{d-1}|}{d+2} \frac{J_{p,d}^{\top}\iota_*\nabla {P}(x)}{d}\epsilon^{d+2}+O(\epsilon^{d+4}), \nonumber\\
&\qquad \frac{|S^{d-1}|}{d+2} \frac{{P}(x)\bar{J}_{p,p-d}^{\top}\mathfrak{N}_0(x)}{2}\epsilon^{d+2}+O(\epsilon^{d+4}) \Big]\!\!\!\Big]\,.\nonumber
\end{align}
We claim that in $\mathbb{E}[(X-\iota(x))\chi_{B_{\epsilon}^{\mathbb{R}^p}(\iota(x))}(X)]^{\top} \mathbf{T}_{\iota(x)}$, the ``fourth order'' terms, i.e., the terms with $\sum_{i=p-l+1}^p$, do not have dominant contribution asymptotically by showing that for each $i=p-l+1,\ldots,p$,
we have
\begin{align} 
 \mathbb{E}[(X-\iota(x))\chi_{B_{\epsilon}^{\mathbb{R}^p}(\iota(x))}(X)]^{\top}\Big[\!\!\!\Big[ X_1\mathsf{S}_{12,2}\bar J_{p,l}^\top e_{i},\, \frac{1}{\epsilon^2}X_{2,2}\bar{J}_{p,l}^\top e_i\Big]\!\!\!\Big] =O(\epsilon^{d+1})\label{case3 denominator}
\end{align}
and
\begin{align}
& \mathbb{E}[(f(X)-f(x))(X-\iota(x))\chi_{B_{\epsilon}^{\mathbb{R}^p}(\iota(x))}(X)]^\top \Big[\!\!\!\Big[ X_1\mathsf{S}_{12,2}\bar J_{p,l}^\top e_{i}, \, \frac{1}{\epsilon^2}X_{2,2}\bar{J}_{p,l}^\top e_i\Big]\!\!\!\Big] =O(\epsilon^{d+3}). \label{case3 numerator}
\end{align}
Since the tangential direction of $\mathbf{T}_{\iota(x)}$ is of order $1$ and the normal direction of $\mathbf{T}_{\iota(x)}$ is of order $\epsilon^{-2}$, it is sufficient to focus on the normal direction in order to show (\ref{case3 denominator}). By Proposition \ref{Lemma:5.1}, the dominant term in the normal direction satisfies 
\begin{align}
\mathfrak{N}^{\top}_0(x) \bar{J}_{p,l}X_{2,2}  \bar{J}_{p,l}^{\top}e_i =0\,,\nonumber
\end{align}
and hence (\ref{case3 denominator}) follows.

To show (\ref{case3 numerator}), 
for each $p-l+1 \leq i \leq p$, by a direct expansion we have
\begin{align}
& \mathbb{E}[(f(X)-f(x))(X-\iota(x))\chi_{B_{\epsilon}^{\mathbb{R}^p}(\iota(x))}(X)]^\top  \Big[\!\!\!\Big[ X_1\mathsf{S}_{12,2}\bar J_{p,l}^\top e_{i},\,\frac{1}{\epsilon^2} X_{2,2} \bar{J}_{p,p-d}^\top e_i\Big]\!\!\!\Big] \nonumber\\
=\,& \Big(\frac{|S^{d-1}|}{d(d+2)} P(x)\iota_*\nabla {f}(x)^\top J_{p,d}X_1\mathsf{S}_{12,2}\bar J_{p,l}^\top e_{i} + \frac{|S^{d-1}|}{d+4}\Omega_f^\top  \bar{J}_{p,l}  X_{2,2}  \bar{J}_{p,l}^\top e_i \Big) \epsilon^{d+2}+O(\epsilon^{d+3}) \,. \nonumber
\end{align}
Again, it is sufficient to focus on the normal direction.
We now claim that 
\begin{align}
\frac{|S^{d-1}|}{d(d+2)} P(x)\iota_*\nabla {f}(x)^\top J_{p,d}X_1\mathsf{S}_{12,2}\bar J_{p,l}^\top e_{i} + \frac{|S^{d-1}|}{d+4}\Omega_f^\top  \bar{J}_{p,l}  X_{2,2}  \bar{J}_{p,l}^\top e_i=0.\label{Expansion:ExpfX:Proof:Theorem2}
\end{align}
Based on Lemma \ref{Lemma:6} and (\ref{Expansion:S122:ForthOrder}), the first part of (\ref{Expansion:ExpfX:Proof:Theorem2}) becomes
\begin{align}
&\frac{|S^{d-1}|P(x)}{d(d+2)}\nabla {f}(x)^\top J_{p,d}X_1\mathsf{S}_{12,2}\bar J_{p,l}^\top e_i \nonumber\\
=&\,-\frac{|S^{d-1}|P(x)}{d(d+2)}\nabla {f}(x)^\top J_{p,d}M_{12,2}^{(2)}X_{2,2}\bar J_{p,l}^\top e_i  \nonumber \\
=&\,-\frac{|S^{d-1}|P(x)}{d(d+2)}\nabla {f}(x)^\top J_{p,d}M_{12}^{(2)}\bar{J}_{p,p-d}^\top \bar J_{p,l}X_{2,2}\bar J_{p,l}^\top e_i  \nonumber \\
=&\,-\frac{|S^{d-1}|}{(d+4)} \Big(\frac{1}{2}\nabla f(x)^\top \mathfrak{M}_2(x) \nabla P(x) 
+\frac{1}{6}P(x)\mathfrak{R}_1(x)\nabla f(x)\Big)^\top  \bar J_{p,l} X_{2,2}\bar J_{p,l}^\top e_i\,, \nonumber
\end{align}
where the second equality comes from the direct expansion that
\begin{equation}
M^{(2)}_{12,2}=M^{(2)}_{12}\bar{J}_{p,p-d}^\top \bar J_{p,l}\nonumber
\end{equation}
and the last equality comes from (\ref{Proposition32:Proof:M2:12:Expansion}). 
For the second part of (\ref{Expansion:ExpfX:Proof:Theorem2}), based on Proposition \ref{Lemma:5.1}, for $p-l+1\leq i \leq p$, we have
\begin{align}
& \frac{|S^{d-1}|}{d+4}\Omega_f^\top  \bar{J}_{p,l} X_{2,2}  \bar{J}_{p,l}^\top e_i\nonumber \\
=\, &\frac{|S^{d-1}|}{(d+4)} \Big(\frac{1}{2}\nabla f(x)^\top \mathfrak{M}_2(x) \nabla P(x) 
+\frac{1}{6}P(x)\mathfrak{R}_1(x)\nabla f(x)\Big)^\top  \bar J_{p,l}X_{2,2}\bar J_{p,l}^\top e_i\nonumber\,.
\end{align}
Thus, two terms in (\ref{Expansion:ExpfX:Proof:Theorem2}) cancel each other and 
(\ref{case3 numerator}) follows.
Based on the above discussion, we know that $\mathbb{E}[(X-\iota(x))\chi_{B_{\epsilon}^{\mathbb{R}^p}(\iota(x))}(X)]^\top  \mathbf{T}_{\iota(x)}$ is dominated by 
\begin{align}
& \mathbb{E}[(X-\iota(x))\chi_{B_{\epsilon}^{\mathbb{R}^p}(\iota(x))}(X)]^\top  \Big[\!\!\!\Big[\frac{J_{p,d}^\top \iota_*\nabla P(x)}{{P}(x)+\frac{d(d+2)}{|S^{d-1}|}\epsilon^{\rho-2}}+\sum_{i=d+1}^{p-l}\frac{\mathfrak{N}^\top _0(x)\tilde{J}X_{2,1}\tilde{J}^\top e_i}{\frac{2}{d}\lambda^{(2)}_i+\frac{2(d+2)}{P(x)|S^{d-1}|}\epsilon^{\rho-4}}X_{2,1}\mathsf{S}_{12,1}\tilde{J}^\top e_i,\nonumber\\
&\qquad\frac{1}{\epsilon^2}\sum_{i=d+1}^{p-l}\frac{\mathfrak{N}^\top _0(x)\tilde{J}X_{2,1}\tilde{J}^\top e_{i}}{\frac{2}{d}\lambda^{(2)}_i+\frac{2(d+2)}{P(x)|S^{d-1}|}\epsilon^{\rho-4}}\begin{bmatrix}X_{2,1} & 0 \\ 0 & X_{2,2} \end{bmatrix}\bar{J}_{p,p-d}^\top e_i\Big]\!\!\!\Big] \nonumber\,,
\end{align}
which is of order $O(\epsilon^{d})$ by a similar argument as in Case 1, and $\mathbb{E}[(f(X)-f(x))(X-\iota(x))\chi_{B_{\epsilon}^{\mathbb{R}^p}(\iota(x))}(X)]^\top  \mathbf{T}_{\iota(x)}$ is dominated by
\begin{align}
& \mathbb{E}[(f(X)-f(x))(X-\iota(x))\chi_{B_{\epsilon}^{\mathbb{R}^p}(\iota(x))}(X)]^\top  \Big[\!\!\!\Big[\frac{J_{p,d}^\top \iota_*\nabla P(x)}{{P}(x)+\frac{d(d+2)}{|S^{d-1}|}\epsilon^{\rho-2}}+\sum_{i=d+1}^{p-l}\frac{\mathfrak{N}^\top _0(x)\tilde{J}X_{2,1}\tilde{J}^\top e_i}{\frac{2}{d}\lambda^{(2)}_i+\frac{2(d+2)}{P(x)|S^{d-1}|}\epsilon^{\rho-4}}X_{2,1}\mathsf{S}_{12,1}\tilde{J}^\top e_i,\nonumber\\
&\qquad\frac{1}{\epsilon^2}\sum_{i=d+1}^{p-l}\frac{\mathfrak{N}^\top _0(x)\tilde{J}X_{2,1}\tilde{J}^\top e_{i}}{\frac{2}{d}\lambda^{(2)}_i+\frac{2(d+2)}{P(x)|S^{d-1}|}\epsilon^{\rho-4}}\begin{bmatrix}X_{2,1} & 0 \\ 0 & X_{2,2} \end{bmatrix}\bar{J}_{p,p-d}^\top e_i\Big]\!\!\!\Big] \nonumber .
\end{align}
which is of order $O(\epsilon^{d+2})$ by using a similar argument in Case 1. 
By putting the above together, we conclude that
\begin{equation}
Qf(x)-f(x)=(\mathfrak{C_1}(x)+\mathfrak{C_2}(x))\epsilon^2+O(\epsilon^3),\nonumber
\end{equation}
where
\begin{align}
\mathfrak{C_1}(x)&=\frac{\frac{1}{d(d+2)}\Big[\frac{1}{2}\Delta {f}(x)+ \frac{\nabla f(x) \cdot \nabla P(x)}{P(x)}-\frac{\nabla {f}(x)\cdot \nabla {P}(x)}{{P}(x)+\frac{d(d+2)}{|S^{d-1}|}\epsilon^{\rho-2}}\Big] }{\frac{1}{d}-\frac{1}{2(d+2) } \sum_{i=d+1}^{p-l}\frac{(\mathfrak{N}^{\top}_0(x)e_{i})^2}{\frac{2}{d}\lambda^{(2)}_i+\frac{2(d+2)}{P(x)|S^{d-1}|}\epsilon^{\rho-4}}}\nonumber
\end{align}
\begin{align}
\mathfrak{C_2}(x)&=-\frac{\frac{1}{4(d+4) } \sum_{i=d+1}^{p-l}\frac{\mathfrak{N}^{\top}_0(x)e_i\mathfrak{H}^{\top}_f(x)e_{i}}{\frac{2}{d}\lambda^{(2)}_i+\frac{2(d+2)}{P(x)|S^{d-1}|}\epsilon^{\rho-4}}}{\frac{1}{d}-\frac{1}{2(d+2) } \sum_{i=d+1}^{p-l}\frac{(\mathfrak{N}^{\top}_0(x)e_{i})^2}{\frac{2}{d}\lambda^{(2)}_i+\frac{2(d+2)}{P(x)|S^{d-1}|}\epsilon^{\rho-4}}}\,,\nonumber
\end{align}
and hence we finish the proof.
\end{proof}

\section{Proof of Theorem 3.1}\label{Section:Theoremt0}

For each $x_k$, denote $\boldsymbol{f}=(f(x_{k,1}),f(x_{k,2}),\ldots,f(x_{k,N}))^{\top}\in\mathbb{R}^N$. By the expansion
\begin{align}
\sum_{j=1}^N w_k(j)f(x_{k,j})&=\frac{\boldsymbol{1}_N^{\top}\boldsymbol{f} - \boldsymbol{1}^{\top}_NG_n^\top \mathcal{I}_{n\epsilon^{d+\rho}}(G_nG_n^{\top}) G_n\boldsymbol{f}}
{N -  \boldsymbol{1}^{\top}_NG_n^\top \mathcal{I}_{n\epsilon^{d+\rho}}(G_nG_n^{\top})G_n\boldsymbol{1}_N },\nonumber
\end{align}
we can write $\sum_{j=1}^N w_k(j)f(x_{k,j})-f(x_k)$ as
\begin{align}
\frac{\frac{1}{n}\sum_{j=1}^N(f(x_{k,j})-f(x_k))- [\frac{1}{n}\sum_{j=1}^N(x_{k,j}-x_k)]^{\top}n\mathcal{I}_{n\epsilon^{d+\rho}}(G_nG_n^{\top})[\frac{1}{n}\sum_{j=1}^N(x_{k,j}-x_k)(f(x_{k,j})-f(x_k))]}
{\frac{N}{n} -  [ \frac{1}{n}\sum_{j=1}^N(x_{k,j}-x_k)]^\top n\mathcal{I}_{n\epsilon^{d+\rho}}(G_nG_n^{\top}) [ \frac{1}{n}\sum_{j=1}^N(x_{k,j}-x_k)] }.\label{Proof:Theorem0:FirstDirectExpansion}
\end{align}
Note that we have
\begin{equation}
n\mathcal{I}_{n\epsilon^{d+\rho}}(G_nG_n^{\top})=\mathcal{I}_{\epsilon^{d+\rho}}(\frac{1}{n}G_nG_n^\top).\nonumber
\end{equation}
Thus, the goal is to relate the finite sum quantity (\ref{Proof:Theorem0:FirstDirectExpansion}) with the following ``expectation''  {
\begin{equation}
\frac{Af(x_k)}{A1(x_k)}-f(x_k)= Qf(x_k)-f(x_k)\,,\label{Proof:Theorem1:FinalExpectationFormulation}
\end{equation}
where $A$ is defined in (\ref{Definition:Aoperator}).}
Note that the LLE is a ratio of two dependent random variables, and the denominator and numerator both involve complicated mixup of sampling points. Therefore, the convergence fluctuation cannot be simply computed. 
We control the size of the fluctuation of the following five terms
\begin{align}
&\frac{1}{n\epsilon^d}\sum_{j=1}^N1\label{Proof:Theorem1:FiniteSum0}\\
&\frac{1}{n\epsilon^{d}}\sum_{j=1}^N(f(x_{k,j})-f(x_k)) 
\label{Proof:Theorem1:FiniteSum1}\\
&\frac{1}{n\epsilon^{d}}\sum_{j=1}^N(x_{k,j}-x_k) \label{Proof:Theorem1:FiniteSum2}\\
&\frac{1}{n\epsilon^{d}}\sum_{j=1}^N(x_{k,j}-x_k)(f(x_{k,j})-f(x_k)) 
\label{Proof:Theorem1:FiniteSum3}\\
&\frac{1}{n\epsilon^{d}}(G_nG_n^{\top}+\epsilon^{\rho}I_{p\times p}) 
\label{Proof:Theorem1:FiniteSum4}
\end{align}
as functions of $n$ and $\epsilon$ by the Bernstein type inequality. Here, we put $\epsilon^{-d}$ in front of each term to normalize the kernel so that the computation is consistent with the existing literature, like \cite{Cheng_Wu:2013,Singer_Wu:2016}. 
The size of the fluctuation of these terms are controlled in the following Lemmas. The term (\ref{Proof:Theorem1:FiniteSum0}) is the usual kernel density estimation, so we have the following lemma.
\begin{lemma}\label{Proof:Theorem1:LemmaF1}
When $n$ is large enough, we have with probability greater than $1-n^{-2}$ that for all $k=1,\ldots,n$ that
\begin{equation}
\left|\frac{1}{n\epsilon^d}\sum_{j=1}^N1 - \mathbb{E}\frac{1}{\epsilon^d}\chi_{{B}^{\mathbb{R}^p}_\epsilon(x_k)}(X)\right| =  O\Big(\frac{\sqrt{\log (n)}}{n^{1/2}\epsilon^{d/2}}\Big)\,.\nonumber
\end{equation}
\end{lemma}

The behavior of (\ref{Proof:Theorem1:FiniteSum1}) is summarized in the following Lemma. Although the proof is standard, we provide it for the sake of self-containedness.
\begin{lemma}\label{Proof:Theorem1:LemmaF2}
When $n$ is large enough, we have with probability greater than $1-n^{-2}$ that for all $k=1,\ldots,n$ that
\begin{equation}
\left|\frac{1}{n\epsilon^d}\sum_{j=1}^N(f(x_{k,j})-f(x_k)) - \mathbb{E}\frac{1}{\epsilon^d}(f(X)-f(x_k))\chi_{{B}^{\mathbb{R}^p}_\epsilon(x_k)}(X)\right| =  O\Big(\frac{\sqrt{\log (n)}}{n^{1/2}\epsilon^{d/2-1}}\Big)\,.\nonumber
\end{equation}
\end{lemma}

\begin{proof}
By denoting 
\begin{equation}
F_{1,j}=\frac{1}{\epsilon^d}(f(x_{j})-f(x_k))\chi_{{B}^{\mathbb{R}^p}_\epsilon(x_k)}(x_{j}), \nonumber
\end{equation}
we have
\begin{equation}
\frac{1}{n\epsilon^d}\sum_{j=1}^N(f(x_{k,j})-f(x_k))=\frac{1}{n}\sum_{j\neq k, j=1}^nF_{1,j}.\nonumber
\end{equation}
Define a random variable
\begin{equation}
F_{1}:=\frac{1}{\epsilon^d}(f(X)-f(x_k))\chi_{{B}^{\mathbb{R}^p}_\epsilon(x_k)}(X).\nonumber
\end{equation}
Clearly, when $j\neq k$, $F_{1,j}$ can be viewed as randomly sampled i.i.d. from $F_{1}$.
Note that we have
\begin{equation}
\frac{1}{n}\sum_{j\neq k, j=1}^nF_{1,j}=\frac{n-1}{n}\left[\frac{1}{n-1}\sum_{j\neq k, j=1}^nF_{1,j}\right]\,. \nonumber
\end{equation}
Since $\frac{n-1}{n}\to 1$ as $n\to\infty$, the error incurred by replacing $\frac{1}{n}$ by $\frac{1}{n-1}$ is of order $\frac{1}{n}$, which is negligible asymptotically. Thus, we can simply focus on analyzing $\frac{1}{n-1}\sum_{j=1,j\neq i}^n F_{1,j}$. 
We have by Lemma \ref{Lemma:5}
\begin{align}
\mathbb{E}[F_{1}]  =& \frac{|S^{d-1}|}{2d(d+2)}\big[\Delta ((f(y)-f(x_k))P(y))|_{y=x_k}\big] \epsilon^{2}+O( \epsilon^{3})\nonumber\\
\mathbb{E}[F_{1}^2]  =\,&  \frac{|S^{d-1}|}{2d(d+2)}\big[\Delta ((f(y)-f(x_k))^2P(y))|_{y=x_k}\big] \epsilon^{-d+2}+O( \epsilon^{-d+3}),\nonumber
\end{align}
where $\Delta$ acts on $y$ and we apply the Lemma by viewing $f(y)P(y)$ as a function and evaluate the integration over the uniform measure. 
Thus, we conclude that
\begin{align}
\sigma_1^2:=\text{Var}(F_{1})  =\,& \frac{|S^{d-1}|}{2d(d+2)}\big[\Delta ((f(y)-f(x_k))^2P(y))|_{y=x_k}\big] \epsilon^{-d+2}+O( \epsilon^{-d+3})\label{Appendix:Proof:Lemma:Variance:F}.
\end{align} 
To simplify the discussion, we assume that $\Delta ((f(y)-f(x_k))^2P(y))|_{y=x_k} \neq 0$ so that $\sigma_1^2  = O(\epsilon^{-d+2})$ when $\epsilon$ is small enough. In the case that $\Delta ((f(y)-f(x_k))^2P(y))|_{y=x_k}=0$, the variance is of higher order, and the proof is the same. 

With the above bounds, we could apply the large deviation theory. First, note that the random variable $F_{1}$ is uniformly bounded by 
\begin{equation}
c_1=2\|f\|_{L^\infty}\epsilon^{-d}\nonumber
\end{equation}
and 
\begin{equation}
\sigma_1^2/c_1\to 0\mbox{ as }\epsilon\to 0,\nonumber
\end{equation}
so we apply Bernstein's inequality to provide a large deviation bound. Recall Bernstein's inequality
\begin{equation}
\Pr \left\{\frac{1}{n-1}\sum_{j\neq k,j=1}^n (F_{1,j} - \mathbb{E}[F_{1}]) > \beta_1 \right\} \leq e^{-\frac{n\beta_1^2}{2\sigma_1^2 + \frac{2}{3}c_1\beta_1}},\nonumber
\end{equation}
where $\beta_1>0$.
Since our goal is to estimate a quantity of order $\epsilon^2$, which is the order that the Laplace-Beltrami operator lives, we need to take $\beta_1=\beta_1(\epsilon)$ much smaller than $\epsilon^2$ in the sense that $\beta_1/\epsilon^2\to 0$ as $\epsilon\to 0$. In this case, $c_1\beta_1$ is much smaller than $\sigma_1^2$, and hence $2\sigma_1^2 + \frac{2}{3}c_1\beta_1\leq 3\sigma_1^2$ when $\epsilon$ is smaller enough. Thus, when $\epsilon$ is smaller enough, the exponent in Bernstein's inequality is bounded from below by
\begin{equation}
\frac{n\beta_1^2}{2\sigma_1^2 + \frac{2}{3}c_1\beta_1} \geq \frac{n \beta_1^2}{3\sigma_1^2} \geq  \frac{n\beta_1^2\epsilon^{d-2}}{3\frac{|S^{d-1}|}{d(d+2)}\big[\Delta ((f(y)-f(x_k))^2P(y))|_{y=x_k}\big] }\,.\nonumber
\end{equation}
Suppose $n$ is chosen large enough so that
\begin{equation}
 \frac{n\beta_1^2\epsilon^{d-2}}{3\frac{|S^{d-1}|}{d(d+2)}\big[\Delta ((f(y)-f(x_k))^2P(y))|_{y=x_k}\big]}= 3\log (n)\,;\nonumber
\end{equation}
that is, the deviation from the mean is set to
\begin{align}\label{proof:alphaChoice}
\beta_1 = \frac{3\sqrt{\log (n)}\sqrt{ \frac{|S^{d-1}|}{d(d+2)}\big[\Delta ((f(y)-f(x_k))^2P(y))|_{y=x_k}\big] }}{n^{1/2}\epsilon^{d/2-1}}=O\Big(\frac{\sqrt{\log (n)}}{n^{1/2}\epsilon^{d/2-1}}\Big)\,,
\end{align}
where the implied constant in $O\Big(\frac{\sqrt{\log (n)}}{n^{1/2}\epsilon^{d/2-1}}\Big)$ is $\sqrt{ \frac{|S^{d-1}|}{d(d+2)}\big[\Delta ((f(y)-f(x_k))^2P(y))|_{y=x_k}\big] }$.
Note that by the assumption that $\epsilon=\epsilon(n)$ so that $\frac{\sqrt{\log (n)}}{n^{1/2}\epsilon^{d/2+1}}\to 0$ as $\epsilon\to 0$, we know that $\beta_1/\epsilon^2= \frac{\sqrt{\log (n)}}{n^{1/2}\epsilon^{d/2+1}}\to 0$.
It implies that the deviation greater than $\beta_1$ happens with probability less than 
\begin{align}
\exp\left(-\frac{n\beta_1^2}{2\sigma_1^2 + \frac{2}{3}c_1\beta_1}\right)&\leq \exp\left(-\frac{n\beta_1^2\epsilon^{d-2}}{3\frac{|S^{d-1}|}{d(d+2)}[\Delta ((f(y)-f(x_k))^2P(y))|_{y=x_k}]}\right)\nonumber\\
&= \exp(-3\log (n))= 1/n^3.\nonumber
\end{align} 
As a result, by a simple union bound, we have
\begin{equation}
\Pr \left\{\frac{1}{n-1}\sum_{j\neq k,\,j=1}^n (F_{1,j} - \mathbb{E}[F_1]) > \beta_1\Big|\,k=1,\ldots,n \right\} \leq ne^{-\frac{n\beta_1^2}{2\sigma_1^2 + \frac{2}{3}c_1\beta_1}}\leq 1/n^2.\nonumber
\end{equation}
\end{proof}

Denote $\Omega_1$ to be the event space that the deviation $\frac{1}{n-1}\sum_{j\neq k,\,j=1}^n (F_{1,j} - \mathbb{E}[F_1])\leq \beta_1$ for all $i=1,\ldots,n$, where $\beta_1$ is chosen in (\ref{proof:alphaChoice}) is satisfied. We now proceed to (\ref{Proof:Theorem1:FiniteSum2}).
In this case, we need to discuss different cases indicated by Condition \ref{Condition:1}.

\begin{lemma}\label{Proof:Theorem1:LemmaF3}
Suppose Case 0 in Condition \ref{Condition:1} holds. When $n$ is large enough, we have with probability greater than $1-n^{-2}$ that for all $k=1,\ldots,n$,
\begin{equation}
e_i^{\top}\left[\frac{1}{n\epsilon^d}\sum_{j=1}^N(x_{k,j}-x_k)    - \mathbb{E}\frac{1}{\epsilon^d}(X-x_k)\chi_{{B}^{\mathbb{R}^p}_\epsilon(x_k)}(X)\right] =  O\Big(\frac{\sqrt{\log (n)}}{n^{1/2}\epsilon^{d/2-1}}\Big)\,,\nonumber
\end{equation}
where $i=1,\ldots,d$.
 
Suppose Case 1 in Condition \ref{Condition:1} holds. When $n$ is large enough, we have with probability greater than $1-n^{-2}$ that for all $k=1,\ldots,n$,
\begin{equation}
e_i^{\top}\left[\frac{1}{n\epsilon^d}\sum_{j=1}^N(x_{k,j}-x_k)    - \mathbb{E}\frac{1}{\epsilon^d}(X-x_k)\chi_{{B}^{\mathbb{R}^p}_\epsilon(x_k)}(X)\right] =  O\Big(\frac{\sqrt{\log (n)}}{n^{1/2}\epsilon^{d/2-1}}\Big)\,,\nonumber
\end{equation}
where $i=1,\ldots,d$ and
\begin{equation}
e_i^{\top}\left[\frac{1}{n\epsilon^d}\sum_{j=1}^N(x_{k,j}-x_k)    - \mathbb{E}\frac{1}{\epsilon^d}(X-x_k)\chi_{{B}^{\mathbb{R}^p}_\epsilon(x_k)}(X)\right] =  O\Big(\frac{\sqrt{\log (n)}}{n^{1/2}\epsilon^{d/2-2}}\Big)\,,\nonumber
\end{equation}
where $i=d+1,\ldots,p$.

Suppose Case 2 in Condition \ref{Condition:1} holds. When $n$ is large enough, we have with probability greater than $1-n^{-2}$ that for all $k=1,\ldots,n$,
\begin{equation}
e_i^{\top}\left[\frac{1}{n\epsilon^d}\sum_{j=1}^N(x_{k,j}-x_k)    - \mathbb{E}\frac{1}{\epsilon^d}(X-x_k)\chi_{{B}^{\mathbb{R}^p}_\epsilon(x_k)}(X)\right] =  O\Big(\frac{\sqrt{\log (n)}}{n^{1/2}\epsilon^{d/2-1}}\Big)\,,\nonumber
\end{equation}
where $i=1,\ldots,d$,
\begin{equation}
e_i^{\top}\left[\frac{1}{n\epsilon^d}\sum_{j=1}^N(x_{k,j}-x_k)    - \mathbb{E}\frac{1}{\epsilon^d}(X-x_k)\chi_{{B}^{\mathbb{R}^p}_\epsilon(x_k)}(X)\right] =  O\Big(\frac{\sqrt{\log (n)}}{n^{1/2}\epsilon^{d/2-2}}\Big)\,,\nonumber
\end{equation}
where $i=d+1,\ldots,p-l$, and
\begin{equation}
e_i^{\top}\left[\frac{1}{n\epsilon^d}\sum_{j=1}^N(x_{k,j}-x_k)    - \mathbb{E}\frac{1}{\epsilon^d}(X-x_k)\chi_{{B}^{\mathbb{R}^p}_\epsilon(x_k)}(X)\right] =  O\Big(\frac{\sqrt{\log (n)}}{n^{1/2}\epsilon^{d/2-3}}\Big)\,,\nonumber
\end{equation}
where $i=p-l+1,\ldots,p$
\end{lemma}

\begin{proof}
First, we prove Case 1. Case 0 is a special case of Case 1. Suppose Case 1 holds. Fix $x_k$. 
By denoting 
\begin{equation}
\frac{1}{n\epsilon^d}\sum_{j=1}^N(x_{k,j}-x_k)=\frac{1}{n}\sum_{j\neq k, j=1}^n\sum_{\ell=1}^pF_{2,\ell,j}e_{\ell}.\nonumber
\end{equation}
where
\begin{equation}
F_{2,\ell,j}:=\frac{1}{\epsilon^d}e_{\ell}^\top (x_{j}-x_k)\chi_{{B}^{\mathbb{R}^p}_\epsilon(x_k)}(x_{j}), \nonumber
\end{equation}
we know that when $j\neq k$, $F_{2,\ell,j}$ is randomly sampled i.i.d. from the random variable
\begin{equation}
F_{2,\ell}:=\frac{1}{\epsilon^d}e_\ell^\top (X-x_k)\chi_{{B}^{\mathbb{R}^p}_\epsilon(x_k)}(X).\nonumber
\end{equation}
Similarly, we can focus on analyzing $\frac{1}{n-1}\sum_{j=1,j\neq i}^n F_{2,\ell,j}$ since $\frac{n-1}{n}\to 1$ as $n\to \infty$. 
By plugging $f=1$ in (\ref{Lemma:4:Statement2:withf}), we have
\begin{align}
\mathbb{E}[F_{2,\ell}]=\frac{|S^{d-1}|\epsilon^{2}}{d+2} e_{\ell}^{\top}\big[\!\!\big[\frac{J_{p,d}^{\top}\iota_*\nabla {P}(x)}{d},\, \frac{{P}(x)\bar{J}_{p,p-d}^{\top}\mathfrak{N}_0(x)}{2}\big]\!\!\big]+O(\epsilon^{4})\nonumber
\end{align}
and by (\ref{Proof:LemmaD5:FinalResult}) we have
\begin{align}
\mathbb{E}[F_{2,\ell}^2]  =\,& 
\left\{
\begin{array}{ll}
\displaystyle \frac{|S^{d-1}|  {P}(x)\epsilon^{-d+2}}{d(d+2)} +O(\epsilon^{-d+4})&\mbox{ when }\ell=1,\ldots,d\\
\displaystyle \frac{P(x)\epsilon^{-d+4}}{4(d+4)}\int_{S^{d-1}} |\langle \Second_{x}(\theta,\theta),e_{\ell} \rangle|^2 d\theta +O(\epsilon^{-d+6}) &\mbox{ when }\ell=d+1,\ldots,p.
\end{array}
\right.\nonumber
\end{align}
Thus, we conclude that
\begin{align}
&\sigma_{2,\ell}^2:=\text{Var}(F_{2,\ell}) \nonumber\\
=\,&\left\{
\begin{array}{ll}
\displaystyle \frac{|S^{d-1}|  {P}(x)\epsilon^{-d+2}}{d(d+2)} +O(\epsilon^{-d+4})&\mbox{ when }\ell=1,\ldots,d\\
\displaystyle \frac{P(x)\epsilon^{-d+4}}{4(d+4)}\int_{S^{d-1}} |\langle \Second_{x}(\theta,\theta),e_{\ell} \rangle|^2 d\theta +O(\epsilon^{-d+6}) &\mbox{ when }\ell=d+1,\ldots,p.
\end{array}
\right.\nonumber
\end{align} 
Note that for $\ell=d+1,\ldots,p$, the variance is of higher order than that of $\ell=1,\ldots,d$. 
By the same argument, Case 0 satisfies $\mathbb{E}[F_{2,\ell}]=\sigma_{2,\ell}^2=0$ for $\ell=d+1,\ldots,p$.

With the above bounds, we could apply the large deviation theory. For $\ell=1,\ldots,d$, the random variable $F_{2,\ell}$ is uniformly bounded by 
$c_{2,\ell}=2\epsilon^{-d+1}$
and 
$\sigma_{2,\ell}^2/c_{2,\ell}\to 0\mbox{ as }\epsilon\to 0$,
so when $\epsilon$ is sufficiently smaller and $n$ is sufficiently large, the exponent in Bernstein's inequality, 
\begin{equation}
\Pr \left\{\frac{1}{n-1}\sum_{j\neq k,j=1}^n (F_{2,\ell,j} - \mathbb{E}[F_{2,\ell}]) > \beta_{2,\ell} \right\} \leq \exp\Big(-\frac{n\beta_{2,\ell}^2}{2\sigma_{2,\ell}^2 + \frac{2}{3}c_{2,\ell}\beta_{2,\ell}}\Big),\nonumber
\end{equation}
where $\beta_{2,\ell}>0$, satisfies
\begin{equation}
\frac{n\beta_{2,\ell}^2}{2\sigma_{2,\ell}^2 + \frac{2}{3}c_{2,\ell}\beta_{2,\ell}} \geq \frac{n \beta_{2,\ell}^2}{3\sigma_{2,\ell}^2} \geq  \frac{n\beta_{2,\ell}^2\epsilon^{d-2}}{ 3\frac{|S^{d-1}|  {P}(x)}{d(d+2)}  }= 3\log (n)\,;\nonumber
\end{equation}
that is, the deviation from the mean is set to
\begin{align}\label{proof:alphaChoiceF2}
\beta_{2,\ell} = \frac{3\sqrt{\log (n)}\sqrt{3\frac{|S^{d-1}|  {P}(x)}{d(d+2)}}}{n^{1/2}\epsilon^{d/2-1}}=O\Big(\frac{\sqrt{\log (n)}}{n^{1/2}\epsilon^{d/2-1}}\Big)\,.
\end{align}
For $\ell=d+1,\ldots,p$, since the variance is of higher order, by the same argument, we have
\begin{align}\label{proof:alphaChoiceF2Higher}
\beta_{2,\ell} = \frac{3\sqrt{\log (n)}\sqrt{3\frac{|S^{d-1}|  {P}(x)}{d(d+2)}}}{n^{1/2}\epsilon^{d/2-1}}=O\Big(\frac{\sqrt{\log (n)}}{n^{1/2}\epsilon^{d/2-2}}\Big)\,.
\end{align}

As a result, in both Case 0 and Case 1, 
by a simple union bound, for $\ell=1,\ldots,d$, we have
\begin{align}
\Pr \left\{\left|\frac{1}{n}\sum_{j\neq k,\,j=1}^n F_{2,\ell,j} - \mathbb{E}[F_{2,\ell}]\right| > \beta_{2,\ell}\Big|\,k=1,\ldots,n \right\}  \leq 1/n^2.\nonumber
\end{align}
where
\begin{align}\label{proof:alphaChoiceF2finald}
\beta_{2,\ell} = \frac{3\sqrt{\log (n)}\sqrt{3\frac{|S^{d-1}|  {P}(x)}{d(d+2)}}}{n^{1/2}\epsilon^{d/2-1}}=O\Big(\frac{\sqrt{\log (n)}}{n^{1/2}\epsilon^{d/2-1}}\Big)\,,
\end{align}
and in Case 1, for $\ell=d+1,\ldots,p$, we have
\begin{align}
\Pr \left\{\left|\frac{1}{n}\sum_{j\neq k,\,j=1}^n F_{2,\ell,j} - \mathbb{E}[F_{2,\ell}]\right| > \beta_{2,\ell}\Big|\,k=1,\ldots,n \right\}  \leq 1/n^2.\nonumber
\end{align}
where
\begin{align}\label{proof:alphaChoiceF2finalp}
\beta_{2,\ell} = \frac{3\sqrt{\log (n)}\sqrt{3\frac{P(x)}{4(d+4)}\int_{S^{d-1}} |\langle \Second_{x}(\theta,\theta),e_{\ell} \rangle|^2 d\theta}}{n^{1/2}\epsilon^{d/2-2}}=O\Big(\frac{\sqrt{\log (n)}}{n^{1/2}\epsilon^{d/2-2}}\Big)\,.
\end{align}

For Case 2, by plugging $f=1$ in (\ref{Lemma:4:Statement2:withf}), we have
\begin{align}
\mathbb{E}[F_{2,\ell}]=\left\{
\begin{array}{ll}
\displaystyle\epsilon^{2} e^\top_{\ell}\frac{|S^{d-1}|\iota_*\nabla {P}(x)}{d(d+2)}+O(\epsilon^4)&\mbox{ when }\ell=1,\ldots,d\\
\displaystyle\epsilon^{2} e^\top_{\ell} \frac{|S^{d-1}|{P}(x)\bar{J}_{p,p-d}^{\top}\mathfrak{N}_0(x)}{d(d+2)}+O(\epsilon^4)&\mbox{ when }\ell=d+1,\ldots,p-l \\
\displaystyle\epsilon^{4} e_{\ell}^\top \frac{\mathfrak{R}_1(x) \nabla P(x)  }{6(d+4)} +O(\epsilon^{5})&\mbox{ when }\ell=p-l+1,\ldots,p.
\end{array}
\right.\nonumber
\end{align}
and by (\ref{Proof:LemmaD5:FinalResult}) we have
\begin{align*}
&\mathbb{E}[F_{2,\ell}^2]  \\
=\,& 
\left\{
\begin{array}{ll}
\displaystyle \frac{|S^{d-1}|  {P}(x)\epsilon^{-d+2}}{d(d+2)} +O(\epsilon^{-d+4})&\mbox{ when }\ell=1,\ldots,d\\
\displaystyle \frac{P(x)\epsilon^{-d+4}}{4(d+4)}\int_{S^{d-1}} |\langle \Second_{x}(\theta,\theta),e_{\ell} \rangle|^2 d\theta +O(\epsilon^{-d+6}) &\mbox{ when }\ell=d+1,\ldots,p-l\\
\displaystyle \frac{P(x)\epsilon^{-d+6}}{36(d+6)} \int_{S^{d-1}} \langle \nabla_{\theta}\Second_{x}(\theta,\theta),e_m \rangle \langle \nabla_{\theta}\Second_{x}(\theta,\theta),e_n \rangle d\theta+O(\epsilon^{-d+8}) & \mbox{ when }\ell=p-l+1,\ldots,p\,,
\end{array}
\right.\nonumber
\end{align*}
Thus, we conclude that
\begin{align*}
&\sigma_{2,\ell}^2:=\text{Var}(F_{2,\ell}) \\
=\,&\left\{
\begin{array}{ll}
\displaystyle \frac{|S^{d-1}|  {P}(x)\epsilon^{-d+2}}{d(d+2)} +O(\epsilon^{-d+4})&\mbox{ when }\ell=1,\ldots,d\\
\displaystyle \frac{P(x)\epsilon^{-d+4}}{4(d+4)}\int_{S^{d-1}} |\langle \Second_{x}(\theta,\theta),e_{\ell} \rangle|^2 d\theta +O(\epsilon^{-d+6}) &\mbox{ when }\ell=d+1,\ldots,p-l\\
\displaystyle \frac{P(x)\epsilon^{-d+6}}{36(d+6)} \int_{S^{d-1}} \langle \nabla_{\theta}\Second_{x}(\theta,\theta),e_m \rangle \langle \nabla_{\theta}\Second_{x}(\theta,\theta),e_n \rangle d\theta +O(\epsilon^{-d+8}) &\mbox{ when }\ell=p-l+1,\ldots,p.
\end{array}
\right.\nonumber
\end{align*} 
By the same large deviation argument that we skip the details, we conclude the claim with
\begin{align}\label{proof:alphaChoiceF2finalpQ}
\beta_{2,\ell} = 
\left\{
\begin{array}{ll}
\displaystyle O\Big(\frac{\sqrt{\log (n)}}{n^{1/2}\epsilon^{d/2-1}}\Big)&\mbox{when }\ell=1,\ldots,d\\
\displaystyle O\Big(\frac{\sqrt{\log (n)}}{n^{1/2}\epsilon^{d/2-2}}\Big)&\mbox{when }\ell=d+1,\ldots,p-l\\
\displaystyle O\Big(\frac{\sqrt{\log (n)}}{n^{1/2}\epsilon^{d/2-3}}\Big)&\mbox{when }\ell=p-l+1,\ldots,p.
\end{array}
\right.
\end{align}
\end{proof}

Denote $\Omega_2$ to be the event space that the deviation $\left|\frac{1}{n}\sum_{j\neq k,\,j=1}^n F_{2,\ell,j} - \mathbb{E}[F_{2,\ell}]\right| \leq \beta_{2,\ell}$ for all $\ell=1,\ldots,p$ and $k=1,\ldots,n$, where $\beta_{2,\ell}$ are chosen in (\ref{proof:alphaChoiceF2finald}) under Case 0 in Condition \ref{Condition:1}, (\ref{proof:alphaChoiceF2finald}) and (\ref{proof:alphaChoiceF2finalp}) under Case 1, and (\ref{proof:alphaChoiceF2finalpQ}) under Case 2.

Denote the eigen-decomposition of $\frac{1}{n\epsilon^{d}}G_nG_n^{\top}$ as $U_n\bar\Lambda_n U_n^{\top}$, where $U_n\in O(p)$ and $\bar\Lambda_n\in\mathbb{R}^{p\times p}$ a diagonal matrix, and the eigen-decomposition of $\frac{1}{\epsilon^{d}}C_x$ as $U\bar\Lambda U^{\top}$, where $U\in O(p)$ and $\bar\Lambda\in\mathbb{R}^{p\times p}$ a diagonal matrix. Note that 
\[
n\epsilon^d \mathcal{I}_{n\epsilon^{d+\rho}}(G_nG_n^{\top})=\mathcal{I}_{\epsilon^{\rho}}(\frac{1}{n\epsilon^d}G_nG_n^\top)\,.
\] 

We first control $\mathcal{I}_{\epsilon^{\rho}}(\bar\Lambda_n)-\mathcal{I}_{\epsilon^{\rho}}(\bar\Lambda)=I_{p,r_n}(\bar\Lambda_n+\epsilon^{\rho})^{-1}I_{p,r_n}-I_{p,r}(\bar\Lambda+\epsilon^{\rho})^{-1}I_{p,r}$ based on the three cases listed in Condition \ref{Condition:1}. By Proposition \ref{Proposition:1}, the first $d$ eigenvalues of $\mathbb{E}F$ are of order $\epsilon^2$. In Case 0, all the remaining eigenvalues are 0; in Case 1, all the remaining eigenvalues are nonzero and of order $\epsilon^4$; in Case 2, there are $l$ nonzero eigenvalues of order $\epsilon^6$ and $p-d-l$ remaining eigenvalues of order $\epsilon^4$.

\begin{lemma}\label{Proof:Theorem1:LemmaF4}
When $n$ is large enough, with probability greater than $1-n^{-2}$, for Case 0 in Condition \ref{Condition:1}, we have
\begin{align}
\big|e_i^\top [\mathcal{I}_{\epsilon^{\rho}}(\bar\Lambda_n)-\mathcal{I}_{\epsilon^{\rho}}(\bar\Lambda)]e_i\big|=O\big(\frac{\sqrt{\log(n)}}{n^{1/2}\epsilon^{d/2-2+2(2\wedge \rho)}}\big)\nonumber
\end{align}
for $i=1,\ldots,d$; for Case 1 in Condition \ref{Condition:1}, we have
\begin{align}
&\big|e_i^\top [\mathcal{I}_{\epsilon^{\rho}}(\bar\Lambda_n)-\mathcal{I}_{\epsilon^{\rho}}(\bar\Lambda)]e_i\big| 
= \left\{
\begin{array}{ll}
\displaystyle O\big(\frac{\sqrt{\log(n)}}{n^{1/2}\epsilon^{d/2-2+2(2\wedge \rho)}}\big)&\mbox{for }i=1,\ldots,d\\
\displaystyle O\big(\frac{\sqrt{\log(n)}}{n^{1/2}\epsilon^{d/2-4+2(4\wedge \rho)}}\big)&\mbox{for }i=d+1,\ldots,p\,;
\end{array}
\right.\nonumber
\end{align} 
for Case 2 in Condition \ref{Condition:1}, we have
\begin{align}
&\big|e_i^\top [\mathcal{I}_{\epsilon^{\rho}}(\bar\Lambda_n)-\mathcal{I}_{\epsilon^{\rho}}(\bar\Lambda)]e_i\big| =\left\{
\begin{array}{ll}
\displaystyle O\big(\frac{\sqrt{\log(n)}}{n^{1/2}\epsilon^{d/2-2+2(2\wedge \rho)}}\big)&\mbox{for }i=1,\ldots,d\\
\displaystyle O\big(\frac{\sqrt{\log(n)}}{n^{1/2}\epsilon^{d/2-4+2(4\wedge \rho)}}\big)&\mbox{for }i=d+1,\ldots,p-l\\
\displaystyle O\big(\frac{\sqrt{\log(n)}}{n^{1/2}\epsilon^{d/2-6+2(6\wedge \rho)}}\big)&\mbox{for }i=p-l+1,\ldots,p\,.
\end{array}
\right.\nonumber
\end{align} 
Moreover, for each case in Condition \ref{Condition:1}, when $n$ is sufficiently large, with probability greater than $1-n^{-2}$, we have {$U_n=U\Theta+\frac{\sqrt{\log(n)}}{n^{1/2}\epsilon^{d/2-2}}U\Theta\mathsf{S}+O\big(\frac{\log(n)}{n\epsilon^{d-4}}\big)$}, where $\mathsf{S}\in\mathfrak{o}(p)$, {and $\Theta \in O(p)$. $\Theta$ commutes with $\mathcal{I}_{\epsilon^{\rho}}(\bar\Lambda)$.}
\end{lemma}

Note that $\frac{\log(n)}{n\epsilon^{d-4}}$ is asymptotically bounded by $\epsilon^6$ due to the assumption that $\frac{\sqrt{\log(n)}}{n^{1/2}\epsilon^{d/2-1}}$ is asymptotically approaching zero as $n\to\infty$.

\begin{proof}
We start from analyzing $\frac{1}{n\epsilon^{d}}G_nG_n^{\top}$. The proof can be found in \cite[(6.12)-(6.19)]{Singer_Wu:2016}, and here we summarize the results with our notations. Denote
\begin{equation}
F_{3,a,b,i}:=\frac{1}{\epsilon^d}e_a^\top (x_{k,i}-x_k)(x_{k,i}-x_k)^{\top}e_b\nonumber
\end{equation}
so that
\begin{equation}
\frac{1}{n\epsilon^{d}}G_nG_n^{\top}=\frac{1}{n}\sum_{a,b=1}^p\sum_{i=1}^NF_{3,a,b,i}e_ae_b^{\top}.\nonumber
\end{equation}
Note that for each $a,b=1,\ldots,p$, $\{F_{3,a,b,i}\}_{i=1}^n$ are i.i.d. realizations of the random variable $F_{3,a,b}=\frac{1}{\epsilon^d}e_a^\top (X-x_k)(X-x_k)^{\top}e_b\chi_{{B}^{\mathbb{R}^p}_\epsilon(x_k)}(X)$. Denote $F_3\in\mathbb{R}^{p\times p}$ so that the $(a,b)$-th entry of $F_3$ is $F_{3,a,b}$. Note that $C_x=\epsilon^d\mathbb{E}F_{3}$.

The random variable $F_{3,a,b}$ is bounded by $c_{3,a,b}=2\epsilon^{-d+2}$ when $a,b=1,\ldots,d$, by $c_{3,a,b}=\mathfrak{c}_{a,b}\epsilon^{-d+4}$ when $a,b=d+1,\ldots,p$, and by $c_{3,a,b}=\mathfrak{c}_{a,b}\epsilon^{-d+3}$ for other pairs of $a,b$, where $\mathfrak{c}_{a,b}$, when $a>d$ or $b>d$, are constants depending on the second fundamental form \cite[(B.33)-(B.34)]{Singer_Wu:2012}. 

The variance of $F_{3,a,b}$, denoted as $\sigma^2_{3,a,b}$, is $\mathfrak{s}_{a,b}\epsilon^{-d+4}$ when $a,b=1,\ldots,d$, $\mathfrak{s}_{a,b}\epsilon^{-d+8}$ when $a,b=d+1,\ldots,p$, and $\mathfrak{s}_{a,b}\epsilon^{-d+6}$ for other pairs of $a,b$ (see \cite[(B.33)-(B.35)]{Singer_Wu:2012} or \cite{Singer_Wu:2016}), where $\mathfrak{s}_{a,b}$ are constants depending on the second fundamental form. Again, to simplify the discussion, we assume that $\mathfrak{c}_{a,b}$ and $\mathfrak{s}_{a,b}$ are not zero for all $a,b=1,\ldots,p$. When the variance is of higher order, the deviation could be evaluated similarly and we skip the details.\footnote{For example, when the manifold is flat around $x_k$ and $\epsilon$ is sufficiently small, $\mathfrak{c}_{a,b}=\mathfrak{s}_{a,b}=0$ when $a>d$ or $b>d$, and the proof of the bound is trivial.} 
Thus, for $\beta_{3,1},\beta_{3,2},\beta_{3,3}>0$, by Berstein's inequality, we have
\begin{align}
\mbox{Pr}\left\{\left|\frac{1}{n}\sum_{i\neq k,\,i=1}^n F_{3,a,b,i}-\mathbb{E} F_{3,a,b}\right|>\beta_{3,1}\right\} \leq\exp\left\{-\frac{(n-1)\beta_{3,1}^2}{\mathfrak{s}_{a,b}\epsilon^{-d+4}+\mathfrak{c}_{a,b}\epsilon^{-d+2}\beta_{3,1}}\right\}\label{Proof:Theorem1:CovConv:1}
\end{align}
when $a,b=1,\ldots,d$,
\begin{align}
\mbox{Pr}\left\{\left|\frac{1}{n}\sum_{i\neq k,\,i=1}^n F_{3,a,b,i}-\mathbb{E} F_{3,a,b}\right|>\beta_{3,2}\right\}\leq \exp\left\{-\frac{(n-1)\beta_{3,2}^2}{\mathfrak{s}_{a,b}\epsilon^{-d+8}+\mathfrak{c}_{a,b}\epsilon^{-d+4}\beta_{3,2}}\right\}\label{Proof:Theorem1:CovConv:2}
\end{align}
when $a,b=d+1,\ldots,p$, and 
\begin{align}
\mbox{Pr}\left\{\left|\frac{1}{n}\sum_{i\neq k,\,i=1}^n F_{3,a,b,i}-\mathbb{E} F_{3,a,b}\right|>\beta_{3,3}\right\}\leq \exp\left\{-\frac{(n-1)\beta_{3,3}^2}{\mathfrak{s}_{a,b}\epsilon^{-d+6}+\mathfrak{c}_{a,b}\epsilon^{-d+3}\beta_{3,3}}\right\}\label{Proof:Theorem1:CovConv:3}
\end{align}
for the other cases. 

Choose $\beta_{3,1}$, $\beta_{3,2}$ and $\beta_{3,3}$ so that $\beta_{3,1}/\epsilon^{2}\to 0$, $\beta_{3,2}/\epsilon^{4}\to 0$ and $\beta_{3,3}/\epsilon^{3}\to 0$ as $\epsilon\to 0$
so that when $\epsilon$ is sufficiently small,
\begin{align}
&\mathfrak{s}_{a,b}\epsilon^{-d+4}+\mathfrak{c}_{a,b}\epsilon^{-d+2}\beta_{3,1}\leq 2\mathfrak{s}_{a,b}\epsilon^{-d+4}\quad\mbox{for all }k,l=1,\ldots,d\nonumber\\
&\mathfrak{s}_{a,b}\epsilon^{-d+8}+\mathfrak{c}_{a,b}\epsilon^{-d+4}\beta_{3,2}\leq 2\mathfrak{s}_{a,b}\epsilon^{-d+8}\quad\mbox{for all }k,l=d+1,\ldots,p\nonumber\\
&\mathfrak{s}_{a,b}\epsilon^{-d+6}+\mathfrak{c}_{a,b}\epsilon^{-d+3}\beta_{3,3}\leq 2\mathfrak{s}_{a,b}\epsilon^{-d+6}\quad\mbox{for other }k,l\,.\nonumber
\end{align} 
To guarantee that the deviation of (\ref{Proof:Theorem1:CovConv:1}), (respectively (\ref{Proof:Theorem1:CovConv:2}) and (\ref{Proof:Theorem1:CovConv:3})) greater than $\beta_{3,1}$ (respectively $\beta_{3,2}$ and $\beta_{3,3}$) happens with probability less than $\frac{1}{n^3}$, $n$ should satisfy $\frac{n\beta_{3,1}^{2}}{\log(n)}\geq 6\mathfrak{s}_{a,b}\epsilon^{-d+4}$ (respectively $\frac{n\beta_{3,2}^{2}}{\log(n)}\geq 6\mathfrak{s}_{a,b}\epsilon^{-d+8}$ and $\frac{n\beta_{3,3}^{2}}{\log(n)}\geq 6\mathfrak{s}_{a,b}\epsilon^{-d+6}$). 
By setting $\beta_{3,1}=\sqrt{6\mathfrak{s}_{a,b}} \frac{\sqrt{\log(n)}}{n^{1/2}\epsilon^{d/2-2}}$, $\beta_{3,2}=\sqrt{6\mathfrak{s}_{a,b}} \frac{\sqrt{\log(n)}}{n^{1/2}\epsilon^{d/2-4}}$, and $\beta_{3,3}=\sqrt{6\mathfrak{s}_{a,b}} \frac{\sqrt{\log(n)}}{n^{1/2}\epsilon^{d/2-3}}$, the conditions $\beta_{3,1}/\epsilon^{3}\to 0$, $\beta_{3,2}/\epsilon^{5}\to 0$ and $\beta_{3,3}/\epsilon^{4}\to 0$ as $\epsilon\to 0$ hold by the assumed relationship between $n$ and $\epsilon$ and the deviations of (\ref{Proof:Theorem1:CovConv:1}), (\ref{Proof:Theorem1:CovConv:2}) and (\ref{Proof:Theorem1:CovConv:3}) are well controlled by $\beta_{3,1}$, $\beta_{3,2}$ and $\beta_{3,3}$ respectively, with probability greater than $1-n^{-3}$. 
Define the deviation of $\frac{1}{n\epsilon^{d}}G_nG_n^{\top}$ from $\mathbb{E}F_3$ as 
\begin{equation}
E:=\frac{1}{n\epsilon^{d}}G_nG_n^{\top}-\mathbb{E}F_3 \in \mathbb{R}^{p\times p}.\label{Proof:Variance:EandGGandEF3}
\end{equation}
As a result, again by a trivial union bound, with probability greater than $1-n^{-2}$, for all $x_k$, we have
\begin{align}\label{Proof:Theorem1:DeviationF3}
\left\{\begin{array}{ll}
\displaystyle |E_{a,b}|\leq \frac{\mathfrak{c}\sqrt{\log(n)}}{n^{1/2}\epsilon^{d/2-2}}&\quad\mbox{when }a,b=1,\ldots,d\\ 
&\\
\displaystyle |E_{a,b}|\leq \frac{\mathfrak{c}\sqrt{\log(n)}}{n^{1/2}\epsilon^{d/2-4}}&\quad\mbox{when }a,b=d+1,\ldots,p\\
&\\
\displaystyle |E_{a,b}|\leq \frac{\mathfrak{c}\sqrt{\log(n)}}{n^{1/2}\epsilon^{d/2-3}}&\quad\mbox{otherwise},
\end{array}\right.
\end{align}  
where 
\begin{equation}
\mathfrak{c}:=\max_{a,b=1,\ldots,p}\sqrt{6\mathfrak{s}_{a,b}}. \label{Proof:Theorem0:Lemma3:Def:c}
\end{equation}
Denote $\Omega_3$ to be the event space that the deviation (\ref{Proof:Theorem1:DeviationF3}) is satisfied. 
With the above preparation, we now start the proof of Lemma \ref{Proof:Theorem1:LemmaF4} case by case.

\medskip
\underline{Case 0 in Condition \ref{Condition:1}.} Note that both $\mathbb{E}F$ and $\frac{1}{n\epsilon^{d}}G_nG_n^{\top}$ are of rank $r=d$ due to the geometric constraints. By the calculation in Section \ref{Section:Perturbation}, when conditional on $\Omega_3$, (\ref{Proof:Variance:EandGGandEF3}) holds, and the nonzero eigenvalues of $\frac{1}{n\epsilon^{d}}G_nG_n^{\top}$ (there are only $d$ such eigenvalues) are deviated from the nonzero eigenvalues of $\mathbb{E}F_3$ by $O(\frac{\sqrt{\log(n)}}{n^{1/2}\epsilon^{d/2-2}})$, which is smaller than $\epsilon^3$ by the assumed relationship between $n$ and $\epsilon$. 
Thus, $r_n=d$ when $\epsilon$ is sufficiently small, and we have
\begin{equation}
I_{p,r_n}(\bar\Lambda_n+\epsilon^{\rho})^{-1}I_{p,r_n}-I_{p,d}(\bar\Lambda+\epsilon^{\rho})^{-1}I_{p,d}=I_{p,d}[(\bar\Lambda_n+\epsilon^{\rho})^{-1}-(\bar\Lambda+\epsilon^{\rho})^{-1}]I_{p,d}\,.\nonumber
\end{equation} 
Denote the $i$-th eigenvalue of $\mathbb{E}F_3=\frac{1}{\epsilon^d}C_x$ as $\bar{\lambda}_i$, where $i=1,\ldots,d$.
By a direct calculation, we have
\begin{align}
\big|e_i^\top [\mathcal{I}_{\epsilon^{\rho}}(\bar\Lambda_n)-\mathcal{I}_{\epsilon^{\rho}}(\bar\Lambda)]e_i\big|=O\big(\frac{\sqrt{\log(n)}}{n^{1/2}\epsilon^{d/2-2+2(2\wedge \rho)}}\big)\nonumber
\end{align}
for $i=1,\ldots,d$ 
when $\epsilon$ is sufficiently small since we have
\begin{align*}
&\frac{1}{\bar \lambda_i+O(\frac{\sqrt{\log n}}{n^{1/2}\epsilon^{d/2-2}})+\epsilon^\rho}-\frac{1}{\bar\lambda_i+\epsilon^\rho}=\frac{1}{\bar\lambda_i+\epsilon^\rho}\Big(\frac{1}{O(\frac{\sqrt{\log n}}{n^{1/2}\epsilon^{d/2-2}(\bar\lambda_i+\epsilon^\rho)})+1}-1\Big)\nonumber\\
=&\,O\Big(\frac{\sqrt{\log n}}{n^{1/2}\epsilon^{d/2-2}(\bar\lambda_i+\epsilon^\rho)^2}\Big)=O\Big(\frac{\sqrt{\log(n)}}{n^{1/2}\epsilon^{d/2-2+2(2\wedge \rho)}}\Big)
\end{align*}
due to the fact that $\bar{\lambda}_i$ is of order $\epsilon^2$ for $i=1,\ldots,d$, $\bar\lambda_i+\epsilon^\rho=O(\epsilon^{2\wedge \rho})$ and $n^{1/2}\epsilon^{d/2+1}\to \infty$ as $n\to \infty$.
{Suppose there are $1\leq l\leq d$ distinct eigenvalues, and the multiplicity of the $j$-th distinct eigenvalue is $p_j\in \mathbb{N}$. Clearly, $\sum_{i=1}^lp_i=p$.
By the calculation in Section \ref{Section:Perturbation} that we skip the details, when conditional on $\Omega_3$, $U_n=U\Theta+\frac{\sqrt{\log(n)}}{n^{1/2}\epsilon^{d/2-2}}U\Theta\mathsf{S}+O(\frac{\log(n)}{n\epsilon^{d-4}})$, where $\mathsf{S}\in\mathfrak{o}(p)$,
\begin{align} 
\Theta=\begin{bmatrix}
X^{(1)} & 0 & \cdots &0 \\
0 & X^{(2)} & \cdots & 0\\
0 & 0 &  \ddots &0\\
0 &0  & \cdots & X^{(l)}\\
\end{bmatrix}\, \in O(p),\label{Definition:XforEigenvectorsinLemmaProof}
\end{align} 
and $X^{(j)}\in O(p_j)$, $j=1,\ldots,l$, comes from the $j$-th distinct eigenvalue. Note that $\Theta$ commutes with $\bar{\Lambda}$ and $\mathcal{I}_{\epsilon^{\rho}}(\bar\Lambda)$.}

\medskip
\underline{Case 1 in Condition \ref{Condition:1}.} By the calculation in Section \ref{Section:Perturbation}, when conditional on $\Omega_3$, the first $d$ eigenvalues of $\frac{1}{n\epsilon^{d}}G_nG_n^{\top}$ are deviated from the first $d$ eigenvalues of $\mathbb{E}F$ by $O(\frac{\mathfrak{c}\sqrt{\log(n)}}{n^{1/2}\epsilon^{d/2-2}})$, which is smaller than $\epsilon^3$, and the left $p-d$ eigenvalues of $\frac{1}{n\epsilon^{d}}G_nG_n^{\top}$ are deviated from the left $p-d$ eigenvalues of $\mathbb{E}F$ by $O(\frac{\mathfrak{c}\sqrt{\log(n)}}{n^{1/2}\epsilon^{d/2-4}})$, which is smaller than $\epsilon^5$. 
Thus, again, when $\epsilon$ is sufficiently small, $r_n=r=p$, and $I_{p,r_n}(\bar\Lambda_n+\epsilon^{\rho})^{-1}I_{p,r_n}-I_{p,r}(\bar\Lambda+\epsilon^{\rho})^{-1}I_{p,r}=[(\bar\Lambda_n+\epsilon^{\rho})^{-1}-(\bar\Lambda+\epsilon^{\rho})^{-1}]$. Therefore, 
\begin{align}
&\big|e_i^\top [\mathcal{I}_{\epsilon^{\rho}}(\bar\Lambda_n)-\mathcal{I}_{\epsilon^{\rho}}(\bar\Lambda)]e_i\big| 
= \left\{
\begin{array}{ll}
\displaystyle O\big(\frac{\sqrt{\log(n)}}{n^{1/2}\epsilon^{d/2-2+2(2\wedge \rho)}}\big)&\mbox{ for }i=1,\ldots,d\\
\displaystyle O\big(\frac{\sqrt{\log(n)}}{n^{1/2}\epsilon^{d/2-4+2(4\wedge \rho)}}\big)&\mbox{ for }i=d+1,\ldots,p
\end{array}
\right.\nonumber
\end{align} 
when $\epsilon$ is sufficiently small. Again, by the calculation in Section \ref{Section:Perturbation}, when conditional on $\Omega_3$, we have {$U_n=U\Theta+\frac{\sqrt{\log(n)}}{n^{1/2}\epsilon^{d/2-2}}U\Theta\mathsf{S}+O(\frac{\log(n)}{n\epsilon^{d-4}})$, where $\mathsf{S}\in\mathfrak{o}(p)$ and $\Theta \in O(p)$ is defined in (\ref{Definition:XforEigenvectorsinLemmaProof}).}

\medskip
\underline{Case 2 in Condition \ref{Condition:1}.} A similar discussion holds. In this case, when conditional on $\Omega_3$, we have
\begin{align}
&\big|e_i^\top [\mathcal{I}_{\epsilon^{\rho}}(\bar\Lambda_n)-\mathcal{I}_{\epsilon^{\rho}}(\bar\Lambda)]e_i\big| =\left\{
\begin{array}{ll}
\displaystyle O\big(\frac{\sqrt{\log(n)}}{n^{1/2}\epsilon^{d/2-2+2(2\wedge \rho)}}\big)&\mbox{for }i=1,\ldots,d\\
\displaystyle O\big(\frac{\sqrt{\log(n)}}{n^{1/2}\epsilon^{d/2-4+2(4\wedge \rho)}}\big)&\mbox{for }i=d+1,\ldots,p-l\\
\displaystyle O\big(\frac{\sqrt{\log(n)}}{n^{1/2}\epsilon^{d/2-6+2(6\wedge \rho)}}\big)&\mbox{for }i=p-l+1,\ldots,p
\end{array}
\right.\nonumber
\end{align} 
when $\epsilon$ is sufficiently small. Similarly, when conditional on $\Omega_3$, {$U_n=U\Theta+\frac{\sqrt{\log(n)}}{n^{1/2}\epsilon^{d/2-2}}U\Theta\mathsf{S}+O(\frac{\log(n)}{n\epsilon^{d-4}})$, where $\mathsf{S}\in\mathfrak{o}(p)$ and $\Theta \in O(p)$ is defined in (\ref{Definition:XforEigenvectorsinLemmaProof}).}

\end{proof}

\medskip\textbf{Back to finish the proof of Theorem \ref{Theorem:t0}.}
Denote $\Omega:=\cap_{i=1,\ldots,4} \Omega_i$. It is clear that the probability of the event space $\Omega$ is great than $1-4n^{-2}$. Below, all arguments are conditional on $\Omega$. When $\epsilon$ is sufficiently small, based on Lemma (\ref{Proof:Theorem1:LemmaF4}), we have
\begin{align}
\mathcal{I}_{\epsilon^{\rho}}(\frac{1}{n\epsilon^{d}}G_nG_n^{\top})= &\,\mathcal{I}_{\epsilon^{\rho}}(\mathbb{E}F)+\bar{\mathcal{E}}_3\,,
\end{align}
where {
\begin{align*}
\bar{\mathcal{E}}_3:=&\, U_n \mathcal{I}_{\epsilon^{\rho}}(\bar\Lambda_n)U_n^\top-U\mathcal{I}_{\epsilon^\rho}(\bar\Lambda_n)U^\top \\
=&\, \Big(U\Theta+\frac{\sqrt{\log(n)}}{n^{1/2}\epsilon^{d/2-2}}U\Theta\mathsf{S}+O(\frac{\log(n)}{n\epsilon^{d-4}})\Big)[\mathcal{I}_{\epsilon^{\rho}}(\bar\Lambda)+\mathcal{E}_{3,1}] \nonumber\\
&\qquad\times\Big(U\Theta+\frac{\sqrt{\log(n)}}{n^{1/2}\epsilon^{d/2-2}}U\Theta\mathsf{S}+O(\frac{\log(n)}{n\epsilon^{d-4}})\Big)^\top - U\mathcal{I}_{\epsilon^{\rho}}(\bar\Lambda)U^\top. \nonumber \\
=&\, \frac{\sqrt{\log(n)}}{n^{1/2}\epsilon^{d/2-2}} U\Theta[S\mathcal{I}_{\epsilon^{\rho}}(\bar\Lambda)+\mathcal{I}_{\epsilon^{\rho}}(\bar\Lambda)S^\top]\Theta^\top U^\top +U\Theta\mathcal{E}_{3,1}\Theta^\top U^\top+\big[\mbox{higher order terms}\big]. \nonumber 
\end{align*}
and $\mathcal{E}_{3,1}:=\mathcal{I}_{\epsilon^{\rho}}(\bar\Lambda_n)-\mathcal{I}_{\epsilon^{\rho}}(\bar\Lambda)$, which bound is provided in Lemma (\ref{Proof:Theorem1:LemmaF4}).
Define 
\begin{equation}\label{Proof:E3Part}
\mathcal{E}_3:=\frac{\sqrt{\log(n)}}{n^{1/2}\epsilon^{d/2-2}} U\Theta[S\mathcal{I}_{\epsilon^{\rho}}(\bar\Lambda)+\mathcal{I}_{\epsilon^{\rho}}(\bar\Lambda)S^\top]\Theta^\top U^\top +U\Theta\mathcal{E}_{3,1}\Theta^\top U^\top.
\end{equation}}

By (\ref{Proof:Theorem1:LemmaF3}), we have
\begin{align}
\frac{1}{n\epsilon^{d}}\sum_{j=1}^N(x_{k,j}-x_k)=\mathbb{E}F_2+\mathcal{E}_2,
\end{align}
where the bound of $\mathcal{E}_2$ is provided in Lemma \ref{Proof:Theorem1:LemmaF3}. Similarly, we have
\begin{align}
\frac{1}{n\epsilon^{d}}\sum_{j=1}^N(x_{k,j}-x_k)(f(x_{k,j})-f(x_k))=\mathbb{E}F_4+\mathcal{E}_4\,,
\end{align}
where the bound of $\mathcal{E}_4$ is the same as that in Lemma \ref{Proof:Theorem1:LemmaF3}.

We could therefore recast $[\frac{1}{n\epsilon^{d}}\sum_{j=1}^N(x_{k,j}-x_k)]^{\top}\mathcal{I}_{\epsilon^\rho}(\frac{1}{n\epsilon^{d}}G_nG_n^{\top})[\frac{1}{n\epsilon^{d}}\sum_{j=1}^N(x_{k,j}-x_k)(f(x_{k,j})-f(x_k))]$ as 
\begin{align*}
&[\mathbb{E}F_2+\mathcal{E}_2]^\top[\mathcal{I}_{\epsilon^\rho}(\mathbb{E}F_3)+\bar{\mathcal{E}}_3][\mathbb{E}F_4+\mathcal{E}_4]\\
=\,&\mathbb{E}F_2^\top \mathcal{I}_{\epsilon^{\rho}}(\mathbb{E}F_3)\mathbb{E}F_4+\big[\mathcal{E}_2^\top\mathcal{I}_{\epsilon^\rho}(\mathbb{E}F_3)\mathbb{E}F_4+\mathbb{E}F_2^\top \mathcal{I}_{\epsilon^{\rho}}(\mathbb{E}F_3)\mathcal{E}_4+\mathbb{E}F_2^\top \mathcal{E}_3\mathbb{E}F_4\big]\nonumber\\
&+\big[\mbox{higher order terms}\big].\nonumber
\end{align*}
We now control the error term $\mathcal{E}_2^\top\mathcal{I}_{\epsilon^\rho}(\mathbb{E}F_3)\mathbb{E}F_4+\mathbb{E}F_2^\top \mathcal{I}_{\epsilon^{\rho}}(\mathbb{E}F_3)\mathcal{E}_4+\mathbb{E}F_2^\top \mathcal{E}_3\mathbb{E}F_4$, which depends on the tangential and normal components. Since the errors are of different orders in the tangential and normal directions, we should evaluate the total error separately.

To avoid tedious description of each Case, we summarize the main order of each term for different Cases in Table \ref{Table:Proof1}.
We mention that in Case 2, if the N1 part is zero; that is, the non-trivial eigenvalues corresponding to the normal bundle are all of order $\epsilon^{d+6}$, the final error rate is $\frac{\sqrt{\log(n)}}{n^{1/2}\epsilon^{d/2-1}}$, which is the same as Case 0.

We only carry out the calculation for Case 1, and skip the details for the other cases since the calculation is the same. By checking the error order in Table \ref{Table:Proof1}, the leading order error term of $\mathcal{E}_2^\top\mathcal{I}_{\epsilon^\rho}(\mathbb{E}F_3)\mathbb{E}F_4+\mathbb{E}F_2^\top \mathcal{I}_{\epsilon^\rho}(\mathbb{E}F_3)\mathcal{E}_4$ is controlled by $\frac{\sqrt{\log(n)}}{n^{1/2}\epsilon^{d/2+(2\wedge \rho)-3}}+\frac{\sqrt{\log(n)}}{n^{1/2}\epsilon^{d/2+(4\wedge \rho)-4}}$, where $\frac{\sqrt{\log(n)}}{n^{1/2}\epsilon^{d/2+(2\wedge \rho)-3}}$ comes from the tangential part, and $\frac{\sqrt{\log(n)}}{n^{1/2}\epsilon^{d/2+(4\wedge \rho)-4}}$ comes from the normal part. Note that the sizes of $(2\wedge \rho)-3$ and $(4\wedge \rho)-4$ depend on the chosen $\rho$, so we keep both. 
On the other hand, {by (\ref{Proof:E3Part}) and Table \ref{Table:Proof1}, the error 
$\mathbb{E}F_2^\top \mathcal{E}_3 \mathbb{E}F_4$ is controlled by 
$\frac{\sqrt{\log(n)}}{n^{1/2}\epsilon^{d/2+(4\wedge2 \rho)-6}}+\frac{\sqrt{\log(n)}}{n^{1/2}\epsilon^{d/2+(8 \wedge 2\rho)-8}}$.} By a direct comparison, it is clear that when $\epsilon$ is sufficiently small, no matter which $\rho$ is chosen, $\mathbb{E}F_2^\top \mathcal{E}_{3} \mathbb{E}F_4$ is dominated by $\mathcal{E}_2^\top\mathcal{I}_{\epsilon^\rho}(\mathbb{E}F_3)\mathbb{E}F_4+\mathbb{E}F_2^\top \mathcal{I}_{\epsilon^\rho}(\mathbb{E}F_3)\mathcal{E}_4$, and hence the total error term is controlled by $\frac{\sqrt{\log(n)}}{n^{1/2}\epsilon^{d/2+(2\wedge \rho)-3}}+\frac{\sqrt{\log(n)}}{n^{1/2}\epsilon^{d/2+(4\wedge \rho)-4}}$.

\begin{table}
\centering
\caption{The relevant items in each error term in $\mathcal{E}_2^\top\mathcal{I}_{\epsilon^\rho}(\mathbb{E}F_3)\mathbb{E}F_4+\mathbb{E}F_2^\top \mathcal{I}_{\epsilon^\rho}(\mathbb{E}F_3)\mathcal{E}_4+\mathbb{E}F_2^\top \mathcal{E}_3\mathbb{E}F_4$. The bounds are for entrywise errors. T means the tangential components in all Cases, N means the normal components in Case 1, and N1 means the first $p-d-l$ normal components of order $\epsilon^{d+4}$, and N2 means the last $l$ normal components of order $\epsilon^{d+6}$ in Case 2. ``Total'' means the overall bound of $\mathcal{E}_2^\top\mathcal{I}_{\epsilon^\rho}(\mathbb{E}F_3)\mathbb{E}F_4+\mathbb{E}F_2^\top \mathcal{I}_{\epsilon^\rho}(\mathbb{E}F_3)\mathcal{E}_4+\mathbb{E}F_2^\top \mathcal{E}_3\mathbb{E}F_4$, where only the major terms depending on $n$ and $\epsilon$ in the leading order terms are shown.}
\begin{tabular}{|c|c|c|c|c|c|c|}
\hline
& Case 0 & \multicolumn{2}{|c|}{Case 1}   & \multicolumn{3}{|c|}{Case 2} \\
\hline\hline
& T & T & N & T & N1 & N2\\
\hline 
& \\[\dimexpr-\normalbaselineskip+1pt]
$\mathbb{E}F_2$ & $\epsilon^2$ & $\epsilon^2 $& $\epsilon^2$  & $\epsilon^2$ & $\epsilon^2$  &  $\epsilon^4$  \\[1pt]
\hline
& \\[\dimexpr-\normalbaselineskip+1pt]
$\mathcal{I}_{\epsilon^\rho}(\mathbb{E}F_3)$ & $\epsilon^{-(2\wedge \rho)}$ & $\epsilon^{-(2\wedge \rho)}$ & $\epsilon^{-(4\wedge \rho)}$  & $\epsilon^{-(2\wedge \rho)}$ & $\epsilon^{-(4\wedge \rho)}$ &  $\epsilon^{-(6\wedge \rho)}$ \\[1pt]
\hline& \\[\dimexpr-\normalbaselineskip+1pt]
$\mathbb{E}F_4$ & $\epsilon^2$ & $\epsilon^2 $& $\epsilon^2$  & $\epsilon^2$ & $\epsilon^2$  &  $\epsilon^4$  \\[1pt]
\hline& \\[\dimexpr-\normalbaselineskip+1pt]
$\mathcal{E}_2$ & $\frac{\sqrt{\log(n)}}{n^{1/2}\epsilon^{d/2-1}}$ & $\frac{\sqrt{\log(n)}}{n^{1/2}\epsilon^{d/2-1}}$ & $\frac{\sqrt{\log(n)}}{n^{1/2}\epsilon^{d/2-2}}$ & $\frac{\sqrt{\log(n)}}{n^{1/2}\epsilon^{d/2-1}}$ & $\frac{\sqrt{\log(n)}}{n^{1/2}\epsilon^{d/2-2}}$ & $\frac{\sqrt{\log(n)}}{n^{1/2}\epsilon^{d/2-3}}$ \\[4pt]
\hline& \\[\dimexpr-\normalbaselineskip+1pt]
$\mathcal{E}_{3,1}$ & $\frac{\sqrt{\log(n)}}{n^{1/2}\epsilon^{d/2-2+2(2\wedge \rho)}}$ & $\frac{\sqrt{\log(n)}}{n^{1/2}\epsilon^{d/2-2+2(2\wedge \rho)}}$ & $\frac{\sqrt{\log(n)}}{n^{1/2}\epsilon^{d/2-4+2(4\wedge \rho)}}$ & $\frac{\sqrt{\log(n)}}{n^{1/2}\epsilon^{d/2-2+2(2\wedge \rho)}}$ & $\frac{\sqrt{\log(n)}}{n^{1/2}\epsilon^{d/2-4+2(4\wedge \rho)}}$ & $\frac{\sqrt{\log(n)}}{n^{1/2}\epsilon^{d/2-6+2(6\wedge \rho)}}$ \\[4pt]
\hline& \\[\dimexpr-\normalbaselineskip+1pt]
$\mathcal{E}_4$ & $\frac{\sqrt{\log(n)}}{n^{1/2}\epsilon^{d/2-1}}$ &$\frac{\sqrt{\log(n)}}{n^{1/2}\epsilon^{d/2-1}}$  & $\frac{\sqrt{\log(n)}}{n^{1/2}\epsilon^{d/2-2}}$ & $\frac{\sqrt{\log(n)}}{n^{1/2}\epsilon^{d/2-1}}$ & $\frac{\sqrt{\log(n)}}{n^{1/2}\epsilon^{d/2-2}}$ & $\frac{\sqrt{\log(n)}}{n^{1/2}\epsilon^{d/2-3}}$ \\[4pt]
\hline\hline& \\[\dimexpr-\normalbaselineskip+1pt]
Total  & $\frac{\sqrt{\log(n)}}{n^{1/2}\epsilon^{d/2+(2\wedge \rho)-3}}$ &  \multicolumn{2}{|c|}{$\frac{\sqrt{\log(n)}}{n^{1/2}\epsilon^{d/2+(2\wedge \rho)-3}}+\frac{\sqrt{\log(n)}}{n^{1/2}\epsilon^{d/2+(4\wedge \rho)-4}}$} &  \multicolumn{3}{|c|}{$\frac{\sqrt{\log(n)}}{n^{1/2}\epsilon^{d/2+(2\wedge \rho)-3}}+\frac{\sqrt{\log(n)}}{n^{1/2}\epsilon^{d/2+(4\wedge \rho)-4}}$}\\[4pt]
\hline
\end{tabular}
\label{Table:Proof1}
\end{table}

Therefore, when conditional on $\Omega$, for all $k=1,\ldots,n$, the deviation of the nominator of (\ref{Proof:Theorem0:FirstDirectExpansion}) from $\mathbb{E}[F_1] - \mathbb{E}F_2^\top \mathcal{I}_{\epsilon^\rho}(\mathbb{E}F_3)\mathbb{E}F_4$ depends on $\rho$ and different Cases in Condition \ref{Condition:1}; for Case 0, it is controlled by $O(\frac{\sqrt{\log (n)}}{n^{1/2}\epsilon^{d/2-1}})+O(\frac{\sqrt{\log(n)}}{n^{1/2}\epsilon^{d/2+(2\wedge \rho)-3}})=O(\frac{\sqrt{\log(n)}}{n^{1/2}\epsilon^{d/2-1}})$ since $(2\wedge \rho)-3=(-1)\wedge (\rho-3)\leq -1$; for Case 1 and Case 2, it is controlled by $O(\frac{\sqrt{\log (n)}}{n^{1/2}\epsilon^{d/2-1}})+O(\frac{\sqrt{\log(n)}}{n^{1/2}\epsilon^{d/2+(2\wedge \rho)-3}})+O(\frac{\sqrt{\log(n)}}{n^{1/2}\epsilon^{d/2+(4\wedge \rho)-4}})=O(\frac{\sqrt{\log(n)}}{n^{1/2}\epsilon^{d/2+[(-1)\vee (0\wedge (\rho-4)]}})$, which comes from the fact that $(2\wedge \rho)-3\leq -1$ and $(4\wedge \rho)-4=0\wedge(\rho-4)$. 
Similarly, the deviation of the denominator of (\ref{Proof:Theorem0:FirstDirectExpansion}) from $\mathbb{E}[F_0] - \mathbb{E}F_2^\top \mathcal{I}_{\epsilon^\rho}(\mathbb{E}F_3)\mathbb{E}F_2$ depends on $\rho$ and different Cases in Condition \ref{Condition:1}; for Case 0, it is controlled by $O(\frac{\sqrt{\log (n)}}{n^{1/2}\epsilon^{d/2}})+O(\frac{\sqrt{\log(n)}}{n^{1/2}\epsilon^{d/2+(2\wedge \rho)-3}})=O(\frac{\sqrt{\log(n)}}{n^{1/2}\epsilon^{d/2}})$ since $(2\wedge \rho)-3\leq -1<0$; for Case 1 and Case 2, it is controlled by $O(\frac{\sqrt{\log (n)}}{n^{1/2}\epsilon^{d/2}})+O(\frac{\sqrt{\log(n)}}{n^{1/2}\epsilon^{d/2+(2\wedge \rho)-3}})+O(\frac{\sqrt{\log(n)}}{n^{1/2}\epsilon^{d/2+(4\wedge \rho)-4}})=O(\frac{\sqrt{\log(n)}}{n^{1/2}\epsilon^{d/2}})$, which comes from the fact that $(4\wedge \rho)-4=0\wedge(\rho-4)\leq 0$.

As a result, when conditional on $\Omega$, for all $k=1,\ldots,n$, we have
\begin{align}
&\sum_{j=1}^N w_k(j)f(x_{k,j})-f(x_k)\\
=\,&
\left\{
\begin{array}{ll}
\displaystyle\frac{\mathbb{E}[F_1] - \mathbb{E}F_2^\top \mathcal{I}_{\epsilon^\rho}(\mathbb{E}F_3)\mathbb{E}F_4+O(\frac{\sqrt{\log (n)}}{n^{1/2}\epsilon^{d/2-1}})}
{\mathbb{E}[F_0] - \mathbb{E}F_2^\top \mathcal{I}_{\epsilon^\rho}(\mathbb{E}F_3)\mathbb{E}F_2+O(\frac{\sqrt{\log (n)}}{n^{1/2}\epsilon^{d/2}})} & \mbox{in Case 0} \\
\vspace{-4pt}&\\
\displaystyle\frac{\mathbb{E}[F_1] - \mathbb{E}F_2^\top \mathcal{I}_{\epsilon^\rho}(\mathbb{E}F_3)\mathbb{E}F_4+O(\frac{\sqrt{\log (n)}}{n^{1/2}\epsilon^{d/2+[(-1)\vee (0\wedge (\rho-4)]}})}
{\mathbb{E}[F_0] - \mathbb{E}F_2^\top \mathcal{I}_{\epsilon^\rho}(\mathbb{E}F_3)\mathbb{E}F_2+O(\frac{\sqrt{\log (n)}}{n^{1/2}\epsilon^{d/2}}) } & \mbox{in Case 1,2} \\
\vspace{-4pt}&
\end{array}
\right.\nonumber
\end{align}
which leads to
\begin{align}
&\sum_{j=1}^N w_k(j)f(x_{k,j})-f(x_k)\\
=\,&
\left\{
\begin{array}{ll}
Qf(x_k)-f(x_k)+O\Big(\frac{\sqrt{\log (n)}}{n^{1/2}\epsilon^{d/2-1}}\Big)&\mbox{in Case 0}\\
Qf(x_k)-f(x_k)+O\Big(\frac{\sqrt{\log (n)}}{n^{1/2}\epsilon^{d/2+[(-1)\vee (0\wedge (\rho-4)]}}\Big)&\mbox{in Case 1,2}
\end{array}
\right.\nonumber
\end{align}
where the equality comes from rewriting (\ref{Proof:Theorem1:FinalExpectationFormulation}) as
\begin{align}
Qf(x_k)-f(x_k)=\frac{\mathbb{E}F_1-\mathbb{E}F_2^\top \mathcal{I}_{\epsilon^\rho}(\mathbb{E}F_3) \mathbb{E}F_4}{\mathbb{E}F_0-\mathbb{E}F_2^\top \mathcal{I}_{\epsilon^\rho}(\mathbb{E}F_3) \mathbb{E}F_2},\nonumber
\end{align}
and the fact that $\mathbb{E}F_0$ is of order 1 and $\mathbb{E}F_1$ is of order $\epsilon^2$.
Hence, we finish the proof.

\section{Proofs of Technical Lemmas}\label{Section:LemmasProof}

To alleviate the notational load, we use the notation introduced in (\ref{Definition:Notation:SmallBallonTangentSpace}).

\subsection{Proof of Lemma \ref{Lemma:3}}

Let $\gamma(t)$ be the geodesic in $\iota(M)$ with $\gamma(0)=\iota(x)$. If $\gamma^{(i)}(0)$ denotes the $i$-th derivative of $\gamma(t)$ with respect to $t$ at $0$, then we have
\begin{align} 
\gamma(t)=&\gamma(0)+\gamma^{(1)}(0)t+\frac{1}{2}\gamma^{(2)}(0)t^2+\frac{1}{6}\gamma^{(3)}(0)t^3\label{expansiongamma}\\
&+\frac{1}{24}\gamma^{(4)}(0)t^4+\frac{1}{120}\gamma^{(5)}(0)t^5+O(t^6).\nonumber
\end{align}
Moreover, since $\gamma(t)$ is a geodesic, if we apply the product rule, we have 
\begin{align} \label{relation gamma}
& \gamma^{(1)}(0) \cdot  \gamma^{(1)} (0)=1, \\
& \gamma^{(2)}(0)  \cdot  \gamma^{(1)}(0)  =0 , \nonumber \\
& \gamma^{(2)}(0)  \cdot  \gamma^{(2)} (0) =-\gamma^{(3)} (0) \cdot  \gamma^{(1)}(0)  ,\nonumber \\
& 3\gamma^{(3)}(0)  \cdot  \gamma^{(2)} (0) =-\gamma^{(4)} (0) \cdot  \gamma^{(1)}(0)   ,\nonumber\\
& 4\gamma^{(4)}(0)  \cdot  \gamma^{(2)}(0) +3\gamma^{(3)} (0) \cdot  \gamma^{(3)} (0)  =-\gamma^{(5)}(0)  \cdot  \gamma^{(1)} (0) \nonumber\,,
\end{align}
where $\gamma^{(l)}$ is the $l$-th derivative of $\gamma$ and $l\in\mathbb{N}$.
From (\ref{expansiongamma}), we have
\begin{align} 
\|\gamma(t)-\gamma(0)\|_{\mathbb{R}^p}^2=&  \gamma^{(1)}(0) \cdot  \gamma^{(1)} (0) t^2+( \gamma^{(2)}(0)  \cdot  \gamma^{(1)}(0) )t^3\label{expansiongammasquare}\\
&+\big(\frac{1}{3}\gamma^{(3)} (0) \cdot  \gamma^{(1)}(0)+ \frac{1}{4}\gamma^{(2)}(0)  \cdot  \gamma^{(2)} (0)\big)t^4\nonumber\\
&+\big(\frac{1}{12}\gamma^{(4)} (0) \cdot  \gamma^{(1)}(0)+\frac{1}{6}\gamma^{(3)}(0)  \cdot  \gamma^{(2)} (0)\big)t^5 \nonumber \\
&\, +\big(\frac{1}{60} \gamma^{(5)}(0)  \cdot  \gamma^{(1)} (0)(0)+\frac{1}{24}\gamma^{(4)}(0)  \cdot  \gamma^{(2)}(0) +\frac{1}{36}\gamma^{(3)} (0) \cdot  \gamma^{(3)} (0) \big)t^6+O(t^7) .\nonumber 
\end{align}
If we substitute (\ref{relation gamma}) into (\ref{expansiongammasquare}), we have 
\begin{align} \label{expansiongammasquare2}
\|\gamma(t)-\gamma(0)\|_{\mathbb{R}^p}^2=& t^2 -\frac{1}{12} \gamma^{(2)}(0)  \cdot  \gamma^{(2)} (0) t^4 -\frac{1}{12} \gamma^{(3)}(0)  \cdot  \gamma^{(2)} (0) t^5 - \\
&\, \big(\frac{1}{40}\gamma^{(4)}(0)  \cdot  \gamma^{(2)}(0) +\frac{1}{45}\gamma^{(3)} (0) \cdot  \gamma^{(3)} \big)t^6+O(t^7).
\nonumber 
\end{align}
Therefore
\begin{align}
\tilde{t} & =\|\gamma(t)-\gamma(0)\|_{\mathbb{R}^p}= t-\frac{1}{24} \gamma^{(2)}(0)  \cdot  \gamma^{(2)} (0) t^3 -\frac{1}{24} \gamma^{(3)}(0)  \cdot  \gamma^{(2)} (0) t^4 \\
&\, -\big(\frac{1}{80}\gamma^{(4)}(0)  \cdot  \gamma^{(2)}(0) +\frac{1}{90}\gamma^{(3)} (0) \cdot  \gamma^{(3)} +\frac{1}{1152} (\gamma^{(2)}(0)  \cdot  \gamma^{(2)} (0))^2\big)t^5+O(t^6). \nonumber
\end{align}
By comparing the order, we have
\begin{align}
t = & \tilde{t} + \frac{1}{24} \gamma^{(2)}(0)  \cdot  \gamma^{(2)} (0) \tilde {t}^3 +\frac{1}{24} \gamma^{(3)}(0)  \cdot  \gamma^{(2)} (0) \tilde{t}^4 +(\frac{1}{80}\gamma^{(4)}(0)  \cdot  \gamma^{(2)}(0) \\
&\, +\frac{1}{90}\gamma^{(3)} (0) \cdot  \gamma^{(3)}+\frac{7}{1152} (\gamma^{(2)}(0)  \cdot  \gamma^{(2)} (0))^2) \tilde{t}^5+O(t^6). \nonumber
\end{align}
Finally, by applying Lemma \ref{Lemma:2} to  (\ref{expansiongamma})  with $\gamma(t)=\iota \circ \exp_x(\theta t)$ and substituting corresponding terms for $\gamma^{(l)}(0)$, the conclusion follows.

\subsection{Proof of Lemma \ref{Lemma:5}}

By Lemma \ref{Lemma:1}, Lemma \ref{Lemma:2} and Lemma \ref{Lemma:3}, 
\begin{align}
&\mathbb{E}[f(X)\chi_{B_{\epsilon}^{\mathbb{R}^p}(\iota(x))}(X)] \nonumber\\
=&\, \int_{\tilde{B}_{\epsilon}(x)} f(y)P(y) dV(y) \nonumber \\
=&\, \int_{S^{d-1}}\int_0^{\tilde{\epsilon}} (f(x) +\nabla_\theta f(x) t+\frac{1}{2} \nabla^2_{\theta,\theta} f(x) t^2+O(t^3)) \nonumber\\
& (P(x)+\nabla_\theta P(x) t+\nabla^2_{\theta,\theta} P(x) t^2 +O(t^3)) ( t^{d-1}- \frac{1}{6}\texttt{Ric}_{x}(\theta,\theta)t^{d+1}+O(t^{d+2})) dtd\theta  \nonumber\\
=&\,A_1+ B_1 +C_1+O(\epsilon^{d+4}),\nonumber
\end{align}
where 
\begin{align}
A_1&:=\int_{S^{d-1}}\int_0^{\tilde{\epsilon}} f(x)P(x) t^{d-1}dtd\theta\\
B_1&:=\int_{S^{d-1}}\int_0^{\tilde{\epsilon}} (\nabla_\theta f(x) P(x)+f(x)\nabla_\theta P(x)) t^d dtd\theta\nonumber\\
C_1&:=\int_{S^{d-1}}\int_0^{\tilde{\epsilon}} \Big[\frac{1}{6}f(x)P(x)\texttt{Ric}_x(\theta,\theta)+\nabla_\theta f(x)\nabla_\theta P(x) \nonumber\\
 &\qquad\qquad + \frac{1}{2}\nabla^2_{\theta,\theta} f(x) P(x)+ \frac{1}{2}\nabla^2_{\theta,\theta} P(x) f(x) \Big]  t^{d+1} dtd\theta\,,\nonumber
 \end{align}
the second equality holds by Lemma \ref{Lemma:3} and the last equality holds due to the symmetry of sphere. Indeed, the symmetry forces all terms of odd order contribute to the $\epsilon^{d+4}$ term; for example,  
\begin{align}
B_1=\,&\int_{S^{d-1}}\int_0^{\tilde{\epsilon}} (\nabla_\theta f(x) P(x)+f(x)\nabla_\theta P(x)) t^d dtd\theta\nonumber\\
=\,&\frac{1}{d+1}\int_{S^{d-1}} (\nabla_\theta f(x) P(x)+f(x)\nabla_\theta P(x))  \Big(\epsilon+\frac{1}{24}\|\Second_x(\theta,\theta)\|^2\epsilon^3+O(\epsilon^4)\Big)^{d+1} d\theta\nonumber\\
=\,&\frac{\epsilon^{d+1}}{d+1}\int_{S^{d-1}} (\nabla_\theta f(x) P(x)+f(x)\nabla_\theta P(x))  \Big(1+\frac{d+1}{24}\|\Second_x(\theta,\theta)\|^2\epsilon^2+O(\epsilon^3)\Big) d\theta\nonumber\\
=\,&O(\epsilon^{d+4})\nonumber
\end{align} 
since $\int_{S^{d-1}} (\nabla_\theta f(x) P(x)+f(x)\nabla_\theta P(x))  \big(1+\frac{d+1}{24}\|\Second_x(\theta,\theta)\|^2\epsilon^2\big)d\theta =0$.
The other even order terms could be expanded by a direct calculation. We have
\begin{align}
A_1=\,&\int_{S^{d-1}}\int_0^{\tilde{\epsilon}} f(x)P(x) t^{d-1}dtd\theta\nonumber\\
=\,&\frac{f(x)P(x)}{d}\int_{S^{d-1}}\big(\epsilon+\frac{1}{24}\|\Second_x(\theta,\theta)\|^2\epsilon^3+O(\epsilon^4)\big)^dd\theta\nonumber\\
=\,&\epsilon^d\frac{f(x)P(x)}{d}\int_{S^{d-1}}\big(1+\frac{d}{24}\|\Second_x(\theta,\theta)\|^2\epsilon^2+O(\epsilon^3)\big)d\theta\nonumber\\
=\,&\epsilon^d|S^{d-1}|f(x)P(x)\Big[\frac{1}{d}+\frac{\omega(x)}{24}\epsilon^{2}\Big]+O(\epsilon^{d+3}).\nonumber
\end{align}
A similar argument holds for $B_1$. By denoting $R_2(\theta):=\frac{1}{6}f(x)P(x)\texttt{Ric}_x(\theta,\theta)+\nabla_\theta f(x)\nabla_\theta P(x) + \frac{1}{2}\nabla^2_{\theta,\theta} f(x) P(x)+ \frac{1}{2}\nabla^2_{\theta,\theta} P(x) f(x)$, we have
\begin{align}
C_1=\,&\int_{S^{d-1}}\int_0^{\tilde{\epsilon}} R_2(\theta) t^{d+1} dtd\theta\nonumber\\
=\,&\frac{1}{d+2}\int_{S^{d-1}}\Big(\epsilon+\frac{1}{24}\|\Second_x(\theta,\theta)\|^2\epsilon^3+O(\epsilon^4)\Big)^{d+2} R_2(\theta) d\theta\nonumber\\
=\,&\frac{\epsilon^{d+2}}{d+2}\int_{S^{d-1}}\Big(1+\frac{d+2}{24}\|\Second_x(\theta,\theta)\|^2\epsilon^2+O(\epsilon^3)\Big) R_2(\theta) d\theta\nonumber\\
=\,&\frac{\epsilon^{d+2}}{d+2}\int_{S^{d-1}} R_2(\theta) d\theta + O(\epsilon^{d+4}).\nonumber
\end{align}
To proceed, note that by expressing $\theta$ in the local coordinate as $\theta^i\partial_i$, we have, for example, 
\begin{align}
&\int_{S^{d-1}}\nabla_\theta f(x)\nabla_\theta P(x)d\theta=\sum_{ij}\int_{S^{d-1}}\partial_if(x)\partial_jP(x)\theta^i\theta^jd\theta\nonumber\\
=\,&\sum_{i}\int_{S^{d-1}}\partial_if(x)\partial_iP(x)(\theta^i)^2d\theta=\frac{|S^{d-1}|}{d}\nabla f(x)\cdot \nabla P(x)\nonumber,
\end{align}
where the second equality holds since odd order terms disappear when integrated over the sphere, and the last equality holds since $\int_{S^{d-1}}(\theta^i)^2d\theta=\frac{1}{d}\int_{S^{d-1}}\sum_{i=1}^d(\theta^i)^2d\theta=\frac{|S^{d-1}|}{d}$ due to again the symmetry of the sphere. The same argument leads to
\begin{align}
&\int_{S^{d-1}}f(x)P(x)\texttt{Ric}_x(\theta,\theta) d\theta=\frac{|S^{d-1}|}{d}f(x)P(x)s(x)\nonumber\\
&\int_{S^{d-1}}\nabla^2_{\theta,\theta} f(x) P(x)d\theta=\frac{|S^{d-1}|}{d}f(x)\Delta P(x)\\
&\int_{S^{d-1}}\nabla^2_{\theta,\theta} P(x) f(x)d\theta=\frac{|S^{d-1}|}{d}P(x)\Delta f(x)\nonumber,
\end{align}
where $s(x)$ is the scalar curvature of $(M,g)$ at $x$.
As a result, we have
\begin{align}
C_1=&\frac{|S^{d-1}|}{d(d+2)}\Big[\frac{1}{2}P(x)\Delta f(x)+\frac{1}{2}f(x)\Delta P(x)\nonumber\\
&+\nabla f(x)\cdot \nabla P(x)+\frac{s(x)f(x)P(x)}{6}\Big]\epsilon^{d+2}+ O(\epsilon^{d+4})\nonumber.
\end{align}
By putting all the above together, we have
\begin{align}
\mathbb{E}&[f(X)\chi_{B_{\epsilon}^{\mathbb{R}^p}(\iota(x))}(X)]\,=\frac{|S^{d-1}|}{d}f(x)P(x)\epsilon^d+\frac{|S^{d-1}|}{d(d+2)}\Big[\frac{1}{2}P(x)\Delta f(x)+\frac{1}{2}f(x)\Delta P(x)\nonumber\\
&+\nabla f(x)\cdot \nabla P(x)+\frac{s(x)f(x)P(x)}{6}+\frac{d(d+2)\omega(x)f(x)P(x)}{24}\Big]\epsilon^{d+2}+O(\epsilon^{d+3}).\nonumber
\end{align}

Next, we evaluate $\mathbb{E}[(X-\iota(x))f(X)\chi_{B_{\epsilon}^{\mathbb{R}^p}(\iota(x))}(X)]$. Again, by Lemma \ref{Lemma:1}, Lemma \ref{Lemma:2} and Lemma \ref{Lemma:3}, we have 
\begin{align}
& \mathbb{E}[(X-\iota(x))f(X)\chi_{B_{\epsilon}^{\mathbb{R}^p}(\iota(x))}(X)] \label{Proof:Lemma4:EquationLabel1} \\
= &\, \int_{\tilde{B}_{\epsilon}(x)} (\iota(y)-\iota(x))f(y)P(y)dV(y)  \nonumber\\
=&\,\int_{S^{d-1}}\int_0^{\tilde{\epsilon}} (\iota_{*}\theta t +\frac{1}{2}\Second_{x}(\theta,\theta)t^2+\frac{1}{6}\nabla_{\theta} \Second_{x}(\theta,\theta) t^3++\frac{1}{24}\nabla_{\theta\theta} \Second_{x}(\theta,\theta) t^4+O(t^5))\nonumber\\
&\qquad\times(f(x) +\nabla_\theta f(x) t+\frac{1}{2} \nabla^2_{\theta,\theta} f(x) t^2+O(t^3)) \nonumber\\
&\qquad\times(P(x)+\nabla_\theta P(x) t+\frac{1}{2} \nabla^2_{\theta,\theta} P(x) t^2 +O(t^3)) \nonumber\\
&\qquad\times( t^{d-1}- \frac{1}{6}\texttt{Ric}_{x}(\theta,\theta)t^{d+1}+O(t^{d+2})) dtd\theta  \nonumber\\
=&\,A_2+B_2+C_2+D_2+O(\epsilon^{d+6}),\nonumber
\end{align}
where
\begin{align}
A_2&:=\int_{S^{d-1}}\int_0^{\tilde{\epsilon}} \iota_*\theta f(x)P(x)t^ddtd\theta\nonumber\\
B_2&:=\int_{S^{d-1}}\int_0^{\tilde{\epsilon}} \big[\iota_*\theta (\nabla_\theta f(x)P(x)+\nabla_\theta P(x)f(x))+ \frac{1}{2}\Second_x(\theta,\theta)f(x)P(x)\big]t^{d+1}dtd\theta,\nonumber\\
C_2&:=\int_{S^{d-1}}\int_0^{\tilde{\epsilon}} \big[\iota_*\theta (\nabla_\theta f(x)\nabla_\theta P(x)+\nabla^2_{\theta,\theta} P(x)f(x)+\nabla^2_{\theta,\theta} f(x)P(x))\nonumber\\
&\qquad\qquad+ \frac{1}{6}\nabla_\theta\Second_x(\theta,\theta)f(x)P(x)-\frac{1}{6}f(x)P(x)\texttt{Ric}_{x}(\theta,\theta)\big]t^{d+2}dtd\theta\nonumber
\end{align}
and
\begin{align}
D_2&:=\int_{S^{d-1}}\int_0^{\tilde{\epsilon}} 
\Big[\iota_*\theta \big(\frac{1}{6}\nabla^3_{\theta,\theta,\theta}f(x)P(x)+\frac{1}{6}\nabla^3_{\theta,\theta,\theta}P(x)f(x)+\frac{1}{2}\nabla^2_{\theta,\theta}f(x)\nabla_\theta P(x)\nonumber\\
&\qquad\qquad\qquad+\frac{1}{2}\nabla^2_{\theta,\theta}P(x)\nabla_\theta f(x)-\frac{1}{6}\texttt{Ric}_{x}(\theta,\theta)[f(x)\nabla_{\theta}P(x)+\nabla_{\theta}f(x)P(x)]\big)
\nonumber\\
&\qquad+\frac{1}{2}\Second_x(\theta,\theta)\big(\nabla_\theta f(x)\nabla_\theta P(x)+\frac{1}{2}[P(x)\nabla^2_{\theta,\theta} f(x)+f(x)\nabla^2_{\theta,\theta} P(x)]-\frac{1}{6}\texttt{Ric}_{x}(\theta,\theta)f(x)P(x)\big)\nonumber\\
&\qquad+\frac{1}{6}\nabla_\theta\Second_x(\theta,\theta)(P(x)\nabla_{\theta} f(x)+f(x)\nabla_{\theta} P(x)) +\frac{1}{24}f(x)P(x)\nabla_{\theta\theta} \Second_{x}(\theta,\theta)\Big]t^{d+3}dtd\theta\nonumber 
\end{align}
and the $O(\epsilon^{d+5})$ term disappears in the last equality due to the symmetry of the sphere.
The main difference between evaluating $\mathbb{E}[(X-\iota(x))f(X)\chi_{B_{\epsilon}^{\mathbb{R}^p}(\iota(x))}(X)]$ and $\mathbb{E}[f(X)\chi_{B_{\epsilon}^{\mathbb{R}^p}(\iota(x))}(X)]$ is the existence of $\iota(y)$ in the integrand in (\ref{Proof:Lemma4:EquationLabel1}). Clearly, $\mathbb{E}[(X-\iota(x))f(X)\chi_{B_{\epsilon}^{\mathbb{R}^p}(\iota(x))}(X)]$ is a vector while $\mathbb{E}[f(X)\chi_{B_{\epsilon}^{\mathbb{R}^p}(\iota(x))}(X)]$ is a scalar. Due to the curvature, $\iota(y)-\iota(x)$ does not always exist on $\iota_*T_xM$ for all $y\in \tilde{B}_\epsilon$, and we need to carefully trace the normal components.
By Lemma \ref{Lemma:4}, 
\begin{align}
A_2 =&\, \int_{S^{d-1}}\int_0^{\tilde{\epsilon}} f(x)P(x)\iota_*\theta t^ddtd\theta\nonumber\\
=\,&\ \frac{ f(x)P(x)}{d+1} \int_{S^{d-1}} \iota_*\theta (\epsilon+ \frac{1}{24} \|\Second_x(\theta,\theta)\|^2 \epsilon^3 +\frac{1}{24} \nabla_{\theta}\Second_x(\theta,\theta)  \cdot  \Second_x(\theta,\theta) \epsilon^4 +O(\epsilon^5))^{d+1}d\theta  \nonumber \\
=\,&\ \frac{ f(x)P(x)}{d+1} \int_{S^{d-1}}  \Big(\iota_*\theta \epsilon^{d+1}+\frac{d+1}{24}\iota_*\theta \|\Second_x(\theta,\theta)\|^2 \epsilon^{d+3}\nonumber\\
&\qquad\qquad\qquad+\frac{d+1}{24} \iota_*\theta \nabla_{\theta}\Second_x(\theta,\theta)  \cdot  \Second_x(\theta,\theta) \epsilon^{d+4}+O(\epsilon^{d+5})\Big) d\theta \nonumber \\
=\,&\ \frac{ f(x)P(x)}{24} \big[ \iota_* \int_{S^{d-1}} \theta \nabla_{\theta}\Second_x(\theta,\theta)  \cdot  \Second_x(\theta,\theta) d\theta \big] \epsilon^{d+4} +O(\epsilon^{d+5}) \nonumber \\
=\,&\ \frac{ f(x)P(x)|S^{d-1}|}{24} \iota_*\mathfrak{R}_0(x) \epsilon^{d+4} +O(\epsilon^{d+5})  \nonumber\,,
\end{align}
where the second last equality holds due to the symmetry of the sphere. We could see that $A_2=O(\epsilon^{d+4})$ and $A_2\in\iota_*T_xM$.  Similarly, we have $C_2=O(\epsilon^{d+6})$, but $C_2$ might not be on $\iota_*T_xM$ due to the term $\nabla_\theta\Second_x(\theta,\theta)f(x)P(x)$.

$B_2$ could be evaluated by a similar direct expansion. 
\begin{align}
B_2 =&\, \int_{S^{d-1}}\int_0^{\tilde{\epsilon}}\big[{P}(x)\iota_{*}\theta (\nabla{f}(x) \cdot \theta)+{f}(x)\iota_{*}\theta (\nabla{P}(x) \cdot \theta)\nonumber\\
&\qquad\qquad+ \frac{1}{2}\Second_x(\theta,\theta)f(x)P(x)\big] t^{d+1}  dtd\theta\nonumber  \\
=&\, \frac{\epsilon^{d+2}}{d+2}  \int_{S^{d-1}}\Big[P(x)\iota_{*}\theta \theta^\top \nabla {f}(x)+f(x)\iota_{*}\theta \theta^\top \nabla {P}(x) + \frac{P(x)}{2}\Second_x(\theta,\theta)f(x)\Big] d\theta  \nonumber \\
&+\frac{\epsilon^{d+4}}{24}  \int_{S^{d-1}}\|\Second_x(\theta,\theta)\|^2\Big[P(x) \iota_{*}\theta \theta^\top \nabla {f}(x)+f(x) \iota_{*}\theta \theta^\top \nabla {P}(x) + \frac{P(x)}{2}\Second_x(\theta,\theta)f(x)\Big] d\theta +O(\epsilon^{d+5} ) ,\nonumber 
\end{align}
which becomes
\begin{align}
&\, \frac{\epsilon^{d+2}}{d+2} \iota_*\int_{S^{d-1}} \theta\theta^\top  d\theta \big[P(x) \nabla {f}(x)+f(x) \nabla {P}(x)\big]+\frac{|S^{d-1}|}{2(d+2)}f(x)P(x)\mathfrak{N}_0(x)\epsilon^{d+2} \nonumber \\
&+\frac{\epsilon^{d+4}}{24} \Big( \iota_{*} \int_{S^{d-1}}\|\Second_x(\theta,\theta)\|^2\theta \theta^\top  d\theta \big[P(x) \nabla {f}(x)+f(x) \nabla {P}(x)\big]\nonumber\\
&\qquad\qquad\qquad+ \frac{f(x)P(x)}{2}\int_{S^{d-1}} \|\Second_x(\theta,\theta)\|^2\Second_x(\theta,\theta) d\theta \Big)+O(\epsilon^{d+5} ) \nonumber \\
=&\, \frac{|S^{d-1}|}{d+2} \Big[ \frac{\big[P(x)\iota_*\nabla {f}(x)+f(x)\iota_*\nabla {P}(x)\big]}{d}+\frac{f(x){P}(x)\mathfrak{N}_0(x)}{2}\Big]\epsilon^{d+2}\nonumber\\
&\,+\frac{|S^{d-1}|}{24} \Big[ \iota_{*} \mathfrak{M}_1(x)\big[P(x)\nabla {f}(x)+f(x)\nabla P(x)\big]+ \frac{f(x)P(x)\mathfrak{N}_1(x)}{2} \Big]\epsilon^{d+4}+O(\epsilon^{d+5}) \nonumber\,,
\end{align}
where the second equality holds by the same argument as that for $B_1$ and the fourth equality holds since $\int_{S^{d-1}} \theta\theta^\top  d\theta=\frac{|S^{d-1}|}{d}I_{d\times d}$.

For $D_2$, we only need to explicitly write down the $\epsilon^{d+4}$ term. 
By the same argument as that for $B_1$, we have
\begin{align}
D_2&=\int_{S^{d-1}}\int_0^{\tilde{\epsilon}} 
\Big[\iota_*\theta \big(\frac{1}{6}\nabla^3_{\theta,\theta,\theta}f(x)P(x)+\frac{1}{6}\nabla^3_{\theta,\theta,\theta}P(x)f(x)+\frac{1}{2}\nabla^2_{\theta,\theta}f(x)\nabla_\theta P(x)\nonumber\\
&\qquad\qquad\qquad+\frac{1}{2}\nabla^2_{\theta,\theta}P(x)\nabla_\theta f(x)-\frac{1}{6}\texttt{Ric}_{x}(\theta,\theta)[f(x)\nabla_{\theta}P(x)+\nabla_{\theta}f(x)P(x)]\big)
\nonumber\\
&\qquad+\frac{1}{2}\Second_x(\theta,\theta)\big(\nabla_\theta f(x)\nabla_\theta P(x)+\frac{1}{2}[P(x)\nabla^2_{\theta,\theta} f(x)+f(x)\nabla^2_{\theta,\theta} P(x)]-\frac{1}{6}\texttt{Ric}_{x}(\theta,\theta)f(x)P(x)\big)\nonumber\\
&\qquad+\frac{1}{6}\nabla_\theta\Second_x(\theta,\theta)(P(x)\nabla_{\theta} f(x)+f(x)\nabla_{\theta} P(x)) +\frac{1}{24}f(x)P(x)\nabla_{\theta\theta} \Second_{x}(\theta,\theta)\Big]t^{d+3}dtd\theta\,,\nonumber 
\end{align}
which becomes
\begin{align}
&\frac{\epsilon^{d+4}}{d+4}\int_{S^{d-1}}
\Big[\iota_*\theta \big(\frac{1}{6}\nabla^3_{\theta,\theta,\theta}f(x)P(x)+\frac{1}{6}\nabla^3_{\theta,\theta,\theta}P(x)f(x)+\frac{1}{2}\nabla^2_{\theta,\theta}f(x)\nabla_\theta P(x)\nonumber\\
&\qquad\qquad\qquad+\frac{1}{2}\nabla^2_{\theta,\theta}P(x)\nabla_\theta f(x)-\frac{1}{6}\texttt{Ric}_{x}(\theta,\theta)[f(x)\nabla_{\theta}P(x)+\nabla_{\theta}f(x)P(x)]\big)
\nonumber\\
&\qquad+\frac{1}{2}\Second_x(\theta,\theta)\big(\nabla_\theta f(x)\nabla_\theta P(x)+\frac{1}{2}[P(x)\nabla^2_{\theta,\theta} f(x)+f(x)\nabla^2_{\theta,\theta} P(x)]-\frac{1}{6}\texttt{Ric}_{x}(\theta,\theta)f(x)P(x)\big)\nonumber\\
&\qquad+\frac{1}{6}\nabla_\theta\Second_x(\theta,\theta)(P(x)\nabla_{\theta} f(x)+f(x)\nabla_{\theta} P(x)) +\frac{1}{24}f(x)P(x)\nabla_{\theta\theta} \Second_{x}(\theta,\theta)\Big]d\theta+O(\epsilon^{d+6})\nonumber\,.
\end{align}
We now simplify this complicated expression. The first term on the right hand side of $D_2$ becomes $|S^{d-1}|\mathfrak{J}_f(x)\in\iota_*T_xM$. For the second term on the right hand side of $D_2$, we rewrite it as
\begin{align}
&\int_{S^{d-1}}\Second_x(\theta,\theta)\Big[\nabla_\theta f(x)\nabla_\theta P(x)+\frac{1}{2}P(x)\nabla^2_{\theta,\theta} f(x)\nonumber\\
&\qquad\qquad+\frac{1}{2}f(x)\nabla^2_{\theta,\theta} P(x)-\frac{1}{6}\texttt{Ric}_{x}(\theta,\theta)f(x)P(x)\Big]d\theta\nonumber\\
=&\,\nabla f(x)^{\top}\int_{S^{d-1}}\Second_x(\theta,\theta)\theta \theta^{\top}d\theta \nabla P(x) +\frac{1}{2}P(x)\texttt{tr}\Big(\int_{S^{d-1}} \Second_x(\theta,\theta) \theta\theta^{\top} d\theta \nabla^2 f(x)\Big)\nonumber\\ 
&+\frac{1}{2}f(x)\texttt{tr}\Big(\int_{S^{d-1}} \Second_x(\theta,\theta) \theta\theta^{\top} d\theta \nabla^2 P(x)\Big)
-\frac{1}{6}f(x)P(x)\int_{S^{d-1}} \Second_x(\theta,\theta)\texttt{Ric}_{x}(\theta,\theta) d \theta \nonumber\\
=&\,|S^{d-1}|\Big[\nabla f(x)^{\top}\mathfrak{M}_2(x) \nabla P(x) +\frac{1}{2}\big(P(x)\texttt{tr}(\mathfrak{M}_2(x) \nabla^2 f(x))\nonumber\\
&\qquad\qquad+\frac{1}{2}f(x)\texttt{tr}(\mathfrak{M}_2(x) \nabla^2 P(x))\big)-\frac{1}{6}f(x)P(x)\mathfrak{N}_2(x)\Big]\nonumber,
\end{align}
which is in $(\iota_*T_xM)^\bot$, where we use the equality $u^{\top}Mv=\texttt{tr}(vu^{\top}M)$, where $M$ is a $d\times d$ matrix and $u,v\in\mathbb{R}^d$. For the third term on the right hand side of $D_2$, it simply becomes 
\begin{align}
&\int_{S^{d-1}}\nabla_\theta\Second_x(\theta,\theta)\big(P(x)\nabla_{\theta} f(x)+f(x)\nabla_{\theta} P(x)\big) d\theta\nonumber\\
=\,&P(x)\int_{S^{d-1}}\nabla_\theta\Second_x(\theta,\theta)\theta^{\top} d\theta\nabla f(x)+f(x)\int_{S^{d-1}}\nabla_\theta\Second_x(\theta,\theta)\theta^{\top} d\theta\nabla P(x)\nonumber\\
=\,&|S^{d-1}|\big[P(x)\mathfrak{R}_1(x)\nabla f(x)+f(x)\mathfrak{R}_1(x) \nabla P(x)\big],\nonumber
\end{align}
which might or might not in $\iota_*T_xM$.
Therefore, we have
\begin{align}
D_2=&\,\frac{|S^{d-1}|}{d+4}\Big(\mathfrak{J}_f(x)+\frac{1}{2}\nabla f(x)^{\top}\mathfrak{M}_2(x) \nabla P(x)-\frac{1}{12}f(x)P(x)\mathfrak{N}_2(x)\nonumber\\
&\quad +\frac{1}{4}\big[P(x)\texttt{tr}\big(\mathfrak{M}_2(x) \nabla^2 f(x)\big)+f(x)\texttt{tr}\big(\mathfrak{M}_2(x) \nabla^2 P(x)\big)\big]\nonumber\\
&\quad+\frac{1}{6}\big[P(x)\mathfrak{R}_1(x)\nabla f(x)+f(x)\mathfrak{R}_1(x) \nabla P(x)\big]+\frac{1}{24}f(x)P(x)\mathfrak{R}_2(x) \Big)\epsilon^{d+4}.\nonumber
\end{align}
As a result, by putting the above together, when expressing $\mathbb{E}[(X-\iota(x))f(X)\chi_{B_{\epsilon}^{\mathbb{R}^p}(\iota(x))}(X)]$ as $[\![v_1,v_2]\!]$, we have
\begin{align}
v_1=\,&J_{p,d}^{\top}\mathbb{E}[(X-\iota(x))f(X)\chi_{B_{\epsilon}^{\mathbb{R}^p}(\iota(x))}(X)]\label{Proof:LemmaD4:v1}\\
=\, &\frac{|S^{d-1}|}{d+2} \frac{J_{p,d}^{\top}\big[P(x)\iota_*\nabla {f}(x)+f(x)\iota_*\nabla {P}(x)\big]}{d}\epsilon^{d+2}\nonumber\\
&+\,\frac{|S^{d-1}|}{24} J_{p,d}^{\top}\iota_{*} \Big(\mathfrak{M}_1(x)\big[P(x)\nabla {f}(x)+f(x)\nabla P(x) \big]+f(x)P(x) \mathfrak{R}_0(x)\Big)\epsilon^{d+4}\nonumber\\
&+ \,\frac{|S^{d-1}|}{d+4}J_{p,d}^{\top}\Big(\mathfrak{J}_f(x) +\frac{1}{6}(P(x)\mathfrak{R}_1(x)\nabla f(x)+f(x)\mathfrak{R}_1(x) \nabla P(x)) \nonumber\\
&\qquad\qquad\qquad+\frac{1}{24}f(x)P(x)\mathfrak{R}_2(x)\Big)\epsilon^{d+4}+O(\epsilon^{d+5})\nonumber,
\end{align}
and
\begin{align}
v_2=\,&\bar J_{p,p-d}^{\top}\mathbb{E}[(X-\iota(x))f(X)\chi_{B_{\epsilon}^{\mathbb{R}^p}(\iota(x))}(X)]\label{Proof:LemmaD4:v2}\\
=\,& \frac{|S^{d-1}|}{d+2} \frac{f(x){P}(x)\bar J_{p,p-d}^{\top}\mathfrak{N}_0(x)}{2}\epsilon^{d+2}+\frac{|S^{d-1}|}{24}\frac{f(x)P(x)\bar J_{p,p-d}^{\top}\mathfrak{N}_1(x)}{2}\epsilon^{d+4}\nonumber\\
&+\,\frac{|S^{d-1}|}{d+4}\bar J_{p,p-d}^{\top}\Big(\frac{1}{2}\nabla f(x)^{\top}\mathfrak{M}_2(x) \nabla P(x) +\frac{1}{4}\big[P(x)\texttt{tr}(\mathfrak{M}_2(x) \nabla^2 f(x))+f(x)\texttt{tr}(\mathfrak{M}_2(x) \nabla^2 P(x))\big] \nonumber \\
&\qquad\qquad\qquad- \frac{1}{12}f(x)P(x)\mathfrak{N}_2(x)\Big)\epsilon^{d+4}\nonumber\\
&+\, \frac{|S^{d-1}|}{6(d+4)}\bar J_{p,p-d}^{\top}\big[P(x)\mathfrak{R}_1(x)\nabla f(x)+f(x)\mathfrak{R}_1(x) \nabla P(x)+\frac{1}{4}f(x)P(x)\mathfrak{R}_2(x) \big]\epsilon^{d+4}  +O(\epsilon^{d+5}).\nonumber
\end{align}

\subsection{Proof of Lemma \ref{Lemma:6}}

We show the lemma case by case, and we will recycle the equations shown in (\ref{Expansion:ProofTheorem2:X}). {Note that although the eigenvectors of $C_x$ might not be unique, we will see that the result is independent of the choice of the eigenvectors.}

\medskip
\underline{Case 0 in Condition \ref{Condition:1}.} In this case, by Proposition \ref{Proposition:2}, denote the $i$-th eigenvector of $C_x$ as $u_i=\begin{bmatrix}X_1J_{p,d}^\top e_i+O(\epsilon^2)\\  0_{(p-d)\times 1}\end{bmatrix}$, where $i=1,\ldots,d$ and $X_1\in O(d)$, and the corresponding eigenvalue $\lambda_i=\frac{|S^{d-1}|  {P}(x)}{d(d+2)}\epsilon^{d+2}+O(\epsilon^{d+4})$. By Lemma \ref{Lemma:5} and Lemma \ref{Proposition:2}, we have $1 \leq i \leq d$
\begin{align} 
& \mathbb{E}[(X-\iota(x))\chi_{B_{\epsilon}^{\mathbb{R}^p}(x_k)}(X)] \cdot u_i  \nonumber\\
=&\, \frac{|S^{d-1}|}{d+2}\Big[\!\!\!\Big[ \frac{J_{p,d}^\top \iota_*\nabla {P}(x)}{d}\epsilon^{d+2} +O(\epsilon^{d+4}),\,\frac{{P}(x)\bar{J}_{p,p-d}^\top \mathfrak{N}_0(x)}{2}\epsilon^{d+2}+O(\epsilon^{d+4})\Big]\!\!\!\Big] \cdot [\![X_1J_{p,d}^\top e_i+O(\epsilon^2), 0]\!] \nonumber \\
=&\, \frac{|S^{d-1}|}{d(d+2)}u_i^\top \iota_*\nabla P(x) \epsilon^{d+2}+O(\epsilon^{d+4}) \nonumber\,,
\end{align} 
where the last equality comes from the fact that $\langle J_{p,d}^\top \iota_*\nabla {P}(x), \,X_1J_{p,d}^\top e_i\rangle=e_i^\top J_{p,d}^\top X_1^\top \iota_*\nabla P(x)=u_i^\top \iota_*\nabla P(x)$.
Thus, 
\begin{align} 
\frac{\mathbb{E}[(X-\iota(x))\chi_{B_{\epsilon}^{\mathbb{R}^p}(x_k)}(X)] \cdot u_i}{\lambda_i+\epsilon^{d+\rho}} &= \frac{\frac{|S^{d-1}|}{d(d+2)}u_i^\top \iota_*\nabla {P}(x) \epsilon^{d+2}+O(\epsilon^{d+4})}{\frac{P(x)|S^{d-1}|}{d(d+2)} \epsilon^{d+2}+\epsilon^{d+\rho}+O(\epsilon^{d+4})}\nonumber\\
&=\frac{u_i^\top \iota_*\nabla {P}(x)}{{P}(x)+\frac{d(d+2)}{|S^{d-1}|}\epsilon^{\rho-2}}+O(\epsilon^2)\,,\nonumber
\end{align} 
where the last expansion holds for all chosen regularization order $\rho$. Specifically, when $\rho>2$, it is trivial; when $\rho\leq 2$, $\frac{u_i^\top \iota_*\nabla {P}(x) +O(\epsilon^{2})}{P(x)+\frac{d(d+2)}{|S^{d-1}|}\epsilon^{\rho-2}+O(\epsilon^{2})}-\frac{u_i^\top \iota_*\nabla {P}(x)}{{P}(x)+\frac{d(d+2)}{|S^{d-1}|}\epsilon^{\rho-2}}$ is of order smaller than $\epsilon^2$ since the denominator is dominated by $\epsilon^{\rho-2}$.
Hence, since $u_i$ for an orthonormal set, we have
\begin{align}
\mathbf{T}_{\iota(x)} &=\sum_{i=1}^d\frac{\mathbb{E}[(X-\iota(x))\chi_{B_{\epsilon}^{\mathbb{R}^p}(x_k)}(X)] \cdot u_i}{\lambda_i+\epsilon^{d+\rho}}u_i\nonumber \\
&=\sum_{i=1}^d\big(\frac{u_i^\top \iota_*\nabla {P}(x)}{{P}(x)+\frac{d(d+2)}{|S^{d-1}|}\epsilon^{\rho-2}}+O(\epsilon^2)\big)[\![X_1J_{p,d}^\top e_i+O(\epsilon^2),\, 0]\!]\nonumber\\
&=\Big[\!\!\!\Big[\frac{J_{p,d}^\top \iota_*\nabla {P}(x)}{{P}(x)+\frac{d(d+2)}{|S^{d-1}|}\epsilon^{\rho-2}},\,0\Big]\!\!\!\Big]+[\![O(\epsilon^2),\,0 ]\!].\nonumber
\end{align}

\medskip
\underline{Case 1 in Condition \ref{Condition:1}.}
The eigenvalues of $C_x$ are $\lambda_i=\frac{|S^{d-1}|  {P}(x)}{d(d+2)}(\epsilon^{d+2}+\lambda^{(2)}_i\epsilon^{d+4}+O(\epsilon^{d+6}))$ for $i=1,\ldots,d$ and $\lambda_i=\frac{|S^{d-1}|  {P}(x)}{d(d+2)}\lambda^{(2)}_i\epsilon^{d+4}+O(\epsilon^{d+6})$ for $i=d+1,\ldots,p$. The eigenvectors of $C_x$ are 
\[
u_i=\begin{bmatrix}X_1J_{p,d}^\top e_i\\  0_{(p-d)\times 1}\end{bmatrix}+\epsilon^2 U_x(0)\mathsf{S}e_i+O(\epsilon^4)=[\![ X_1J_{p,d}^\top e_i+O(\epsilon^2),\,O(\epsilon^2)]\!]
\] 
for $i=1,\ldots,d$, where $U_x(0)=\begin{bmatrix}X_1 & 0 \\ 0 & X_2 \end{bmatrix}\in O(p)$, and 
\[
u_i=\begin{bmatrix}0_{d\times 1} \\X_2 \bar{J}_{p,p-d}^\top e_i \end{bmatrix}+\epsilon^2 U_x(0)\mathsf{S}e_i+O(\epsilon^4)=[\![ J_{p,d}^\top U_x(0)\mathsf{S} e_i\epsilon^2+O(\epsilon^4),\,X_2\bar{J}_{p,p-d}^\top e_i+O(\epsilon^2)]\!]
\] 
for $i=d+1,\ldots,p$, $X_1 \in O(d)$ and $X_2 \in O(p-d)$.
For $1 \leq i \leq d$, we have
\begin{align} 
& \mathbb{E}[(X-\iota(x))\chi_{B_{\epsilon}^{\mathbb{R}^p}(x_k)}(X)] \cdot u_i  \nonumber\\
=&\, \frac{|S^{d-1}|}{d+2}\Big[\!\!\!\Big[ \frac{J_{p,d}^\top \iota_*\nabla {P}(x)}{d}\epsilon^{d+2} +O(\epsilon^{d+4}),\,\frac{{P}(x)\bar{J}_{p,p-d}^\top \mathfrak{N}_0(x)}{2}\epsilon^{d+2}+O(\epsilon^{d+4})\Big]\!\!\!\Big] \cdot [\![X_1J_{p,d}^\top e_i+O(\epsilon^2),\,O(\epsilon^2)]\!] \nonumber \\
=&\, \frac{|S^{d-1}|}{d(d+2)}\iota_*\nabla P(x)^\top J_{p,d}X_1J_{p,d}^\top e_i \epsilon^{d+2}+O(\epsilon^{d+4}) \nonumber,
\end{align} 
and hence
\begin{align} 
&\frac{\mathbb{E}[(X-\iota(x))\chi_{B_{\epsilon}^{\mathbb{R}^p}(x_k)}(X)] \cdot u_i}{\lambda_i+\epsilon^{d+\rho}}u_i\nonumber\\
=& \frac{\frac{|S^{d-1}|}{d(d+2)}(\iota_*\nabla P(x))^\top   J_{p,d}X_1J_{p,d}^\top e_i \epsilon^{d+2}+O(\epsilon^{d+4})}{\frac{|S^{d-1}|  {P}(x)}{d(d+2)}\epsilon^{d+2}+\epsilon^{d+\rho}+\lambda^{(2)}_i\epsilon^{d+4}+O(\epsilon^{d+6})}[\![ X_1J_{p,d}^\top e_i+O(\epsilon^2),\,O(\epsilon^2)]\!]\nonumber\\
=&\Big[\!\!\!\Big[\frac{(\iota_*\nabla P(x))^\top  J_{p,d}X_1J_{p,d}^\top e_i}{{P}(x)+\frac{d(d+2)}{|S^{d-1}|}\epsilon^{\rho-2}}  X_1J_{p,d}^\top e_i,\,0\Big]\!\!\!\Big]+[\![O(\epsilon^2),\,O(\epsilon^2)]\!].\nonumber
\end{align} 
Since columns of $J_{p,d}X_1$ form an orthonormal basis of $\iota_*T_xM$, we have
\begin{align} 
&\sum_{i=1}^d\frac{\mathbb{E}[(X-\iota(x))\chi_{B_{\epsilon}^{\mathbb{R}^p}(x_k)}(X)] \cdot u_i}{\lambda_i+\epsilon^{d+\rho}}u_i\nonumber\\
=&\sum_{i=1}^d \Big[\!\!\!\Big[\frac{(\iota_*\nabla P(x))^\top  J_{p,d}X_1J_{p,d}^\top e_i}{{P}(x)+\frac{d(d+2)}{|S^{d-1}|}\epsilon^{\rho-2}}  X_1J_{p,d}^\top e_i,\,0\Big]\!\!\!\Big]+[\![O(\epsilon^2),\,O(\epsilon^2)]\!] \nonumber \\
=& \Big[\!\!\!\Big[\frac{J_{p,d}^\top \iota_*\nabla P(x)}{{P}(x)+\frac{d(d+2)}{|S^{d-1}|}\epsilon^{\rho-2}}+O(\epsilon^2),\,O(\epsilon^2)  \Big]\!\!\!\Big]. \nonumber 
\end{align} 
For $d+1 \leq i \leq p$, similarly we have
\begin{align}
& \mathbb{E}[(X-\iota(x))\chi_{B_{\epsilon}^{\mathbb{R}^p}(x_k)}(X)] \cdot u_i \nonumber\\
=&\, \frac{|S^{d-1}|}{d+2}[\![ \frac{J_{p,d}^\top \iota_*\nabla {P}(x)}{d}\epsilon^{d+2} +O(\epsilon^{d+4}),\,\frac{{P}(x)\bar{J}_{p,p-d}^\top \mathfrak{N}_0(x)}{2}\epsilon^{d+2}+O(\epsilon^{d+4})]\!]\nonumber\\
&\qquad \cdot [\![J_{p,d}^\top U_x(0)\mathsf{S}e_i\epsilon^2+O(\epsilon^4),\,X_2\bar{J}_{p,p-d}^\top e_i +O(\epsilon^2)]\!]  \nonumber \\
=&\,\frac{|S^{d-1}|}{2(d+2)}{P}(x)\mathfrak{N}^\top _0(x)\bar{J}_{p,p-d}X_2\bar{J}_{p,p-d}^\top  e_i  \epsilon^{d+2}+ O(\epsilon^{d+4})  \nonumber\,,
\end{align}
and hence
\begin{align}
& \frac{\mathbb{E}[(X-\iota(x))\chi_{B_{\epsilon}^{\mathbb{R}^p}(x_k)}(X)] \cdot u_i}{\lambda_i+\epsilon^{d+\rho}}u_i \nonumber\\
=&\, \frac{ \mathfrak{N}^\top _0(x)\bar{J}_{p,p-d}X_2\bar{J}_{p,p-d}^\top  e_i  \epsilon^{d+2}+ O(\epsilon^{d+4})}{\frac{2}{d}\lambda^{(2)}_i\epsilon^{d+4}+\frac{2(d+2)}{P(x)|S^{d-1}|}\epsilon^{d+\rho}+O(\epsilon^{d+6})} [\![J_{p,d}^\top U_x(0)\mathsf{S}e_i\epsilon^2+O(\epsilon^4),\,X_2\bar{J}_{p,p-d}^\top e_i +O(\epsilon^2)]\!]\nonumber \\
=&\, 
\Big[\!\!\!\Big[\frac{\mathfrak{N}^\top _0(x)\bar{J}_{p,p-d}X_2\bar{J}_{p,p-d}^\top e_i}{\frac{2}{d}\lambda^{(2)}_i+\frac{2(d+2)}{P(x)|S^{d-1}|}\epsilon^{\rho-4}} J_{p,d}^\top U_x(0)\mathsf{S}e_i,\,\nonumber\\
&\qquad\qquad\frac{\mathfrak{N}^\top _0(x)\bar{J}_{p,p-d}X_2\bar{J}_{p,p-d}^\top e_{i}X_2\bar{J}_{p,p-d}^\top e_{i}}{\frac{2}{d}\lambda^{(2)}_i+\frac{2(d+2)}{P(x)|S^{d-1}|}\epsilon^{\rho-4}}\frac{1}{\epsilon^2}\Big]\!\!\!\Big]+[\![O(\epsilon^2),\,O(1)]\!].\nonumber
\end{align}
As a result, by the fact that $J_{p,d}^\top U_x(0)\mathsf{S}e_i=X_1\mathsf{S}_{12}\bar{J}_{p,p-d}^\top e_i $ when $i=d+1,\ldots,p$, in this case we have
\begin{align}
\mathbf{T}_{\iota(x)}
=&\,\Big[\!\!\!\Big[\frac{J_{p,d}^\top \iota_*\nabla P(x)}{{P}(x)+\frac{d(d+2)}{|S^{d-1}|}\epsilon^{\rho-2}}+\sum_{i=d+1}^p\frac{\mathfrak{N}^\top _0(x)\bar{J}_{p,p-d}X_2\bar{J}_{p,p-d}^\top e_i}{\frac{2}{d}\lambda^{(2)}_i+\frac{2(d+2)}{P(x)|S^{d-1}|}\epsilon^{\rho-4}} X_1\mathsf{S}_{12}\bar{J}_{p,p-d}^\top e_i,\nonumber\\
&\qquad\frac{1}{ \epsilon^2}\sum_{i=d+1}^p\frac{\mathfrak{N}^\top _0(x)\bar{J}_{p,p-d}X_2\bar{J}_{p,p-d}^\top e_{i}}{\frac{2}{d}\lambda^{(2)}_i+\frac{2(d+2)}{P(x)|S^{d-1}|}\epsilon^{\rho-4}}X_2\bar{J}_{p,p-d}^\top e_{i}\Big]\!\!\!\Big]+[\![O(\epsilon^2),\,O(1)]\!].\nonumber
\end{align}

\medskip
\underline{\textbf{Case 2 in Condition \ref{Condition:1}}.}
In this case, the eigenvalues of $C_x$ are 
\begin{align}
\lambda_i=\left\{
\begin{array}{ll}
\frac{|S^{d-1}|  {P}(x)}{d(d+2)}\epsilon^{d+2}+O(\epsilon^{d+4})&\mbox{ for }i=1,\ldots,d\\
\frac{|S^{d-1}|  {P}(x)}{d(d+2)}\lambda^{(2)}_i\epsilon^{d+4}+O(\epsilon^{d+6})&\mbox{ for }i=d+1,\ldots,p-l,\\ 
\frac{|S^{d-1}|  {P}(x)}{d(d+2)}\lambda^{(4)}_i\epsilon^{d+6}+O(\epsilon^{d+8})&\mbox{ for }i=p-l+1,\ldots,p.
\end{array}\right.\nonumber
\end{align}  
Adapt notations from Proposition \ref{Proposition:2} and use
\begin{align}
U_x(0)&=\begin{bmatrix}
X_1 & 0 & 0\\
0 & X_{2,1} & 0\\
0 & 0 & X_{2,2} \\
\end{bmatrix}\in O(p)\,,\quad
\mathsf{S}=
\begin{bmatrix}
\mathsf{S}_{11} & \mathsf{S}_{12,1} & \mathsf{S}_{12,2} \\ 
\mathsf{S}_{21,1} & \mathsf{S}_{22,11}  & \mathsf{S}_{22,12} \\
\mathsf{S}_{21,2} & \mathsf{S}_{22,21} & \mathsf{S}_{22,22}
\end{bmatrix}\in\mathfrak{o}(p),\nonumber
\end{align}
where $X_1 \in O(d)$, $X_{2,1} \in O(p-d-l)$ and $X_{2,2} \in O(l)$. The eigenvectors of $C_x$, on the other hand, are $u_i=\begin{bmatrix}X_1J_{p,d}^\top e_i\\  0_{(p-d)\times 1}\end{bmatrix}+\epsilon^2 U_x(0)\mathsf{S}e_i+O(\epsilon^4)$ 
 for $i=1,\ldots,d$, $u_i=\begin{bmatrix}0_{d\times 1} \\X_{2,1}\tilde{J}^\top e_i \end{bmatrix}+\epsilon^2U_x(0)\mathsf{S}e_i+O(\epsilon^4)$ for $i=d+1,\ldots,p-l$, and $u_i=\begin{bmatrix}0_{d\times 1} \\ X_{2,2}  \bar{J}_{p,l}^\top e_i \end{bmatrix}+\epsilon^2U_x(0)\mathsf{S}e_i+O(\epsilon^4)$ for $i=p-l+1,\ldots,p$.  

Similar to Case 1, we could evaluate $\frac{\mathbb{E}[(X-\iota(x))\chi_{B_{\epsilon}^{\mathbb{R}^p}(x_k)}(X)] \cdot u_i}{\lambda_i+\epsilon^{d+\rho}}u_i$ for $i=1,\ldots, p-l$, and have
\begin{align}
& \sum_{i=1}^{p-l} \frac{\mathbb{E}[(X-\iota(x))\chi_{B_{\epsilon}^{\mathbb{R}^p}(\iota(x_k))}(X)] \cdot u_i}{\lambda_i+\epsilon^{d+\rho}}u_i \nonumber\\
=&\,\Big[\!\!\!\Big[\frac{J_{p,d}^\top \iota_*\nabla P(x)}{{P}(x)+\frac{d(d+2)}{|S^{d-1}|}\epsilon^{\rho-2}}+\sum_{i=d+1}^{p-l}\frac{\mathfrak{N}^\top _0(x)\tilde{J}X_{2,1}\tilde{J}^\top e_i}{\frac{2}{d}\lambda^{(2)}_i+\frac{2(d+2)}{P(x)|S^{d-1}|}\epsilon^{\rho-4}}X_{2,1}\mathsf{S}_{12,1}\tilde{J}^\top e_i,\nonumber\\
&\qquad\frac{1}{\epsilon^2}\sum_{i=d+1}^{p-l}\frac{\mathfrak{N}^\top _0(x)\tilde{J}X_{2,1}\tilde{J}^\top e_{i}}{\frac{2}{d}\lambda^{(2)}_i+\frac{2(d+2)}{P(x)|S^{d-1}|}\epsilon^{\rho-4}}X_{2,1} \tilde{J}^\top e_i\Big]\!\!\!\Big]+[\![O(\epsilon^2),\,O(1)]\!]. \nonumber
\end{align}
For $p-l+1 \leq i \leq p$, base on Proposition \ref{Lemma:5.1}, we have
\begin{align}
\mathfrak{N}^\top _0(x) \bar{J}_{p,p-d}\begin{bmatrix}X_{2,1} & 0 \\ 0 & X_{2,2} \end{bmatrix} \bar{J}_{p,p-d}^\top e_i =\mathfrak{N}^\top _0(x)\bar J_{p,l}X_{2,2}\bar J_{p,l}^\top e_{i}=0\,,\label{Proof:Lemma6:Case2:Cancel}
\end{align}
and hence 
\begin{align}
& \mathbb{E}[(X-\iota(x))\chi_{B_{\epsilon}^{\mathbb{R}^p}(\iota(x_k))}(X)] \cdot u_i\label{Definition:Alphai:ProofC6}\\
=&\, \frac{|S^{d-1}|}{d+2}\Big[\!\!\!\Big[ \frac{J_{p,d}^\top \iota_*\nabla {P}(x)}{d}\epsilon^{d+2} +O(\epsilon^{d+4}),\,\frac{{P}(x)\bar{J}_{p,p-d}^\top \mathfrak{N}_0(x)}{2}\epsilon^{d+2}+O(\epsilon^{d+4})\Big]\!\!\!\Big]\nonumber\\
&\qquad \cdot [\![X_1\mathsf{S}_{12,2}\bar J_{p,l}^\top e_{i}\epsilon^2+O(\epsilon^4),\,X_{2,2}  \bar{J}_{p,l}^\top e_i +O(\epsilon^2)]\!]  \nonumber \\
=&\, \alpha_i \epsilon^{d+4}+ O(\epsilon^{d+6})  \nonumber ,
\end{align}
where we use the fact that $J_{p,d}^\top U_x(0)\mathsf{S}e_i=X_1\mathsf{S}_{12,2}\bar J_{p,l}^\top e_{i}$ when $i=p-l+1,\ldots,p$, and $\alpha_i\in\mathbb{R}$ is the coefficient of the order $\epsilon^{d+4}$ term. Note that the $\epsilon^{d+2}$ term disappears due to (\ref{Proof:Lemma6:Case2:Cancel}). We mention that since $\alpha_i$ will be canceled out in the main Theorem, we do not spell it out explicitly. 
Therefore,
\begin{align}
& \sum_{i=p-l+1}^{p} \frac{\mathbb{E}[(X-\iota(x))\chi_{B_{\epsilon}^{\mathbb{R}^p}(\iota(x_k))}(X)] \cdot u_i}{\lambda_i+\epsilon^{d+\rho}}u_i\nonumber \\
=&\, \sum_{i=p-l+1}^{p}\frac{\alpha_i \epsilon^{d+4}+ O(\epsilon^{d+6}) }{\frac{|S^{d-1}|  {P}(x)}{d(d+2)}\lambda^{(4)}_i\epsilon^{d+6}+\epsilon^{d+\rho}+O(\epsilon^{d+8})} \nonumber\\
&\qquad\qquad\times [\![X_1\mathsf{S}_{12,2}\bar J_{p,l}^\top e_{i}\epsilon^2+O(\epsilon^4),\, X_{2,2} \bar{J}_{p,l}^\top e_i +O(\epsilon^2)]\!]\nonumber \\
=&\, 
\sum_{i=p-l+1}^{p}\frac{\alpha_i }{\frac{|S^{d-1}|  {P}(x)}{d(d+2)}\lambda^{(4)}_i+\epsilon^{\rho-6}} \Big[\!\!\!\Big[X_1\mathsf{S}_{12,2}\bar J_{p,l}^\top e_{i},\,X_{2,2} \bar{J}_{p,l}^\top e_i\frac{1}{\epsilon^2}\Big]\!\!\!\Big] +[\![O(\epsilon^2),\,O(1)]\!].\nonumber
\end{align}
As a result, in this case we have
\begin{align}
\mathbf{T}_{\iota(x)}
=&\,[\![v_1,\,v_2]\!]+[\![O(\epsilon^2),\,O(1)]\!],\nonumber
\end{align}
where 
\begin{align}
v_1=\,&\frac{J_{p,d}^\top \iota_*\nabla P(x)}{{P}(x)+\frac{d(d+2)}{|S^{d-1}|}\epsilon^{\rho-2}}+\sum_{i=d+1}^{p-l}\frac{\mathfrak{N}^\top _0(x)\tilde{J}X_{2,1}\tilde{J}^\top e_i}{\frac{2}{d}\lambda^{(2)}_i+\frac{2(d+2)}{P(x)|S^{d-1}|}\epsilon^{\rho-4}}X_{2,1}\mathsf{S}_{12,1}\tilde{J}^\top e_i\nonumber\\
&\qquad+\sum_{i=p-l+1}^p\frac{\alpha_i }{\frac{|S^{d-1}|  {P}(x)}{d(d+2)}\lambda^{(4)}_i+\epsilon^{\rho-6}} X_1\mathsf{S}_{12,2}\bar J_{p,l}^\top e_{i}\nonumber
\end{align}
and
\begin{align}
v_2=\,&\frac{1}{\epsilon^2}\sum_{i=d+1}^{p-l}\frac{\mathfrak{N}^\top _0(x)\tilde{J}X_{2,1}\tilde{J}^\top e_{i}}{\frac{2}{d}\lambda^{(2)}_i+\frac{2(d+2)}{P(x)|S^{d-1}|}\epsilon^{\rho-4}}X_{2,1}\tilde{J}^\top e_i\nonumber\\
&\quad+\frac{1}{\epsilon^2}\sum_{i=p-l+1}^p\frac{\alpha_i}{\frac{|S^{d-1}|  {P}(x)}{d(d+2)}\lambda^{(4)}_i+\epsilon^{\rho-6}}X_{2,2} \bar{J}_{p,l}^\top e_i.\nonumber
\end{align}

\section{Calculation of Examples}\label{Section:Calculation}

\subsection{Sphere}
The calculation flow could serve as a simplified proof of Theorem \ref{Theorem:t1} under the special manifold setup, so we provide the details here. Consider the unit sphere $S^{p-1} \subset \mathbb{R}^p$. We assume that the center of the sphere is at $[0,\cdots,0, 1]$ the data set $\{x_i\}_{i=1}^n$ is uniformly sampled from $S^{p-1}$, and $x_k$ is at the origin. 
To simplify the calculation, for $v\in\mathbb{R}^p$, denote $v_1\in \mathbb{R}^{p-1}$ to be the first $p-1$ coordinates of $v$ and $v_2\in\mathbb{R}$ to be the last coordinate of $v$, and use the notation
$v=[\![v_1,\,v_2]\!]$.
We  parametrize $S^{p-1}\setminus [0,\cdots,0, 2] $ by the normal coordinates at $x_k$ via
\begin{equation}
\theta t \rightarrow [\![ \theta \sin(t), 1-\cos(t) ]\!]\in S^{p-1}\setminus [0,\cdots,0, 2]\,,\nonumber
\end{equation} 
where $\theta \in S^{p-2} \subset T_{x_k}S^{p-1}\approx \mathbb{R}^{p-1}$ and $t \in [0, \pi)$ is the geodesic distance. The volume form is 
\begin{align}
dV=\sin^{p-2}(t)dt d\theta\,. \nonumber
\end{align}

Denote $r:=r(\epsilon)$ to be the radius of the ball $\exp_{x_k}^{-1}(B_{\epsilon}^{\mathbb{R}^p}(x_k) \cap S^{p-1})$ in $T_{x_k}S^{p-1}$, where $\epsilon$ is assumed to be sufficiently small.
By a direct calculation, we have
\begin{align}
& \mathbf{E}[XX^\top\chi_{B_{\epsilon}^{\mathbb{R}^p}(x_k)}(X)]\nonumber  \\
=&\,  
\begin{bmatrix}
\displaystyle \int_{S^{p-2}}\int_{0}^{r} \theta \theta^\top \sin^2(t) \sin^{p-2}(t) dt d\theta  &  \displaystyle \int_{S^{p-2}}\int_{0}^{r} \theta^\top (\sin(t)-\sin(t)\cos(t)) \sin^{p-2}(t) dt d\theta  \\
\\
\displaystyle \int_{S^{p-2}}\int_{0}^{r} \theta (\sin(t)-\sin(t)\cos(t)) \sin^{p-2}(t) dt d\theta  &  \displaystyle \int_{S^{p-2}}\int_{0}^{r} (1-\cos(t))^2  \sin^{p-2}(t) dt d\theta 
\end{bmatrix} \nonumber\,.
\end{align}
Since
$\int_{S^{p-2}} \theta \theta^\top  d\theta  =\frac{|S^{p-2}|}{p-1} I_{(p-1) \times (p-1)}$  and $\int_{S^{p-2}} \theta d\theta  =0$,
we conclude that
\begin{align}
C_{x_k}=&\, \mathbf{E}[XX^\top\chi_{B_{\epsilon}^{\mathbb{R}^p}(x_k)}(X)] \nonumber \\
=&\,  
\begin{bmatrix}
\displaystyle\big(\frac{|S^{p-2}|}{p-1} \int_{0}^{r}  \sin^p(t)  dt \big)  I_{(p-1) \times (p-1)} & 0\\
\\
0 & \displaystyle |S^{p-2}| \int_{0}^{r} (1-\cos(t))^2  \sin^{p-2}(t) dt \\
\end{bmatrix} \nonumber,
\end{align}
which is a diagonal matrix containing the eigenvalues of $C_{x_k}$, and we can choose $\{e_i\}_{i=1}^p$ to be its orthonormal eigenvectors.
Next, we have
\begin{align}
 \mathbf{E}[X\chi_{B_{\epsilon}^{\mathbb{R}^p}(x_k)}(X)]  =&\,   \big[\!\!\big[ \int_{S^{p-2}}\int_{0}^{r} \theta \sin^{p-1}(t)  dt d\theta  , \,  \int_{S^{p-2}}\int_{0}^{r} (1-\cos(t))\sin^{p-2}(t) dt d\theta  \big]\!\!\big]  \nonumber \\
=&\,    [\![ 0  , \,  |S^{p-2}| \int_{0}^{r} (1-\cos(t))\sin^{p-2}(t) dt ]\!].  \nonumber
\end{align}
We now choose $\rho=8$. Therefore, by definition,
\begin{align}
\mathbf{T}_{x_k}&\,=\mathcal{I}_{\epsilon^{p+5}}(C_{x_k})\big[\mathbb{E}X\chi_{B_{\epsilon}^{\mathbb{R}^p}(x_k)}\big]=\big[\!\!\big[ 0, \,\frac{\int_{0}^{r} (1-\cos(t))\sin^{p-2}(t) dt}{\int_{0}^{r} (1-\cos(t))^2  \sin^{p-2}(t) dt+\epsilon^{p+7}} \big]\!\!\big]\nonumber\\
&\,=\big[\!\!\big[ 0, \,\frac{\int_{0}^{r} (1-\cos(t))\sin^{p-2}(t) dt}{\int_{0}^{r} (1-\cos(t))^2  \sin^{p-2}(t) dt+r^{p+7}+O(r^{p+9})} \big]\!\!\big],\nonumber
\end{align} 
where the last equality holds since $r=r(\epsilon)=\epsilon+O(\epsilon^3)$ and hence $\epsilon^{p+7}=r^{p+7}+O(r^{p+9})$. 
Thus, the kernel centered at $x_k=0$ and evaluated at $y=[\![ \theta \sin(t), 1-\cos(t) ]\!]\in \mathbb{R}^p$ satisfies
\begin{align} 
&\,K_{\texttt{LLE}}(x_k,y) =\, 1-[\![ \theta \sin(t), 1-\cos (t) ]\!] \cdot \mathbf{T}_{x_k}\nonumber\\
 =&\, 1- (1-\cos (t))\frac{\int_{0}^{r} (1-\cos(t))\sin^{p-2}(t) dt}{\int_{0}^{r} (1-\cos(t))^2  \sin^{p-2}(t) dt+r^{p+7}+O(r^{p+9})}  \nonumber \\
  =&\, 1-(1-\cos(t)) \bigg(\frac{2(p+3)}{(p+1)r^2}+(\frac{p^2+14p-3}{6(p+1)(p+5)}) +O(r^2)\bigg) \nonumber.
\end{align}
Suppose $f \in C^5 (S^{p-1})$, we are going to calculate $\frac{\int K_{\texttt{LLE}}(x_k,y)f(y) dV(y)}{\int K_{\texttt{LLE}}(x_k,y) dV(y)}-f(x_k)$.
The evaluation of $\int K_{\texttt{LLE}}(x_k,y) dV(y)$ is direct, and we have
\begin{align} 
 &\int K_{\texttt{LLE}}(x_k,y) dV(y)\nonumber \\
=&\,  \int_{S^{p-2}}\int_{0}^{r}  \bigg( 1-(1-\cos(t)) \bigg(\frac{2(p+3)}{(p+1)r^2}+(\frac{p^2+14p-3}{6(p+1)(p+5)}) +O(r^2)\bigg) \bigg) \sin^{p-2}(t) dt d\theta  \nonumber \\
=&\, \Big(\frac{4|S^{p-2}|}{(p+1)(p^2-1)}\Big)r^{p-1}+O(r^{p+1}) \nonumber\,.
\end{align} 
On the other hand, we have
\begin{align}
&\int K_{\texttt{LLE}}(x_k,y)(f(y)-f(x_k)) dV(y) \nonumber\\
=&\,  \int_{S^{p-2}}\int_{0}^{r} (\nabla_\theta f(x_k) t + \frac{1}{2}\nabla^2_{\theta\theta} f(x_k) t^2+\frac{1}{6}\nabla^3_{\theta\theta\theta} f(x_k) t^3+\frac{1}{24}\nabla^4_{\theta\theta\theta\theta}  f(x_k) t^4+O(t^5))  \nonumber \\
&\quad \times\bigg(1-(1-\cos(t)) \Big[\frac{2(p+3)}{(p+1)r^2}+(\frac{p^2+14p-3}{6(p+1)(p+5)}) +O(r^2)\Big] \bigg)\sin^{p-2}(t) dt d\theta  \nonumber .
\end{align}
We calculate each part in the above integration by using the symmetry of sphere $S^{p-2}$ in the tangent space. Specifically, we have
\begin{align}
& \int_{S^{p-2}}\int_{0}^{r} \Big(\nabla_\theta f(x_k) t + \frac{1}{2}\nabla^2_{\theta\theta} f(x_k) t^2+\frac{1}{6}\nabla^3_{\theta\theta\theta} f(x_k) t^3 \nonumber\\
&\qquad+\frac{1}{24}\nabla^4_{\theta\theta\theta\theta}  f(x_k) t^4+O(t^5)\Big) \sin^{p-2}(t) dt d\theta  \nonumber\\
= &\, \frac{\int_{S^{p-2}} \nabla^2_{\theta\theta} f(x_k)  d\theta }{2(p+1)}r^{p+1}+\big(\frac{\int_{S^{p-2}} \nabla^4_{\theta\theta\theta\theta}  f(x_k) d\theta }{24(p+3)}-\frac{(p-2)\int_{S^{p-2}} \nabla^2_{\theta\theta} f(x_k)  d\theta }{12(p+3)}\big)r^{p+3}+O(r^{p+5}) \nonumber 
\end{align}
and
\begin{align}
& \int_{S^{p-2}}\int_{0}^{r} \Big(\nabla_\theta f(x_k) t + \frac{1}{2}\nabla^2_{\theta\theta} f(x_k)  t^2+\frac{1}{6}\nabla^3_{\theta\theta\theta} f(x_k)  t^3  \nonumber\\
&\qquad+\frac{1}{24}\nabla^4_{\theta\theta\theta\theta}  f(x_k) t^4+O(t^5)\Big)(1-\cos(t)) \sin^{p-2}(t) dt d\theta \nonumber \\
= &\, \frac{\int_{S^{p-2}} \nabla^2_{\theta\theta} f(x_k)  d\theta }{4(p+3)}r^{p+3}+\frac{\int_{S^{p-2}} \nabla^4_{\theta\theta\theta\theta}  f(x_k) -(2p-3)\nabla^2_{\theta\theta} f(x_k) d\theta }{48(p+5)}r^{p+5}+O(r^{p+7}) \nonumber.
\end{align}
Due to $\frac{2(p+3)}{(p+1)r^2}$, the term of order $r^{p+1}$ is cancelled and we obtain
\begin{align}
&\int K_{\texttt{LLE}}(x_k,y)(f(y) -f(x_k)) dV(y) \nonumber  \\
=&\, \frac{-1}{6(p+1)(p+3)(p+5)}\bigg(\int_{S^{p-2}} \nabla^4_{\theta\theta\theta\theta}  f(x_k) d\theta +\int_{S^{p-2}}\nabla^2_{\theta\theta}  f(x_k) d\theta \bigg) r^{p+3}+O(r^{p+5}). \nonumber
\end{align}
We use the fact that $r=r(\epsilon)=\epsilon+O(\epsilon^3)$ and summarize the result as follows:
\begin{align}
& \frac{\int K_{\texttt{LLE}}(x_k,y)f(y) dV(y) }{\int K_{\texttt{LLE}}(x_k,y) dV(y)}-f(x_k)
= \frac{\int  K_{\texttt{LLE}}(x_k,y)(f(y)-f(x_k)) dV(y) }{\int K_{\texttt{LLE}}(x_k,y) dV(y) }  \nonumber  \\
=&\,\frac{-(p^2-1)}{24|S^{p-2}|(p+3)(p+5)}\bigg(\int_{S^{p-2}} \nabla^4_{\theta\theta\theta\theta}  f(x_k) d\theta +\int_{S^{p-2}} \nabla^2_{\theta\theta}  f(x_k) d\theta \bigg) \epsilon^4+O(\epsilon^6). \nonumber
\end{align}
Finally, if we use formulas
\begin{align}
\left\{
\begin{array}{ll}
\displaystyle\int_{S^{p-2}} x_i^4 d\theta= \frac{3}{p+1}|S^{p-2}| &\\
\displaystyle\int_{S^{p-2}} x_i^2x_j^2 d\theta= \frac{1}{p+1}|S^{d-2}| & i \not=j \\
\displaystyle\int_{S^{p-2}} x_k^2x_ix_j d\theta=0 & i \not=j\,,
\end{array}\right.\label{Proof:Expression:FourthOrderSpherical}
\end{align}
we obtain the expansion (\ref{Example:Sp:Expansion})

\subsection{Torus}\label{Section:Example:Torus}
Consider the torus $\mathbb{T}^2 \subset R^3$, which is the level set of
\begin{equation}
(1-\sqrt{z^2+y^2})^2+x^2=\frac{1}{4},\nonumber
\end{equation}
where $(x,y,z)\in \mathbb{R}^3$.
In other words, the distance from the center of the tube to the center of the torus is $1$, and the radius of the tube is $\frac{1}{2}$. Suppose the data set $\{x_i\}_{i=1}^n$ is sampled from $\mathbb{T}^2$ based on a p.d.f. $P$ and $\epsilon$ is sufficiently small. Let $\{e_1,e_2,e_3\}$ be the standard orthonormal basis of $\mathbb{R}^3$. We calculate the asymptotical result of the LLE at two different typical points on the torus by choosing $\rho=5$.

First, we consider $x_k=(0,0,-\frac{3}{2})$.  We identify $T_{x_k}\mathbb{T}^2$ as the subspace of $\mathbb{R}^3$ which is generated by $\{e_1, e_2\}$. Observe that $e_3$ is the unit normal vector at $x_k$. 
Note that $B_{\epsilon}^{\mathbb{R}^3}(x_k) \cap \mathbb{T}^2$ can be parametrized locally as
\begin{align}
\phi(x,y): D(0) \subset T_{x_k}\mathbb{T}^2 \rightarrow B_{\epsilon}^{\mathbb{R}^3}(x_k) \cap \mathbb{T}^2,\nonumber
\end{align}
where $D(0)$ is the unit disk centered at $0$, 
so that $\phi(0,0)=x_k$ and $\phi(x,y)=(x,y,z)$, where $z=z(x,y)=-\sqrt{\Big(1+\frac{1}{2}\sqrt{1-4x^2}\Big)^2-y^2}$.
Thus, locally we have
\begin{equation}
z=-\frac{3}{2}+x^2+\frac{1}{3}y^2+x^4+\frac{2}{9}x^2y^2+\frac{1}{27}y^4+\mbox{[higher order terms]}.\nonumber
\end{equation}
Let $\theta=(x,y) \in S^1 \subset T_{x_k}\mathbb{T}^2$. Since $T_{x_k}\mathbb{T}^2$ is identified as the subspace of $\mathbb{R}^3$, which is generated by $\{e_1, e_2\}$, we have
\begin{align}
& \Second_{x_k}(\theta,\theta)=(2x^2+\frac{2}{3}y^2)e_3\,,\quad \langle \nabla_{\theta}\Second_{x_k}(\theta,\theta),e_3 \rangle =0\,,\nonumber \\
& \mathfrak{N}_0(x_k)=\frac{1}{2\pi}\int_{S^1}\Second_{x_k}(\theta,\theta) d\theta=\frac{4}{3}e_3\,, \quad \int_{S^1} \|\Second_{x_k}(\theta,\theta)\|^2 d\theta= 4\pi.\nonumber
\end{align}
These quantities, when plugged in to Proposition \ref{Proposition:1}, lead to 
\begin{equation}
\mathbf{E}[(X-x_k)(X-x_k)^\top\chi_{B_{\epsilon}^{\mathbb{R}^3}(x_k)}(X-x_k)] =  \frac{\pi P(x_k)}{4} \epsilon^{4} \Big({
\left[ \begin{array}{cc}
I_{d \times d} & 0  \\
0& 0  \\
\end{array}
\right ]}+\epsilon^{2}
{\left[ \begin{array}{cc}
M^{11} & 0  \\
0 &  \frac{2}{3} \\
\end{array}
\right ]}
+O(\epsilon^{4})\Big).\nonumber
\end {equation}
Take $f \in C^3(\mathbb{T}^2)$. We parametrize $T_{x_k}\mathbb{T}^2$ by Euclidean coordinate and $\mathbb{T}^2$ is locally parametrized by the exponential map, then from Theorem \ref{Theorem:t1}, we have
 
\begin{align}
\mathfrak{H}_f(x)=&\, \frac{1}{2\pi}\int_{S^1}  (\partial^2_x f (x_k)x^2+\partial^2_{xy} f (x_k) xy+\partial^2_y (f (x_k)y^2 )(2x^2+\frac{2}{3}y^2)  d\theta   e^3 \nonumber\\
=&\, (\frac{5}{6}\partial^2_x f (x_k)+\frac{1}{2}\partial^2_y f(x_k))e_3 ,\nonumber 
\end{align}
and hence
\begin{equation}
\lim_{n \rightarrow \infty} \sum_{j=1}^n w_k(j)f(x_j)=f(x_k)-\bigg(\frac{1}{24}\partial^2_x f (x_k)-\frac{1}{8}\partial^2_y f(x_k)\bigg)\epsilon^2+O(\epsilon^4)\,.\nonumber
\end{equation}

Second, we consider $x_k=(0,0,-\frac{1}{2})$. We identify $T_{x_k}\mathbb{T}^2$ as the subspace of $\mathbb{R}^3$ which is generated by $\{e_1, e_2\}$. Let $\theta=(x,y) \in S^1 \subset T_{x_k}\mathbb{T}^2$, we use the similar method as before to show that 
\begin{align}
& \Second_{x_k}(\theta,\theta)=(-2x^2+2y^2)e_3 \,.\nonumber
\end{align}
Hence, we have $\mathfrak{N}_0(x)=0$. From Theorem \ref{Theorem:t1}, we conclude
\begin{equation}
\lim_{n \rightarrow \infty} \sum_{j=1}^n w_k(j)f(x_j)=f(x_k)+\frac{1}{8}\Delta f(x_k)\epsilon^2+O(\epsilon^4).\nonumber
\end{equation}
The evaluation of the LLE at these two typical points again shows that when the regularization term is chosen too small, the asymptotical behavior of the LLE is controlled by the curvature.

\end{document}